\documentclass[a4paper]{article}

\usepackage{verbatim}

\usepackage{url}
\usepackage[numbers]{natbib}
\newcommand\citeayn[2][]{\textnormal{\citeauthor*{#2} (\citeyear{#2}) \cite[#1]{#2}}}

\newcommand\citeyn[2][]{\textnormal{(\citeyear{#2})~\cite{#2}}}
\newcommand\citeay[2][]{\textnormal{\citeauthor*{#2} (\citeyear{#2})}}

\usepackage{amsmath}
\usepackage{amssymb}
\usepackage{mathabx}
\usepackage{amsthm}
\usepackage{latexsym}
\usepackage{enumitem}
\usepackage{eurosym}
\usepackage{dsfont}
\usepackage{appendix}
\usepackage{color} 

\usepackage{frcursive}
\usepackage[utf8]{inputenc}
\usepackage{geometry}
\usepackage{multirow}
\usepackage{todonotes}
\usepackage{lmodern}
\usepackage{anyfontsize}
\usepackage{stmaryrd}
\usepackage{cases}
\usepackage{calc}
\usepackage{pdfpages}
\usepackage{indentfirst}

\usepackage[unicode]{hyperref}
\usepackage[nameinlink]{cleveref}
\newcommand*{\email}[1]{\href{mailto:#1}{\nolinkurl{#1}}} 

\Crefname{equation}{Condition}{Conditions}

\usepackage{mathrsfs}

\usepackage[english]{babel}
\usepackage[english=british]{csquotes}

\DeclareMathOperator*{\essinf}{ess\,inf}
\DeclareMathOperator*{\argmax}{arg\,max}

\definecolor{red}{rgb}{0.7,0.15,0.15}
\definecolor{green}{rgb}{0,0.5,0}
\definecolor{blue}{rgb}{0,0,0.7}
\hypersetup{colorlinks, linkcolor={blue},citecolor={green}, urlcolor={red}}
			
\makeatletter \@addtoreset{equation}{section}

\usepackage[justification=centering]{caption}
\newtheorem{theorem}{Theorem}[section]
\newtheorem{assumption}[theorem]{Assumption}

\newtheorem{lemma}[theorem]{Lemma}
\newtheorem{proposition}[theorem]{Proposition}

\newtheorem{definition}[theorem]{Definition}
\newtheorem{remark}[theorem]{Remark}

\makeatletter
\renewenvironment{proof}[1][\relax]{\par
  \pushQED{\qed}%
  \normalfont \topsep6\p@\@plus6\p@\relax
  \trivlist
  \item[\hskip\labelsep\itshape
    \ifx#1\relax \proofname\else\proofname{} of #1\fi\@addpunct{.}]\ignorespaces
}{%
  \popQED\endtrivlist\@endpefalse
}
\makeatother

\DeclareUnicodeCharacter{014D}{\=o}
\def \D{\mathbb{D}}
\def \E{\mathbb{E}}
\def \F{\mathbb{F}}
\def \G{\mathbb{G}}
\def \H{\mathbb{H}}
\def \I{\mathbb{I}}

\def \M{\mathbb{M}}
\def \N{\mathbb{N}}

\def \P{\mathbb{P}}

\def \R{\mathbb{R}}
\def \S{\mathbb{S}}

\def \U{\mathbb{U}}
\def \V{\mathbb{V}}

\def \X{\mathbb{X}}

\def\Pk{{\mathfrak P}}
\def\Xk{{\mathfrak X}}

\def\Ac{{\cal A}}

\def\Cc{{\cal C}}

\def\Ec{{\cal E}}
\def\Fc{{\cal F}}
\def\Gc{{\cal G}}
\def\Hc{{\cal H}}

\def\Kc{{\cal K}}

\def\Mc{{\cal M}}

\def\Oc{{\cal O}}
\def\Pc{{\cal P}}

\def\Sc{{\cal S}}

\def\Uc{{\cal U}}
\def\Vc{{\cal V}}

\def\Xc{{\cal X}}
\def\Yc{{\cal Y}}
\def\Zc{{\cal Z}}

\def\eps{\varepsilon}
\def\erm{\mathrm{e}}
\def\drm{\mathrm{d}}

\def\eps{\varepsilon}

\usepackage{authblk}
\title{Continuous--time incentives in hierarchies\thanks{Research supported by the ANR project PACMAN ANR--16--CE05--0027, the FACE Foundation -- Thomas Jefferson Fund \& the Mobility Grant of Université Gustave Eiffel.}}

\author{Emma {\sc Hubert}\thanks{ \email{emma.hubert@univ-paris-est.fr}}}
\affil{LAMA, Universit\'e Gustave Eiffel, Marne--la--Vall\'ee, France.}

\date{\today}
\begin{document}

\maketitle

\begin{abstract}

This paper studies continuous--time optimal contracting in a hierarchy problem which generalises the model of \citeayn{sung2015pay}. The hierarchy is modeled by a series of interlinked principal--agent problems, leading to a sequence of Stackelberg equilibria. More precisely, the principal can contract with the managers to incentivise them to act in her best interest, despite only observing the net benefits of the total hierarchy. Managers in turn subcontract with the agents below them. Both agents and managers independently control in continuous time a stochastic process representing their outcome. First, we show through a continuous--time adaptation of \citeauthor{sung2015pay}'s model that, even if the agents only control the drift of their outcome, their manager controls the volatility of their continuation utility. This first simple example justifies the use of recent results on optimal contracting for drift and volatility control, and therefore the theory of second--order backward stochastic differential equations, developed in the theoretical part of this paper, dedicated to a more general model. The comprehensive approach we outline highlights the benefits of considering a continuous--time model and opens the way to obtain comparative statics. We also explain how the model can be extended to a large--scale principal--agent hierarchy. Since the principal's problem can be reduced to only an $m$--dimensional state space and a $2m$--dimensional control set, where $m$ is the number of managers immediately below her, and is therefore independent of the size of the hierarchy below these managers, the dimension of the problem does not explode.

\bigskip


\medskip


\medskip

\end{abstract}

% \tableofcontents

\section{Introduction}\label{sec:introduction}

\textcursive{A little bit of history.}\vspace{-0.2em} The desire to optimise the organisation of work, in a scientific manner, has its origins in the early $20$th century, through the work of Frederick Winslow Taylor, pioneer of the theory of \textit{scientific management}. The main objective of this theory is to improve economic efficiency, especially labor productivity, and was one of the earliest attempts to apply science to the engineering of management processes. The main objective of Taylor's model, developed in his monograph entitled \textit{The principles of scientific management}, could be summarised as: \textit{How to make workers perform in the employer's interest, i.e., in the most cost--efficient way and with the least possible resistance?} To answer this question, Taylor promotes the supervision of workers, in opposition to \blockquote[{\citeayn[pp. 34]{taylor1911principles}}][]{the management of initiative and incentive}. 

\medskip

Nevertheless, during the course of the 20th century, management has noticeably evolved. 
This shift of paradigm is primarily due to the fact that the very nature of work has changed. Indeed, with the introduction of new technologies at all levels of production, the predominance of indirect labour requires forms of management that break with the classical model of Taylorian organisation of work. 
Furthermore, the recognition of the importance of the psychological climate, and in particular the idea that the recognition of workers stimulates their productivity, also represents an important advance compared to the Taylorian approach.
Both the progressive elimination of simple jobs and the desire to empower the employee have made the supervision of workers difficult and even counterproductive. These two transformations imply that management nowadays relies more on employee initiative, and on the development of incentives to bring the interests of the employee and the employer together, than on the supervision of workers. However, the hierarchical organisation of work \textit{à la} Taylor is still considered as the standard structure in companies, and only a few dare to adopt a different organisation.

\medskip

\textcursive{A little bit of context.}\vspace{-0.2em} In an organisation, a hierarchy usually consists of a power entity at the top with subsequent levels of power underneath. This structure is the dominant mode in our contemporary society. Indeed, most companies, governments, and even criminal organisations have a hierarchical structure, with different levels of management or authority. This particular structure of organisations raises many questions: on its efficiency, its cost, its optimal size... To answer these questions, an abundant literature has emerged in the last century in a wide variety of fields, from philosophy to mathematics, through social and management sciences. 
The first mathematical model for the study of the optimal structure of a hierarchy seems to be the work of \citeayn{williamson1967hierarchical}, but, as he mentioned, this question, which presents a serious dilemma for business theory, was originally introduced by Knight in 1921 (see \citeayn{knight2012risk} for a recent edition). Many authors have followed this trend, including the models of \citeauthor{calvo1978supervision} \citeyn{calvo1978supervision, calvo1979hierarchy} and \citeayn{keren1979optimum}, as well as a generalisation by \citeayn{qian1994incentives} to take into account the notion of \textit{incentives}. 
% As mentioned in \cite{qian1994incentives}, these models use the approach of the \textit{grand contract}, which differs from the more recent approach that interests us here, namely the approach of the \textit{network of contracts}.

\medskip

The first attempt to define a mathematical framework for incentives in management is attributed to \citeayn{barnard1938functions}. In particular, he advocates the need to create hierarchical relationships within organisations. Although he highlighted the serious issues associated with moral hazard, this very concept was introduced into the literature on management control almost thirty years later by \citeayn{arrow1963uncertainty}.
Mathematical models on incentive theories then became more widespread in the 1970s, especially through the work of \citeayn{mirrlees1971exploration}, and were applied a few years later to hierarchical organisations by \citeayn{stiglitz1975incentives} and \citeayn{mirrlees1976optimal}. Incentive theory is strongly related to contract theory and principal--agent problems, and is associated with a vast literature that cannot be mentioned here for the sake of conciseness.\footnote{We refer the interested reader to the seminal books by \citeayn{bolton2005contract}, \citeayn{salanie2005economics}, or \citeayn{laffont2009theory} for more references.}
In the case of a hierarchy, we are dealing with a succession of interlinked principal--agent problems, or in other words, a sequence of nested Stackelberg equilibria. 
The interest of this mathematical formalism lies in the modelling of information asymmetries within a hierarchy, whether they are \textit{ex--ante} (adverse selection) or \textit{ex--post} (moral hazard) the signing of contracts between the entities constituting the hierarchy. Works in this direction include, among others, \citeayn{tirole1986hierarchies}, \citeayn{demski1987hierarchical}, \citeayn{baiman1987optimal} and \citeayn{kofman1993collusion} on collusion and auditing within a hierarchy, as well as \citeayn{melumad1995hierarchical}, \citeayn{mcafee1995organizational}, \citeayn{laffont1997firm}, \citeayn{mookherjee2006decentralization} on adverse selection. 
In our framework, we will focus on moral hazard within a hierarchy, as in the work of \citeayn{laffont1990analysis}, \citeayn{yang1995degree}, \citeayn{macho--stadler1998centralized}, \citeayn{itoh2001job} and \citeayn{jost2010organization}.\footnote{The essay by \citeayn{miller2006principal} presents, however, some limitations to the use of incentives within a hierarchy, through a simple principal--agent model.} 
However, it should be noted that the above--mentioned models are discrete--time models, mostly consisting of a single period.

\medskip

\textcursive{Moving to continuous--time.}\vspace{-0.2em} In the late 1980s, the literature on contract theory expanded to include continuous--time models. The first, and seminal, paper on continuous--time principal--agent problems is by \citeayn{holmstrom1987aggregation}. This work was then extended, and main contributors in these regards are \citeayn{schattler1993first}, \citeayn{sannikov2008continuous}, \citeayn{biais2010large} as well as \citeayn{cvitanic2012contract}.\footnote{We can also mention in a non--exhaustive way the works of \citeauthor{sung1995linearity} \citeyn{sung1995linearity,sung1997corporate}, \citeauthor{muller1998first} \citeyn{muller1998first,muller2000asymptotic}, \citeayn{hellwig2002discrete} and \citeayn{hellwig2007role}, that are based on an extension of the first--order approach, popular in static cases. More recently, 
\citeauthor{williams2009dynamic} \citeyn{williams2009dynamic,williams2011persistent,williams2015solvable} and \citeauthor*{cvitanic2006optimal} \citeyn{cvitanic2006optimal,cvitanic2008principal,cvitanic2009optimal} characterise the optimal compensation for more general utility functions by using the stochastic maximum principle and forward--backward stochastic differential equations.}
More recently, \citeayn{cvitanic2018dynamic} have developed a general theory that allows to address a wide spectrum of principal--agent problems. The basic idea is to identify a sub--class of contracts offered by the principal, which are revealing in the sense that the best--reaction function of the agent, and his optimal control, can be computed straightforwardly, and then proving that restricting one's attention to this class is without loss of generality. With this approach, the problem faced by the principal becomes a standard optimal control problem. More importantly, this method allows one to address volatility control problems. It has subsequently been extended and applied in many different situations. We can mention in a non--exhaustive way the applications to finance by \citeayn{cvitanic2017moral} and 
%\citeayn{eleuch2018optimal}, and \citeayn{baldacci2019market}
\citeayn{cvitanic2018asset};
the works of \citeayn{aid2019optimal} and \citeayn{alasseur2019principal} for applications related to the energy sector; as well as other various extensions, \textit{e.g.}, the works of
%applications, \textit{e.g.}, the works of 
%\citeayn{hajjej2017optimal} and \citeayn{kharroubi2020regulation}, and other particular extensions by 
\citeayn{hernandezsantibanez2019contract} and \citeayn{hu2019continuous}.

\medskip

Recently, principal--agent problems in continuous time have been extended to models with several principals, through the works of  \citeayn{mastrolia2018principal} and \citeayn{hu2019principal} for example. In our context, we are particularly interested in the extension to several agents, as by \citeayn{koo2008optimal}, \citeayn{elie2019contracting} and \citeayn{baldacci2019optimal}, and possibly to a continuum of agents with mean--field interactions by \citeayn{elie2018tale}, \citeayn{carmona2018finite} and \citeayn{elie2019mean}. This latter extension to a large number of agents is a significant step towards the application of continuous--time contract theory to hierarchies. Nevertheless, this type of model seems for the moment to be countable on the fingers of a single hand.
First, \citeayn{miller2015optimal} consider a hierarchy of $N+1$ players, each with a principal--agent relationship. Using the approach of \citeayn{evans2015concavity}, they identify conditions under which a dynamic programming construction of an optimal contract can be reduced to only a one--dimensional state space and one--dimensional control set, independent of the size of the hierarchy.
However, the approach in \cite{evans2015concavity} to characterise optimal contracts in continuous time is less general than the one in \cite{cvitanic2018dynamic}, on which we will rely in this present paper. In particular, it does not allow for volatility control, which seems inescapable in our framework. 
Then, \citeayn{li2018forward} develop a method using forward--backward stochastic differential equations to characterise the equilibrium of a generalised Stackelberg game with multi--level hierarchy in a linear--quadratic setting. Finally, \citeayn{keppo2020dynamic} model the relationships between an investor, a partner, and a fund manager as a hierarchical principal--agent problem. More precisely in their model, the manager is an agent for the partner, the partner is a principal for the manager and an agent for the investor, and the investor is a principal for the partner. The framework is similar to the two aforementioned models, but the approach is different and related to \citeayn{cvitanic2017moral} (and thus to \cite{cvitanic2018dynamic}), to take into account the fact that the partner controls the volatility of the output. Nevertheless, in these three hierarchical models in continuous time, it is assumed that the entities of the hierarchy control and observe the same output process, while moral hazard prevents them from directly observing the controls. 

\medskip

\textcursive{Sung's model.}\vspace{-0.2em} The present work is inspired by the model developed by \citeayn{sung2015pay}. In this model, a \textit{top} manager is hired by a principal to subcontract with $N$ \textit{middle} managers (agents). Each worker (top and middle managers) controls his own output process, and all outputs are assumed to be independent.
His model includes a \textit{bi--level} moral hazard. First, the (top) manager does not observe the effort of the agents, but only the resulting outputs. Second, the principal observes only the total benefit of the hierarchy, \textit{i.e.}, the difference between the sum of the outputs of all workers and the sum of the contracts paid to the agents. 
Instead of studying a continuous--time version of the model, \citeauthor{sung2015pay} considers that the one--period model is simpler and \textit{without loss of generality}: 
\blockquote[{\citeayn[pp. 2]{sung2015pay}}][]{[f]or ease of exposition and without loss of generality, we formulate a discrete--time model which is analogous to its continuous--time counterpart}. 
Extending the reasoning of \citeayn{holmstrom1987aggregation}, he therefore restricts the study to \textit{linear contracts}, in the sense that they are linear with respect to the outcome, and states that \blockquote[{\citeayn[pp. 3]{sung2015pay}}][]{[t]his assumption is without loss of generality, as long as our results are interpreted in the context of continuous--time models}.

\medskip

However, while the restriction to linear contracts can be justified in \citeauthor{sung2015pay}'s framework for the first Stackelberg equilibrium, this is no longer the case for the contract offered by the principal to the (top) manager. More precisely, even if the workers are only controlling the drift of their outcomes, the manager controls both the drift and the volatility of the net benefit. Therefore, according to the work of \citeayn{cvitanic2018dynamic}, it appears that the type of contracts considered by \citeauthor{sung2015pay} is sub--optimal. Indeed, in continuous time, it is not sufficient to limit oneself to linear contracts (in the sense of \citeayn{holmstrom1987aggregation}) when the volatility of the state variable is controlled. More precisely, the optimal form of contracts should contain an additional part indexed on the quadratic variation of the net benefit. However, in the one--period model of \citeauthor{sung2015pay}, this controlled quadratic variation cannot be estimated (unlike in continuous time), which leads to a fundamental gap between these two frameworks. From our point of view, this gap motivates a full study of \citeauthor{sung2015pay}'s model in continuous time.

\medskip

\textcursive{Main contributions.} In this paper, we provide a systematic method to solve any hierarchy problems of this sort, including those in which workers can also control the volatility of the output process, and not just the drift. 
The main result is that optimal contracts in continuous time for the manager are not those expected by \citeayn{sung2015pay}, since they have to be indexed on the quadratic variation of the state variable, in the spirit of \cite{cvitanic2018dynamic}. The search for the optimal contract therefore requires the application of the theory of second order backward stochastic differential equations (2BSDEs for short), subject to a slight extension to take into account the plurality of agents in the hierarchy. 
Furthermore, we show that, in a general way, the contract offered by the manager to one of his subordinate agents must be indexed not only on the performance of this particular agent, but also on the performance of other workers.
However, several hypotheses are necessary to complete our study, notably on the shape of the dynamics of the state variables. Nevertheless, we will see that these assumptions are satisfied in the most common examples.

\medskip

The theoretical model we develop allows us to determine the optimal form of incentives for a particular hierarchical structure, which can be extended in a straightforward way to a larger scale hierarchy. Although theoretical, the results we obtain give intuitions based on solid theoretical considerations to know which levers could be activated to incentivise workers within a hierarchy. In particular, the indexation of the contract on the quadratic variation of net profits for the managers argues in favour of remunerating them through stock options. These results can be applied to problems of incentives within a firm with a hierarchical structure, but also and above all as soon as work is delegated to an external entity. For example, these multi--layered incentive problems can be used to model the relationships between a firm and its subsidiaries or subcontractors, or the relationships between an investor, an investment company, and a fund manager, as in the model by \citeayn{keppo2020dynamic}.

\medskip

\textcursive{An overview of the paper.}\vspace{-0.25em} This work consists of two parts. In the first part, we study \citeauthor{sung2015pay}'s model and some extensions. More precisely, the continuous--time version of \citeauthor{sung2015pay}'s model is introduced and solved in \Cref{sec:sungmodel}. This opening example highlights the differences between the discrete--time model and its continuous--time equivalent, concerning the volatility control and the form of the contracts. In particular, this example leads to the conclusion that in order to rigorously study a continuous--time hierarchy problem, it is not possible to consider the associated discrete--time model with linear contracts. In our opinion, these conclusions justify the use of the theory of 2BSDEs to tackle problems of moral hazard within a hierarchy.
In \Cref{sec:extensions}, we consider some extensions, $(i)$ by adding an ability parameter for the manager to justify his position in the hierarchy, $(ii)$ by looking at different types of reporting from the manager to the principal (other than the reporting of the net profit), $(iii)$ by extending to a more general hierarchy structure. The details and proofs of the main results in this first part can be found in \Cref{app:sungs_model}.

\medskip

The second part of this paper is devoted to the study of the most general model possible. In particular, the workers (middle and top managers) can also control the volatility of their output, in addition to the drift. Moreover, we consider general utilities, allowing us to recover the exponential utility functions (CARA) of \citeauthor{sung2015pay}'s model \cite{sung2015pay}, but also other cases such as the one of risk--neutral workers. Finally, we consider a more general hierarchy since the principal contracts with $m$ managers, who in turns subcontract with the agents in their teams. The model is described in \Cref{sec:general_model}, and then solved in \Cref{sec:solving_general_model}, going up the hierarchy. More precisely, we first solve the problem of agents (see \Cref{ss:agent_problem}), then the problem of their supervisors, \textit{i.e.}, the managers (see \Cref{ss:manager_problem}), and finally, we end with the principal's problem (see \Cref{ss:principal_problem}). 
As previously mentioned, although some intuitions are presented in \Cref{app:intuition_extension}, the resolution of this model is based on the theory of 2BSDE, whose presentation and theoretical results are postponed to \Cref{sec:2BSDEs}.
Finally, \Cref{sec:conclusion} concludes and provides some extensions.

\subsection*{Notations}

Let $\mathbb{N}^\star:=\mathbb{N}\setminus\{0\}$ be the set of positive integers. % and $\mathbb{R}_+^\star$ the set of real positive numbers. 
For every $d$--dimensional (column) vector $b$ with $d\in \mathbb{N}^\star$, we denote by $b^{1},\ldots,b^{d}$ its coordinates and for $(\alpha,\beta) \in \R^d\times\R^d$, we denote by $\alpha \cdot \beta$ the usual inner product, with associated norm $\|\cdot\|$, which we simplify to $|\cdot|$ when $d$ is equal to $1$. Let ${\bf 0}_d$ and ${\bf 1}_d$ be the vectors of size $d$ whose coordinates are all equal to respectively $0$ and $1$.
The zero element in $\R^d$ is denoted by $\mathbf{0}_d$. For any $(\ell,c) \in \N^\star \times \N^\star$, $\M^{\ell,c}$ will denote the space of $\ell\times c$ matrices with real entries. 
%Elements of the matrix $M \in \M^{\ell,c}$ will be denoted by $(M^{ij})_{1\leq i\leq \ell,\; 1\leq j\leq c}$, and 
The transpose of $M \in \M^{\ell,c}$ will be denoted by $M^\top$ and the zero element in $\M^{\ell,c}$ is denoted by $\mathbf{0}_{\ell,c}$. When $\ell=c$, we let $\M^{\ell}:=\M^{\ell,\ell}$. We also identify $\M^{\ell,1}$ and $\R^\ell$. The identity matrix in $\M^\ell$ will be denoted by $\mathrm{I}_\ell$. We also denote by $\S^\ell$ (resp. $\S^\ell_+$) the set of symmetric positive (resp. symmetric definite positive) matrices in $\M^\ell$. The trace of a matrix $M \in \M^\ell$ will be denoted by $\mathrm{Tr} [M]$. %, and for any $(x,y) \in \M^{\ell,c} \times \M^{c, \ell}$, we use the notation $x : y$ for $\mathrm{Tr} [xy]$. 
For any vector $x \in \R^d$, we denote by $\textnormal{diag}[x]$ the diagonal matrix in $\M^d$ whose entries are the $d$ elements of the vector $x$.

\medskip

Throughout this paper, $T > 0$ denotes some maturity fixed in the contract. For any positive integer $d$, let $\Cc ( [0,T], \R^d)$ denote the set of continuous functions from $[0,T]$ to $\R^d$. On $\Cc ( [0,T], \R^d)$, define the evaluation mappings $\pi_t$ by $\pi_t(x) = x_t$ and the truncated supremum norms $\| \cdot \|_t$ by $\| x \|_t = \sup_{s \in [0,t]} \|x_s\|$, for $t \in [0,T]$.
% \begin{align*}
%   \| x \|_t = \sup_{s \in [0,t]} \|x_s\|, \; \text{ for } t \in [0,T].
% \end{align*}
Unless otherwise stated, $\Cc ( [0,T], \R^d)$ is endowed with the norm $\| \cdot \|_T$.

\section{An opening example: Sung's model in continuous time}\label{sec:sungmodel}

In order to justify the motivation of this work, we present in this section a simple hierarchical contracting problem, similar to the one considered by \citeayn{sung2015pay}. In this model, the principal contracts one manager who in turn subcontracts with many agents. The hierarchy is illustrated in \Cref{fig:sung_model}. Despite its simplicity, this illuminating example shows the need to take volatility control into account, and therefore justifies the use of the 2BSDEs theory in the rest of this paper. The reasoning will remain informal throughout \Cref{sec:sungmodel,sec:extensions}, the reader is referred to the theoretical part of this paper, from \Cref{sec:general_model} onwards, for a rigorous model setting in continuous time.

\medskip

\begin{figure}[ht!]
\begin{center}
\begin{tikzpicture}
% style des nœuds
\tikzstyle{principal}=[rectangle,draw,rounded corners=4pt,fill=cyan!50]
\tikzstyle{agents}=[rectangle,draw,rounded corners=4pt, fill=green!50]
\tikzstyle{manager}=[rectangle,draw,rounded corners=4pt, fill=red!25]
\tikzstyle{invisible}=[rectangle]

% style des flèches
%\tikzstyle{suite}=[->,>=stealth’,thick,rounded corners=4pt]
% placement des nœuds
\node[principal] (principal) at (0,0) {Principal};
\node[manager] (manager) at (0,-2) {Manager};
\node[agents] (agent1) at (-4,-4) {Agent $1$};
\node[agents] (agent2) at (-2,-4) {Agent $2$};
\node[invisible] (agenti) at (0,-4) {$\dots$};
\node[agents] (agentn1) at (2,-4) {Agent $n-1$};
\node[agents] (agentn) at (4,-4) {Agent $n$};

% Placement des flèches
\draw[->] (principal) -- (manager) node[midway,fill=white]{$\xi^0$};
\draw[->] (manager) -- (agent1) node[midway,fill=white]{$\xi^{1}$};
\draw[->] (manager) -- (agent2) node[midway,fill=white]{$\xi^{2}$};
\draw[->, dashed] (manager) -- (agenti);
\draw[->] (manager) -- (agentn1) node[midway,fill=white]{$\xi^{n-1}$};
\draw[->] (manager) -- (agentn) node[midway,fill=white]{$\xi^{n}$};
\end{tikzpicture}
\caption{Hierarchy in \citeauthor{sung2015pay}'s model}
\label{fig:sung_model}
\end{center}
\end{figure}
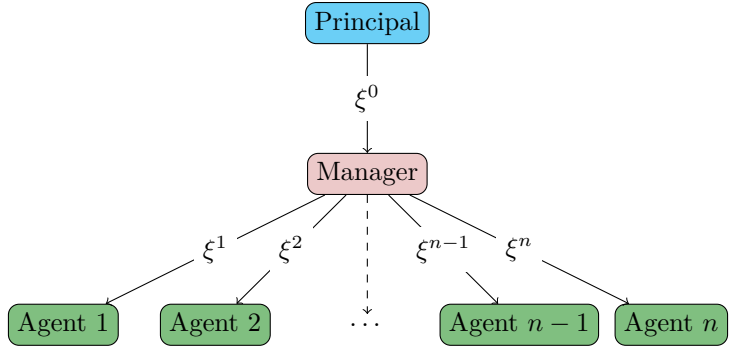

The only difference with \citeauthor{sung2015pay}'s model is that we consider here a continuous--time model: between $0$ and some time $T > 0$, denoting the maturity fixed in the contract, the firm has $n+1$ tasks which have to be carried out by $n+1$ workers. The outputs of the tasks are represented by $n+1$ stochastic processes, denoted by $X^i$, with dynamic
\begin{align}\label{eq:dynamic_Xi}
    \drm X^i_t = \alpha^i_t \drm t + \sigma^i \drm W^i_t, \; \sigma^i > 0, \; t \in [0,T],
\end{align}
for $i \in \{0, \dots, n\}$. More precisely, the $i$--th worker carries out the task with outcome $X^i$ by choosing a costly effort $\alpha^i \in \Ac^i$ with values in $\R$, where $\Ac^i$ is the set of control processes\footnote{The set of admissible control processes will be rigorously defined, in weak formulation, in the theoretical part of this paper, more precisely in \Cref{ss:theoretical_formulation}.}. For simplicity, we assume that $W^i$ for $i \in \{0, \dots, n\}$ are independent Brownian motions and that the efforts of a worker only impact his own project, which means that all projects are independent and no workers collude. Again for the sake of simplicity in this part, we will consider the following quadratic cost of effort:
\begin{align}\label{eq:cost_function}
    c^i (a) = \dfrac{1}{2} \dfrac{a^2}{k^i}, \; k^i > 0, \; \text{ for } \; a \in \R, \; i \in \{0, \dots, n\}.
\end{align}
Still following \citeauthor{sung2015pay}'s framework in \cite{sung2015pay}, we also assume that the benefit of each worker $i$ is represented by a CARA utility function with risk aversion coefficient $R^i > 0$.
The holder of the firm (the principal) is risk--neutral and seeks to maximise the expected difference between the sum of the outcomes and the sum of the compensations paid to the workers. The minimum level of utility that must be guaranteed by a contract to make it acceptable to a worker, \textit{i.e.}, his reservation utility, is defined by his utility without contract.

\medskip

We define the \textit{direct contracting case} (DC case) as the case where the principal can directly contract with the agents, without the help of a manager. In this setting, the optimal efforts of the workers are deterministic constant processes given by the result of \citeayn{holmstrom1987aggregation}, and summarised in \Cref{lem:DC_case}.
This case is also discussed by \citeauthor{sung2015pay}, but the main point of his paper \cite{sung2015pay}, and thus of ours, is to study the case where the principal contracts with a manager, who in turn subcontracts with the agents.

\medskip

In the \textit{hierarchical contracting case} (HC case) considered by \citeauthor{sung2015pay}, the principal cannot directly contract with the workers. She hires a manager (the worker indexed by $i=0$) who:
\begin{enumerate}[label=$(\roman*)$]
    \item carries out his own task by choosing an effort process $\alpha^0 \in \Ac^0$;
    \item hires $n$ agents to carry out the $n$ remaining tasks: each agent $i$ handles the outcome $X^i$, by choosing his effort level $\alpha^i \in \Ac^i$, for $i \in \{1, \dots n\}$;
    \item reports to the principal the total benefit, that is the difference between the sum of the outcomes and the sum of the compensations to be paid to the agents.
\end{enumerate}
We will show that, in the continuous--time framework, we can improve \citeauthor{sung2015pay}'s results in \cite{sung2015pay} by considering a more general form of contracts.

\subsection{A continuous--time principal--manager--agents problem}

As already mentioned, we are faced with a bi--level principal--agent problem, in the sense that a principal contracts with a manager who in turn subcontracts with many agents. In this section we define the continuous--time equivalent formulation of the value functions considered in \cite{sung2015pay}, as well as the admissible set of contracts.

\medskip

\textcursive{Agent's problem.}\vspace{-0.3em} Consider first $i \in \{1, \dots, n\}$ to focus on the $i$--th agent's problem. He controls his own output $X^i$ with dynamic \eqref{eq:dynamic_Xi} by choosing an effort $\alpha^i \in \Ac^i$. 
Given a contract $\xi^{i}$ offered by his supervisor, namely the manager, the $i$--th agent's value function is simply defined by:
\begin{align}\label{eq:agent_pb_sung}
    V_0^i (\xi^{i}) := \sup_{\alpha^i \in \Ac^i}  J_0^i \big( \xi^{i}, \alpha^i \big), \; \text{ where } \; J_0^i \big( \xi^{i}, \alpha^i \big) := \E^{\P^i} \Big[ - \erm^{ - R^i \big( \xi^{i} - \int_0^T c^i(\alpha_t^i) \drm t \big) } \Big],
\end{align}
where $\P^i$ is the probability associated to the effort $\alpha^i \in \Ac^i$.
We assume that the manager cannot directly observe the efforts of the agents, which implies a first level of moral hazard in our framework. The manager only observes the outcome processes $X^i$ for $i \in \{1, \dots, n\}$.
In order to follow \citeauthor{sung2015pay}'s model as closely as possible, we also assume that the compensation for the $i$--th agent can only be indexed on his performance, \textit{i.e.}, his outcome process $X^i$, and denote the set of admissible contracts by $\Cc^{i}$.\footnote{It is worth noticing that this restriction is not admissible in a more general model, as we will see in \Cref{rk:measurability_issues}.} Note that since the reservation utility of the $i$--th agent is defined as his utility without contract, it is given by $V^i_0(0) =-1$.

\medskip

\textcursive{Manager's problem.}\vspace{-0.3em} The manager controls his own output $X^0$ with dynamic \eqref{eq:dynamic_Xi} by choosing an effort $\alpha^0 \in \Ac^0$. He also designs the compensations for the agents, namely a collection of contracts 
\begin{align*}
    \xi^{\rm A} \in \Cc^{\rm A} := \big\{ (\xi^i)_{i=1}^n, \; \text{s.t.} \; \; \xi^{i} \in \Cc^i \; \forall \, i \in \{1, \dots, n\}\big\}.
\end{align*}
Although we consider that the manager designs the contracts for the agents, all compensations, whether for the agents or the manager, are paid by the principal. Given a contract $\xi^{\rm b}$ designed by his supervisor (the principal), the manager's value function is defined by 
\begin{align}\label{eq:manager_pb_sung}
    V_0^{\rm 0, b} (\xi^{\rm b}) := \sup_{(\alpha^{0}, \xi^{\rm A}) \in \Ac^{0} \times \Cc^{\rm A}}  J_0^{0} \big( \xi^{\rm b}, \alpha^{0}, \xi^{\rm A} \big), \; \text{ where } \; J_0^{0} \big( \xi^{\rm b}, \alpha^{0}, \xi^{\rm A} \big) := \E^{\P^{0}} \Big[ - \erm^{- R^{0} \big( \xi^{\rm b} - \int_0^T c^{0}(\alpha_t^{0}) \drm t \big) } \Big],
\end{align}
where, informally, $\P^0$ is the probability associated to both the effort $\alpha^0 \in \Ac^0$ and the choice of the contracts $\xi^{\rm A} \in \Cc^{\rm A}$, under the optimal efforts of the $n$ agents. 
As in \cite{sung2015pay}, the second level of moral hazard in our framework is linked to the fact that the manager only reports (in continuous time) to the principal the total benefit $\zeta^{\rm b}$, \textit{i.e.}, the difference between the sum of the outcomes and the sum of the compensations to be paid:
\begin{align}\label{eq:zeta_b}
    \zeta^{\rm b}_t = \sum_{i=0}^n X_t^i - \sum_{i=1}^n \xi_t^i, \; t \in [0,T].
\end{align}
More precisely, in this setting, the principal cannot independently observe the agents' outcomes $X^i$ or the certainty equivalent of their continuation utility, $\xi^i$. Nor does she observe the manager's outcome $X^0$, nor his effort $\alpha^0$. Therefore, she can only index the contract $\xi^{\rm b}$ for the manager on the total benefit, \textit{i.e.}, the variable $\zeta^{\rm b}$. The contract $\xi^{\rm b}$ is thus a measurable function of $\zeta^{\rm b}$, and the corresponding set of admissible contract is denoted by $\Cc^{\rm b}$.

\medskip

\textcursive{Principal's problem.}\vspace{-0.3em} Finally, we consider a risk--neutral principal whose problem is to maximise the sum of the outcomes minus the sum of the compensations to be paid to each worker, by choosing the optimal contract $\xi^{\rm b}$ for the manager. Mathematically speaking, we define her value function as follows:
\begin{align}\label{eq:principal_pb_sung}
    V_0^{\rm P, b} := \sup_{\xi^{\rm b} \in \Cc^{\rm b}}  J_0^{\rm P} (\xi^{\rm b}), \; \text{ where } \; 
    J_0^{\rm P} (\xi^{\rm b}) := \E^{\P^{\rm b}} \bigg[ \sum_{i=0}^n X_T^i - \sum_{i=1}^n \xi_T^i - \xi_T^{\rm b} \bigg] = \E^{\P^{\rm b}} \big[ \zeta^{\rm b}_T - \xi_T^b \big],
\end{align}
where $\P^{\rm b}$ is the probability associated to the choice of the contract $\xi^{\rm b} \in \Cc^{\rm b}$, under optimal efforts of the workers. 

\begin{remark}
It should be noted that the three value functions defined above by \textnormal{(\ref{eq:agent_pb_sung}--\ref{eq:manager_pb_sung}--\ref{eq:principal_pb_sung})} should be rigorously written in weak formulation. However, for the sake of simplicity in this section, we avoid this technical side for the moment and refer the reader to the theoretical part of this paper \textnormal{(}from \textnormal{\Cref{sec:general_model}} onwards\textnormal{)} for a more correct writing of these three value functions.
\end{remark}

\subsection{Solving the sequence of Stackelberg equilibria}

In order to solve this principal--manager--agents problem, we follow the general theory developed by \citeayn{cvitanic2018dynamic} and the application in \cite{cvitanic2017moral} to a framework with CARA utility functions. More precisely, for each Stackelberg equilibrium, starting with the manager--agent problem, we have to:
\begin{enumerate}[label=$(\roman*)$]
    \item identify a sub--class of contracts, offered to a considered worker by his supervisor, which are revealing in the sense that the best--reaction function of the worker and his optimal control can be computed straightforwardly;
    \item prove that the restriction to revealing contracts is without loss of generality;
    \item solve the supervisor's problem, which boils down to a standard optimal control problem.
\end{enumerate}

\medskip

\textcursive{Revealing contract for an agent.}\vspace{-0.3em} Consider $i \in \{1, \dots, n\}$ to focus on the contract for the $i$--th agent. Recall that his contract $\xi^i$ is assumed to be a measurable function of his output $X^i$. 
By applying classical results of contract theory for drift control only (see, \textit{e.g.}, the work by \citeayn{sannikov2008continuous}), the optimal form of contracts is the terminal value $\xi^i_T$ of the certainty equivalent of the continuation utility, which is defined for all $t \in [0,T]$ as follows:
\begin{align}\label{eq:contract_drift}
    \xi^{i}_t = \xi_0^{i} - \int_0^t \Hc^i \big( Z^{i}_s \big) \drm s + \int_0^t Z^{i}_s \drm X_s^i + \dfrac{1}{2} R^i \int_0^t \big| Z^{i}_s \big|^2 \drm \langle X^i \rangle_s, \; \xi_0^i \in \R,
\end{align}
where $Z^i$ is a payment rate chosen by the manager, and $\Hc^i (z) := \sup_{a \in \R} \big\{a z - c^i (a) \big\}$ for all $z \in \R$ is the Hamiltonian of the $i$--th agent. This form of contract is exactly the continuous--time equivalent of the linear contract considered by \citeauthor{sung2015pay} in \cite[Equation (7)]{sung2015pay}. Note that this form of compensation includes in particular a fixed part $\xi_0^i$, which is chosen so as to satisfy the $i$--th agent's participation constraint. Assuming for simplicity that $c^i$ is a standard quadratic cost function defined by \eqref{eq:cost_function}, we can establish the following result, whose proof is straightforward\footnote{As the proof of this result is very similar to that of \Cref{prop:manager_sung} below, providing a similar result for the manager, we have chosen not to detail it here. The reader is thus referred to \Cref{app:ss:proof_sung} for a sketch of the proof.} in the light of the choice of contract's form \eqref{eq:contract_drift}.

\begin{proposition}\label{prop:agent_sung}
    Fix $i \in \{1, \dots, n\}$. Let $\xi^i_0 \in \R$, $Z^i$ be an $\R$--valued process, predictable with respect to the filtration generated by $X^i$, satisfying appropriate integrability conditions\footnote{\label{fn:integrability}We have to require minimal integrability on the process $Z$ so that the stochastic integral with $X$ is well--defined. Nevertheless, since this section is informal, the conditions of integrability are ignored for the time being. The reader can refer to the general model for a rigorous definition of the admissible process $Z$.}, and consider the associated contract $\xi^i$ defined through \eqref{eq:contract_drift}. Given this contract, the optimal effort of the $i$--th agent is given for $t \in [0,T]$ by
    \begin{align*}
        \alpha^{i, {\rm HC}}_t = a^{i,{\rm HC}}(Z^i_t), \; \text{ where } \; a^{i,{\rm HC}}(z) := k^i z, \; z \in \R.
    \end{align*}
    Moreover, under the probability $\P^{i, {\rm HC}}$ associated to the optimal effort $\alpha^{i,{\rm HC}}$, the dynamics of $X^i$ and $\xi^{i}$ satisfy
    \begin{align*}
        \drm X^i_t &= k^i Z_t^{i} \drm t + \sigma^i \drm W_t^i \;
        \text{ and } \; \drm \xi^{i}_t = \frac{1}{2} \widetilde R^i \big| Z^{i}_t \big|^2 \drm t 
        + Z^{i}_t \sigma^i \drm W_t^i, \; t \in [0,T], \; \text{ where } \; \widetilde R^i = k^i + R^i |\sigma^i |^2.
    \end{align*}
\end{proposition}

% \medskip

\textcursive{Simplification of the manager's problem.}\vspace{-0.3em} 
Recall that the manager, in addition to controlling his own output $X^0$ with dynamic \eqref{eq:dynamic_Xi} by choosing an effort $\alpha^0 \in \Ac^0$, also designs the collection of contracts $\xi^{\rm A} \in \Cc^{\rm A}$ for the agents. 
As mentioned above, instead of studying all possible contracts $\xi^i \in \Cc^i$ for the $i$--th agent, it has been proved by \citeayn{sannikov2008continuous} that it is sufficient to restrict the study to contracts of the form \eqref{eq:contract_drift}. Therefore, the manager's problem boils down to a standard optimal control problem: to design the compensation for the $i$--th agent, the manager only have to choose the payment rate $Z^i$. We thus define by $\Vc^0$ the collection of all processes $Z : [0,T] \times \Cc([0,T],\R) \longrightarrow \R^n$, where each $Z^i$ is predictable with respect to the filtration generated by $X^i$, and satisfies appropriate integrability conditions. We can now rewrite the manager's problem defined by \eqref{eq:manager_pb_sung} in a more standard way:
\begin{align*}
    V_0^{\rm 0, b} (\xi^{\rm b}) := \sup_{(\alpha^{0}, Z) \in \Ac^{0} \times \Vc^0}  J_0^{0} \big( \xi^{\rm b}, \alpha^{0}, \xi^{\rm A} \big).
\end{align*}
Recall that we consider that the manager designs the contracts for the agents, but he does not pay them. All compensations, whether for the agents or the manager, are paid by the principal. For this reason, we assume that the principal chooses, for all $i \in \{1, \dots, n\}$, the constant $\xi_0^i$ in the $i$--th agent's contract, noticing that this constant must be chosen in order to ensure that his participation constraint is satisfied.

\medskip

Since the manager only reports to the principal the total benefit $\zeta^{\rm b}$ in continuous time, his compensation $\xi^{\rm b}$ offered by the principal can only be a measurable function of $\zeta^{\rm b}$. Then, given the form of the manager's utility and the dynamic of the output $X^i$ for $i \in \{1, \dots, n\}$, $\zeta^{\rm b}$ is the only state variable of his control problem. Therefore, his optimal control, namely $(\alpha^0, Z) \in \Ac^0 \times \Vc^0$, will naturally be adapted to the filtration generated by $\zeta^{\rm b}$.

\begin{remark}\label{rk:measurability_issues}
    In fact, the set of admissible control processes for the manager cannot be properly defined in this framework. Indeed, recall that we choose to restrict the contract for an agent $i$ to his own $X^i$. Therefore, the payment rate $Z^i$ should not depend on anything other than $X^i$. In particular, it cannot be predictable with respect to the filtration generated by $\zeta^{\rm b}$, since it contains information generated by the outputs of other workers. However, under this assumption, the optimal control of the manager on the $i$--th agent's contract, which will be denoted by $Z^{i, \rm b}$, should be a function of $X^i$, and thus cannot be computed by the principal, since she only observes $\zeta^{\rm b}$. Moreover, we will generically have $Z_t^{i, \rm b} := z^{i, \rm b} (Z_t, \Gamma_t)$, $t \in [0,T]$, where $Z$ and $\Gamma$ are two processes chosen by the principal, assumed to be predictable with respect to what she observes, \textit{i.e.}, with respect to the filtration generated by $\zeta^{\rm b}$. Indeed, the manager's contract is restricted to a measurable function of $\zeta^{\rm b}$, therefore the payment rates $Z$ and $\Gamma$ indexing the contract on $\zeta^{\rm b}$ should also be functions of $\zeta^{\rm b}$. We thus obtain a contradiction. Nevertheless, in this particular example, every optimal efforts and controls controls turn out to be be deterministic \textnormal{(}and even constant\textnormal{)}, and we can thus index the contract for the $i$--th agent only on his own output $X^i$. This model has been chosen in this section to easily compare our results with those of \textnormal{\citeayn{sung2015pay}}. However, in a more general case, we will not be able to restrict the study to such contracts. More precisely, we will be forced to consider that each agent knows the output of other workers, and that his contract can be indexed on it. The controls of the manager will thus be predictable with respect to the filtration generated by $\zeta^{\rm b}$, and can then be computed at the optimum by the principal. 
\end{remark}

\textcursive{Towards volatility control.}\vspace{-0.3em} Let us set aside for the moment the previous remark, it will be dealt with in the general model (from \Cref{sec:general_model} onwards). The most important thing to notice at this stage is that, even if the agents are only controlling the drift of their outcomes, the manager controls both the drift and the volatility of $\zeta^{\rm b}$, as we can see from its dynamic under optimal efforts of the agents, which is as follows:
\begin{align}\label{eq:dyn_zeta_b}
    \drm \zeta^{\rm b}_t = \bigg[ \alpha_t^{0} + \sum_{i=1}^n \Big( k^i Z_t^{i} - \frac{1}{2} \widetilde R^i | Z^{i}_t |^2 \Big) \bigg] \drm t + \sigma^{0} \drm W_t^{0} + \sum_{i=1}^n \sigma^i ( 1 - Z^i_t) \drm W_t^i,
\end{align}
recalling that $\widetilde R^i = k^i + R^i |\sigma^i |^2$ for all $i \in \{1, \dots, n\}$.
Indeed, by choosing optimally the payment rate in each agent's contract ($Z^i$ for all $i \in \{1, \dots, n\}$), the manager controls in a way the volatility of the certainty equivalent of agents' continuation utilities $\xi^i$ (through the term $ Z^{i}_t \sigma^i \drm W_t^i$), and thus the volatility of $\zeta^{\rm b}$. Therefore, we must consider a more extensive class of contracts than the one used by \citeauthor{sung2015pay} in \cite{sung2015pay}. Indeed, in continuous time, it is not sufficient to limit oneself to linear contracts (in the sense of \citeayn{holmstrom1987aggregation}) when the volatility of the state variable is controlled, as demonstrated by \citeayn{cvitanic2018dynamic}. This is where our model and \citeauthor{sung2015pay}'s will diverge. Instead of studying the model in continuous time, \citeauthor{sung2015pay} considers that the one--period model is simpler and without loss of generality.
% \footnote{\blockquote[\citeayn{sung2015pay}][]{For ease of exposition and without loss of generality, we formulate a discrete--time model which is analogous to its continuous--time counterpart}.} 
He therefore continues to restrict the study to contracts that are linear with respect to the outcome, insisting that this restriction is \blockquote[\citeayn{sung2015pay}][.]{without loss of generality, as long as our results are interpreted in the context of continuous--time models}. According to our study in continuous time, it appears that the type of contracts considered by \citeauthor{sung2015pay} is sub--optimal (see \Cref{ss:non_optimal} for the analysis of the results).

\medskip

\textcursive{Revealing contract for the manager.}\vspace{-0.3em} Let $\V^{\rm b}$ be the set of all $(z,\gamma) \in \R^2$ such that $\widetilde R^i z - |\sigma^i|^2 \gamma > 0$ for all $i \in \{1, \dots, n\}$. We define by $\Vc^{\rm b}$ the collection of all processes $(Z,\Gamma) : [0,T] \times \Cc([0,T],\R) \longrightarrow \V^{\rm b}$, predictable with respect to the filtration generated by $\zeta^{\rm b}$, satisfying appropriate\footnote{Similarly as noticed in \Cref{fn:integrability}, we have to require minimal integrability on the process $Z$ so that the stochastic integral with $\zeta^{\rm b}$ is well--defined. The reader can refer to the general model for a rigorous definition of the set of admissible control processes $\Vc^{\rm b}$.} integrability conditions. The set $\Vc^{\rm b}$ represents the admissible control processes for the principal, when she only observes the variable $\zeta^{\rm b}$. Taking into account the previous discussion, it is necessary to use recent results on optimal contracting for drift and volatility control, and therefore the theory of 2BSDEs, to state the following result. We refer to the previously mentioned works of \citeayn{cvitanic2018dynamic} for the general result, \cite{cvitanic2017moral} for an application with exponential utilities, as well as \citeayn{lin2020random} for an extension to random time horizon or \citeayn{elie2019mean} for a model with a continuum of agents with mean--field interaction.

\begin{proposition}\label{prop:contract_vol}
    Assuming that the principal only observes $\zeta^{\rm b}$, the optimal form of contracts offered by the principal to the manager is given by
    \begin{align}\label{eq:contract_vol}
        \xi^{\rm b}_T = &\ 
        \xi^{\rm b}_0
        - \int_0^T \Hc^{\rm b} (Z_s, \Gamma_s) \drm s
        + \int_0^T  Z_s \drm \zeta^{\rm b}_s
        + \dfrac{1}{2} \int_0^T \big( \Gamma_s + R^{0} | Z_s|^2 \big) \drm \langle \zeta^{\rm b} \rangle_s, \; \xi^{\rm b}_0 \in \R,
    \end{align}
    where $\Hc^{\rm b}$ is the manager's Hamiltonian, and $(Z, \Gamma) \in \Vc^{\rm b}$ is a pair of processes optimally chosen by the principal. In addition, similar to the agent's contract form \eqref{eq:contract_drift}, $\xi^{\rm b}_0$ represents a fixed part of the compensation, which is chosen so as to satisfy the manager's participation constraint.
\end{proposition}

Given the dynamic of the state variable $\zeta^{\rm b}$ and its quadratic variation, the manager's Hamiltonian is defined, for any $(z, \gamma) \in \V^{\rm b}$, as follows:
\begin{align}\label{eq:hamiltonian_manager_sung}
    \Hc^{\rm b} (z, \gamma) = \dfrac{1}{2} \gamma | \sigma^{0} |^2 + \sup_{a \in \R} \big\{ 
    a z - c^{0} (a) \big\}
    + \sum_{i=1}^n \sup_{z^i \in \R} \bigg\{ 
    z \Big(  k^i z^{i} - \frac{1}{2} \widetilde R^i |z^{i}|^2  \Big)
    + \dfrac{1}{2} \gamma |\sigma^i|^2  |1 - z^i|^2 \bigg\}.
\end{align}
Considering any contracts of the form \eqref{eq:contract_vol}, we can easily solve the manager's problem, mainly by maximising the previous Hamiltonian. The proof of the following proposition is therefore a direct consequence of the considered form of contracts, and is detailed in \Cref{app:ss:proof_sung}.
\begin{proposition}\label{prop:manager_sung}
    Let $(Z,\Gamma) \in \Vc^{\rm b}$. The optimal effort on the drift and the optimal control on the $i$--th agent's compensation $(i \in \{1, \dots, n\})$ chosen by the manager are respectively given by $\alpha_t^{\rm b} := k^{0} Z_t$ and $Z^{i, \rm b}_t := z^{i, \rm b} (Z_t,\Gamma_t)$ for all $t \in [0,T]$, where, for all $(z, \gamma) \in \V^{\rm b}$,
    \begin{align}\label{eq:zib_star}
        z^{i, \rm b} (z,\gamma) &:= \dfrac{k^i z - | \sigma^i |^2 \gamma }{ \widetilde R^i z - | \sigma^i |^2 \gamma}.
    \end{align}
    Under the optimal probability $\P^{\rm b}$ associated to the optimal efforts of both the agents and the manager, the dynamics of $\zeta^{\rm b}$ and $\xi^{\rm b}$ are respectively given, for all $t \in [0,T]$, by:
    \begin{align*}
        \drm \zeta^{\rm b}_t = &\ \bigg[ k^{0} Z_t 
        + \sum_{i=1}^n \Big( k^i Z_t^{i,  \rm b} 
        - \frac{1}{2} \widetilde R^i \big| z^{i, \rm b}_t \big|^2 \Big) \bigg] \drm t 
        + \sigma^{0} \drm W_t^{0} 
        + \sum_{i=1}^n \sigma^i ( 1 - z^{i, \rm b}_t) \drm W_t^i, \\
        \text{and } \; \drm \xi^{\rm b}_t = &\ \dfrac12  Z_t^2 \bigg( k^0  
        + R^{0} | \sigma^{0} |^2 
        + R^{0} \sum_{i=1}^n | \sigma^i |^2 \big| 1 - z^{i, \rm b}_t \big|^2 \bigg) \drm t
        +  Z_t \sigma^{0} \drm W_t^{0} +  Z_t \sum_{i=1}^n \sigma^i ( 1 - z^{i, \rm b}_t) \drm W_t^i.
    \end{align*}
\end{proposition}

\textcursive{Solving the principal's problem.}\vspace{-0.3em} \Cref{prop:contract_vol} states that it is sufficient to restrict the space of contracts to those of the form \eqref{eq:contract_vol}, and thus simplifies the principal's problem. Recall that we assume that the principal chooses all the constants $\xi_0^i$ in each agent's contract, as well as the constant $\xi^{\rm b}_0$ in the manager's contract. Informally, these constants have to be chosen such that each contract satisfies the participation constraint of the corresponding worker. Given the form of the manager's and agents' utility, and in particular since the manager does not pay the compensation for his agents, he is indifferent to the value of $\xi_0^i$, for $i \in \{1, \dots, n\}$, as long as the agents accept the contracts. This is why we consider that the principal chooses it, and we denote by $\xi_0 \in \R^{n+1}$ the collection of $\xi^{\rm b}_0$ and $\xi_0^i$, for $i \in \{1, \dots, n\}$. Moreover, recall that we have assumed as in \cite{sung2015pay} that the reservation utility level of each worker is given by his utility without any contract, thus equal to $-1$, and that the initial outcomes $(X_0^i)_{i=0}^n$ are equal to zero.

\begin{proposition}\label{prop:principal_sung}
    The principal's problem defined by \eqref{eq:principal_pb_sung} is reduced to $V_0^{\rm b}  = \sup_{(\xi_0, Z, \Gamma) \in \R^{n+1} \times \R \times \Vc^{\rm b}}  J_0^{\rm P} (\xi^{\rm b})$. By solving standard control problem, we determine her optimal controls, the optimal contracts and her utility.
    \begin{enumerate}[label=$(\roman*)$]
        \item The optimal payment rates for the manager are given by the constant processes $Z^{\rm b} := z^{\rm b}$ and $\Gamma^{\rm b} := - R^0 (z^{\rm b})^3$, where $z^{\rm b}$ is solution of the following maximisation problem
    \begin{align}\label{eq:sup_z_sung_case}
        \sup_{z > 0} \bigg\{ 
        k^{0} z 
        - \dfrac12 \widetilde R^0 |z|^2
        + \sum_{i=1}^n h^{i,\rm b} \big(z,-R^0 z^3 \big) 
        \bigg\},
    \end{align}
    where $\widetilde R^0 := k^0 + R^{0} | \sigma^{0} |^2$ and, for all $i \in \{1, \dots, n\}$ and any $(z,\gamma) \in \V^{\rm b}$,
\begin{align}\label{eq:hib}
    h^{i,\rm b}(z,\gamma) := k^i z^{i, \rm b} (z,\gamma)
    - \frac{1}{2} \widetilde R^i \big| z^{i, \rm b} (z,\gamma) \big|^2
    - \dfrac12 R^{0} | \sigma^i |^2  |z|^2 \big| 1 - z^{i, \rm b} (z,\gamma) \big|^2.
\end{align}
    \item The optimal contract offered by the principal to the manager is given by:
    \begin{align*}%\label{eq:contract_vol_optimal}
        \xi^{\rm b}_T =
        - \Hc^{\rm b} \big( z^{\rm b}, - R^0 (z^{\rm b})^3 \big) T
        + z^{\rm b} \zeta^{\rm b}_T %- \zeta^{\rm b}_0)
        + \dfrac{1}{2} R^0 | z^{\rm b} |^2 (1 - z^{\rm b}) \langle \zeta^{\rm b} \rangle_T,
    \end{align*}
    where $\Hc^{\rm b}$ is the manager's Hamiltonian defined by \eqref{eq:hamiltonian_manager_sung}. In particular, the optimal choice of the fixed part of the compensation $\xi^{\rm b}_0$ is the one that saturates the manager's participation constraint, \textit{i.e.}, such that he obtains exactly his reservation utility. In this case, since his utility reservation is equal to $-1$, the optimal $\xi^{\rm b}_0$ is $0$.
    % , for any $(z, \gamma) \in \V^{\rm b}$, as follows:
    % \begin{align*}%\label{eq:hamiltonian_manager_sung_optimal}
    %     \Hc^{\rm b} (z, \gamma) = &\ \dfrac{1}{2} \gamma \big(\sigma^{0}\big)^2 
    %     + \dfrac12 k^0 z^2
    %     + \sum_{i=1}^n \bigg[ z \bigg(  k^i z^{i,\star} (z,\gamma) - \frac{1}{2} \big(z^{i,\star} (z,\gamma) \big)^2 \Big(  k^i + R^i \big( \sigma^i \big)^2 \Big) \bigg) 
    %     + \dfrac{1}{2} \gamma \big(\sigma^i\big)^2 \big( 1 - z^{i,\star} (z,\gamma) \big)^2 \bigg];
    % \end{align*}
    % Under optimal effort of the manager and his agents, the optimal contract $\xi^0$ is given by:
    % \begin{align*}
    %     \xi^{0}_T = \xi^0_0 + \dfrac12 \int_0^T \big( Z_t^{0,\star} \big)^2 \bigg( k^0  
    %     + R^{0} \big( \sigma^{0} \big)^2 + R^{0} \sum_{i=1}^n \big( \sigma^i \big)^2 \big( 1 - Z^{i, \star}_t \big)^2 \bigg) \drm t
    %     + \sigma^{0} \int_0^T Z_t^{0,\star} \drm W_t^{0} +  \sum_{i=1}^n \sigma^i \int_0^T Z_t^{0,\star}  ( 1 - Z^{i, \star}_t) \drm W_t^i.
    % \end{align*}
    \item For all $i \in \{1, \dots, n\}$, the optimal contract offered by the manager to the $i$--th agent is given by:
    \begin{align*}
        \xi^{i}_T = - \Hc^i \big(z^{i, \rm b} \big( z^{\rm b},-R^0 (z^{\rm b})^3 \big) \big) T + z^{i, \rm b} \big( z^{\rm b},-R^0 (z^{\rm b})^3 \big) X_T^i %- X^i_0 \big) 
        + \dfrac{1}{2} R^i \big| z^{i, \rm b} \big( z^{\rm b},-R^0 (z^{\rm b})^3 \big) \big|^2  \langle X^i \rangle_T,
    \end{align*}
    where $\Hc^i$ is the $i$--th agent's Hamiltonian, and recalling that $z^{i, \rm b}$ is defined in \textnormal{\Cref{prop:manager_sung}}. In particular, as for the manager, the optimal choice of the fixed part of the compensation $\xi_0^i$ is $0$.
    \item Finally, the value function of the principal is given by:
    \begin{align*}
        V_0^{\rm b} =  T \bigg( k^{0} z^{\rm b}
        - \dfrac12 \widetilde R^{0}  |z^{\rm b}|^2
        + \sum_{i=1}^n h^{i,\rm b} \big(z^{\rm b}, - R^0 ( z^{\rm b})^3 \big) \bigg).
    \end{align*}
    \end{enumerate}
\end{proposition}
\noindent The proof of this proposition is detailed in \Cref{app:ss:proof_sung}.

\subsection{The benefits of continuous time}

\subsubsection{Non--optimality of linear contracts in continuous time}\label{ss:non_optimal}

\Cref{prop:contract_vol,prop:principal_sung} require that the optimal contracts for the manager must be indexed on the quadratic variation of the net profit $\zeta^{\rm b}$ through the parameter $\Gamma^{\rm b} := -R^0 (z^{\rm b})^3$. However, in \cite{sung2015pay}, \citeauthor{sung2015pay} restricts the analysis to linear contracts: although he remarks that decisions on middle managerial contracts are affecting the volatility of the net profit of the firm, he chooses to view them as a case similar to unobservable project choice decisions, modelled by \citeayn{sung1995linearity}. More precisely, he states the following: 
\blockquote[{\citeayn[pp. 3]{sung2015pay}}][.]{As shall be seen in our hierarchical contracting problem, the top manager turns out to choose not only the mean of the outcome of his own effort but, in effect, the volatility of the total profit of the firm as he chooses middle managerial contracts. Thus, our problem turns out to be similar to the unobservable project choice problem in \citeauthor{sung1995linearity} [21]}
In the aforementioned article \cite{sung1995linearity}, \citeauthor{sung1995linearity} studies a principal--agent problem in continuous time where the volatility can be controlled. He distinguishes two cases.
\begin{enumerate}[label=$(\roman*)$]
    \item One where the variance is observed, but since the Brownian motion is only one dimensional, there is no moral hazard on the volatility's effort anymore. Indeed, in this case, the variance is equal to the square of the volatility effort, and since the variance is observed, the effort is easily computable by the principal. Therefore, the principal directly controls the volatility's effort of the agent and the model degenerates to the first--best case (no moral hazard) for the volatility.
    \item One where the variance is not observed by the principal, and therefore she cannot index the contract on the quadratic variation of the outcomes, which obviously leads to consider only linear contracts.
\end{enumerate}
In \cite{sung2015pay}, as \citeauthor{sung2015pay} considers that the variance is not observed, the principal cannot offer a contract to the manager indexed on it, which is equivalent to forcing $\Gamma^{\rm b} + R^0 | Z^{\rm b} |^2 = 0$ in our extended class of contracts defined by \eqref{eq:contract_vol}. Therefore, he does not optimise the utility of the principal with respect to $\Gamma^{\rm b}$, since he forces $\Gamma^{\rm b} := - R^0 | Z^{\rm b} |^2$.

\medskip

Nevertheless, in continuous time, it seems natural to consider that the principal observes the quadratic variation of the total benefit, $\langle \zeta^{\rm b} \rangle$, and can therefore contract on it. Indeed, she observes $\zeta^{\rm b}$ in continuous time and can therefore estimate the quadratic variation through the sum of the squared increments. Moreover, a result of \citeayn{bichteler1981stochastic} (see \citeayn[Proposition 6.6]{neufeld2014measurability} for a modern presentation) states that this quadratic variation, even controlled, can be defined independently of the probability associated to the effort. Therefore, contrary to $(ii)$ above, the contract can be indexed on the quadratic variation. Moreover, since the process $\zeta^{\rm b}$ is naturally driven by $n+1$ independent Brownian motions, the principal does not perfectly observe the controls $Z^i$ of the manager, but only a functional of $Z^i$. This prevents the volatility control case from degenerating into the first--best case, contrary to what is mentioned in $(i)$ above.
    
\medskip

Therefore, \citeauthor{sung2015pay}'s argument in \cite{sung2015pay} to justify restricting the study to linear contracts, namely that his model has to be understood as a continuous--time model in which linear contracts are supposedly optimal, seems not to be valid. One way to fix this problem in the one--period model would be to propose contracts indexed on the variance of $\zeta^{\rm b}$. However, as this variance is controlled, it depends on the probability chosen by the manager, which is unknown to the principal when the efforts are not optimal. Indeed, unlike in continuous time, where the quadratic variation, even controlled, can be defined independently of the effort probability, this is not the case for the variance in the one--period model. Therefore, it is not easy to find an equivalent to the contract indexed on the quadratic variation for the one--period model.\footnote{It is worth noticing that in a discrete--time framework, but with multiple periods, one could also approximate the variance.} However, in any case, it is a well--known result that it is not possible to find an optimal contract in the one--period model, even without volatility control, as soon as the monotone likelihood ratio\footnote{The monotone likelihood ratio is defined in this case by the ratio between the derivative of the considered process density with respect to the effort and the density itself.} is not bounded from below, which is the case in \cite{sung2015pay} since the output processes are Gaussian. Indeed, \citeayn{mirrlees1999theory} shows that, in this case, there is a sequence of contracts, called forcing contracts, that allows to obtain the results of the first--best case (when there is no moral hazard) at the limit, but that there is no optimal contract. Restricting oneself to linear contracts in the case of drift control only in the one--period model is justified because these are the optimal contracts in continuous time, but, unfortunately, this reasoning is no longer valid in the case of volatility control. 

\medskip

In conclusion, unlike the case of drift control only, in the case of volatility control it is not possible to consider the one--period model by limiting the study to linear contracts, and expect to obtain the same results as in continuous time. This result therefore justifies the full study of continuous--time models and the use of the recent theory of 2BSDEs, from a theoretical point of view. In the following, we will see through numerical results that it is obviously beneficial in a practical way for the principal to consider the problem in continuous time.

% \begin{align*}
%     \xi^{\rm b}_T = &\ 
%     \xi^{\rm b}_0
%     - \int_0^T \bigg(
%     \alpha_t^{0,\star} Z_t - c^{0} \big( \alpha_t^{0,\star} \big) 
%     + \sum_{i=1}^n \bigg\{ 
%     Z_t \bigg[  k^i Z_t^{i} - \frac{1}{2} \big( Z_t^{i} \big)^2 \Big(  k^i + R^i \big( \sigma^i \big)^2 \Big) \bigg]
%     \bigg\} \bigg) \drm s \\
%     &+ \int_0^T  Z_t \drm \zeta^{\rm b}_t
%     + \dfrac{1}{2} R^{0} \int_0^T Z_t^2 \bigg( \big( \sigma^{0} \big)^2 \drm t + \sum_{i=1}^n \big( \sigma^i \big)^2 ( 1 - Z^i_t)^2 \drm t \bigg).
% \end{align*}

\subsubsection{Numerical results}\label{sss:numerical_results}

In \cite[Theorem 2]{sung2015pay}, \citeauthor{sung2015pay} presents some interesting facts such as the decrease in the efforts of the manager and agents when the total number of workers increases, as well as their limits for an infinitely large company. These facts seem also be true in our framework, but we do not believe it is necessary to dwell on proving the same results. We find it more interesting to focus on the differences between the two models, and of course on the benefits of the approach being considered in this paper. In our opinion, the simplest way to achieve this goal is to present in this subsection some numerical simulations.

\medskip

Therefore, to illustrate the benefits induced by considering contracts with the quadratic variation term, we decide to perform some numerical simulations based on the parameters chosen in \cite[Section 5]{sung2015pay}, in particular in the case of identical workers. More precisely, we let for all $i \in \{0, \dots, n\}$, $k^i = k$, $R^i = R$ and $\sigma^i = \sigma$, where $(k, R, \sigma) := (1000, 50, 1)$. We thus represent in the left graphs of \Cref{fig:effort_agents,fig:effort_manager,,fig:principal_value}, respectively the optimal \textit{pay--for--performance sensitivities} for the agents, for the manager, as well as the value per workers for the principal, in three cases:
\begin{enumerate}[label=$(\roman*)$]
    \item in the DC case (blue line), \textit{i.e.}, without any manager (see \Cref{lem:DC_case} for theoretical results);
    \item in the case of \citeauthor{sung2015pay}'s (orange curve), \textit{i.e.}, when the manager's contract is linear;
    \item in our framework (green curve), \textit{i.e.}, when his contract is more sophisticated with an indexation on the quadratic variation.
\end{enumerate}
All curves are represented with respect to the number of workers, starting from $2$ (\textit{i.e.} $n=1$), to consider at least two agents in the DC case or one agent and one manager in the HC case.
Pay--for--performance sensitivity (PPS for short) is a common proxy for the strength of incentives (see \citeayn{gryglewicz2020growth}). In our framework, this sensitivity is directly related to the efforts of the workers. Indeed, for all $i \in \{1,\dots,n\}$, the $i$--th agent's optimal effort is given by $\alpha^{i, {\rm HC}} := k^i Z^i$, where $Z^i$ is precisely the PPS for the $i$--th agent's contract. A similar relation stands for the manager. We could have similarly represented workers' efforts, but we decide to use this indicator to simplify the comparison with the results of \citeauthor{sung2015pay} given in \cite[Table 1]{sung2015pay}, although they have been recalculated in our case for different numbers of agents. On the right graphs of the three figures below, we represent the relative gain induced by considering sophisticated contracts versus linear ones.

\begin{figure}[!ht]
\includegraphics[scale=0.55]{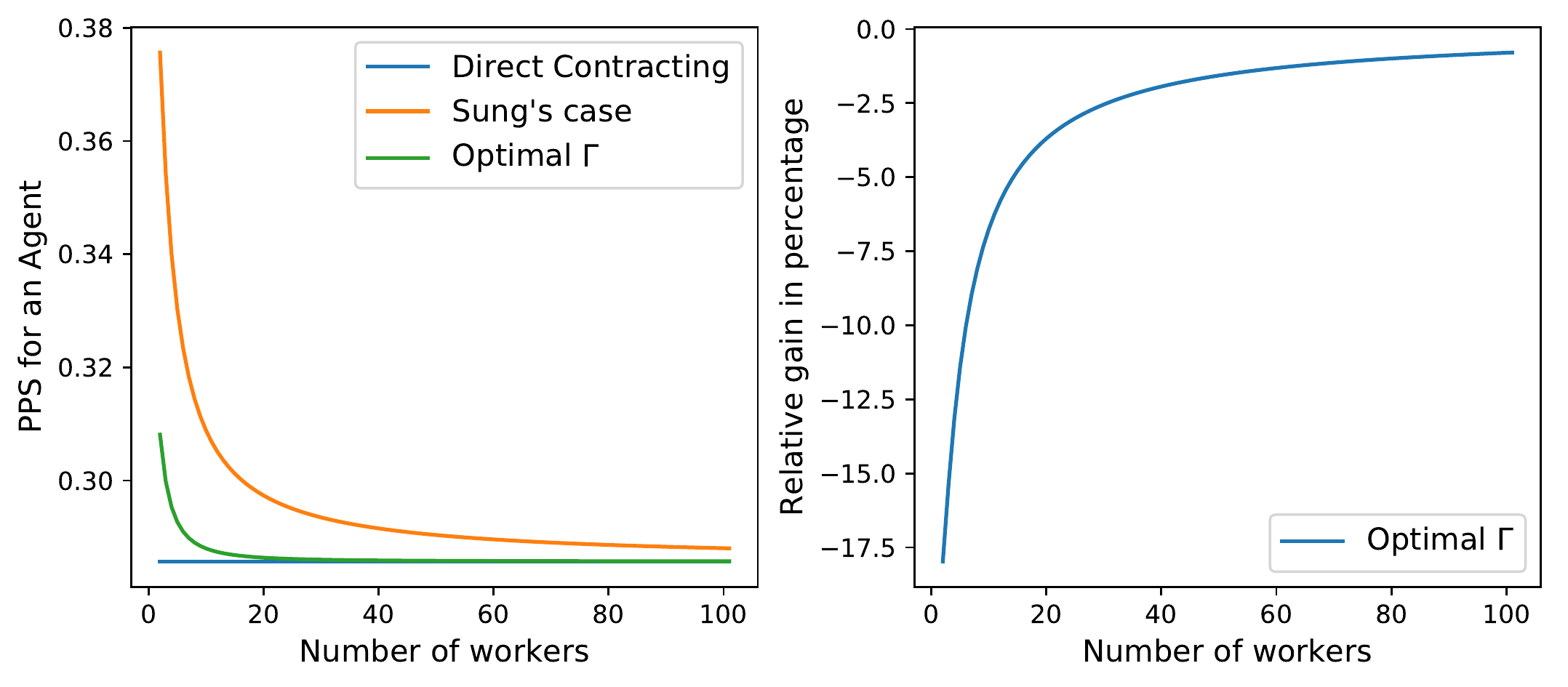}
\centering
\caption{PPS for an agent and relative gain.}
\label{fig:effort_agents}
\end{figure}

Our results obviously present the same features as those outlined in \cite{sung2015pay}. Since the manager can subcontract with the agents, he can benefit, to some extent, from the results of agents' efforts and transfers his own compensation risk to them. As a consequence, agents are induced to work harder than implied in the direct contracting case (see \Cref{fig:effort_agents}, left). To counterbalance this undesirable risk--shifting motivation of the risk--averse manager, the principal set his contract sensitivity to a level lower than that of the contract in the DC case. Consequently, the manager makes less effort than what would be required in the direct contracting situation (see \Cref{fig:effort_manager}, left). In addition, the larger the size of the company, the more motivated the manager is to shift the risk onto the agents. Consequently, the larger the number of workers, the lower--powered the managerial incentive contract. \citeauthor{sung2015pay} concludes that the results obtained with this model on the low managerial effort can serve as an explanation of the empirical finding of \citeayn{jensen1990performance} that the average CEO contract sensitivity of large firms is lower than that of small firms. 

\begin{figure}[!ht]
\includegraphics[scale=0.55]{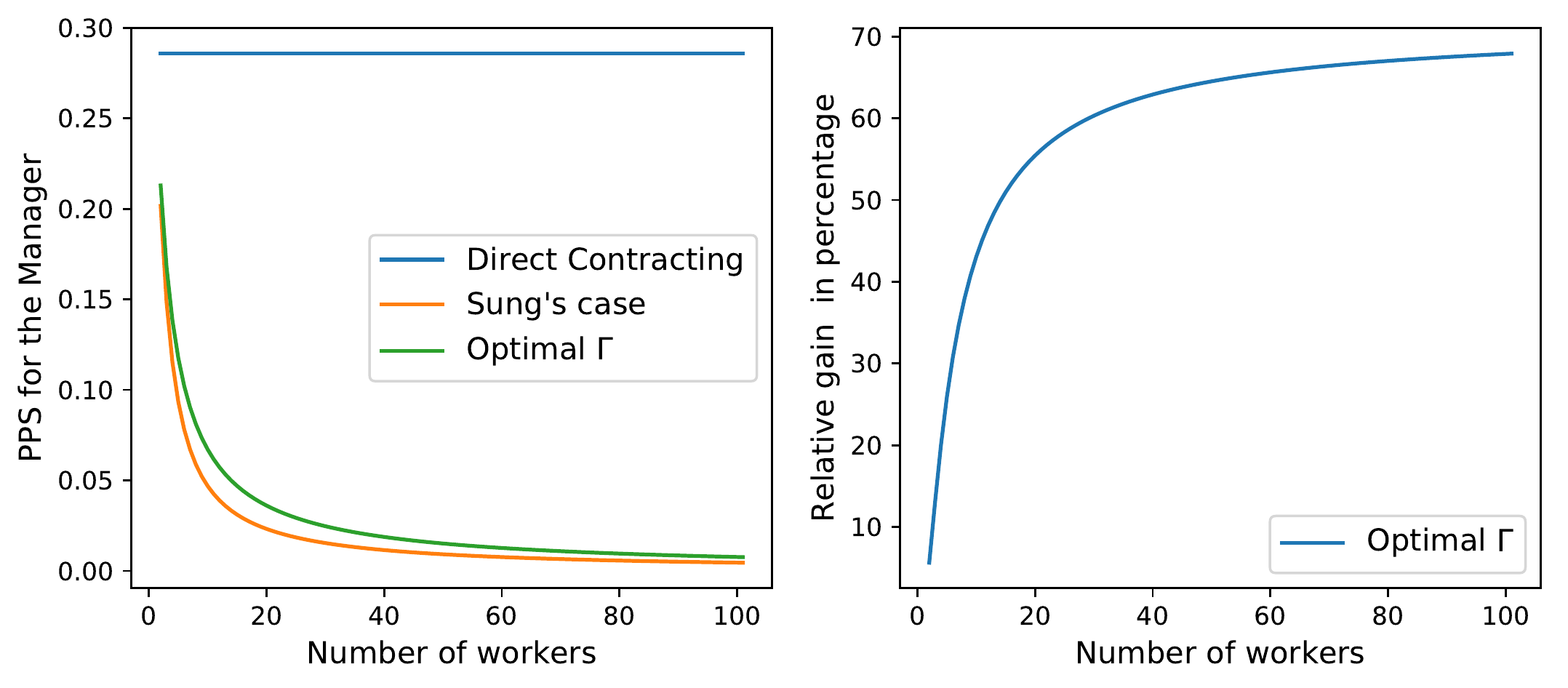}
\centering
\caption{PPS for the manager and relative gain.}
\label{fig:effort_manager}
\end{figure}

\medskip

Nevertheless, our sophisticated contracts allow an improvement of the results. More specifically, the PPSs we obtain, both for the manager and for the agents, are closer to the PPS in the DC case, compared to those obtained by \citeauthor{sung2015pay}. More precisely, using a contract with the quadratic variation for the manager allows the principal to better monitor his own performance. This results in a higher PPS for the manager (see \Cref{fig:effort_manager}, left), and therefore forces him to make more effort. The relative gain (\Cref{fig:effort_manager}, right) is increasing with respect to the number of workers and reaches for example $60$\% for $30$ workers in the company ($29$ agents in addition to the manager). The new contracts we consider therefore mitigates the undesirable risk--shifting motivation of the risk--averse manager. The manager still benefits from the results of agents' efforts and transfers a part of his own compensation risk to them, but less than with linear contracts. Consequently, agents are still induced to work harder than implied in the DC case (see \Cref{fig:effort_agents}, left), but less than in \citeauthor{sung2015pay}'s framework. 

\begin{figure}[!ht]
\includegraphics[scale=0.55]{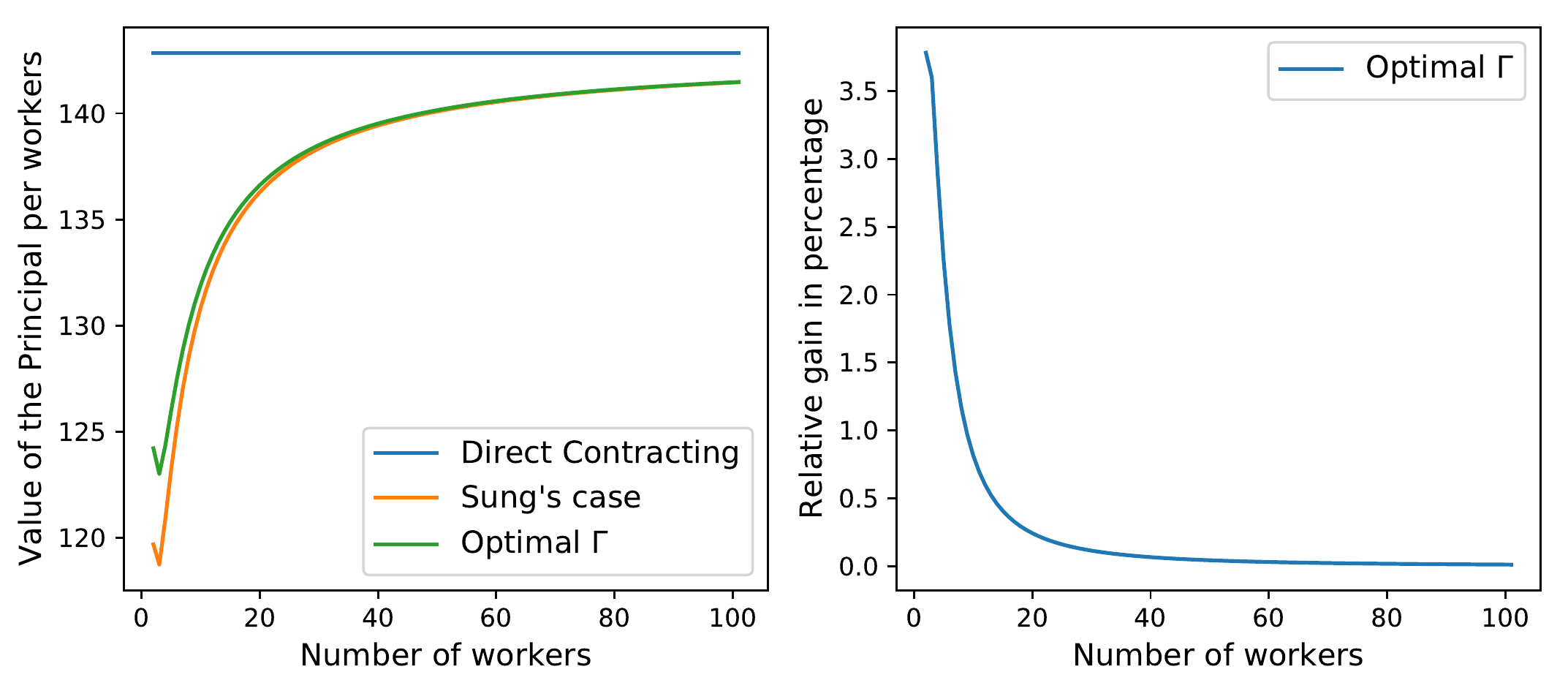}
\centering
\caption{Principal's value per workers.}
\label{fig:principal_value}
\end{figure}

\Cref{fig:principal_value} (left) represents the value function of the principal per workers. With the sophisticated contracts, this value is obviously higher than with linear contracts, which confirms the interest of our study. Even if the relative gain seems small (see \Cref{fig:principal_value}, right), this result motivates a full study with even more sophisticated contracts, in the theoretical part of this paper (from \Cref{sec:general_model} onwards). Indeed, even if this only leads to a small increase in the principal's value per worker when the number of workers is large, the gain has to be multiplied by the number of workers. Moreover, when the number of workers is small, the gain is significant nonetheless, but above all it allows to reduce the effort gaps between the agents and the manager. It is interesting to consider the benefit of these contracts not only from the principal's point of view, but also from a global managerial perspective. Indeed, by developing this type of contracts, the principal better monitors the manager's efforts, and therefore regulates his risk--shifting motivation, which results in improved conditions for the agents and a better division of work and risk between the agents and the manager. 

\begin{remark}
One can notice that the principal's profit per worker is not monotonous when the number of workers is small. In particular, her profit is higher when there are two workers instead of three, while it is then increasing with the number of workers. This is actually explained by the fact that when there is only one agent supervised by the manager, there is less loss of information when going up the hierarchy. Indeed, since the principal can estimate the quadratic variation of $\zeta^{\rm b}$, given in this case by:
\begin{align*}
    \drm \langle \zeta^{\rm b} \rangle_t = \big( | \sigma^0|^2 + |\sigma^1|^2 |1 - Z^{1}_t |^2 \big) \drm t, \; t \in [0,T], 
\end{align*}
she has access \textnormal{(}up to a sign\textnormal{)} to the volatility control of the manager, \textit{i.e.}, the indexation parameter $Z^{1}$. Therefore, in this particular case, there is 'less' moral hazard on volatility control, and we could expect that the model degenerates towards the first--best case regarding volatility control. This fact should $($at least partially$)$ explain the higher profit of the principal.
\end{remark}

\section{Explicit extensions of Sung's model}\label{sec:extensions}

In this section, we propose some basic extensions of \citeauthor{sung2015pay}'s model developed in \Cref{sec:sungmodel}. The first possible extension is to add a coefficient of ability for the manager, allowing to highlight the interest for the principal to delegate the management of the agents. Indeed, a good manager should have a positive impact on the work of the agents below him, and improve the efficiency of the hierarchical organisation to the point that above a certain number of workers, it becomes more profitable to group them in a working team led by a manager. The second extension we propose is to consider a different reporting: the manager will report the sum of the cost and the sum of the outcomes, instead of only reporting the difference between the two sums (the benefits). We will see that more precise reporting will lead to a degeneracy of the model towards the direct contracting case. The last extension considers a more complicated hierarchy, with a top manager in--between the principal and the $m$ managers. In this case, we show that, although it is more complicated to obtain analytical results, the resolution tools remain the same and it is thus theoretically very simple to add a level in the hierarchy using our approach.

\subsection{On the positive impact of hierarchical organisations}

As we have seen in the previous numerical results, the hierarchical organisation considered in \citeauthor{sung2015pay}'s model is not recommended compared to the direct contracting case. Indeed, the utility obtained by the principal, when she hires a manager, is smaller compared to the case when he contracts directly with all the workers. This decrease in utility is linked to the fact that the 'severity' of moral hazard increases with the number of levels in the hierarchy, but, above all, because we do not model any specific ability of the manager. In reality, hierarchical structures appear for logistical reasons, since it would be complicated for a principal to supervise a large number of workers, but also because a manager should have more ability to manage a small group of workers than the owner (or the investors) of the firm. We therefore provide a simple extension to take into account the manager's ability to improve the productivity of his workers.

\medskip

We can say that the manager's skills are defined by a pair $(m,\widetilde m)$ where $m$ measures the help he provides to the agents under his supervision, and $\widetilde m$ is a penalty suffered because of his management activities. Indeed, it seems natural to consider that, by helping the agents, the manager has less time for his own work. We consider that the manager's skill $m > 0$ and $\widetilde m \in [0,1)$ respectively affect the cost functions of the agents and the manager as follows
\begin{align*}
    c^i (a) = \dfrac{1}{2} \dfrac{|a|^2}{k^i \big(1 + \frac{m}{n} \big)}, \; \text{ for } \; i \in \{1, \dots, n\} \; \text{ and } \;
    c^0 (a) = \dfrac{1}{2} \dfrac{|a|^2}{k^0 (1 - \widetilde m)}.
\end{align*}
In the one hand, this means that the manager's ability $m$ decreases the cost of effort of each agent under his supervision. However, the more agents he is responsible for, the weaker the effect is. On the other hand, the parameter $\widetilde m$ increases its own cost, representing the fact that helping agents leaves him less time for his own work. In other words, devoting time to helping the agents penalises his own result. We obtain the same form of solutions as in the previous section, more precisely by replacing $k^i$ and $k^0$ respectively by 
\begin{align*}
    \widetilde k^i := k^i \Big(1 + \frac{m}{n} \Big) \; \text{ and } \; \widetilde k^0 := k^0 (1 - \widetilde m).
\end{align*}
In order to evaluate the effects of the manager's ability parameters, we present some numerical results, with identical workers and the same set of parameters as in the previous simulations (see parameters in \Cref{sss:numerical_results}).

\medskip

In \Cref{fig:with_aptitude}, we represent, in addition to the previous results, the PPS for an agent (on the left), for the manager (in the middle), as well as the value of the principal (on the right) when we take into account the skills of the manager (red curves). Specifically, we set $(m, \widetilde m) = (0.6, 0.1)$. We can see that the agents are incentivised to work more than without the help of their manager, but still less than in \citeauthor{sung2015pay}'s framework. The results are the opposite for the manager. The key point to observe here is that the value of the principal is higher than in the DC case, which means that a hierarchical structure can be beneficial for the principal when the manager has an ability to supervise the agents.

\begin{figure}[!ht]
\includegraphics[width=0.99\linewidth]{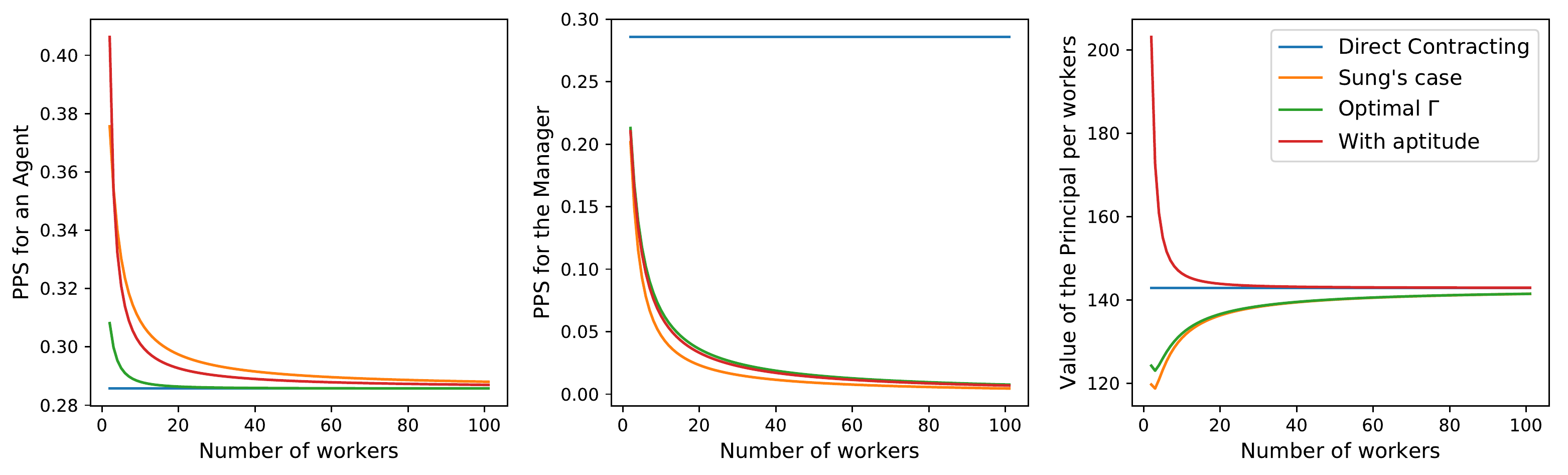}
\centering
\caption{PPS and principal's value per workers.}
\label{fig:with_aptitude}
\end{figure}

\medskip

Therefore, this modeling highlights the interest for the principal to delegate the management of the agents. As can be seen in \Cref{fig:with_aptitude_compare}, if the manager has good management skills, \textit{i.e.}, $m$ large enough, it becomes more cost--effective for the principal to form a working team headed by a manager. 
The right graph of \Cref{fig:with_aptitude_compare} shows that the influence of the parameter $\widetilde m$ is negligible, in the sense that if $m$ is large enough, then the value of the principal is higher than in the DC case regardless of the value of $\widetilde m$. This is in line with the previous results. Indeed, since under a hierarchical contracting, the manager works less than in the DC case, the fact that its cost is higher matters less, as long as his ability is sufficiently beneficent to the agents.

\begin{figure}[!h]
\includegraphics[scale=0.5]{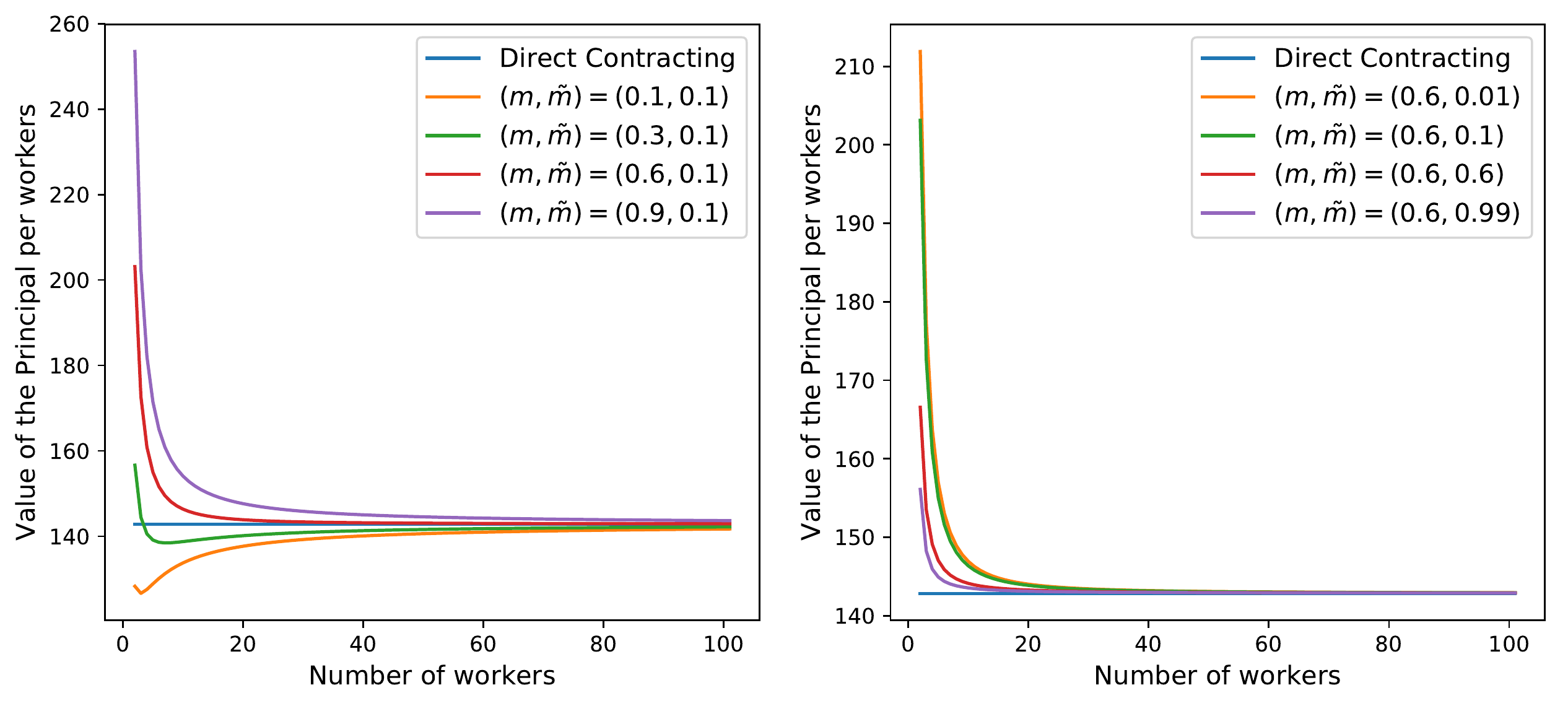}
\centering
\caption{Principal's value per workers for different values of $(m, \widetilde m)$.}
\label{fig:with_aptitude_compare}
\end{figure}

% \todo[inline]{Comparative statics ? Faire une jolie proposition une fois qu'on aura des vrais résultats propres... Dire aussi quels sont les paramètres importants dans la borne inf pour m : $1 + R \sigma^2/k$. Apparait dans le Z optimal de HM ? Qté connue ? Développer ! Quand le coût du manager est très grand par rapport aux agents ?}
% To simplify, we consider that all the agents are identical, in the sense that for all $i \in \{1, \dots, n\}$, $k^i = k$, $R^i = R$ and $\sigma^i = \sigma$, for some $(k, R, \sigma) \in (\R^\star_+)^3$. 

% Nevertheless, for $n$ big enough, the utility of the principal can be approximated by:
% \begin{align*}
%     V^b_0 \approx \dfrac{n}{2} \dfrac{k^2}{k + R \sigma^2} + \dfrac{1}{2} m k^2 \dfrac{k+2R \sigma^2}{\big( k + R \sigma^2 \big)^2} 
% \end{align*}
% and therefore is bigger than the utility in the DC case given by \eqref{eq:value_DC} as soon as
% \begin{align*}
%     m \geq \dfrac{\big(k^0\big)^2}{k^2} \dfrac{\big(k + R \sigma^2 \big)^2}{\big( k + 2R \sigma^2 \big) \big( k^0 + R^0 | \sigma^0 |^2 \big)}.
% \end{align*}

\begin{remark}
    We could consider another case where the manager's skills $m$ and $\widetilde m$ respectively affect the outcomes' drift of the agents and the manager in the following way:
\begin{align*}
    \drm X^i_t =  \bigg( \alpha^i_t + \dfrac{m}{n} \bigg) \drm t + \sigma^i \drm W^i_t, \; \text{ for } \; i \in \{1, \dots, n\} \; \text{ and } \;
    \drm X^0_t = \big( \alpha^0_t - \widetilde m \big) \drm t + \sigma^0 \drm W^0_t, \; t \in [0,T].
\end{align*}
This basic extension finally leads to the same problem and solution as \citeauthor{sung2015pay}'s model, only the utility of the principal is increased by $m - \widetilde m$ \textnormal{(}is decreased if $m - \widetilde m < 0)$. However, this model is not necessarily very realistic, because if the agents do not make any effort, the manager's ability is sufficient to increase the outcomes. The previous model therefore seems more interesting, both in terms of interpretations and results.
\end{remark}

\subsection{On other types of reporting}

Throughout \Cref{sec:sungmodel}, we have assumed that the manager reports only the net benefit $\zeta^{\rm b}$ to his supervisor, the principal. The goal of this subsection is to provide some interesting results on other types of reporting.

\subsubsection{Reporting of profits and costs}\label{sss:2statevariables}

Within the same hierarchical structure as before, we consider here that the manager reports to the principal the sum of the profits and the sum of the costs (and not only the net profit $\zeta^{\rm b}$, \textit{i.e.}, the difference between the two values). Therefore his contract will be indexed on the following $2$--dimensional state variable
\begin{align*}
    \zeta^{\rm pc} = \bigg( \sum_{i=0}^n X^i, \sum_{i=1}^n \xi^i \bigg)^\top.
\end{align*}
Since the agents' problem remain unchanged, the dynamic of $\zeta^{\rm pc}$ under their optimal efforts is given by
\renewcommand*{\arraystretch}{1.5}
\begin{align*}
    \drm \zeta^{\rm pc}_t = 
    \begin{pmatrix}
        \displaystyle \alpha^0_t + \sum_{i=1}^n k^i Z_t^{i} \\
        \displaystyle \dfrac{1}{2} \sum_{i=1}^n | Z^{i}_t |^2 \widetilde R^i
    \end{pmatrix}
    \drm t +
    \begin{pmatrix}
        \sigma^0 & \sigma^1 & \cdots & \sigma^n \\
        0 & \sigma^1  Z^1_t & \cdots & \sigma^n  Z^n_t
    \end{pmatrix}
    \begin{pmatrix}
        \drm W^0_t \\
        \vdots \\
        \drm W^n_t
    \end{pmatrix}, t \in [0,T].
\end{align*}
We consider the same criterion for the manager as before. Following the reasoning behind \Cref{prop:contract_vol}, we are led to consider contracts similar to \eqref{eq:contract_vol}, but now indexed on the $2$--dimensional variable $\zeta^{\rm pc}$:
\begin{align}\label{eq:contract_vol_2D}
    \xi^{\rm pc}_T =  
    \xi^{\rm pc}_0
    - \int_0^T \Hc^{\rm pc} (Z_s, \Gamma_s) \drm s
    + \int_0^T  Z_s \cdot \drm \zeta^{\rm pc}_s
    + \dfrac{1}{2} \int_0^T \mathrm{Tr} \Big[ \big( \Gamma_s + R^{0}  Z_s ( Z_s)^\top \big) \drm \langle \zeta^{\rm pc} \rangle_s \Big],
\end{align}
where $\Hc^{\rm pc}$ is the manager's Hamiltonian, and $(Z, \Gamma)$
is a tuple of parameters optimally chosen by the principal. We thus denote by $\Vc^{\rm pc}$ the set of admissible control processes for the principal, defined as the collection of all processes $(Z,\Gamma) : [0,T] \times \Cc([0,T],\R^2) \longrightarrow \V^{\rm pc}$, predictable with respect to the filtration generated by $\zeta^{\rm pc}$ and satisfying appropriate integrability conditions, where 
\begin{align*}
    \V^{\rm pc} := \Big\{ (z,\gamma) \in \R^2 \times \R^{2 \times 2} \; \text{ s.t. } \forall i \in \{1, \dots, n\}, \; \widetilde R^i z^{2} - | \sigma^i|^2 \gamma^{22} > 0\Big\}.
\end{align*}

Given the dynamic of the state variable $\zeta^{\rm pc}$ and its quadratic variation, we can compute the manager's Hamiltonian and thus establish a result similar to \Cref{prop:manager_sung}.
% \begin{align*}
%     \Hc^{0} \big(z^0, \gamma^0 \big) = &\ \dfrac{1}{2} \gamma^{0,11} \bigg( \big(\sigma^{0}\big)^2 + \sum_{i=1}^n \big( \sigma^i \big)^2 \bigg) + \sup_{a \in A^{0}} \Big\{ 
%     a z^{0,1} - c^{0} (a) \Big\} \\
%     &+ \sum_{i=1}^n \sup_{z^i \in \R} \bigg\{ 
%     z^{i} \Big( z^{0,1} k^i + \gamma^{0,12} \big(\sigma^i\big)^2 \Big) + \frac{1}{2} (z^{i})^2 \Big(z^{0,2} \Big(  k^i + R^i \big( \sigma^i \big)^2 \Big) + \gamma^{0,22} \big(\sigma^i\big)^2 \Big) \bigg\}.
% \end{align*}
\begin{proposition}
    Let $(Z,\Gamma) \in \Vc^{\rm pc}$. The optimal drift effort and the optimal control on the $i$--th agent's compensation, chosen by the manager, are respectively given by $\alpha_t^{\rm pc} := k^0 Z^{1}_t$ and $Z^{i,\rm pc}_t := z^{i,\rm pc} (Z_t,\Gamma_t)$, where
    \begin{align*}
        z^{i,\rm pc} (z, \gamma ) := - \dfrac{k^i z^{1} + |\sigma^i|^2 \gamma^{12} }{ \widetilde R^i z^{2} + | \sigma^i |^2 \gamma^{22}}, \; \text{for} \; (z, \gamma ) \in \V^{\rm pc} \text{ and all }\; i \in \{1, \dots, n\}.
    \end{align*}
    Under the probability $\P^{\rm pc}$ associated to the optimal effort and controls of the manager \textnormal{(}and the agents\textnormal{)}, the dynamics of $\zeta^{\rm pc}$ and $\xi^{\rm pc}$ are respectively given by:
    \renewcommand*{\arraystretch}{1.5}
    \begin{align*}
        \drm \zeta^{\rm pc}_t = &
        \begin{pmatrix}
            \displaystyle k^0 Z^{1}_t + \sum_{i=1}^n k^i z^{i, \rm pc} (Z_t, \Gamma_t) \\
            \displaystyle \dfrac{1}{2} \sum_{i=1}^n \widetilde R^i\big| z^{i, \rm pc} (Z_t, \Gamma_t) \big|^2
        \end{pmatrix}
        \drm t +
        \begin{pmatrix}
            \sigma^0 & \sigma^1 & \cdots & \sigma^n \\
            0 & \sigma^1 z^{1, \rm pc} (Z_t, \Gamma_t) & \cdots & \sigma^n z^{n, \rm pc} (Z_t, \Gamma_t)
        \end{pmatrix}
        \begin{pmatrix}
            \drm W^0_t \\
            \vdots \\
            \drm W^n_t
        \end{pmatrix}, \\
        \displaystyle \drm \xi_t^{\rm pc} = &\ 
        \dfrac{\widetilde R^0}{2} | Z_t^{1} |^2 \drm t
        + \dfrac{R^{0}}{2}  \sum_{i=1}^n |\sigma^i|^2 \big| 
        Z_t^{1} + Z_t^{2} z^{i, \rm pc} (Z_t, \Gamma_t) \big|^2 \drm t
        + Z_t^{1} \sigma^0 \drm W^0_t
        + \sum_{i=1}^n \sigma^i \big( Z_t^{1} + Z_t^{2} z^{i, \rm pc} (Z_t, \Gamma_t) \big)  \drm W^i_t,
    \end{align*}
    where $\widetilde R^0 := k^0 + R^{0} | \sigma^0|^2$.
\end{proposition}

The principal's problem remains essentially unchanged compared to \eqref{eq:principal_pb_sung}, she still maximises the difference between the sum of the outcomes and sum of the compensations owed to the manager and the agents. The only difference is the contract space on which she optimises her criterion. In particular, given the optimal form \eqref{eq:contract_vol_2D} of the contracts, her problem becomes:
\begin{align*}
    V_0^{\rm pc} &:= \sup_{(\xi_0, Z, \Gamma) \in \R^{n+1} \times \Vc^{\rm pc}} \E^{\P^{\rm pc}} \bigg[ \sum_{i=0}^n X_T^i - \sum_{i=1}^n \xi^i - \xi^{\rm pc} \bigg].
\end{align*}
To solve the previous optimisation problem, it is in fact equivalent to consider the following maximisation problem:
\begin{align}\label{eq:sup_2statevariable}
    \sup_{(z, \gamma) \in \V^{\rm pc}} \bigg\{ 
     k^0 z^{1} 
    - \dfrac12 \widetilde R^0 |z^{1}|^2 
    + \sum_{i=1}^n h^{i, \rm pc} (z, \gamma)
    \bigg\},
\end{align}
where $h^{i, \rm pc}$ is defined for all $i \in \{1, \dots, n\}$ and $(z,\gamma)\in \V^{\rm pc}$ by 
\begin{align}\label{eq:hipc}
    h^{i, \rm pc} (z, \gamma) :=  k^i z^{i, \rm pc} (z, \gamma)
    - \frac{1}{2} \widetilde R^i \big| z^{i, \rm pc} (z, \gamma) \big|^2
    - \dfrac{1}{2} R^{0} | \sigma^i|^2 
    \big| z^{1} + z^{2} z^{i, \rm pc} (z, \gamma) \big|^2.
\end{align}

\begin{remark}
    In this setting, we could also consider the case where the manager directly pays the agents he manages \textnormal{(}not only designs their contracts\textnormal{)}. His value function is thus defined by:
    \begin{align*}
        \widetilde V_0^{0, \rm pc} (\widetilde \xi^{\rm pc}) &:= \sup_{(\alpha^0, Z) \in \Ac^0 \times \Vc^{0}} \E^{\P^{\rm pc}} \bigg[ - \exp \bigg(- R^{0} \bigg(\widetilde \xi^{\rm pc} - \sum_{i=1}^n \xi^i - \int_0^T c^{0}(\alpha_t^{0}) \drm t \bigg) \bigg) \bigg],
    \end{align*}
    and the optimal form of contract for the manager is as follows:
    \begin{align*}
        \widetilde \xi^{\rm pc} =  
        \widetilde \xi^{\rm pc}_0
        + \sum_{i=1}^n \xi^i
        - \int_0^T \Hc^{\rm pc} (Z_s, \Gamma_s) \drm s
        + \int_0^T  Z_s \cdot \drm \zeta^{\rm pc}_s
        + \dfrac{1}{2} \int_0^T \mathrm{Tr} \Big[ \big( \Gamma_s + R^{0}  Z_s ( Z_s)^\top \big) \drm \langle \zeta^{\rm pc} \rangle_s \Big].
    \end{align*}
    Since the manager is directly paying the agents, the problem of the principal is only to maximise the difference between the sum of the outcomes and the compensation due to the manager:
    \begin{align*}
        \widetilde V_0^{\rm pc} &:= \sup_{(Z, \Gamma) \in \Vc^{\rm pc}} \widetilde J_0^{\rm P} \big( \widetilde \xi^{\rm pc} \big), \; \text{ where } \; \widetilde J_0^{\rm P} \big(  \widetilde \xi^{\rm pc} \big) := \E^{\P^{\rm pc}} \bigg[ \sum_{i=0}^n X_T^i - \widetilde \xi^{\rm pc} \bigg],
    \end{align*}
    which leads to the exact same maximisation as before. Indeed, in this case, the compensation for the manager $\widetilde \xi^{\rm pc}$ is equal to the sum of the compensations for the agents and the contract $\xi^{\rm pc}$ of the previous case: 
    \begin{align*}
        \widetilde \xi^{\rm pc} = \xi^{\rm pc} + \sum_{i=1}^n \xi^i,
    \end{align*}
    and thus the two frameworks are strictly equivalent, since 
    \begin{align*}
        \widetilde J_0^{\rm P} \big(  \widetilde \xi^{\rm pc} \big) = \E^{\P^{\rm pc}} \bigg[ \sum_{i=0}^n X_T^i - \sum_{i=1}^n \xi^i - \xi^{\rm pc} \bigg].
    \end{align*}
\end{remark}

The supremum given by \eqref{eq:sup_2statevariable} is not easily computable in the general case, but if all agents have the same characteristics, we obtain the following result, whose proof is postponed to \Cref{app:ss:proof_extension}.
\begin{proposition}\label{prop:cas_pc0_identical}
    If all agents are identical, the principal can achieve her utility in the DC case, denoted by $V^{\rm DC}$ and given by \eqref{eq:value_DC} in \textnormal{\Cref{lem:DC_case}}.
\end{proposition}

\subsubsection{Separate reporting of the manager's performance}\label{sss:reporting_denegerate}

For now, we focused our study on two frameworks, one where the manager reports only net benefits (\Cref{sec:sungmodel}), the other where he reports the sum of the total profit (sum of outcomes) and the total cost (sum of payments) separately (\Cref{sss:2statevariables}). We could consider other scenarii where the manager reports more information to the principal, for example if he reports his personal outcome $X^0$ separately. However, with this reporting, the HC case degenerates towards the DC case, in the sense that optimal efforts of the agents and the manager, as well as the value of the principal, can be equal to those in the DC case, given by \Cref{lem:DC_case}. The proof of the following proposition is postponed to \Cref{app:ss:proof_extension}.

\begin{proposition}\label{prop:other_reporting}
    If the manager reports to the principal his own outcome $X^0$ separately, the problem degenerates towards the DC case. More precisely:
    \begin{enumerate}[label=$(\roman*)$]
        \item if the manager reports $\zeta^{\rm b,0} = \big( \sum_{i=1}^n X^i - \sum_{i=1}^n \xi^i, X^0 \big)$, it is possible to find a sequence of contracts such that, at the limit, all workers apply the optimal efforts of the DC case, and the principal receives the maximum utility possible, \textit{i.e.}, $V^{\rm DC}$;
        \item if the manager reports $\zeta^{\rm pc,0} = \big( \sum_{i=1}^n X^i, \sum_{i=1}^n \xi^i, X^0 \big)$, the result of $(i)$ holds and moreover, if agents are identical, we can find a contract which allows to attain the DC case.
    \end{enumerate}
\end{proposition}
Since these reporting leads to a degeneration towards the DC case, they are less interesting mathematically speaking. However, from a managerial point of view, it is relevant to observe that reporting the manager's output separately makes it possible to reduce the moral hazard within the hierarchy.

\subsection{On a more complex hierarchy}\label{ss:complexe_hierarchy}

We consider in this section a more complex hierarchy illustrated by \Cref{fig:complex_hierarchy}: the principal hires a top manager, who hires $m$ managers, and each manager $j$ hires $n_j$ agents. This new hierarchy requires small adjustments in notations, which will be reused in the general model. First, the top manager controls his own outcome $X^{0}$ and receives the compensation $\xi^{0}$ designed by the principal. Then, the $m$ managers, indexed by $j \in \{1, \dots, m\}$, each carry out their own outcome $X^{j,0}$ and receive the compensation $\xi^{j,0}$ designed by the top manager. For $j \in \{1, \dots, m\}$ and $i \in \{1, \dots, n_j \}$, the $(j,i)$--th agent is the $i$--th agent of the $j$--th manager. He controls the output $X^{j,i}$ and will receive the compensation $\xi^{j,i}$ designed by his manager. The dynamics of the output processes are given by
\begin{align*}
    \drm X^{0}_t = \alpha^0 \drm t + \sigma^{0} \drm W_t^{0} \; 
    \text{ and } \;
    \drm X^{j,i}_t = \alpha_t^{j,i} \drm t + \sigma^{j,i} \drm W_t^{j,i}, \;  \; t \in [0,T],
\end{align*}
for all $j \in \{1, \dots, m\}$ and $i \in \{0, \dots, n_j \}$, where $W^{0}$ and all $W^{j,i}$ are independent Brownian motions.

\begin{figure}[h!]
\begin{center}
\begin{tikzpicture}
% style des nœuds
\tikzstyle{principal}=[rectangle,draw,rounded corners=4pt,fill=cyan!50]
\tikzstyle{agents}=[rectangle,draw,rounded corners=4pt, fill=green!50]
\tikzstyle{manager}=[rectangle,draw,rounded corners=4pt, fill=red!25]
\tikzstyle{topmanager}=[rectangle,draw,rounded corners=4pt, fill=magenta!25]
\tikzstyle{invisible}=[rectangle]

% style des flèches
%\tikzstyle{suite}=[->,>=stealth’,thick,rounded corners=4pt]
% placement des nœuds
\node[principal] (principal) at (0,2) {Principal};
\node[topmanager] (topmanager) at (0,0) {Top manager};
\node[manager] (manager1) at (-4,-2) {Manager $1$};
\node[invisible] (manageri) at (0,-2) {$\dots$};
\node[invisible] (agentii) at (0,-4) {$\dots$};
\node[manager] (managerm) at (4,-2) {Manager $m$};
\node[agents] (agent11) at (-5.5,-4) {Agent $1,1$};
\node[invisible] (agent1i) at (-4,-4) {$\dots$};
\node[agents] (agent1n) at (-2.5,-4) {Agent $1,n_1$};
\node[agents] (agentm1) at (2.5,-4) {Agent $m,1$};
\node[invisible] (agentmi) at (4,-4) {$\dots$};
\node[agents] (agentmn) at (5.5,-4) {Agent $m,n_m$};

% Placement des flèches
\draw[->] (principal) -- (topmanager) node[midway,fill=white]{$\xi^0$};
\draw[->] (topmanager) -- (manager1) node[midway,fill=white]{$\xi^{1,0}$};
\draw[->, dashed] (topmanager) -- (manageri);
\draw[->] (topmanager) -- (managerm) node[midway,fill=white]{$\xi^{m,0}$};

\draw[->, dashed] (manageri) -- (agentii);

\draw[->] (manager1) -- (agent11) node[midway,fill=white]{$\xi^{1,1}$};
\draw[->, dashed] (manager1) -- (agent1i);
\draw[->] (manager1) -- (agent1n) node[midway,fill=white]{$\xi^{1,n_1}$};

\draw[->] (managerm) -- (agentm1) node[midway,fill=white]{$\xi^{m,1}$};
\draw[->, dashed] (managerm) -- (agentmi);
\draw[->] (managerm) -- (agentmn) node[midway,fill=white]{$\xi^{m,n_m}$};

\end{tikzpicture}
\caption{A more complex hierarchy}\label{fig:complex_hierarchy}
\end{center}
\end{figure}
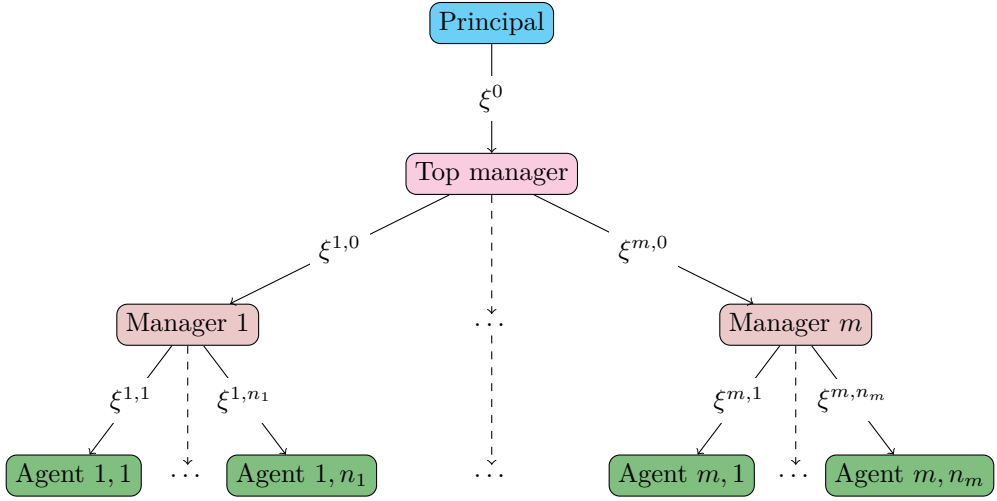

\textcursive{Agent's problem.}\vspace{-0.2em} Apart from these notation changes, the problem for the agents remains the same. Therefore, for $j \in \{1, \dots, m\}$ and $i \in \{1, \dots, n_j \}$, the optimal contract form for the $(j,i)$--th agent is similar to \eqref{eq:contract_drift}:
\begin{align*}
    \xi^{j,i}_t = \xi_0^{j,i} - \int_0^t \Hc^{j,i} (Z^{j,i}_s) \drm s + \int_0^t Z^{j,i}_s \drm X_s^{j,i} + \dfrac{1}{2} R^{j,i} \int_0^t \big| Z^{j,i}_s \big|^2 \drm \langle X^{j,i} \rangle_s, \; t \in [0,T],
\end{align*}
where $Z^{j,i}$ is an $\R$--valued process, predictable with respect to the filtration generated by $X^{j,i}$, satisfying appropriate integrability conditions, chosen by the $j$--th manager. This contract leads to the optimal effort $\alpha^{j,i, \star} = k^{j,i} Z^{j,i}$.

\medskip

\textcursive{Manager's problem.}\vspace{-0.4em} Let $j \in \{1, \dots, m\}$. The problem of the $j$--th manager is also equivalent to the manager's problem in \citeauthor{sung2015pay}'s model. In addition to choosing his effort $\alpha^{j,0}$, he designs the compensation $\xi^{j,i}$ for the $i$--th agent under his supervision by choosing the payment rate $Z^{j,i}$, for $i \in \{1, \dots, n_j \}$, and receives the payment $\xi^{j,0}$ offered by his supervisor (the top manager). We assume that the $j$--th manager only reports in continuous time to his supervisor the total benefit of his working team, \textit{i.e.}, the following variable
\begin{align*}
    \zeta^j_t = \sum_{i=0}^{n_j} X_t^{j,i} - \sum_{i=1}^{n_j} \xi_t^{j,i}, \; t \in [0,T].
\end{align*}
Under optimal efforts of the agents, and using the notation $\widetilde R^{j,i} := k^{j,i} + R^{j,i} | \sigma^{j,i} |^2$, the dynamic of $\zeta^j$ is given by
\begin{align*}
    \drm \zeta^j_t = \bigg( \alpha_t^{j,0} + \sum_{i=1}^{n_j} \Big( k^{j,i} Z_t^{j,i} - \frac{1}{2} \widetilde R^{j,i} | Z^{j,i}_t|^2 \Big) \bigg) \drm t + \sigma^{j,0} \drm W_t^{j,0} + \sum_{i=1}^{n_j} \sigma^{j,i} ( 1 - Z^{j,i}_t) \drm W_t^{j,i}, \; t \in [0,T].
\end{align*}

We assume that the contract for each manager is indexed on the total benefit of his working team, \textit{i.e.}, each contract $\xi^{j,0}$ is a measurable function of $\zeta^j$ only. Given the form of his value function, $\zeta^j$ is the only state variable of the $j$--th manager's control problem. Since he controls both the drift and the volatility of $\zeta^j$, the optimal form of contract is given by:
\begin{align}\label{eq:contractmanager}
    \xi^{j,0}_t =  
    \xi^{j,0}_0
    - \int_0^t \Hc^{j} (Z_s^{j}, \Gamma_s^{j}) \drm s
    + \int_0^t  Z_s^{j} \drm \zeta^j_s
    + \dfrac{1}{2} \int_0^t \big( \Gamma_s^{j} + R^{j} | Z_s^{j}|^2 \big) \drm \langle \zeta^j \rangle_s, \; t \in [0,T],
\end{align}
where $\Hc^{j}$ is his Hamiltonian and $(Z^{j}, \Gamma^{j}) \in \Vc^j$ is a pair of parameters optimally chosen by the top manager. More precisely, we define by $\Vc^{j}$ the collection of all processes $(Z,\Gamma) : [0,T] \times \Cc([0,T],\R) \longrightarrow \V^{j}$, predictable with respect to the filtration generated by $\zeta^{j}$ and satisfying appropriate integrability conditions, where 
\begin{align*}
    \V^{j} := \Big\{ (z,\gamma) \in \R \times \R \; \text{ s.t. } \forall i \in \{1, \dots, n_j\}, \; \widetilde R^{j,i} z - |\sigma^{j,i}|^2 \gamma > 0\Big\}.
\end{align*}
The set $\Vc^0 := \prod_{j = 1}^m \Vc^{j}$ thus represents the admissible control processes for the top manager. By computing and maximising the $j$--th manager's Hamiltonian, we obtain the following proposition.
\begin{proposition}
    Let consider a contract of the form \eqref{eq:contractmanager} indexed by a pair $(Z^j, \Gamma^j) \in \Vc^j$. Then, for all $t \in [0,T]$, the optimal effort on the drift of the $j$--th manager's outcome $X^j$ is $\alpha_t^{j,0,\star} := k^{j,0} Z^j_t$, and for all $i \in \{1, \dots, n_j \}$, his optimal control on the $i$--th agent's compensation is $Z^{j,i,\star}_t := z^{j,i,\star} (Z^{j}_t,\Gamma^{j}_t)$ where
    \begin{align*}
        z^{j,i,\star} (z,\gamma) := \dfrac{k^{j,i} z - | \sigma^{j,i} |^2 \gamma }{ \widetilde R^{j,i} z - |\sigma^{j,i}|^2 \gamma}, \text{ for all } (z, \gamma) \in \V^j.
    \end{align*}
\end{proposition}

\begin{remark}\label{rk:measurability_issues_extension}
The problem highlighted in \textnormal{\Cref{rk:measurability_issues}} obviously also arises here. Indeed, to restrict the contract $\xi^{j,i}$ for the $(j,i)$--th agent to a measurable function of $X^{j,i}$, the payment process $Z^{j,i}$ must be predictable with respect to the filtration generated by $X^{j,i}$. Similarly, to restrict the contract $\xi^{j,0}$ for the $j$--th manager to a measurable function of $\zeta^{j}$, the payment rate processes $Z^{j}$ and $\Gamma^j$ must be predictable with respect to the filtration generated by $\zeta^{j}$. Since the optimal payment rate $Z^{j,i,\star}$ is a function of $Z^j$ and $\Gamma^j$, the model is consistent if and only if $Z^j$ and $\Gamma^j$ are deterministic functions of time only, which is actually the case in this example \textnormal{(}they are even constant\textnormal{)}.
\end{remark}

\textcursive{Top manager's problem.}\vspace{-0.1em} The top manager carries out his own output $X^{0}$, by choosing his effort level $\alpha^{0}$, and designs the contracts for the $m$ managers. Like other managers, he has a CARA utility function with a risk--aversion parameter $R^{0}$, and maximises the utility of the difference between the payment he receives from the principal, $\xi^{0}$, and his cost of effort:
\begin{align*}
    V_0^{0} (\xi^{0}) &:= \sup_{(\alpha^{0}, \Zc^0) \in \Ac^{0} \times \Vc^0} \E^{\P^{0}} \Big[ - \erm^{- R^{0} \big( \xi^{0} - \int_0^T c^{0}(\alpha_t^{0}) \drm t \big) }\Big],
\end{align*}
where $\P^0$ is the probability associated to the effort $\alpha^0 \in \Ac^0$ and the choice of the process $\Zc^0 := (Z^j, \Gamma^j)_{j=1}^m \in \Vc^0$, under the optimal efforts of the $m$ managers and their agents. 
In this setting, the top manager observes in continuous time the net benefit of each working team led by a manager, that is the tuple $(\zeta^j)_{j=1}^m$. 
Moreover, like every managers, the top manager reports in continuous time to the principal the benefits of his team of workers composed of all managers and agents below him, namely the following variable:
\begin{align*}
    \zeta^0_t =  X_t^0 + \sum_{j=1}^m \zeta_t^j - \sum_{j=1}^m \xi_t^{j}, \; t \in [0,T].
\end{align*}

Therefore the principal can only offer to the top manager a contract indexed on $\zeta^0$, and, since he controls the volatility of $\zeta^0$ through his choice of contracts for the managers, the optimal form of his compensation is the same as \eqref{eq:contractmanager} but indexed on the variable $\zeta^0$:
\begin{align*}
    \xi^{0} :=  
    \xi^{0}_0
    - \int_0^T \Hc^{0} (Z_s, \Gamma_s) \drm s
    + \int_0^T  Z_s \drm \zeta^0_s
    + \dfrac{1}{2} \int_0^T  \big( \Gamma_s + R^{0} |Z_s|^2 \big) \drm \langle \zeta^0 \rangle_s,
\end{align*}
where $(Z, \Gamma) \in \Vc$ is a pair of processes optimally chosen by the principal. More precisely, we define by $\Vc$ the collection of all processes $(Z,\Gamma) : [0,T] \times \Cc([0,T],\R) \longrightarrow \V$, predictable with respect to the filtration generated by $\zeta^{0}$ and satisfying appropriate integrability conditions, where $\V$ is the set of all $(z,\gamma) \in \R^2$ such that the top manager's Hamiltonian $\Hc^0$ defined below by \eqref{eq:top_manager_hamiltonian} is well defined.

\medskip

Under optimal efforts and controls of the managers and the agents (and associated probability $\P^0$), the dynamic of $\zeta^0$ is given for all $t \in [0,T]$ by:
\begin{align*}
    \drm \zeta^0_t =  \bigg( \alpha^0_t
    + \sum_{j=1}^m  h^{0,j} \big(Z^{j}_t,\Gamma^{j}_t \big)  \bigg) \drm t
    + \sigma^0 \drm W^0_t
    + \sum_{j=1}^m
    ( 1 - Z_t^{j}) \bigg( \sigma^{j} \drm W_t^{j} 
    + \sum_{i=1}^{n_j} \sigma^{j,i} \big( 1 - z^{j,i,\star} \big(Z^{j}_t,\Gamma^{j}_t \big) \big) \drm W_t^{j,i} 
    \bigg),
\end{align*}
where, for all $j \in \{1, \dots, m\}$, $\widetilde R^{j} := k^j + R^{j} | \sigma^{j} |^2$ and in addition for all $(z,\gamma) \in \V^j$,
\begin{align}\label{eq:h0j}
    h^{0,j} (z, \gamma) :=  k^j z
    - \dfrac{\widetilde R^{j}}{2} | z |^2 
    + \sum_{i=1}^{n_j} \bigg( k^{j,i} z^{j,i,\star} (z,\gamma) 
    - \frac{\widetilde R^{j,i}}{2} \big| z^{j,i,\star}  (z,\gamma)  \big|^2
    - \dfrac{ R^{j}}{2} |z |^2 |\sigma^{j,i}|^2 \big| 1 - z^{j,i,\star}  (z,\gamma) \big|^2 \bigg).
\end{align}
% \begin{align*}
%     \drm \zeta^j_t = &\ \bigg[ k^j Z^j_t + \sum_{i=1}^{n_j} \bigg( k^{j,i} z^{j,i,\star} (Z^{j}_t,\Gamma^{j}_t) - \frac{1}{2} \big( z^{j,i,\star} (Z^{j}_t,\Gamma^{j}_t) \big)^2 \Big(  k^{j,i} + R^{j,i} \big( \sigma^{j,i} \big)^2 \Big) \bigg) \bigg] \drm t \\
%     &+ \sigma^{j} \drm W_t^{j} + \sum_{i=1}^{n_j} \sigma^{j,i} \big( 1 - z^{j,i,\star} \big(Z^{j}_t,\Gamma^{j}_t \big) \big) \drm W_t^{j,i}.
% \end{align*}
% and
% \begin{align*}
%     \drm \xi^{j}_t = &\ \dfrac12 \big( Z_t^j \big)^2 \bigg( k^j  
%     + R^{j} \big( \sigma^{j} \big)^2 + R^{j} \sum_{i=1}^{n_j} \big( \sigma^{j,i} \big)^2 \big( 1 - z^{j,i,\star} (Z^{j}_t,\Gamma^{j}_t) \big)^2 \bigg) \drm t
%     +  Z_t^{j} \sigma^{j} \drm W_t^{j} +  Z_t^{j} \sum_{i=1}^{n_j} \sigma^{j,i} ( 1 - z^{j,i,\star} (Z^{j}_t,\Gamma^{j}_t) ) \drm W_t^{j,i}.
% \end{align*}
Therefore, the top manager's Hamiltonian is defined as follows:
\begin{align}\label{eq:top_manager_hamiltonian}
    \Hc^{0} (z, \gamma) :=  
    & \sup_{a \in A^0} \big\{ a z - c^0 (a) \big\}
    + \sum_{j=1}^m \sup_{(z^j, \gamma^j) \in \V^j} \bigg\{ 
    z h^{0,j} (z^j, \gamma^j)
    + \dfrac12 \gamma 
    | 1 - z^{j} |^2 \bigg( | \sigma^{j} |^2 
    + \sum_{i=1}^{n_j} |\sigma^{j,i}|^2 \big| 1 - z^{j,i,\star} (z^{j},\gamma^{j}) \big|^2 \bigg) \bigg\} \nonumber \\
    &+ \dfrac12 \gamma | \sigma^0 |^2.
\end{align}
The first supremum is attained for the optimal effort $\alpha_t^{0,\star} = k^0 Z_t$ for $t \in [0,T]$. In addition, the optimal control of the top manager for the managers' contracts are given for all $j \in \{1, \dots m\}$ and all $t \in [0,T]$ by $Z_t^{j,\star} := z^{j,\star} (Z_t, \Gamma_t)$ and $\Gamma_t^{j,\star} := \gamma^{j,\star} (Z_t, \Gamma_t)$, for
\begin{align*}
    \gamma^{j, \star} (z,\gamma)  = - R^{j} \big( z^{j, \star} (z,\gamma) \big)^3 + \dfrac{\gamma}{z}  z^{j, \star} (z,\gamma) \big| 1 - z^{j, \star} (z,\gamma) \big|^2,
\end{align*}
where $z^{j, \star} (z,\gamma)$ can be numerically computed as the maximiser of the previous Hamiltonian $\Hc^0$, for all $(z, \gamma) \in \V$. Abusing notations slightly for simplicity, we will denote in the following, for all $(z, \gamma) \in \V$,
\begin{align}\label{eq:zjistar}
    h^{0,j,\star} (z, \gamma) := h^{0,j} \big(z^{j, \star} (z,\gamma),\gamma^{j, \star} (z,\gamma) \big) \; \text{ and } \; z^{j,i,\star} (z,\gamma) := z^{j,i,\star} \big(z^{j, \star} (z,\gamma),\gamma^{j, \star} (z,\gamma) \big).
\end{align} 

\begin{remark}\label{rk:measurability_issues_extension_top}
    In addition to the measurability issues of the manager's control, mentioned in \textnormal{\Cref{rk:measurability_issues_extension}},
    we also have, in this more complex model, a measurability problem for the top manager's control. Indeed, since we restrict the space of the contracts for the $j$--th manager to measurable contracts with respect to his own $\zeta^j$, the processes $Z^j$ and $\Gamma^j$ defining the contract must be adapted to the filtration generated by $\zeta^j$. In fact, these processes, chosen by the top manager, are functions of $Z$ and $\Gamma$, the principal's controls. Since the principal only observes $\zeta^0$, the processes $Z$ and $\Gamma$ should be adapted to the filtration generated by $\zeta^0$, contradicting the fact that $Z^j$ and $\Gamma^j$ are adapted to the filtration generated by $\zeta^j$.
    Again, this question of measurability is actually not a problem in this particular case since all optimal controls are deterministic or even constant, but will be in a more general framework.
\end{remark}

\textcursive{Principal's problem.}\vspace{-0.1em} We consider the same problem for the principal as before, namely that she maximises the difference between the sum of the outcomes and the sum of the costs at terminal time $T$, which can be summarised by $\zeta_T^0 - \xi^0$. Her problem is reduced to the optimal choice of the indexation parameters in the contract $\xi^0$, the pair $(Z, \Gamma)$, and thus to the following maximisation problem:
\begin{align*}
    \sup_{(z, \gamma) \in \V} \Bigg\{ 
    k^0 z
    - \dfrac{\widetilde R^{0}}{2} |z|^2
    + \sum_{j=1}^m \bigg( h^{0,j,\star} (z, \gamma) 
    - \dfrac{R^{0}}{2} |z|^2 \big|1-z^{j, \star}(z, \gamma)\big|^2
    \bigg( |\sigma^{j}|^2
    + \sum_{i=1}^{n_j} |\sigma^{j,i}|^2 \big| 1 - z^{j,i,\star}(z,\gamma) \big|^2 \bigg)
    \bigg)
    \Bigg\},
\end{align*}
where $\widetilde R^{0} := k^{0} + R^{0} | \sigma^{0}|^2$ and using the notations defined by \eqref{eq:zjistar}.
Optimal parameters $(z,\gamma) \in \V$ are constant over time, and their values can be obtained thanks to a simple numerical optimisation.

\medskip

By comparing to the first example where there is no top manager and only one manager, we can see that the structure of the problem is the same. Adding a level in the hierarchy is no more complicated, all it takes is writing an additional control problem. With this in mind, and in order to avoid overloading the notations, we will consider only three levels in the hierarchy for the general model, \textit{i.e.}, one principal, $m$ managers, with a fixed number of agents each.

\section{The general model: preliminaries}\label{sec:general_model}

For the general model, we focus on the following hierarchy, represented in \Cref{fig:hierarchy_general}: the principal contracts with $m$ managers, and each manager $j$ for $j \in \{1, \dots, m\}$ in turn subcontracts with $n_j$ agents, indexed by $(j,i)$ for $i \in \{1, \dots, n_j \}$. The $(j,i)$--th agent is therefore the $i$--th agent of the $j$--th manager. The term \textit{workers} will refer to the actors in the hierarchy who are in charge of managing a project, \textit{i.e.}, both agents and managers. The total number of workers is given by $w := m + \sum_{j = 1}^m n_j$. 
Fix throughout the general model a positive integer $d$, which represents the dimension of the noise which affects each project.\footnote{We assume here that $d$ does not depend on a specific agent. This is without loss of generality, as we can always add unused coordinates to a given project.} The $(j,i)$--th agent will manage the project with output $X^{j,i}$, while the $j$--th manager is in charge of the project with output $X^{j,0}$. 
We assume here for simplicity that the outputs are one--dimensional\footnote{We could consider that each output $X^{j,i}$ is $k$--dimensional for $k > 0$. In this case, each coordinate of $X^{j,i}$ can be interpreted as the profit generated by a task managed by the $(j,i)$--th worker. Nevertheless, at some point we would be led to consider the total profit generated by a worker, which will naturally corresponds to the sum of the coordinates of his output. Therefore, to simplify, we choose to directly consider each output $X^{j,i}$ as the total profit generated by the $(j,i)$--th worker, and thus avoid increasing the notations by considering $k$--dimensional vectors and finally taking their inner product with $\mathbf{1}_k$.} and uncorrelated, meaning that tasks to be performed have been clearly separated. Moreover, each worker in the hierarchy can only impact directly his own project. This set up can be justified by the fact that each worker have a specific set of skills, implying that they are the only ones able to perform their tasks. In fact, we could let them interact, but this would make the Nash equilibrium hard to solve. In addition, interactions between workers will naturally occur at the level of managers, and therefore a way to handle this issue will be explained at that time.
The most important aspect that we wish to address in this paper is the loss of information by moving up the hierarchy. To model this, we assume that each manager $j$ only reports the results of his team to the principal through a (possibly multidimensional) variable $\zeta^j$, as in \citeauthor{sung2015pay}'s model detailed in the previous sections. Thus, the principal does not have a separate access to the results of each worker, which could lead to a degeneracy of the problem towards the DC case, as we have seen in particular in \Cref{sss:reporting_denegerate}. 

\begin{figure}[h!]
\begin{center}
\begin{tikzpicture}
% style des nœuds
\tikzstyle{principal}=[rectangle,draw,rounded corners=4pt,fill=cyan!50]
\tikzstyle{agents}=[rectangle,draw,rounded corners=4pt, fill=green!50]
\tikzstyle{manager}=[rectangle,draw,rounded corners=4pt, fill=red!25]
% \tikzstyle{topmanager}=[rectangle,draw,rounded corners=4pt, fill=magenta!25]
\tikzstyle{invisible}=[rectangle]

% style des flèches
%\tikzstyle{suite}=[->,>=stealth’,thick,rounded corners=4pt]
% placement des nœuds
\node[principal] (principal) at (0,0) {Principal};
% \node[topmanager] (topmanager) at (0,0) {top manager};
\node[manager] (manager1) at (-4,-2) {Manager $1$};
\node[invisible] (manageri) at (0,-2) {$\dots$};
\node[manager] (managerm) at (4,-2) {Manager $m$};
\node[invisible] (agentii) at (0,-4) {$\dots$};
\node[agents] (agent11) at (-5.5,-4) {Agent $(1,1)$};
\node[invisible] (agent1i) at (-4,-4) {$\dots$};
\node[agents] (agent1n) at (-2.5,-4) {Agent $(1,n_1)$};
\node[agents] (agentm1) at (2.5,-4) {Agent $(m,1)$};
\node[invisible] (agentmi) at (4,-4) {$\dots$};
\node[agents] (agentmn) at (5.5,-4) {Agent $(m,n_m)$};

% Placement des flèches
% \draw[->] (principal) -- (topmanager) node[midway,fill=white]{$\xi^0$};
\draw[->] (principal) -- (manager1) node[midway,fill=white]{$\xi^{1}$};
\draw[->, dashed] (principal) -- (manageri);
\draw[->] (principal) -- (managerm) node[midway,fill=white]{$\xi^{m}$};

\draw[->, dashed] (manageri) -- (agentii);

\draw[->] (manager1) -- (agent11) node[midway,fill=white]{$\xi^{1,1}$};
\draw[->, dashed] (manager1) -- (agent1i);
\draw[->] (manager1) -- (agent1n) node[midway,fill=white]{$\xi^{1,n_1}$};

\draw[->] (managerm) -- (agentm1) node[midway,fill=white]{$\xi^{m,1}$};
\draw[->, dashed] (managerm) -- (agentmi);
\draw[->] (managerm) -- (agentmn) node[midway,fill=white]{$\xi^{m,n_m}$};

\end{tikzpicture}
\caption{Hierarchy for the general model}\label{fig:hierarchy_general}
\end{center}
\end{figure}
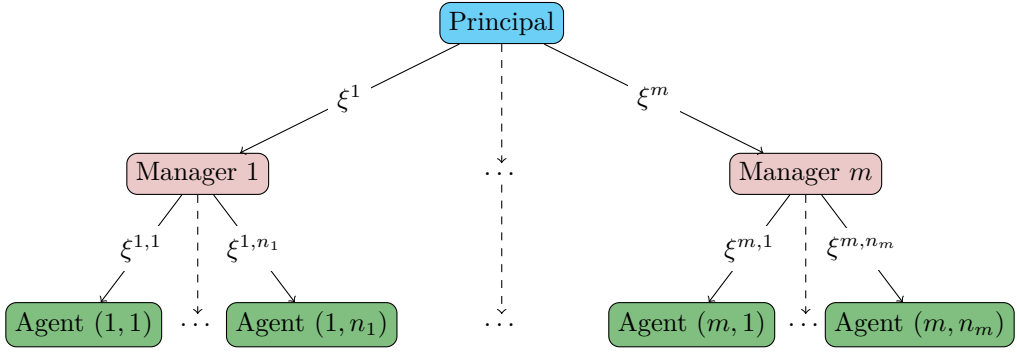

\subsection*{Additional notations}

Recall that $T > 0$ denotes some maturity fixed in the contract, and that for any positive integer $k$, $\Cc ( [0,T], \R^k)$ denotes the set of continuous functions from $[0,T]$ to $\R^k$. We will denote by $\Cc^2_b(\R^k,\R)$ the set of bounded twice continuously differentiable functions from $\R^k$ to $\R$, whose first and second derivatives are also bounded. For a probability space of the form $\Omega := \Cc ( [0,T], \R^k) \times \widetilde \Omega$ and an associated filtration $\F$, we will have to consider processes $\psi : [0,T] \times \Cc ( [0,T], \R^k) \longrightarrow E$, taking values in some Polish space $E$, which are $\F$--optional, \textit{i.e.}, $\Oc(\F)$--measurable where $\Oc(\F)$ is the so-called optional $\sigma$--field generated by $\F$--adapted right--continuous processes. In particular, such a process $\psi$ is non--anticipative in the sense that $\psi (t,x) = \psi(t, x_{\cdot \wedge t})$, for all $t\in[0,T]$ and $x \in \Cc ( [0,T], \R^k)$.

\medskip

For $j \in \{1, \dots, m\}$, $i \in \{0, \dots, n_j \}$ and $x^{j,i} \in S^{j,i}$ for some set $S^{j,i}$, we will make use of the following notations for vectors:
\begin{align}\label{eq:notations_x}
    x^{j} := (x^{j,i})_{i \in \{0, \dots, n_j \}}, \;
    x := (x^{j})_{j \in \{1, \dots, m \}}, \; 
    x^{\rm M} := (x^{j,0})_{j \in \{1, \dots, m \}} \; \text{and} \;
    x^{-j} := (x^{\ell})_{\ell \in \{1, \dots, m \} \setminus \{j\}},
\end{align}
and their corresponding sets:
\begin{align}\label{eq:notations_set}
    x^j \in S^j := \prod_{i=0}^{n_j} S^{j,i}, \; 
    x \in S := \prod_{j=1}^m \prod_{i=0}^{n_j} S^{j,i}, \; 
    x^{\rm M} \in  S^{\rm M} := \prod_{j=1}^m S^{j,0}, \; \text{ and } \; 
    x^{-j} \in S^{-j} := \prod_{\ell = 1, \, \ell \neq j}^{m} S^{j}.
\end{align}
In a similar way, we also define
\begin{itemize}[font=\scriptsize]
    \item $x^{\rm A} \in S^{\rm A}$, the vector obtained by suppressing the elements $x^{\rm M}$ of $x$;
    % S^{A} := \prod_{j \in \{1, \dots, m\}} \prod_{i \in \{1, \dots, n_j\}} S^{j,i}, \; 
    \item $x^{-(j,i)} \in  S^{-(j,i)}$, the vector obtained by suppressing the element $x^{j,i}$ of $x^{\rm A}$;
    \item $x^{j \setminus 0} := (x^{j,i})_{i \in \{1, \dots, n_j \}} \in  S^{j \setminus 0}$ and $x^{-j \setminus 0} := (x^{\ell \setminus 0})_{\ell \in \{1, \dots, m \} \setminus \{j\} } \in  S^{-j \setminus 0}$.
\end{itemize}
We will also use the following notations for sums: 
\begin{align}\label{eq:notation_xbar}
    \widebar{x}^{j} := \sum_{i=0}^{n_j} x^{j,i} \in \widebar S^j, \; 
    \text{and} \;
    \widebar{x}^{-j} := \big(\widebar x^{\ell} \big)_{\ell \in \{1, \dots, m \} \setminus \{j\}} \in \widebar S^{-j}.
\end{align}

\subsection{Theoretical formulation for the workers} \label{ss:theoretical_formulation}

Fix throughout this section $j \in \{1, \dots, m\}$ and $i \in \{0, \dots, n_j\}$, to consider all the workers, \textit{i.e.}, both the agents and the managers. Each worker $(j,i)$ takes care of his own task by choosing a pair $\nu^{j,i} := (\alpha^{j,i}, \beta^{j,i})$, where $\alpha^{j,i}$ and $\beta^{j,i}$ are respectively $A^{j,i}$-- and $B^{j,i}$--valued, for some subsets $A^{j,i}$ and $B^{j,i}$ of Polish spaces. 
More specifically, $\alpha^{j,i}$ and $\beta^{j,i}$ represent the effort of the worker $(j,i)$ to impact respectively the expected value and the variability of his outcome.\footnote{If the outcome $X^{j,i}$ represents the value added by the $(j,i)$--th worker, then naturally $\alpha$ represents an effort to increase the average and $\beta$ an effort to decrease the volatility. However, in more general terms, the outcome may represent different measures of a worker's performance, and it is therefore possible that the worker may need to decrease the average of the outcome or increase its volatility.} 
In addition, we will consider the following functions, assumed to be bounded:
\begin{align*}
    \lambda^{j,i} : [0,T] \times A^{j,i} \longrightarrow \R^d \; \text{ and } \; 
    \sigma^{j,i} : [0,T] \times B^{j,i} \longrightarrow \R^{d},
\end{align*}
More precisely, the scalar product between the two functions will represent the drift of the outcome of the $(j,i)$--th worker, while $\sigma^{j,i}$ will represent its volatility. We will denote for simplicity $U^{j,i}:= A^{j,i} \times B^{j,i}$, as well as $U$ the Cartesian product of the sets $U^{j,i}$, following the notations defined by \eqref{eq:notations_set}.
To easily write the dynamic of the column vector $X$ composed by the collection of all the $X^{j,i}$, we define by $\Lambda : [0,T] \times U \longrightarrow \R^{dw}$ and $\Sigma : [0,T] \times U \longrightarrow \M^{dw, w}$ the functions that will correspond respectively to the drift and the volatility of the process $X$. These functions $\Lambda$ and $\Sigma$ will be defined for all $t \in [0,T]$ and $u = (a,b) \in U$ respectively by:
\begin{align}\label{eq:def_drift_X}
    \Lambda (t,u) := 
    \begin{pmatrix}
        \big( \sigma^{j,i} \big(t, b^{j,i} \big) \cdot \lambda^{j,i} \big(t, a^{j,i} \big) \big)_{j,i}
    \end{pmatrix},
\end{align}
in the sense that $\Lambda$ is a column vector composed by the collection of all the scalar product $\sigma^{j,i} \big(t, b^{j,i} \big) \cdot \lambda^{j,i} \big(t, a^{j,i} \big)$, meaning that $\Lambda^{j,i} (t,u) := \sigma^{j,i} \big(t, b^{j,i} \big) \cdot \lambda^{j,i} \big(t, a^{j,i} \big)$, and
\begin{align}\label{eq:def_vol_X}
    \Sigma(t,b) := \bigoplus_{j =1}^m \bigoplus_{i = 0}^{n_j} \sigma^{j,i} (t, b^{j,i}),
    % := 
    % \begin{pmatrix}
    %     \sigma^{1,0}(t,b^{1,0}) & \textbf{0}_k &  \cdots & \textbf{0}_k \\
    %     \textbf{0}_k & \sigma^{1,1} (t,b^{1,1}) & \ddots & \vdots \\
    %     \vdots & \ddots & \ddots & \textbf{0}_k \\
    %     \textbf{0}_k & \cdots & \textbf{0}_k & \sigma^{m, n_m} (t,b^{m,n_m})
    % \end{pmatrix}.
\end{align}
where $\oplus$ symbolises direct sum\footnote{The symbol $\oplus$ denote for direct sum of matrices, which is defined for two matrix $A$ and $B$ by:
\renewcommand*{\arraystretch}{1}
\begin{align*}
    A \oplus B := 
    \begin{pmatrix}
        A & \textbf{0} \\
        \textbf{0} & B
    \end{pmatrix}.
\end{align*}} of matrices (vectors in this case).
% In particular, we will use this notation to consider the following matrix:
% \begin{align}\label{eq:bigoplus}
%     \bigoplus_{j = 1}^m \bigoplus_{i=0}^{n_j} x^{j,i} := 
%     \begin{pmatrix}
%         x^{1,0} & \textbf{0}_k &  \cdots & \textbf{0}_k \\
%         \textbf{0}_k & x^{1,1} & \ddots & \vdots \\
%         \vdots & \ddots & \ddots & \textbf{0}_k \\
%         \textbf{0}_k & \cdots & \textbf{0}_k & x^{m, n_m}
%     \end{pmatrix}.
% \end{align} 
To be consistent with the weak formulation of control problems, we need to define the canonical space $\Omega$ for the workers. Nevertheless, before that, we should discuss about what should be observed by the agents and their managers.

\subsubsection{Intuition}

According to the intuitions developed in \Cref{rk:measurability_issues,rk:measurability_issues_extension_top}, we cannot assume that the $(j,i)$--th agent only observes his own output $X^{j,i}$, since his contract cannot be restricted to a measurable function of this output. Indeed, in the general case, the parameters $Z^{j,i}$ and $\Gamma^{j,i}$ chosen by his manager and indexing the contract on his output would not be \textit{a priori} deterministic functions of time. Moreover, if we restrict the payment rates $Z^{j,i}$ and $\Gamma^{j,i}$ for the $(j,i)$--th agent to be adapted only to the filtration generated by $X^{j,i}$, the principal, who does not observe $X^{j,i}$, will not be able to compute the $j$--th manager's Hamiltonian. In fact, to better understand this measurability problem, we have to approach it from the top of the hierarchy. 

\medskip

We want to study a case of loss of information by proceeding up the hierarchy, modelled by the fact that each manager $j$ reports only the variable $\zeta^j$ to the principal, representing the performance of his work team. 
Indeed, if we consider for example that the principal represents the company's shareholders as in \cite{sung2015pay}, it is logical to assume that she is not aware of the precise results of each team led by a manager, and that she is probably only interested in the profits and costs, or even the net profits/benefit of each team, represented by the vector $\zeta$. This assumption is particularly relevant if we consider, for example, that each team is a department of the company (or a subsidiary of the parent company), and that shareholders can only compare the benefits of the different departments to optimise their investments and the importance given to each department.

\medskip

Therefore, the principal only observes the collection of the $\zeta^j$, for $j \in \{1, \dots, m \}$. Under some restrictive conditions\footnote{These conditions could be that the principal's problem is separable in each $\zeta^j$, with $\zeta^j$ independent and independently controlled by each manager. To illustrate a separable problem for the principal, one can consider \citeauthor{sung2015pay}'s model developed in \Cref{sec:sungmodel} and adding other working teams, led by managers, also reporting the net benefit of their working team to the principal. In this case, if each manager controls only his own $\zeta^j$ and if all $\zeta^j$ are independent, since the net benefit is just a difference of sums, and the principal is risk--neutral, her problem is completely separable in each $\zeta^j$.}, she may offer a contract for the $j$--th manager which only depends on the result $\zeta^j$ of his working team. However, more generally, her controls would be adapted to any information available to her, \textit{i.e.}, the filtration generated by $\zeta := (\zeta^j)_{j=1}^m $, and therefore it makes little sense to restrict the space of contracts for the $j$--th manager to measurable functions of $\zeta^j$. We are thus led to study a more general space of contracts for the managers, measurable with respect to the filtration generated by $\zeta$. 

\medskip

Therefore, each manager receives a contract indexed on $\zeta$, and we should make some assumptions ensuring that it is the only state variable of his control problem, in order to avoid the more challenging case where the manager's problem depend on another process, unobservable by the principal. Under these assumptions (see in particular \Cref{ass:zeta_dynamic} in the following), even if a manager observes independently the outcomes of his agents, his optimal controls will be adapted to the filtration generated by $\zeta$, and thus computable at the optimum by the principal. Nevertheless, the managers can use the detailed information they have to index the contract for their agents on it. 
Indeed, if a manager indexes his agents' contracts only on $\zeta$, he does not benefit from the information he knows over that known by the principal, and there is then no loss of information between the manager and the principal. Therefore, it must be in the interest of the $j$--th manager to index the compensation for his workers on all the information he has. In particular, we will assume that the $j$--th manager observes in continuous time:
\begin{itemize}[font=\scriptsize]
    \item the output produced by the workers of his team, including his own, \textit{i.e.}, the $(n_j+1)$--dimensional process $X^j$;
    \item the sum of the results of each of the other teams, \textit{i.e.}, the $(m-1)$--dimensional process $\widebar X^{-j}$;
\end{itemize}
using the notations defined by \eqref{eq:notations_x} and \eqref{eq:notation_xbar}.
Indeed, in the context of a hierarchy in a company, it seems quite logical to assume that a manager is well informed about the results of his agents, and that during meetings between managers, everyone communicates only the overall result of his team. 
The canonical space of each agent should contain every processes observable by his manager. Indeed, we want to focus on the loss of information by going up the hierarchy, between the agents and the managers, and then between the managers and the principal. Considering in addition a loss of information by going down the hierarchy would seriously complicate the problem, and would require further study.

\medskip

We could also assume that the $j$--th manager observes in particular the output produced by the workers of the other teams, \textit{i.e.}, the process $X^{-j}$, instead of the sum process $\widebar X^{-j}$, and so that the agents of his team have access to it as well. In practice, it seems difficult to imagine that a manager would have access to the individual results of the agents from another team. Nevertheless, for the sake of clarity and simplicity, we will consider that an agent observes all the outputs of all workers. Indeed, since each agent can only impact its own project, and its objective function depends only on its project and remuneration, the state variables of his problem will only be the one on which the contract is indexed. This consideration is therefore without loss of generality and makes it possible to define a single canonical space for all agents, regardless of the team they are in.

\medskip

In summary, the following framework is considered: the agents, at the bottom of the hierarchy, observe all the workers' output, \textit{i.e.}, the global output process $X$ taking values in $\R^{w}$. As explained above, they will not use all the information contained in $X$, the important thing is that they have access to the information held by their manager. Each manager perfectly observes the outputs of his agents, as well as his own. However, he does not have access to the detailed results of the other teams, but only to the sum of the outputs produced per team. Finally, the principal only observes $m$ variables, namely $\zeta := (\zeta^j)_{j=1}^m$, each representing the aggregation of a team's results.

\subsubsection{Canonical space}\label{sss:can_space_init}

Following the previous reasoning, we are thus led to consider the following canonical space
\[
    \Omega := \Cc ([0,T],\R^w) \times \Cc ([0,T], \R^{dw}) \times \U,
\]
where $\mathbb U$ is the collection of all finite and positive Borel measures on $[0,T] \times U$, whose projection on $[0,T]$ is the Lebesgue measure. In other words, every $q \in \mathbb \U$ can be disintegrated as $q(\mathrm{d}t,\mathrm{d} u)=q_t(\mathrm{d} u)\mathrm{d}t$, for $t \in [0,T]$ and an appropriate Borel measurable kernel $q_t$. The weak formulation requires to consider a subset of $\U$, namely the set $\U_0$ of all $q \in \U$ such that the kernel $q_t$ is of the form $\delta_{\phi_t}(\mathrm{d} v)$ for some Borel function $\phi$, where as usual, $\delta_{\phi_t}$ is the Dirac mass at $\phi_t$.
This space is supporting a canonical process $(X, W, \Pi)$, where for any $(t, x, w, q) \in [0,T] \times \Omega$,
\begin{align*}
    X_t(x,w,q) &:=x(t), \; W_t(x,w,q):=w(t), \; \Pi(x,w,q):=q.
\end{align*}
Less formally, $X$ represents the collection of the $w$ one--dimensional outcomes $X^{j,i}$ controlled by the workers. Each $X^{j,i}$ is affected by a $d$--dimensional idiosyncratic noise $W^{j,i}$, and $W$ is the collection of the $w$ noises $W^{j,i}$, using the notations defined in \eqref{eq:notations_x}.
Then, the canonical filtration $\F:=(\Fc_t)_{t\in[0,T]}$ is defined by
\begin{align*}
    \mathcal F_t:=\sigma\Big( \big( X_s, W_s, \Delta_s(\varphi) \big) \text{ s.t. }  (s,\varphi)\in[0,t]\times \Cc_b \big( [0,T]\times U, \R \big) \Big),\; t\in [0,T],
\end{align*}
where $\Cc_b([0,T]\times U,\R)$ is the set of all bounded continuous functions from $[0,T]\times U$ to $\R$, and for any $(s,\varphi)\in[0,T]\times \Cc_b([0,T]\times U,\R)$,
\begin{align*}
    \Delta_s(\varphi):=\int_0^s\int_{U} \varphi(r,u) \Pi (\mathrm{d}r,\mathrm{d} u).
\end{align*}
%We will also need a smaller filtration containing only the information generated by $X^{j,i}$, namely $\G^{j,i} := (\Gc^{j,i}_t)_{t \in [0,T]}$, which represents the information observable by the $j$--th manager on the $(j,i)$--th agent's results.
For any $(t,\psi)\in[0,T]\times \Cc^2_b(\R^{w} \times \R^{dw}, \R)$, we set
\begin{align}\label{eq:Mphi_ji}
    M_t^{\rm A}(\psi):= \psi (X_t, W_t) - \int_0^t \int_{U} \bigg( \widetilde \Lambda (s, u) \cdot \nabla \psi (X_s, W_s)
    + \frac12  {\rm Tr} \Big[ \nabla^2 \psi (X_s, W_s) \widetilde \Sigma (s, u)  \widetilde \Sigma (s, u)^\top \Big] \bigg) \Pi (\mathrm{d}s, \mathrm{d} u),
\end{align}
where $\nabla^2 \psi$ denotes the Hessian matrix of $\psi$, $\widetilde \Lambda$ and $\widetilde \Sigma$ are respectively the drift vector and the diffusion matrix of the $(w + d w)$--dimensional vector process $(X, W)^\top$, defined for all $s \in [0,T]$ and $u:=(a,b) \in U$ by:
\renewcommand*{\arraystretch}{1}
\begin{align*}
    \widetilde \Lambda(s, u) := 
    \begin{pmatrix}
        \Lambda (s,u) \\
        \mathbf{0}_{w d}
    \end{pmatrix},
    \; 
    \widetilde \Sigma(s,u) :=
    \begin{pmatrix}
        \mathbf{0}_{w, w} & \Sigma (s,b)^\top  \\
        \mathbf{0}_{w d, w} & \mathrm{I}_{w d} \\
    \end{pmatrix},
\end{align*}
where $\Lambda$ and $\Sigma$ are respectively defined by \eqref{eq:def_drift_X} and \eqref{eq:def_vol_X}, so that $\widetilde \Lambda (s,u) \in \R^{w + wd}$ and $\widetilde \Sigma (s,u) \in \M^{w + wd}$.

\medskip

We fix some initial conditions, namely $x_0 \in \R^{w}$ representing the initial value of $X$, and let $\Mc$ be the set of all probability measures on $(\Omega,\Fc_T)$.
\begin{definition}\label{def:weak_formulation}
The subset $\Pc \subset \Mc$ is composed of all $\P$ such that
\begin{enumerate}[label=$(\roman*)$]
    \item $M^{\rm A}(\psi)$ is a $(\F, \P)$--local martingale on $[0,T]$ for all $\psi \in \Cc^2_b(\R^{w} \times \R^{dw},\R)$;
    \item there exists some $w_0 \in \R^{wd}$ such that $\P [(X_0,W_0) = (x_0,w_0)]=1$;
    \item $\P\big[\Pi \in \U_0]=1$.
\end{enumerate}
\end{definition}
The previous definition does not give us access directly to the dynamic of $X$. It is however a classical result that, enlarging the canonical space if necessary, one can construct Brownian motions allowing to write rigorously the dynamic of $X$, see for instance \citeayn[Theorem 4.5.2]{stroock1997multidimensional}. It turns out here that, since we enlarged the canonical space right from the start to account for the idiosyncratic noises, any further enlargement is not required. Indeed, arguing as in the proof of \citeayn[Lemma 2.2]{lin2018second}, we can prove the following.

\begin{lemma}\label{lem:no_enlarge}
For all $\P \in \Pc$, we have $\Pi(\mathrm{d}s,\mathrm{d} u) = \delta_{\nu^\P_s}(\mathrm{d} u)\mathrm{d}s $ $\P$--{\rm a.s.} for some $\F$--predictable control process $\nu^\P := ( \alpha^{j,i,\P}, \beta^{j,i,\P} )_{j,i}$, and we obtain the following representation for $X$:
\begin{align}\label{eq:def_X}
    X_t = x_0 + \int_0^t \Lambda \big( s, \nu_s^{\P} \big) \drm s + \int_0^t \Sigma \big(s, \beta_s^{\P} \big)^\top \drm W_s, \; t\in[0,T],\; \P \textnormal{--a.s.}
\end{align}
More precisely, for any $j \in \{1, \dots, m\}$, $i \in \{0, \dots, n_j\}$,
\begin{align}\label{eq:XW}
    X_t^{j,i}=x_0^{j,i}+\int_0^t \sigma^{j,i} \big(s, \beta_s^{j,i,\P} \big) \cdot \Big[ \lambda^{j,i} \big(s, \alpha_s^{j,i,\P} \big) \mathrm{d}s +  \mathrm{d}W^{j,i}_s \Big],\; t\in[0,T],\; \P \textnormal{--a.s.}
\end{align}
\end{lemma}

\begin{remark}\label{rk:SDE_simplify}
    By construction, the collection of outputs $X^{j,i}$ are not defined as solutions of stochastic differential equations \textnormal{(SDE} for short\textnormal{)}, as in more general framework of principal--agent models \textnormal{(}see \textnormal{\citeayn{cvitanic2018dynamic}} for example\textnormal{)}, they are just standard It\=o processes. The reader is referred to \textnormal{\Cref{rk:SDE_simplify_why}} for the motivations of this assumption.
\end{remark}

\subsection{The principal--managers--agents hierarchy}\label{ss:hierarchy_theoretical}

The hierarchy is modeled by a series of interlinked principal--agent problems: the principal contracts with $m$ managers, and each manager $j$, for $j \in \{1, \dots, m\}$, hires $n_j$ agents, indexed by $(j,i)$ for $i \in \{1, \dots, n_j \}$. More precisely, for $j \in \{1,\dots,m\}$, the $j$--th manager receives a compensation $\xi^j$ from the principal and must design what will be the remuneration $\xi^{j,i}$ of each agent $(j,i)$, for $i \in \{1, \dots, n_j\}$. Finally, as in the model developed in the previous sections, we assume that the principal chooses the continuation utilities of the agents at time $t=0$  (in addition to those of the managers), \textit{i.e.}, the expected value finally obtained by the agents at time $T$, denoted by $Y^{j,i}_0 \in \R$ for the $(j,i)$--th agent.

\medskip

Once again, it seems more appropriate to think about this hierarchy from top to bottom. First, and as mentioned above, the principal chooses the initial values of the agents' continuation utility. This will lead us to fix, until the principal's problem, $Y^{j,i}_0 \in \R$ for all $j \in \{1, \dots, m\}$ and $i \in \{1, \dots, n_j\}$, denoted by $Y_0^{\rm A} \in \R^{w-m}$. The principal also offers to each manager a contract indexed on the variable $\zeta$ she observes, which leads to the first Stackelberg game. Given his contract and the choices of other managers, the $j$--th manager will choose:
\begin{enumerate}[label=$(\roman*)$]
    \item an optimal effort $\alpha^{j,0}$ to improve the expected value of his own output, and the associated probability;
    \item a compensation scheme for the agents he manages, \textit{i.e.}, $(\xi^{j,i})_{i=1}^{n_j}$.
\end{enumerate}
Then, given the choices of the managers, and in particular the agents' contracts, the agents will determine their optimal efforts. To write their problems, it is therefore necessary to fix a contract and an effort probability chosen by the managers. In addition, since an agent receives a contract which depends on his colleagues' output, his optimal response must therefore also be defined in relation to the efforts of the other agents. The intuition is thus to define the optimal response of an agent given:
\begin{enumerate}[label=$(\roman*)$]
    \item a contract;
    \item the efforts of the managers;
    \item the efforts of the other agents.
\end{enumerate}

In a classical way, the two Stackelberg games are solved from the bottom to the top. First we look for the optimal response of an agent to arbitrary choices of others. Next, the Nash equilibrium between the agents can be solved, under fixed contracts and efforts of the managers. Then, knowing the optimal response of the agents, each manager will then be able to choose his efforts and the remuneration for his agents, in order to optimise his criterion, given the choices of other managers, and the contracts designed by the principal. Since the manager's choices depend on those of other managers, it is also necessary to find a Nash equilibrium between managers. Finally, given the optimal response of each manager to a contract, the principal will be able to design the optimal contract for each manager in order to optimise her own criterion.

\subsubsection{A Nash equilibrium between the agents}\label{sss:nash_theoretical}

To consider a particular agent, we fix throughout this section $j \in \{1, \dots, m \}$ and $i \in \{1, \dots, n_j\}$.
Following the previous reasoning, we have to define two subsets of the canonical space $\Omega$, in order to look for the optimal response of the $(j,i)$--th agent, given the choices of other workers (managers and agents). Using notations \eqref{eq:notations_set}, we define the sets $U^{\rm M}$ and $U^{-(j,i)}$ from the collection of sets $U^{j,i}$. In the same manner we have defined $\U$ and the corresponding set $\U_0$ for the canonical space in \Cref{sss:can_space_init}, we define:
\begin{enumerate}[label=$(\roman*)$]
    \item $\mathbb U^{-(j,i)}$, the collection of all finite and positive Borel measures on $[0,T] \times U^{-(j,i)}$, whose projection on $[0,T]$ is the Lebesgue measure, and the associated subset $\U^{-(j,i)}_0$;
    \item $\mathbb U^{\rm M}$, the collection of all finite and positive Borel measures on $[0,T] \times U^{\rm M}$, whose projection on $[0,T]$ is the Lebesgue measure, and the associated subset $\U^{\rm M}_0$.
\end{enumerate}
Informally, the set $\U_0$ will allow us to define the set of admissible efforts of all the workers, while the set $\U^{-(j,i)}_0$ will be used to define the efforts of other agents, apart from the $(j,i)$--th agent, and the set $\U^{\rm M}_0$ will be used to define the efforts of the managers. In the same way we have defined $\Omega$, with its canonical process, its canonical filtration and the appropriate subset of probability $\Pc$ by \Cref{def:weak_formulation}, we can define the two following canonical spaces:
\begin{enumerate}[label=$(\roman*)$]
    \item $(\Omega^{-(j,i)},\Fc^{-(j,i)}_T)$ the canonical space of other agents (apart from the agent $(j,i)$), where
\begin{align*}
    \Omega^{-(j,i)} := \Cc ([0,T],\R^{w-m-1}) \times \Cc ([0,T], \R^{d(w-m-1)}) \times \U^{-(j,i)},
\end{align*}
with canonical process $( X^{-(j,i)}, W^{-(j,i)}, \Pi^{-(j,i)})$, and canonical filtration $\F^{-(j,i)}:=(\Fc^{-(j,i)}_t)_{t\in[0,T]}$, on which we define the subset $\Pc^{-(j,i)}$ of probability measures satisfying the appropriate properties;
    \item $(\Omega^{\rm M},\Fc^{\rm M}_T)$ the canonical space of managers, where
\begin{align*}
    \Omega^{\rm M} := \Cc ([0,T],\R^m) \times \Cc ([0,T], \R^{dm}) \times \U^{\rm M},
\end{align*}
with canonical process $(X^{\rm M}, W^{\rm M}, \Pi^{\rm M})$, and canonical filtration $\F^{\rm M}:=(\Fc^{\rm M}_t)_{t\in[0,T]}$, on which we define the subset $\Pc^{\rm M}$ of probability measures satisfying the appropriate properties.
\end{enumerate}
Informally, the canonical space $\Omega^{-(j,i)}$ contains the same information as $\Omega$, except that the components concerning the $(j,i)$--th agent and the managers are removed, and thus, $\Omega^{-(j,i)}$ is a subset of $\Omega$. Similarly, the space $\Omega^{\rm M}$ contains the information of $\Omega$ concerning the managers. 
We can now define, for the $(j,i)$--th agent, the set of his admissible response to others.

\begin{definition}[Admissible Response]\label{def:admissible_agent}
    Consider two probability measures, $(\P^{\rm M}, \P^{-(j,i)}) \in \Pc^{M} \times \Pc^{-(j,i)}$, respectively chosen by the managers and the other agents. The set of admissible response of the $(j,i)$--th agent, denoted by $\Pc^{j,i} (\P^{-(j,i)}, \P^{\rm M})$, is given by all probability measures $\P \in \Pc$ on $(\Omega,\Fc_T)$ satisfying:
    \begin{enumerate}[label=$(\roman*)$]
        \item the restriction of $\P$ to $(\Omega^{-(j,i)},\Fc^{-(j,i)}_T)$ is $\P^{-(j,i)}$;
        \item the restriction of $\P$ to $(\Omega^{\rm M},\Fc^{\rm M}_T)$ is $\P^{\rm M}$.
    \end{enumerate}
\end{definition}

Following the discussion in the beginning of \Cref{ss:theoretical_formulation}, we have assumed that the $j$--th manager observes the output of the agents under his supervision, but also the sum of the outputs of each of the other teams. Formally, this means that the $j$--th manager proposes to the $(j,i)$--th agent a contract $\xi^{j,i}$, which is a random variable measurable with respect to the natural filtration generated by $X^{j}$ and $\widebar X^{-j}$, denoted $\G^j$. In other words, $\xi^{j,i}$ must be a measurable functional of the paths of $X^{j}$ and $\widebar X^{-j}$:
\begin{align}\label{eq:contract_agent_dependency}
 \nonumber  \xi^{j,i} : \Cc \big([0,T],\R^{n_j+1} \big) \times \Cc \big([0,T],\R^{m-1} \big) & \longrightarrow \mathbb{R}, \\ \big( X^{j}, \widebar X^{-j} \big) \qquad \qquad \qquad  & \longmapsto \xi^{j,i} \big( X^{j}, \widebar X^{-j} \big).
\end{align}
One can notice that the contracts for the agents of the team $j$ are measurable with respect to $\G^j$, \textit{i.e.}, all agents in a team will receive a contract indexed on the same results. The set of admissible contracts for the $(j,i)$--th agent will be denoted by $\Cc^{j,i}$, and we refer to \Cref{def:contract_agent_admissible} below for a rigorous description.

\medskip

Given this contract and two probability measures $(\P^{-(j,i)}, \P^{\rm M}) \in \Pc^{-(j,i)} \times \Pc^{\rm M}$ chosen by others, we introduce the $(j,i)$--th agent's objective function
\begin{align}\label{eq:objectiv_agent}
    J^{j,i} \big(\P, \xi^{j,i} \big):= \E^\P \bigg[ \Kc^{j,i, \P}_{0,T} \; g^{j,i} \big(X^{j,i}_{\cdot \wedge T}, \xi^{j,i} \big) - \int_0^T \Kc^{j,i,\P}_{0,s} c^{j,i} \big(s, X^{j,i}, \nu_s^{j,i,\P} \big) \mathrm{d}s \bigg], \; \text{for} \; \P \in \Pc^{j,i} \big( \P^{-(j,i)}, \P^{\rm M} \big),
\end{align}
where
\begin{enumerate}[label=$(\roman*)$]
    \item $g^{j,i} : \Cc([0,T], \R) \times \R \longrightarrow \R$ is a utility function assumed to be Borel measurable in each arguments, such that for any $x \in \Cc([0,T], \R)$, the map $\xi \longmapsto g^{j,i} (x, \xi)$ is invertible and we will denote by $\widebar g^{j,i}$ its inverse.
    \item $c^{j,i}:[0,T] \times \Cc([0,T], \R) \times U^{j,i} \longrightarrow \R$ is a cost function assumed to be Borel measurable in each arguments, such that for any $u\in U^{j,i}$, the map $(t,x) \longmapsto c^{j,i}(t,x,u)$ is $\F$--optional, and there exists some $p > 1$ such that
    \begin{align}\label{eq:integ_cond_cost}
    % \sup_{\P \in \Pc} \E^\P \bigg[ \int_0^T \big| c^{j,i} \big(t, X^{j,i}, \nu^{j,i,\P}_t \big) \big|^p \drm t \bigg] < + \infty;
        \sup_{\P\in \Pc} \E^\P \bigg[ \int_0^T \big| c^{j,i} \big(t, X^{j,i}, \nu^{j,i,\P}_t \big) \big|^p \drm t \bigg] < + \infty;
    \end{align}
    \item the discount factor
    \[
    \Kc^{j,i, \P}_{t,s} :=\exp\bigg(-\int_t^s \int_{U} k^{j,i} \big(v, X^{j,i}, u^{j,i} \big) \Pi (\mathrm{d}v,\mathrm{d} u) \bigg), \; \text{ for } \; 0\leq t\leq s\leq T \; \text{ and } \; \P \in \Pc^{j,i} \big( \P^{-(j,i)}, \P^{\rm M} \big),
    \]
    is defined by means of the function $k^{j,i}: [0,T] \times \Cc([0,T], \R) \times  U^{j,i} \longrightarrow \R$ assumed to be bounded and Borel measurable in each arguments, and such that for any $u\in U^{j,i}$, the map $(t,x)\longmapsto k^{j,i}(t,x,u)$ is $\F$--optional.
\end{enumerate}
The previous setting and assumptions are relatively standard in control theory. The map $g^{j,i}$ is assumed to be invertible so that one can recover the contract $\xi^{j,i}$ from the continuation utility of the agent. More precisely, our first goal is to obtain the dynamic of the continuation utility of the $(j,i)$--th agent, which will be denoted $(Y_t^{j,i})_{t \in [0,T]}$, satisfying at the end of the contracting period the following equality: $Y_T^{j,i} = g^{j,i} (X^{j,i}, \xi^{j,i})$. The contract will thus be given by $\xi^{j,i} = \widebar g^{j,i} (X^{j,i}, Y^{j,i}_T)$, recalling that $\widebar g^{j,i} : \Cc([0,T], \R) \times \R \longrightarrow \R$ is the inverse of $g^{j,i}$. However, we are forced to make an additional hypothesis, and the reader is referred to \Cref{rk:separability_why} for the motivation.
\begin{assumption}\label{ass:separability}
    There exists $c_{\rm x}^{j,i}, k_{\rm x}^{j,i} : [0,T] \times \Cc([0,T], \R) \longrightarrow \R$ and $c_{\rm u}^{j,i}, k_{\rm u}^{j,i} : [0,T] \times U^{j,i} \longrightarrow \R$ such that for any $(t, x, u) \in [0,T] \times \Cc([0,T], \R) \times U^{j,i}$, we can write:
    \begin{align*}
        c^{j,i}(t,x,u) = c_{\rm x}^{j,i}(t,x) + c_{\rm u}^{j,i}(t,u) \; \text{ and } \; k^{j,i}(t,x,u) = k_{\rm x}^{j,i}(t,x) + k_{\rm u}^{j,i}(t,u).
    \end{align*}
\end{assumption}

Despite this unusual assumption, the framework under consideration can at least accommodate the two most standard cases for the agent's utility function, \textit{i.e.}, both the risk--averse case from \cite{sung2015pay} and the risk--neutral case. Indeed, to recover for example \citeauthor{sung2015pay}'s model, it suffices to set, for all $(t, x,y) \in [0,T] \times \R^2$, $u := (a,b) \in \U^{j,i}$, $c^{j,i} (t,x,u) = k_{\rm x}^{j,i} (t,x) \equiv 0$, $g^{j,i} (x, y) := - \erm^{- R^{j,i} y}$ and $ k_{\rm u}^{j,i} (t,u) := - R^{j,i} a^2/ 2k^{j,i}$, for some $(R^{j,i}, k^{j,i}) \in (\R_+^\star)^2$.

\medskip

Given a compensation $\xi^{j,i}$ promised by his manager, as well as a pair of probabilities $(\P^{-(j,i)}, \P^{\rm M}) \in \Pc^{-(j,i)} \times \Pc^{\rm M}$ chosen by other workers, the optimisation problem faced by the $(j,i)$--th agent is defined by
\begin{align}\label{eq:pb_agent}
    V^{j,i}_0 \big(\P^{-(j,i)}, \P^{\rm M}, \xi^{j,i} \big) := \sup_{\P \in \Pc^{j,i} (\P^{-(j,i)}, \P^{\rm M}) } J^{j,i} \big(\P, \xi^{j,i} \big).
\end{align}
For $V_0^{j,i}(\P^{-(j,i)}, \P^{\rm M}, \xi^{j,i})$ to make sense, we require minimal integrability on the contracts, and thus impose that there is some $p>1$ such that
\begin{align}\label{eq:integrability_contract_agent}
    \sup_{\P\in \Pc} \E^\P \bigg[ \big| g^{j,i} \big( X^{j,i}, \xi^{j,i} \big) \big|^p \bigg] < + \infty.
    \tag{${\rm I}^p_{\rm A}$}
\end{align}

\begin{definition}[Optimal Response]\label{def:optimal_response_agent}
    A probability measure $\P \in \Pc$ is an optimal response to the probabilities $\P^{-(j,i)}$ and $\P^{\rm M}$ chosen by others and to a contract $\xi^{j,i} \in \Cc^{j,i}$ if $(i)$ $\P$ is admissible, \textit{i.e.}, $\P \in \Pc^{j,i} (\P^{-(j,i)}, \P^{\rm M})$; and $(ii)$ $V^{j,i}_0 \big(\P^{-(j,i)}, \P^{\rm M}, \xi \big) = J^{j,i} \big(\P, \P^{-(j,i)}, \P^{\rm M}, \xi \big)$. We denote by $\Pc^{j,i,\star} (\P^{-(j,i)}, \P^{\rm M}, \xi)$ the collection of all such optimal probability measures.
\end{definition}

Using the notation \eqref{eq:notations_x}, we define by $\xi^{\rm A}$ the collection of contracts for all agents, \textit{i.e.}, $\xi^{j,i} \in \Cc^{j,i}$ for $j \in \{1, \dots, m\}$ and $i \in \{1, \dots, n_j \}$, and the corresponding set is denoted $\Cc^{\rm A}$. A Nash equilibrium between the agents can thus be defined as an optimal response to the manager's choices for any agents:

\begin{definition}[Nash Equilibrium]\label{def:nash_agent}
    Fix a probability measure $\P^{\rm M} \in \Pc^{\rm M}$ chosen by the managers and a collection $\xi^{\rm A} \in \Cc^{\rm A}$ of contracts for the agents. We denote by $\Pc^{\rm A,\star} (\P^{\rm M}, \xi^{\rm A})$ the set of Nash equilibria between agents, \textit{i.e.} the set of probability measure $\P \in \Pc$ such that for any $j \in \{1, \dots, m\}$ and any $i \in \{1, \dots, n_j\}$, $\P \in \Pc^{j,i,\star} (\P^{-(j,i)}, \P^{\rm M}, \xi^{j,i})$, where $\P^{-(j,i)}$ is the restriction of $\P$ on $\Omega^{-(j,i)}$.
\end{definition}

To simplify the scope of our study and avoid unnecessary complexities, we will subsequently require that all eligible contracts for agents are those that induce a unique Nash equilibrium between agents, and we thus define the set of admissible contract as follows:
\begin{definition}[Admissible contracts]\label{def:contract_agent_admissible}
Fix $Y^{\rm A}_0 \in \R^{w-m}$ the collection of $Y_0^{j,i} \in \R$ for all $j \in \{1, \dots, m\}$, $i \in \{1, \dots, n_j\}$ and $\P^{\rm M} \in \Pc^{\rm M}$.
A contract of the form \eqref{eq:contract_agent_dependency} satisfying \textnormal{\Cref{eq:integrability_contract_agent}} is called admissible. The corresponding class is denoted by $\Cc^{j,i}$. Moreover, a collection $\xi^{\rm A}$ of contracts for the agents is admissible if $(i)$ the induced Nash equilibrium $\P^\star$ between the agents is unique; and $(ii)$ for all $j \in \{1, \dots, m\}$, $i \in \{1, \dots, n_j\}$, $\xi^{j,i} \in \Cc^{j,i}$ and the $(j,i)$--th agent's value function at equilibrium is equal to $Y^{j,i}_0$. In this case, we will write $\xi^{\rm A} \in \Cc^{\rm A}$.\footnote{Note that the space of admissible contracts $\Cc^{\rm A}$ depends on the collection of $Y_0^{\rm A} \in \R^{w-m}$. Nevertheless, in order to lighten the notations, we omit this dependency.}
\end{definition}

In addition, similarly to classical principal--agent problems, we assume that the $(j,i)$--th agent has a reservation utility level $\rho^{j,i} \in \R$ below which he will refuse to work, and thus decline the contract offered by his manager. Mathematically speaking, a contract should thus satisfy the following inequality:
\begin{align}\label{eq:participation_agent}
    V^{j,i}_0 \big(\P^{-(j,i)}, \P^{\rm M}, \xi^{j,i} \big) \geq \rho^{j,i}.
    \tag{${\rm PC_A}$}
\end{align}
Note that, if the collection $\xi^{\rm A}$ of contracts for the agents is admissible in the sense of the previous definition, then there exists a unique Nash equilibrium $\P^\star$ between agents, which satisfies in addition $V^{j,i}_0 (\P^{-(j,i),\star}, \P^{\rm M}, \xi^{j,i}) = Y_0^{j,i}$. Therefore, in order to ensure that the participation constraint \eqref{eq:participation_agent} of the $(j,i)$--th agent is satisfied, the principal only has to choose $Y_0^{j,i} \geq \rho^{j,i}$ in the end.

\subsubsection{A Nash equilibrium between the managers}\label{sss:def_manager_pb}

Throughout the following, we fix $j \in \{1, \dots, m\}$ to informally define the $j$--th manager's optimisation problem. The weak formulation of this problem will be rigorously defined in \Cref{ss:manager_problem}, since it will be necessary to consider a new canonical space taking into account the optimal response of the agents.
Let us just recall for the moment that the $j$--th manager is in charge of a task that generates an output $X^{j,0}$. His effort to improve his output $X^{j,0}$ is defined in an informal way by a pair $\nu^{j,0} := (\alpha^{j,0}, \beta^{j,0}) \in \Uc^{j,0}$ taking values in $U^{j,0} := A^{j,0} \times B^{j,0}$.

\medskip

We suppose that the $j$--th manager reports in continuous time to the principal the variable $\zeta^j$, assumed to be $h$--dimensional for some $h > 0$, and measuring the global result of his entire working team (including himself). Therefore, $\zeta^j$ will depend on the outcomes of his team, namely $X^{j}$ (the outcomes of the agents he manages and his own) and the collection $\xi^{j \setminus 0} := (\xi^{j,i})_{i= 1}^{n_j}$ of compensations to be paid to the agents. Therefore, the principal only knows the result of each team, \textit{i.e.}, the process $\zeta := (\zeta^j)_{j=}^m$, and thus the contract designed by the principal for the $j$--th manager depends exclusively on the path of $\zeta$:
\begin{align}\label{eq:contract_manager_dependency}
    \xi^{j,0} : \zeta \in \Cc \big([0,T],\R^{hm} \big) \longmapsto  \xi^{j,0} ( \zeta) \in \mathbb{R}.
\end{align}
In other words, we can only consider contracts for managers which are $\Gc_T$--measurable, where $\G := (\Gc_t)_{t\in [0,T]}$ is the natural filtration of $\zeta$. The set of admissible contracts for the $j$--th manager will be denoted by $\Cc^{j,0}$, and we refer to \Cref{def:contract_manager_admissible} below for a rigorous description.

\medskip

We define the characteristics of the $j$--th manager as follows:
\begin{enumerate}[label=$(\roman*)$]
    \item a utility function $g^{j,0} : \Cc([0,T], \R^h) \times \R \longrightarrow \R$, assumed to be Borel measurable in each arguments, such that for any $x \in \Cc([0,T], \R^h)$, the map $y \longmapsto g^{j,0} (x, y)$ is invertible and we will denote by $\widebar g^{j,0}$ its inverse;
    \item a cost function $c^{j,0} : [0,T] \times \Cc([0,T], \R^h) \times U^{j,0} \longrightarrow \R$, assumed to be Borel measurable in each arguments, such that for any $u\in U^{j,0}$, the map $(t,x)\longmapsto c^{j,0} (t,x,u)$ is $\G$--optional, and satisfying for some $p > 1$
    \begin{align}\label{eq:integ_cond_cost_manager}
        \sup_{\P\in \Pc} \E^\P \bigg[ \int_0^T \big| c^{j,0} \big(t, \zeta^j, \nu^{j,0}_t \big) \big|^p \drm t \bigg] < + \infty;
    \end{align}
    \item a discount factor $k^{j,0}: [0,T] \times \Cc([0,T], \R^h) \times U^{j,0} \longrightarrow \R$, assumed to be bounded and Borel measurable in each arguments, and such that for any $u \in U^{j,0}$, the map $(t,x) \longmapsto k^{j,0} (t,x,u)$ is $\G$--optional,
    and its associated quantity
    \[
    \Kc^{j,0, \P}_{t,s} := \exp \bigg(-\int_t^s \int_{U} k^{j,0} \big(v, \zeta^j, u^{j,0} \big) \Pi (\mathrm{d}v,\mathrm{d} u) \bigg), \; \text{ for } \; 0\leq t\leq s\leq T \; \text{ and } \; \P \in \Pc;
    \]
    \item a reservation utility level $\rho^{j,0} \in \R$ below which he will refuse to work.
\end{enumerate}

Given a probability $\P \in \Pc$ on the canonical space, as well as a contract $\xi^{j,0}$ designed for him by the principal, the criterion of the $j$--th manager is defined as follows:
\begin{align}\label{eq:criterion_manager}
    J^{j,0} \big(\P, \xi^{j,0} \big) &:= \E^{\P} \bigg[ \Kc^{j, 0, \P}_{0,T}  g^{j,0} \big(\zeta^j, \xi^{j,0} \big) - \int_0^T \Kc^{j, 0, \P}_{0,t}  c^{j,0} \big(t, \zeta^{j}, \nu^{j,0, \P}_t \big) \mathrm{d} t \bigg].
\end{align}

\begin{remark}\label{rk:measurability_manager}
    The attentive reader will have noticed that the characteristics of the manager, \textit{i.e.}, the functions $g^{j,0}$, $c^{j,0}$ and $k^{j,0}$, are assumed to be functions of $\zeta^j$, and not of the collection $X$ of outputs. Indeed, these functions can only depend on the variable observed by the principal. Otherwise, if these functions depend on some variable which is not observed by the principal, she cannot compute the managers' Hamiltonians, even for their optimal efforts. 
    This would raise major issues, not yet addressed in the literature on continuous time $($to our knowledge$)$, which would require a full study before it can be considered in our case. Nevertheless, it is worth noticing that some works attempt to address similar problems, such as the paper by \citeayn{huang2017optimal}.
    Nevertheless, for the sake of generality, we could let them depend on the collection of the $\zeta^j$, since the manager's contract will be a measurable function of $\zeta := (\zeta^j)_{j=1}^m$, and $\zeta$ will thus be the natural state variable of each manager's problem.
\end{remark}

The $j$--th manager must optimise the specific criterion defined by \eqref{eq:criterion_manager}, given the contract he receives, but also given the choices of the other managers. 
More precisely, let us fix a contract $\xi^{j,0}$ of the form \eqref{eq:contract_manager_dependency}, as well as the decision of other managers, namely, for all $\ell \in \{1, \dots, m \} \setminus \{j\}$,
\begin{enumerate}[label=$(\roman*)$]
    \item the effort $\nu^{\ell, 0} \in \Uc^{\ell, 0}$ of the $\ell$--th manager;
    \item the collection of contracts $\xi^{\ell \setminus 0} \in \Cc^{\ell \setminus 0}$ offered by the $\ell$--th manager to his $n_\ell$ agents.
\end{enumerate}
Given this, the $j$--th manager must thus choose an optimal control $\nu^{j,0} \in \Uc^{j,0}$, as well as a contract $\xi^{j,i} \in \Cc^{j,i}$ for each agent $(j,i)$ under his supervision. We summarise the controls of the managers by a tuple $\widetilde \chi \in \widetilde{\mathscr X}$, such that
\begin{align*}
    \widetilde \chi^j := \big( \nu^{j,0}, (\xi^{j,i})_{i=1}^{n_j} \big)  \in \widetilde{\mathscr X}^j := \Uc^{j,0} \times \prod_{i=1}^{n_j} \Cc^{j,i},
\end{align*}
is the control of the $j$--th manager, for all $j \in \{1, \dots, m\}$. Informally, the probability $\P^{\rm M} \in \Pc^{\rm M}$ results from the effort choice of all managers, namely $\widetilde \chi \in \widetilde{\mathscr X}$. Therefore, the collection of all contracts chosen by the managers for their agents, denoted $\xi^{\rm A}$, depends on $\widetilde \chi$. As mentioned in \Cref{def:contract_agent_admissible}, we require that all eligible contracts for agents are those that induce a unique Nash equilibrium between agents, \textit{i.e.}, $\Pc^{\rm A,\star} (\P^{\rm M}, \xi^{\rm A}) = \{ \P^\star (\widetilde \chi)\}$, where the notation is made to highlight the dependence of the probability $\P^\star$ on the control $\widetilde \chi$.

\begin{remark}
    It is well--known that the uniqueness of a Nash equilibrium is more the exception than the rule. However, if we do not want to assume uniqueness, this leads to considerations that can become complex. Indeed, assume that there is no uniqueness of the Nash equilibrium. In usual principal--agent problems\footnote{We refer here to \textit{usual} principal--agent problems in the sense that there is no limited liability, contrary to \citeayn{sannikov2008continuous} or \citeayn{demarzo2017learning}, and the agent does not control the discount factor.}, the agents actually receive exactly their reservation utility. If this is true in our case, each equilibrium should therefore give the same utility to the agents. 
    In this case, it would therefore be reasonable to assume that each manager chooses the most advantageous equilibrium between his agents from his point of view, and that the principal then chooses the best equilibrium between the managers. 
    Unfortunately, to do this, it would first be necessary to define a Nash equilibrium between the agents per team $($\textit{i.e.} for all $j \in \{1, \dots, m\})$, and then look at the Nash equilibrium between the teams, before we address the Nash equilibrium between the managers. 
    Moreover, at this stage of the study, it is also possible that the various equilibria give distinct utilities to the agents. In this case, if one equilibrium Pareto--dominates the others from the agents' point of view, then it would make sense for the agents to choose it, and the manager would not be able to optimise over all possible equilibria. 
    In our opinion, introducing the possibility of multiple Nash equilibria would therefore only overload the framework, while it is not the major focus of this study.
    %and seems superfluous since the agents only interact with each other through the contracts. %Moreover, in the examples we have in mind, the Nash equilibrium between the agents is actually unique.
\end{remark}

The $j$--th manager's optimisation problem can then be informally written as follows:
\begin{align}\label{eq:pb_manager}
    V^{j,0}_0 \big(\xi^{j,0}, \widetilde \chi^{-j} \big) &:= \sup_{\widetilde \chi^j \in \widetilde{\mathscr X}^j } 
    %\sup_{\P \in \Pc^{\rm A,\star} (\P^{\rm M}, \xi^{\rm A}) } 
    J^{j,0} \big(\P^{\star} \big(\widetilde \chi \big), \xi^{j,0} \big).
\end{align}
Similar to the agents problem, we require minimal integrability on the contracts:
\begin{align}\label{eq:integrability_contract_manager}
    \sup_{\P \in \Pc} 
    \E^{\P} \bigg[ \Big| g^{j,0} \big(\zeta^j, \xi^{j,0} \big) \Big|^p \bigg] < + \infty, \; \text{for some} \; p >1.
    \tag{${\rm I}^p_{\rm M}$}
\end{align}

\begin{definition}[Nash equilibrium between managers]\label{def:nash_manager}
Given a collection $\xi^{\rm M} := (\xi^{j, 0})_{ j=1}^m \in \Cc^{\rm M}$ of contracts, a Nash equilibrium between the managers is a tuple of control $\widetilde \chi \in \widetilde{\mathscr X}$, such that for all $j \in \{1, \dots, m\}$,
\begin{align*}
    V^{j,0}_0 \big(\xi^{j,0}, \widetilde \chi^{-j} \big) =
    J^{j,0} \big(\P^{\star} \big(\widetilde \chi \big), \xi^{j,0} \big).
\end{align*}
We denote by $\Pc^{\rm M,\star} (\xi^{\rm M})$ the set of Nash equilibria given $\xi^{\rm M} \in \Cc^{\rm M}$.
\end{definition}

Similar to the previous Stackelberg game, to simplify the scope of our study, we will require that all eligible contracts for managers are those that induce a unique Nash equilibrium between them, \textit{i.e.}, $\Pc^{\rm M,\star} (\xi^{\rm M}) = \{ \P^\star (\xi^{\rm M})\}$. 

\begin{definition}\label{def:contract_manager_admissible}
A contract of the form \eqref{eq:contract_manager_dependency} satisfying \textnormal{\Cref{eq:integrability_contract_manager}} is called admissible. The corresponding class is denoted by $\Cc^{j,0}$. The product of the sets $\Cc^{j,0}$ for $j \in \{1, \dots, m\}$, such that the resulting collection $\xi^{\rm M}$ of contracts for the managers induces a unique Nash equilibrium between them, will be denoted by $\Cc^{\rm M}$.
\end{definition}

In addition, similarly to the agent's problems, we assume that the $j$--th manager has a reservation utility level $\rho^{j,0} \in \R$ below which he will refuse to work, and thus decline the contract offered by the principal. Mathematically speaking, a contract should also satisfy the following inequality:
\begin{align}\label{eq:participation_manager}
    V^{j,0}_0 \big( \xi^{j,0} \big) \geq \rho^{j,0}. \tag{${\rm PC_M}$}
\end{align}
As with the agents' admissible contracts, we could have assumed that the collection of admissible contracts for the managers is such that for any $j \in \{ 1,\dots, m\}$, the $j$--th manager's value function equals a particular $Y^{j,0}_0 \in \R$ at equilibrium. Then, the principal would optimise these $Y^{j,0}_0 \in \R$ in order to satisfy the manager's participation constraint. Nevertheless, contrary to the agents' problem, here the principal chooses both the value of the manager's continuation utility and his contract. Therefore, there is no need to first consider a subset of admissible contracts for a given level of continuation utility and then optimise this level.

\subsubsection{A principal at the top}\label{sss:principal_pb}

It remains to define, still informally at this point, the principal's problem. Contrary to the managers, her problem is more classical, as she does not directly control any process, she just designs the collection of contracts $\xi^{\rm M}$ for the $m$ managers. Her criterion is defined by:
\begin{align}\label{eq:principal_criterion}
    J^{\rm P} (\xi^{\rm M}) :=
    %\sup_{(\P^{\rm M}, \xi^{\rm A}) \in \Pc^{\rm M,\star} (\xi^{\rm M}) } 
    %\sup_{\P \in \Pc^{\rm A,\star} (\P^{\rm M}, \xi^{\rm A}) }
    % \E^{\P^{A,\star} (\P^{\rm M}, \xi^{\rm A})} \bigg[ \Kc^{\rm P}_{0,T} \; g^{\rm P} \big(\zeta, \xi^{\rm M} \big) \bigg],
    \E^{\P (\xi^{\rm M})} \Big[ \Kc^{\rm P}_{0,T} \; g^{\rm P} \big(\zeta, \xi^{\rm M} \big) \Big],
\end{align}
where $\P (\xi^{\rm M})$ can be seen informally as the probability resulting from the optimal controls of all managers and agents under her supervision given the contracts $\xi^{\rm M}$, and
\begin{enumerate}[label=$(\roman*)$]
    \item $g^{\rm P} : \R^{hm} \times \R^m \longrightarrow \R$ is a given utility function, non--increasing and concave in the second argument;
    \item the discount factor 
    \begin{align*}
        \Kc^{\rm P}_{s,t} := \exp \bigg( \int_s^t k^{\rm P}(r, \zeta) \drm r \bigg), \; 0 \leq s \leq t \leq T, 
    \end{align*}
    is defined by means of a bounded discount rate function $k^{\rm P} : [0,T] \times \Cc([0,T], \R^h) \longrightarrow \R$, $\G$--optional, recalling that $\G := (\Gc_t)_{t\in [0,T]}$ is the natural filtration of $\zeta$.
\end{enumerate}

The principal must therefore optimise the specific criterion defined by \eqref{eq:principal_criterion} by choosing the collection of contracts $\xi^{\rm M} \in \Cc^{\rm M}$, as well as the initial value of the agents' continuation utility $Y_0^{\rm A} \in \R^{w - m}$. Since we have assumed uniqueness of the Nash equilibrium between managers given a collection of admissible contracts, the principal's problem is simply given by:
\begin{align}\label{eq:pb_principal}
    V^{\rm P}_0 := \sup_{Y_0^{\rm A} \in \R^{w-m}} \sup_{\xi^{\rm M} \in \Cc^{\rm M}} J^{\rm P} (\xi^{\rm M}).
\end{align}
Note that the choice $Y_0^{\rm A}$ does not directly appear in the principal's optimisation problem. Nevertheless, this choice does have an impact on the optimal controls of both the managers and agents, as well as on the value of $\zeta$. However, to lighten the notations, we choose not to emphasise this dependency.

\section{Reduction to a standard stochastic control problem}\label{sec:solving_general_model}

Applying the recent results of \citeayn{cvitanic2018dynamic}, it is relatively simple to solve the $(j,i)$--th agent's problem. Indeed, by limiting our study to contracts indexed on the dynamics of the outputs $X^j$ and $\widebar X^{-j}$ and their quadratic variations, the optimal efforts are given by the maximisers of his Hamiltonian. It is then sufficient to show that the restriction to contracts of this form for the manager is without loss of generality. 

\medskip

The main result of this paper is that this reasoning can be extended to the Stackelberg game between the managers and the principal. Therefore, following the same intuition, we will limit our study to contracts for the managers indexed on a well--chosen state variable, namely $\zeta$, and its quadratic variation, through a tuple of parameters $\Zc \in \Vc$ chosen by the principal. We will prove that the problem of the managers is relatively simple to solve for this particular class of contract, and that this restriction is in fact without loss of generality for the principal. More precisely, we establish in \Cref{thm:main_principal} that, at the end of the day, the principal's problem defined by \eqref{eq:pb_principal} boils down to the following standard control problem:
\begin{align*}
V^{\rm P}_0
= \sup_{ Y_0 \geq \rho} 
\sup_{\Zc \in \Vc} 
\E^{\P^\star (\Zc)} \Big[ \Kc^{\rm P}_{0,T} g^{\rm P} (\zeta, \xi^{\rm M}) \Big],
\end{align*}
where 
\begin{enumerate}[label=$(\roman*)$]
    \item the inequality $Y_0 \geq \rho$ has to be understood componentwise, \textit{i.e.}, for all $j \in \{1, \dots, m\}$ and $i \in \{0, \dots, m\}$, $Y_0^{j,i} \geq \rho^{j,i}$, and ensure that the participation constraint of all workers is satisfied;
    \item $\P^\star (\Zc)$ is the unique Nash equilibrium between the workers, given the control $\Zc \in \Vc$ chosen by the principal;
    \item $\xi^{\rm M}$ is the collection of \textit{revealing contracts} for the managers, thoroughly characterised by the choice of $Y_0^{\rm M} \in \R^m$ and $\Zc \in \Vc$.
\end{enumerate}

\subsection{Contracting with the agents}\label{ss:agent_problem}

This subsection is devoted to solving the problem of a particular agent. With this in mind, we fix throughout the following $j \in \{1, \dots, m \}$ and $i \in \{1, \dots, n_j \}$, $Y^{j,i}_0 \in \R$, as well as the probabilities $\P^{-(j,i)}$ and $\P^{\rm M}$ chosen by other workers, and thus the associated efforts $\widehat \nu := (\nu^{-(j,i)}, \nu^{\rm M}) \in \Uc^{-(j,i)} \times \Uc^{\rm M}$.
%To easily distinguish the efforts of the considered agent from the efforts of others, we refer to $\P^{-(j,i)}, \P^{\rm M}$ as this previous pair of probabilities, and $\widehat \nu$ as the associated efforts of other workers. 
We have assumed previously that, given his manager's observation, an admissible contract $\xi^{j,i} \in \Cc^{j,i}$ for the $(j,i)$--th agent is restricted to functions of the form \eqref{eq:contract_agent_dependency}, and more precisely satisfies \Cref{def:contract_agent_admissible}. Therefore, in view of his objective function \eqref{eq:objectiv_agent}, we can already point out that the state variables of his optimisation problem \eqref{eq:pb_agent} are $X^j$ and $\widebar X^{-j}$. By considering the dynamic version $Y_t^{j,i}$ of his value function, we should have:
\begin{align}\label{eq:dynamic_value_function}
    Y_0^{j,i} = V_0^{j,i} \big( \P^{-(j,i)}, \P^{\rm M}, \xi^{j,i} \big), \text{ and } Y_T^{j,i} = g^{j,i} \big( X_{\cdot \wedge T}^{j,i}, \xi^{j,i} \big).
\end{align}
From this definition, we notice that we have an explicit relationship between the compensation $\xi^{j,i}$ and the terminal value function $Y_T^{j,i}$.
Given the probabilities chosen by other workers, and the associated efforts, we first write the Hamiltonian of the considered agent $(j,i)$ in \Cref{sss:hamiltonian}. Intuitively, this Hamiltonian appears by simply applying It\=o's formula to the dynamic function of the consumer and by considering the associated Hamilton--Jacobi--Bellman (HJB for short) equation. The next step is to derive a class of so--called revealing contracts, thus extending to a many--agents framework the results of \citeayn{cvitanic2018dynamic}, which considered general moral hazard problems with one agent, or similarly extending to volatility control the results of \citeayn{elie2019contracting}, where the agents controlled only the drift of the output process $X$. Similarly, the class of revealing contracts can be obtained by applying It\=o's formula, still in an informal way. For more details about this intuition, the reader may refer to \Cref{app:intuition_markovian}, where the reasoning is rigorously written in the Markovian framework. 
Finally, we will see in \Cref{thm:main_manager} that the restriction to revealing contracts is in fact without loss of generality. 

\subsubsection{Agent's Hamiltonian}\label{sss:hamiltonian}

Recalling that the agent's problem is defined by \eqref{eq:pb_agent}, and since his contract is restricted to functions of the form \eqref{eq:contract_agent_dependency}, the state variables of his problem are $X^j$ and $\widebar X^{-j}$. Nevertheless, the agent only controls the process $X^{j,i}$, while the dynamic of the other state variables are fixed through the probabilities $\P^{-(j,i)}$ and $\P^{\rm M}$. 
Following the reasoning in \cite{cvitanic2018dynamic}, as well as the intuition in the Markovian case developed in \Cref{app:intuition_markovian}, the agent's Hamiltonian is a sum of two components: 
\begin{enumerate}[label=$(\roman*)$]
    \item one is the classical Hamiltonian part as in \cite{cvitanic2018dynamic}, given by the supremum on the agent's effort $u := (a,b) \in U^{j,i}$ of the following function:
    \begin{align}\label{eq:agent_hamiltonian_tomax}
        h^{j,i}(t, x, y, z, \gamma, u) &:= - c^{j,i}(t,x,u) - k^{j,i}(t,x,u)y + \Lambda^{j,i} (t,u) z + \dfrac{1}{2} \big\| \sigma^{j,i}(t,b) \big\|^2 \gamma,
    \end{align}
    for $(t, x) \in [0,T] \times \Cc([0,T], \R)$, $(y, z, \gamma) \in \R^3$ and recalling that $\Lambda^{j,i} (t,u) := \sigma^{j,i}(t,b) \cdot \lambda^{j,i}(t,a)$;
    \item the second part is related to the indexation of the contract on the outputs of the other workers, and thus indexed on their effort $\widehat \nu := (\nu^{-(j,i)}, \nu^{\rm M}) \in \Uc^{-(j,i)} \times \Uc^{\rm M}$, fixed by $\P^{-(j,i)}$ and $\P^{\rm M}$:
    \begin{align}\label{eq:agent_hamiltonian_1}
        H^{j,i} \big( t, z, \widetilde z, \widehat \nu \big) := 
        z \cdot \Big( \Lambda^{j,\ell} \big(t, \widehat \nu_t^{j,\ell} \big) \Big)^{n_j}_{\ell = 0, \, \ell \neq i}
        + \widetilde z \cdot \bigg( \sum_{\ell = 0}^{n_k} \Lambda^{k,\ell} \big(t, \widehat \nu_t^{k,\ell} \big) \bigg)_{k = 1, \, k \neq j}^{m},
    \end{align}
    % \begin{align}\label{eq:agent_hamiltonian_1}
    %     H^{j,-i} (t, z, \gamma, \nu^{j,-i}) := \sum_{\ell = 0, \; \ell \neq i}^{n_j} \Big[ \sigma^{j,\ell} \big(t,\beta^{j, \ell} \big) \cdot \lambda^{j,\ell} \big( t,\alpha^{j, \ell} \big) z^\ell + \dfrac{1}{2} \big\| \sigma^{j,\ell} \big(t,\beta^{j, \ell} \big) \big\|^2 \gamma^\ell \Big]
    % \end{align}
    for $t \in [0,T]$ and $(z, \widetilde z) \in \R^{n_j} \times \R^{m-1}$.
\end{enumerate}

\begin{remark}\label{rk:SDE_simplify_why}
Without the assumption on the independence of the drift and the volatility with respect to the outputs $X$, the second part of $H^{j,i}$, defined by \eqref{eq:agent_hamiltonian_1}, would depend on the outputs of the other teams, which are not supposed to be observable by the manager of the $j$--th team. Unfortunately, the manager would not be able to compute the Hamiltonian of his agents in this case. This would lead to the more challenging case, already mentioned in \textnormal{\Cref{rk:measurability_manager}}, where the agent’s problem depends on another process, unobservable by the manager. Again, to our knowledge, this problem is not yet addressed in the literature in continuous time $($see the aforementioned paper \textnormal{\cite{huang2017optimal}} for a particular example$)$, and would require a full study before it could be considered in our case. 
\end{remark}

%In an attempt to simplify the notations, we denote by $\widehat \nu$ the collection of efforts of others, defined on the space $\widehat \U_0$. 
The Hamiltonian $\Hc^{j,i}$ of the agent $(j,i)$ is the supremum on his effort of the sum of the two previous terms:
\begin{align}\label{eq:agent_hamiltonian}
    \Hc^{j,i} \big(t, x, y, z, \widetilde z, \gamma, \widehat \nu \big) := \sup_{u \in U^{j,i}} h^{j,i} \big(t, x, y, z^i, \gamma, u \big) + H^{j,i} \big( t, z^{-i}, \widetilde z, \widehat \nu \big),
\end{align}
defined for any $(t, x, y, z, \widetilde z, \gamma) \in [0,T] \times \Cc([0,T], \R) \times \R \times \R^{n_j+1} \times \R^{m-1} \times \R$ and along $\widehat \nu \in \Uc^{-(j,i)} \times \Uc^{\rm M}$. By definition of $h^{j,i}$ in \eqref{eq:agent_hamiltonian_tomax}, and using \Cref{ass:separability}, we can already notice that a maximiser of the Hamiltonian, if it exists, can be written as a function $u^{j,i,\star} : [0,T] \times \R^3 \longrightarrow A^{j,i} \times B^{j,i}$:
\begin{align}\label{eq:hamiltonian_maximiser}
    u^{j,i, \star} \big(t, y, z^i, \gamma \big) := \big( a^{j,i, \star}, b^{j,i, \star} \big) \big(t, y, z^i, \gamma \big).
\end{align}
This maximiser, which will be proved later to be the optimal effort of the $(j,i)$--th agent, only depends on the time, the variable $y$ (which will be the agent's continuation utility), and the parameters $z^{i}$ and $\gamma$ (which will represent the indexation of the contract on respectively his output $X^{j,i}$ and its quadratic variation).
In particular, this maximiser does not depend on the effort of the other workers $\widehat \nu$, nor even on the indexation of the agent's contract on the outputs of the others, represented by the parameter $\widetilde z^{-j}$. Therefore, an agent will in fact optimise his efforts independently of others.
% \begin{align*}
%     \sup_{u \in U^{j,i}} h^{j,i}(t, x, y, z, \gamma, u) := &\ \sup_{u \in U^{j,i}} \big\{  - c_{\rm u}^{j,i}(t,u) - k_{\rm u}^{j,i}(t,u) y + \sigma^{j,i}(t,b) \cdot \lambda^{j,i}(t,a) z + \dfrac{1}{2} \big\| \sigma^{j,i}(t,b) \big\|^2 \gamma \big\} \\
%     &- c_{\rm x}^{j,i}(t,x) - k_{\rm x}^{j,i}(t,x) y.
% \end{align*}

\begin{remark}\label{rk:separability_why}
Without \textnormal{\Cref{ass:separability}}, the maximiser of the Hamiltonian would also be a function of $X^{j,i}$. In this case, even if the agent's supervisor could still compute his optimal effort, the manager of another team could not, which would be an issue, similar to the one mentioned in \textnormal{\Cref{rk:SDE_simplify_why}}.
\end{remark}

Similarly, we can define the Hamiltonian of any agents in the same way $\Hc^{j,i}$ is defined for the $(j,i)$--th agent. Therefore, the function \eqref{eq:hamiltonian_maximiser} is defined for any $j \in \{1, \dots, m\}$ and $i \in \{1, \dots, n_j\}$. Apart from the $(j,i)$--th agent, we will denote by $\nu^{-(j,i), \star}$ the collection of the effort of other agents such that their Hamiltonians are maximised.
To simplify the reasoning from now on, we will make the following assumption.
\begin{assumption}\label{ass:unicity_maximiser_agents}
    For all $j \in \{1, \dots, m\}$ and $i \in \{1, \dots, n_j\}$, there exists a unique Borel--measurable map $u^{j,i,\star} : [0,T] \times \R^3 \longrightarrow U^{j,i}$, defined by \eqref{eq:hamiltonian_maximiser}, maximising the Hamiltonian $\Hc^{j,i}$ given by \eqref{eq:agent_hamiltonian}.
\end{assumption}

\subsubsection{A relevant form of contracts leading to a Nash equilibrium}\label{sss:nash_equilibrium_agent}

We define the relevant subset of contracts, similarly as \citeayn[Definition 3.2]{cvitanic2018dynamic}, but extended to a multi--agents framework, in the spirit of \citeayn{elie2019contracting}, although with volatility control in addition. As for the definition of the Hamiltonian, this relevant form of contract is intuited from the Markovian framework, developed in \Cref{app:intuition_markovian}.

\medskip

Let $\V^j := \R^{n_j+1} \times \R^{m-1} \times \R$. For any $\V^j$--valued $\G^{j}$--predictable processes $\Zc := (Z,\widetilde Z, \Gamma)$, and any $Y_0^{j,i} \in \R$, let us define $\P$--a.s. for all $\P \in \Pc^{j,i}(\P^{-(j,i)}, \P^{\rm M})$ the process $Y^{j,i}$ by:
\begin{align}\label{eq:continuation_utility_agent}
    Y_{t}^{j,i}
    :=  Y_{0}^{j,i} - \int_0^t \Hc^{j,i} \big( r, X^{j,i}, Y_r^{j,i}, \Zc_r, \widehat \nu_r^{\star} \big) \drm r
    + \int_0^t Z_r \cdot \drm X_r^{j} 
    + \int_0^t \widetilde Z_r \cdot \drm \widebar X_r^{-j}
    + \dfrac{1}{2} \int_0^t \Gamma_r \drm \langle X^{j,i} \rangle_r,
\end{align}
for all $t \in [0,T]$, where $\Hc^{j,i}$ is defined by \eqref{eq:agent_hamiltonian} and $\widehat \nu^\star := (\nu^{-(j,i),\star}, \nu^{\rm M})$ is the collection of optimal efforts of the other agents, given by their Hamiltonian maximisers through \eqref{eq:hamiltonian_maximiser}, and fixed effort of the managers. This process $Y^{j,i}$ will represent the continuation utility of the $(j,i)$--th agent, given the action of others.

\begin{remark}\label{rk:solution_ODE}
Note that the process $Y^{j,i}$ is defined by \eqref{eq:continuation_utility_agent} as a solution of an \textnormal{ODE} with random coefficients. It is therefore necessary to mention that this \textnormal{ODE} is well defined so that the solution exists and is unique. This is indeed the case because the Hamiltonian $\Hc^{j,i}$ defined by \eqref{eq:agent_hamiltonian} is Lipschitz continuous in the variable $y$ \textnormal{(}due to the discount factor $k^{j,i}$ being bounded\textnormal{)}, thus guaranteeing that $Y^{j,i}$ is well--defined as the unique solution of \eqref{eq:continuation_utility_agent}.
\end{remark}

\begin{definition}\label{def:contractji}
Let $Y_0^{j,i} \in \R$. We denote by $\Vc^{j,i}$ the set of $\V^j$--valued $\G^{j}$--predictable process $\Zc$, such that
% \begin{enumerate}[label=$(\roman*)$]
%     \item
the process $Y^{j,i}$ defined by \eqref{eq:continuation_utility_agent} satisfies the following integrability condition:
\begin{align}\label{eq:integrability_continuation_utility_agent}
    \sup_{\P \in \Pc} \E^\P \bigg[ \sup_{t \in [0,T]} \big| Y^{j,i}_t \big|^p \bigg] < + \infty,
    \tag{${\rm J}^p_{\rm A}$}
\end{align}
for some $p > 1$. % is the same as in Condition \eqref{eq:integrability_contract_agent};
% \item there exists a weak solution $\P^{\star} \in \Pc^{j,i} (\P^{-(j,i)}, \P^{\rm M})$ such that the following equality
% \begin{align*}
%     \Hc^{j,i} (t, X^{j,i}, Y^{j,i}_t, Z_t, \widetilde Z_t, \Gamma_t, \nu^{-(j,i),\star}_t, \nu_t^M) =
%     h^{j,i} \big(t, X^{j,i}, Y^{j,i}_t, Z^{i}_t, \Gamma_t, \nu_t^{\P^\star} \big) + H^{j,i} \big( t, Z_t^{-i}, \widetilde Z_t, \nu_t^{-(j,i),\star}, \nu_t^M \big)
% \end{align*}
% holds $\drm t \otimes \P^\star$--a.e. on $[0,T] \times \Cc([0,T], \R^d)$.
% \end{enumerate}
For any $\Zc \in \Vc^{j,i}$, we call random variables of the form $\xi^{j,i} = \widebar g^{j,i} (X^{j,i}, Y_T^{j,i})$ \textbf{revealing contracts} for the $(j,i)$--th agent, and denote the corresponding set by $\Xi^{j,i}$.
\end{definition}

% \todo[inline]{Dans le cas général, on ne pourra pas avoir la forme explicite du contrat... Ok, mais pour le principal on va vouloir faire la somme... Pour simplifier, on peut potentiellement demander à ce que $g$ soit séparable (produit ou somme), pour obtenir plus facilement $\xi$ ? Ou sinon on dit juste que sous des hypothèses raisonnables ($g$ inversible par rapport à la deuxième variable), on peut écrire $\xi$ comme une fonction $u$ de $Y$ et $X$, et après on peut appliquer It\=o à $u$. On aurait un truc relativement explicite, mais qui dépend des dérivées de $u$ (donc des dérivées de $g$). Peut être la solution la moins pire.}

By considering revealing contracts, we are able to compute the optimal efforts of each agent, which were given informally by \eqref{eq:hamiltonian_maximiser}: intuitively, maximising each agent's Hamiltonian is sufficient to obtain his optimal efforts. 
Since the agent's optimal efforts do not depend on the efforts of the others, it simplifies the task of characterising a Nash equilibrium, in the sense of \Cref{def:nash_agent}. In other words, each agent controls his output independently of others. 
Still informally, \Cref{ass:unicity_maximiser_agents} is in force to ensure existence and uniqueness of the Nash equilibrium, thus avoiding technical considerations at the level of the managers' problem, which, in our opinion, are not relevant for this analysis. These results are rigorously presented in the following proposition.

\begin{proposition}\label{prop:nash_agent}
    Fix $\P^{\rm M}$ the probability chosen by the managers. For all $j \in \{1, \dots, m\}$ and $i \in \{1, \dots, n_j\}$, let $Y^{j,i}_0 \in \R$ and $\Zc^{j,i} := (Z^{j,i}, \widetilde Z^{j,i}, \Gamma^{j,i}) \in \Vc^{j,i}$. Define $Y^{j,i}$ and the associated contract $\xi^{j,i} \in \Xi^{j,i}$ as in \textnormal{\Cref{def:contractji}}. Denote by $\xi^{\rm A} \in \Xi^{\rm A}$ the subsequent collection of agents' contracts. Then, $\xi^{\rm A} \in \Cc^{\rm A}$ in the sense of \textnormal{\Cref{def:contract_agent_admissible}}, and the unique Nash equilibrium $\P^\star \in \Pc^{\rm A,\star} (\P^{\rm M}, \xi^{\rm A})$ is characterised by, for all $j \in \{1, \dots, m\}$ and $i \in \{1, \dots, n_j\}$,
    \begin{enumerate}[label=$(\roman*)$]
    \item the optimal effort of the $(j,i)$--th agent is given by the unique maximiser of his Hamiltonian, defined by \eqref{eq:hamiltonian_maximiser}:
    \begin{align*}
        \nu^{j,i,\star}_t := u^{j,i, \star} \big(t, Y^{j,i}_t, \big( Z_t^{j,i} \big)^i, \Gamma^{j,i}_t \big), \; \drm t \otimes \P^\star\text{--a.s. for all } t \in [0,T];
    \end{align*}
    \item $Y^{j,i}_0 = V_0^{j,i} \big(\P^{-(j,i),\star}, \P^{\rm M}, \xi^{j,i} \big)$.
% \item $\P^{\star}$ corresponds to the law of $X$ driven by optimal effort of all agents $\nu^{A,\star}$, when the efforts of the managers are fixed through the probability $\P^{\rm M}$.
\end{enumerate}
\end{proposition}

While the result given by the previous proposition is relatively intuitive, its formal proof is based on the 2BSDE theory, and is thus reported to \Cref{ss:tech_proof_agents}. It follows the same reasoning as the one developed for example by \citeayn[Proof of Theorem 3.4]{elie2019mean}: the crux of the argument is to show that we can construct directly a solution to a 2BSDE (more precisely, to \ref{eq:2bsde} defined in \Cref{sss:agent_best_reac}) whenever $\xi^{j,i} \in \Xi^{j,i}$, and this for all $j \in \{1, \dots, m\}$ and $i \in \{1, \dots, n_j\}$. 
Finally, the main point is to prove that the restriction of our study to revealing contracts is not restrictive from the managers' point of view. This is precisely the purpose of the following section.

\subsubsection{Optimality of the revealing contracts}

Let us fix throughout this section $j \in \{1, \dots, m\}$, in order to focus on the $j$--th managers' problem. Given an admissible contract $\xi^{j,0} \in \Cc^{j,0}$, in the sense of \Cref{def:contract_manager_admissible}, as well as the decision of other managers summarised by
% , the effort $\nu^{\ell, 0} \in \Uc^{\ell, 0}$ and the collection of contracts $\xi^{\ell \setminus 0} \in \Cc^{\ell \setminus 0}$ for all $\ell \in \{1, \dots, m \} \setminus \{j\}$, summarised by 
the control $\widetilde \chi^{-j}$, we recall that the $j$--th manager's optimisation problem is defined by \eqref{eq:pb_manager}. 
Following the general approach by \citeayn{cvitanic2018dynamic}, we can prove that there is no loss of generality for the manager to restrict to contracts in $\Xi^{j \setminus 0}$, in the sense of \Cref{def:contractji}, instead of contracts in $\Cc^{j \setminus 0}$. More precisely, to incentivise the $(j,i)$--th agent under his supervision, it is sufficient to offer him a revealing contract $\xi^{j,i} \in \Xi^{j,i}$, parametrised by a process $\Zc^{j,i} \in \Vc^{j,i}$, instead of considering all admissible contracts in $\Cc^{j,i}$. Therefore, the $j$--th manager has to choose in an optimal way:
\begin{enumerate}[label=$(\roman*)$]
    \item his own effort $\nu^{j,0} \in \Uc^{j,0}$;
    \item the triple of payment rates for each agent under his supervision, \textit{i.e.}, $\Zc^{j,i} \in \Vc^{j,i}$ for all $i \in \{1, \dots, n_j\}$.
\end{enumerate}
The control of the $j$--th manager can be summarised by a process $\chi^j := (\nu^{j,0}, (\Zc^{j,i})_{i=1}^{n_j} ) \in \mathscr X^j$ and, extending this reasoning to the other managers, the controls of all managers will be denoted by $\chi$, defined as follows:
\begin{align*}
    \chi := \big( \chi^j \big)_{j = 1}^{m} \in  \mathscr X, \; \textnormal{ where } \; 
    % j \in \{1, \dots,m\}, \;
    % \chi^j := ( \nu^{j,0}, \Zc^{j \setminus 0} ) \in \Uc^{j,0} \times \Vc^{j \setminus 0} =:  \mathscr X^j, \textnormal{ and }
    \mathscr X := \prod_{j=1}^m  \mathscr X^j, \text{ and } \; \mathscr X^j := \Uc^{j,0} \times \prod_{i=1}^{n_j} \Vc^{j,i}.
\end{align*}
Note that for all $j \in \{1, \dots, m\}$, the process $\chi^j$ takes values in $\mathfrak X^j := U^{j,0} \times (\V^j)^{n_j}$. The process $\chi$ thus takes values in $\Xk$, naturally defined as the Cartesian product of all $\Xk^j$. 
%Moreover, we recall that an agent will refuse the contract if his participation constraint, given by \Cref{eq:participation_agent}, is not satisfied. Therefore, the $j$--th manager also has to choose in an optimal way the collection of constant payments $Y_0^{j \setminus 0} := (Y^{j,i}_0)_{i=1}^{n_j} \in \R^{n_j}$ so that the participation constraint of each agent is satisfied.
With this in hand, we can now turn to the main theorem of this first Stackelberg game, whose proof is postponed to \Cref{ss:tech_proof_agents}. 
\begin{theorem}\label{thm:main_manager}
Consider a collection of admissible contracts $\xi^{\rm M} := (\xi^{j,0})_{j=1}^{m} \in \Cc^{\rm M}$ for the managers, in the sense of \textnormal{\Cref{def:contract_manager_admissible}}, then the following equality holds
\begin{align}\label{eq:main_manager}
V_0^{j,0} \big(\xi^{j,0}, \chi^{-j} \big) 
% = & \sup_{(\nu^{j,0}, \xi^{j \setminus 0} ) \in \Uc^{j,0} \times \Xi^{j \setminus 0} } 
% J^{j,0} \big(\P^{\star} \big(\P^{\rm M}, \xi^{\rm A} \big), \xi^{j,0} \big)
= \sup_{\chi^j \in  \mathscr X^j } 
%\sup_{ Y_0^{j \setminus 0} \geq \rho^{j \setminus 0} }
J^{j,0} \big(\P^{\star} (\chi), \xi^{j,0} \big), \text{ for all } \, j \in \{1,\dots,m\},
\end{align}
where 
%$(i)$ the inequality $Y_0^{j \setminus 0} \geq \rho^{j \setminus 0}$ has to be understood componentwise, \textit{i.e.}, for all $i \in \{1, \dots, n_j\}$, $Y_0^{j,i} \geq \rho^{j,i}$, and $(ii)$ 
$\P^\star (\chi)$ is the unique Nash equilibrium between the agents, given the control $\chi \in \mathscr X$ of the managers.
\end{theorem}

To summarise the agents' problem, first, \Cref{prop:nash_agent} solves the Nash equilibrium for a probability $\P^{\rm M}$ and a collection of revealing contracts $\xi^{\rm A} \in \Xi^{\rm A}$ chosen by the managers. Then, \Cref{thm:main_manager} states that the restriction to revealing contracts is without loss of generality. Using the previous results and notations, we can write the value function of each agent at equilibrium as follows:
\begin{align}\label{eq:V_ji_star}
    % V_0^{j,i,\star} \big( \P^{\rm M}, \xi^{\rm A} \big) 
    V_0^{j,i,\star} (\chi) 
    := V_0^{j,i} \big(\P^{-(j,i),\star}, \P^{\rm M}, \xi^{j,i} \big), \; \text{ for all } \; j \in \{1, \dots, m \}, \; i \in \{1, \dots, n_j\}.
\end{align}

\subsection{Contracting with the managers}\label{ss:manager_problem}

Throughout this subsection, we fix $j \in \{1, \dots, m\}$ and focus our attention on the $j$--th manager's problem. Recall that the $j$--th manager controls his own project with outcome $X^{j,0}$ as well as the compensations for his $n_j$ agents
\Cref{thm:main_manager} ensures that it is sufficient to restrict the admissible contract space for the $(j,i)$--th agent to $\Xi^{j,i}$, and thus limit the $j$--th manager's optimisation problem to choosing an optimal effort $\nu^{j,0} \in \Uc^{j,0}$, as well as $n_j$ triples $\Zc^{j,i} := (Z^{j,i}, \widetilde Z^{j,i}, \Gamma^{j,i}) \in \Vc^{j,i}$, for $i \in \{1, \dots, n_j \}$, to set up the contracts of the agents under his supervision.
Therefore, the manager's goal is to choose his control process $\chi^j \in \mathscr X^j$ optimally, given the results of his working team $\zeta^j$ and his compensation $\xi^{j,0}$ chosen by the principal. 

\medskip

Through the equality \eqref{eq:main_manager}, the manager's value function is similar to that of an agent in a classical principal--agent problem, since given a contract $\xi^{j,0}$, the manager chooses his optimal controls. 
However, the state variable $\zeta^j$, as a function of $\xi^{j \setminus 0}$, $X^j$ and $\widebar X^{-j}$, seems to be considered partially in strong and weak formulation. 
Indeed, $\xi^{j \setminus 0}$ is considered in strong formulation (indexed by the control $\chi^{j} \in \mathscr X^{j}$), while the vector of outputs $X$ is considered in weak formulation (the control $\chi^{j} \in \mathscr X^{j}$ only impacts the distribution of $X$ through $\P^{\star}$). It makes little sense to consider a control problem of this form directly, and we should adopt the weak formulation to state the problem of each manager, since this is the one which makes sense for the agents' problem.

\medskip

However, before turning to the weak formulation, it is necessary to determine which are the state variables of the manager problem. Recall that the contract $\xi^{j,0}$ for the manager can only be indexed on $\zeta := (\zeta^j)_{j=1}^m$, as defined in \eqref{eq:contract_manager_dependency}, since it is the only variable observable by the principal. Nevertheless, each $\zeta^j$ measures the global result of the entire $j$--th working team (including the manager) and therefore depends on the outcomes of the team, namely $X^{j}$, and the collection of payments $\xi^{j,i}$, for $i \in \{1, \dots, n_j\}$, to be made to the agents, denoted by $\xi^{j \setminus 0}$. More precisely, we can state that, for all $t \in [0,T]$, $\zeta^j_t = f^j \big(t, X^{j}, \xi^{j \setminus 0} \big)$, for some function $f^{j} : [0,T] \times \R^{n_j +1} \times \R^{n_j} \longmapsto \R^h$. 

\medskip

In all generality and without any particular assumptions on the function $f^j$, there is no reason the dynamic of $\zeta^j$ should not depend on the collections of all outputs and continuation utilities, which would therefore constitute the state variables of $j$--th manager. Unfortunately, this would lead us, once again, to the case where the principal does not observe all the state variables of the managers' problem. 
As already explained in \Cref{rk:measurability_manager,rk:SDE_simplify_why}, this would raise major issues requiring a comprehensive study before being addressed in our framework. We are therefore forced to make the following major assumption on the shape of the induced dynamic for $\zeta$. 
\begin{assumption}\label{ass:zeta_dynamic}
    The dynamic of $\zeta$ depends only on time, on $\zeta$ itself, and on the managers' controls summarised by the process $\chi \in \mathscr X$. More precisely, there exist two bounded functions:
    \begin{align*}
        \Lambda_{\rm M} &: [0,T] \times \Cc([0,T], \R^{m h}) \times \Xk \longrightarrow \R^{dw} \; \textnormal{and} \; \Sigma_{\rm M} : [0,T] \times \Cc([0,T], \R^{m h}) \times \Xk \longrightarrow \M^{hm, dw},
    \end{align*}
    satisfying $\Lambda_{\rm M} (\cdot, z)$ and $\Sigma_{\rm M} (\cdot, z)$ $\G$--optional for any $z \in \Xk$, such that $\zeta$ is solution to the following \textnormal{SDE}:
    \begin{align}\label{eq:dyn_zeta_simple}
        \drm \zeta_t := \Sigma_{\rm M} (t, \zeta, \chi_t) \big( \Lambda_{\rm M} (t, \zeta, \chi_t) \drm t + \drm W_t \big), \; \text{ for all } \; t \in [0,T].
    \end{align}
\end{assumption}
\noindent
Although restrictive, this assumption nevertheless allows the study of interesting frameworks that we may have in mind, including the context described by \citeayn{sung2015pay}. The reader is referred to \Cref{app:assumption_zeta}, where interesting cases in which this hypothesis is satisfied are presented (see \Cref{lem:linear_case,lem:exponential_case} for the linear and exponential utility functions respectively). 
Throughout the following, we will assume that \Cref{ass:zeta_dynamic} holds. 

\medskip

Even under \Cref{ass:zeta_dynamic}, the dynamic of $\zeta$ is controlled by all managers in a non--trivial way. Indeed, while the agents control uniquely their own outcome, in the sense that the $(j,i)$--th agent only impacts the dynamic of $X^{j,i}$, the $j$--th manager does not only control the variable $\zeta^j$, but also the other components of the vector $\zeta$. This is due to the fact that the component $\zeta^k$ (for $k \neq j$) depends on the collection of contracts $\xi^{k \setminus 0}$, which are indexed in particular on $\widebar X^{-k}$, and thus depend in particular on the optimal effort of the $(j,i)$--th agent given by
\begin{align*}
    \nu^{j,i,\star}_t := u^{j,i, \star} \big(t, Y^{j,i}_t, \big( Z_t^{j,i} \big)^i, \Gamma^{j,i}_t \big), \; \drm t \otimes \P^\star\text{--a.s. for all } t \in [0,T].
\end{align*}
Since the pair $(Z^{j,i}, \Gamma^{j,i})$ is chosen by the $j$--th manager, he somehow controls the volatility of $\zeta^k$, for all $k \neq j$. Note that he also controls the volatility through his own effort $\nu^{j,0}$. This leads us to a control problem with interacting agents, as for example, in the work of \citeayn{elie2019contracting} (see also \citeayn{elie2018tale} for the case of an infinite number of interacting agents), but with volatility control in addition. The following weak formulation is inspired by the work of \citeayn[Section 6.1]{possamai2018zero}, where a zero--sum game is considered between two players, controlling both the drift and the volatility of the same output process. Therefore, it only requires to extend their formulation to a nonzero--sum game with $m$ interacting agents.

\subsubsection{Canonical space for the managers}\label{sss:can_space_init_managers}

Following the previous reasoning, in particular under \Cref{ass:zeta_dynamic}, and given the form \eqref{eq:contract_manager_dependency} for the managers' contract, $\zeta$ is clearly the only state variable of the managers' problems. We are thus led to consider the following canonical space,
\begin{align*}
    \Omega^{\rm M} := \Cc \big( [0,T], \R^{hm} \big) \times \Cc \big( [0,T], \R^{wd} \big) \times \X,
\end{align*}
where, similarly as for the initial canonical space defined in \Cref{sss:can_space_init}, $\X$ is the collection of all finite and positive Borel measures on $[0,T] \times \Xk$, whose projection on $[0,T]$ is the Lebesgue measure. 
The weak formulation requires to consider a subset of $\X$, namely the set $\X_0$ of all $q \in \X$ such that $q(\drm s,\drm u) = \delta_{\phi_s}(\drm u) \drm s$, for some Borel function $\phi$. The canonical process is denoted by $(\zeta, W, \Pi^{\rm M})$ where, for any $(t, \omega, \varpi, q) \in [0,T] \times \Omega$,
\begin{align*}
    \zeta_t (\omega, \varpi, q) := \omega_t, \; W_t (\omega, \varpi, q) := \varpi_t, \text{ and } \Pi^{\rm M} (\omega, \varpi, q) := q.
\end{align*}
The associated canonical filtration is defined by $\F^{\rm M} := (\Fc^{\rm M}_t)_{t \in [0,T]}$ with
\begin{align*}
    \mathcal F^{\rm M}_t := \sigma \bigg( \bigg( \zeta_s, \int_0^s\int_{\Xk} \varphi(r,u) \Pi^{\rm M}(\mathrm{d}r,\mathrm{d} u) \bigg) \text{ s.t. } (s,\varphi) \in [0,t] \times \Cc_b \big( [0,T] \times \Xk, \R \big) \bigg),\; t\in [0,T].
\end{align*}
Then, for any $(t,\psi) \in [0,T] \times \Cc^2_b(\R^{hm} \times \R^{dw}, \R)$, we set
\begin{align*}
    M^{\rm M}_t (\psi) := \psi (\zeta_t, W_t) - \int_0^t \int_{\Xk} \Big( \widetilde \Lambda_{\rm M} \big(s, \zeta, u \big) \cdot \nabla \psi (\zeta_s, W_s) 
    + \frac12  {\rm Tr} \big[ \nabla^2 \psi (\zeta_s, W_s) \big( \widetilde \Sigma_{\rm M} \widetilde \Sigma_{\rm M}^\top \big) (s, \zeta, u) \big] \Big) \Pi^{\rm M} (\mathrm{d}s, \mathrm{d} u).
\end{align*}
where $\widetilde \Lambda_{\rm M}$ and $\widetilde \Sigma_{\rm M}$ are respectively the drift vector and the diffusion matrix of the $(hm + d w)$--dimensional vector process $(\zeta, W)^\top$, defined for all $s \in [0,T]$, $x \in \Cc([0,T],\R^{hm}) $ and $u \in \Xk$ by:
\renewcommand*{\arraystretch}{1}
\begin{align*}
    \widetilde \Lambda_{\rm M} (s,x,u) := 
    \begin{pmatrix}
        \big(\Sigma_{\rm M} \Lambda_{\rm M}\big) (s,x,u) \\
        \mathbf{0}_{w d}
    \end{pmatrix},
    \; 
    \widetilde \Sigma_{\rm M} (s,x,u) :=
    \begin{pmatrix}
        \mathbf{0}_{hm, hm} & \Sigma_{\rm M} (s,x,u) \\
        \mathbf{0}_{wd, hm} & \mathrm{I}_{w d} \\
    \end{pmatrix},
\end{align*}
where $\Lambda_{\rm M}$ and $\Sigma_{\rm M}$ are defined in \Cref{ass:zeta_dynamic}.

\medskip

We fix an initial condition for the process $\zeta$, namely $\varrho_0 \in \R^{hm}$. Similarly as in \Cref{def:weak_formulation} for the initial control problem, we define the subset $\Pc^{\rm M}$ of probability measures $\P$ on $(\Omega^{\rm M},\Fc^{\rm M}_T)$ satisfying the following conditions:
\begin{enumerate}[label=$(\roman*)$]
    \item $M^{\rm M}(\psi)$ is a $(\F^{\rm M}, \P)$--local martingale on $[0,T]$ for all $\psi \in \Cc^2_b(\R^{hm} \times \R^{dw},\R)$;
    \item $\P [(\zeta_0,W_0) = (\varrho_0,w_0)]=1$;
    \item $\P\big[\Pi^{\rm M} \in \X_0]=1$.
\end{enumerate}
Similarly to \Cref{lem:no_enlarge}, we know that for all $\P \in \Pc^{\rm M}$, we have $\Pi^{\rm M} (\mathrm{d}s,\mathrm{d} u) = \delta_{\chi^{\P}_s}(\mathrm{d} u)\mathrm{d}s$ $\P$--a.s. for some $\F^{\rm M}$--predictable control process $\chi^{\P}$, and we thus obtain the representation \eqref{eq:dyn_zeta_simple} for the dynamic of $\zeta$, but controlled by $\chi := \chi^{\P}$.
However, this representation only gives access to an admissible set of controls in terms of probability measures for all managers, namely $\Pc^{\rm M}$. 

\subsubsection{Weak formulation of a manager's problem}\label{sss:weak_manager}

To properly define the choice of a particular manager, in response to the choices of others, we need to define its own canonical space. With this in mind, we fix throughout the following $j \in \{1, \dots, m\}$, as well as the controls $\chi^{-j} \in \mathscr X^{-j}$ chosen by other managers. The canonical space for the $j$--th manager is defined by
\begin{align*}
    \Omega^j := \Cc \big( [0,T], \R^{hm} \big) \times \Cc \big( [0,T], \R^{wd} \big) \times \X^j,
\end{align*}
where $\X^j$ is the collection of all finite and positive Borel measures on $[0,T] \times \Xk^j$, whose projection on $[0,T]$ is the Lebesgue measure. Similarly as before, we define the corresponding set $\X^j_0$ of all $q \in \X^j$ such that $q(\drm s,\drm u) = \delta_{\phi_s} (\drm u) \drm s$ for some Borel function $\phi$. The canonical process is denoted by $(\zeta, W, \Pi^j)$ where,
\begin{align*}
    \zeta_t (\omega, \varpi, q) := \omega_t, \; W_t (\omega, \varpi, q) := \varpi_t, \text{ and } \Pi^j (\omega, \varpi, q) := q, \; \text{ for any } \; (t, \omega, \varpi, q) \in [0,T] \times \Omega^j.
\end{align*}
The associated canonical filtration is defined by $\F^j := (\Fc_t^j)_{t \in [0,T]}$ with
\begin{align*}
    \mathcal F^{j}_t := \sigma \bigg( \bigg( \zeta_s, \int_0^s\int_{\Xk^{j}} \varphi(r,u) \Pi^{j}(\mathrm{d}r,\mathrm{d} u) \bigg) \text{ s.t. } (s,\varphi) \in [0,t] \times \Cc_b \big( [0,T] \times \Xk^{j}, \R \big) \bigg),\; t\in [0,T].
\end{align*}
Then, for any $(t,\psi) \in [0,T] \times \Cc^2_b(\R^{hm} \times \R^{dw}, \R)$, we set
\begin{align*}
    M^{j}_t (\psi) := \psi (\zeta_t, W_t) - \int_0^t \int_{\Xk^j} \Big( &\ \widetilde \Lambda_{\rm M} \big(s, \zeta, u \otimes_j \chi^{-j}_s \big) \cdot \nabla \psi (\zeta_s, W_s) \\
    &+ \frac12  {\rm Tr} \big[ \nabla^2 \psi (\zeta_s, W_s) \big( \widetilde \Sigma_{\rm M} \widetilde \Sigma_{\rm M}^\top \big) (s, \zeta, u \otimes_j \chi^{-j}_s) \big] \Big) \Pi^j (\mathrm{d}s, \mathrm{d} u).
\end{align*}
where $(u \otimes_j \chi^{-j}_t )^j = u$ and $(u \otimes_j \chi^{-j}_t)^k = \chi_t^k$ for $k \neq j$, recalling that $\chi^{-j}$ is fixed throughout this section.

\medskip

We can then define the subset $\Pc^j (\chi^{-j})$ of probability measures $\P$ on $(\Omega^j,\Fc^j_T)$ satisfying the following conditions:
\begin{enumerate}[label=$(\roman*)$]
    \item $M^j(\psi)$ is a $(\F^j, \P)$--local martingale on $[0,T]$ for all $\psi \in \Cc^2_b(\R^{hm} \times \R^{dw},\R)$;
    \item there exists some $\iota \in \R^{wd}$ such that $\P \circ (\zeta_0, W_0)^{-1} = \delta_{(\varrho,\iota)}$;
    \item $\P\big[\Pi^j \in \X_0^j]=1$.
\end{enumerate}
We know that for all $\P \in \Pc^j (\chi^{-j})$, we have $\Pi^j (\mathrm{d}s,\mathrm{d} u) = \delta_{\chi^{j,\P}_s}(\mathrm{d} u)\mathrm{d}s$ $\P$--a.s. for some $\F^j$--predictable control process $\chi^{j,\P}$, and we obtain the representation \eqref{eq:dyn_zeta_simple} for the dynamic of $\zeta$, but controlled by $\chi := \chi^{j,\P} \otimes_j \chi^{-j}$.

\medskip

Therefore, given $\chi^{-j} \in \mathscr X^{-j}$ chosen by other managers, the $j$--th manager must choose an optimal probability measure $\P \in \Pc^j (\chi^{-j})$, which leads to consider the following weak formulation for his optimisation problem \eqref{eq:pb_manager}:
\begin{align}\label{eq:manager_weak_formulation}
    V^{j,0}_0 \big( \xi^{j,0}, \chi^{-j} \big) &:= 
    %\sup_{ Y_0^{j \setminus 0} \geq \rho^{j \setminus 0} }
    \sup_{\P \in \Pc^j (\chi^{-j}) } 
    %\sup_{\P \in \Pc^{\rm A,\star} (\P^{\rm M}, \xi^{\rm A}) } 
    J^{j,0} \big(\P, \xi^{j,0} \big),
\end{align}
recalling that $J^{j,0}$ is defined by \eqref{eq:criterion_manager}.
We can then adapt \Cref{def:nash_manager} to define a Nash equilibrium between the managers in weak formulation.

\begin{definition}[Nash equilibrium between managers, in weak formulation]\label{def:nash_manager_weak}
    Fix a collection of contracts designed by the principal for the managers, namely $\xi^{\rm M} := (\xi^{j, 0})_{j=1}^m \in \Cc^{\rm M}$. A Nash equilibrium between the managers is a control $\chi \in \mathscr X$, such that there exists a probability measure $\P^\star \in \Pc^{\rm M}$ satisfying, for all $j \in \{1, \dots, m\}$, $\P^\star \in \Pc^j (\chi^{-j})$ and such that the supremum in \eqref{eq:manager_weak_formulation} is attained for this $\P^\star$.
    We denote by $\Pc^{\rm M,\star} (\xi^{\rm M})$ the set of Nash equilibria.
\end{definition}

\subsubsection{Relevant form of contracts for the managers}\label{sss:hamiltonian_manager}

Recall that \Cref{ass:zeta_dynamic} is enforced to ensure that $\zeta$ is the only state variable of the $j$--th manager's optimisation problem.
Following the line developed in \Cref{sss:hamiltonian} for the agents (based on \cite{cvitanic2018dynamic}), the Hamiltonian of the $j$--th manager is defined by:
\begin{align}\label{eq:hamiltonian_manager}
    \Hc^{j} ( t, x, y, z, \gamma, \chi^{-j}) 
    :=  \sup_{u \in \Xk^j} h^{j} ( t, x, y, z, \gamma, \chi^{-j}, u),
\end{align}
for $(t, x, y) \in [0,T] \times \Cc([0,T], \R^{hm}) \times \R$, $(z, \gamma) \in \R^{hm} \times \M^{hm}$ and $\chi^{-j} \in \mathscr X^{-j}$ chosen by other managers, where in addition, for $u \in \Xk^j$,
\begin{align*}
    h^{j} ( t, x, y, z, \gamma, \chi^{-j}, u) := &- \big( c^{j,0} + y k^{j,0} \big) (t,x^j,u) + \big( \Sigma_{\rm M} \Lambda_{\rm M} \big) \big( t,x, u \otimes_j \chi^{-j} \big) \cdot z 
    + \dfrac{1}{2} {\rm Tr} \big[ \big( \Sigma_{\rm M} \Sigma_{\rm M}^\top \big) \big( t,x,u \otimes_j \chi^{-j} \big) \gamma \big].
\end{align*}
Similarly, we can define the Hamiltonian of any managers in the same way $\Hc^{j}$ is defined for the $j$--th manager. We can already notice that a maximiser of the $j$--th manager's Hamiltonian, if it exists, depends on
\begin{enumerate}[label=$(\roman*)$]
    \item the time;
    \item the paths of the state variable $\zeta$;
    \item the parameter $y$, which will be the manager's continuation utility;
    \item the parameters $z$ and $\gamma$, which will represent the indexation of the contract on respectively the reporting $\zeta$ and its quadratic variation;
    \item the efforts of other managers, namely $\chi^{-j}$.
\end{enumerate}

\begin{assumption}\label{ass:existence_maximiser_manager}
For all $j \in \{1, \dots, m\}$, there exists at least a Borel--measurable map $u^j : [0,T] \times \Cc( [0,T], \R^{hm}) \times \R \times \R^{hm} \times \M^{hm} \times \mathscr X^{-j} \longrightarrow \Xk^j$, satisfying:
\begin{align*}
    \Hc^{j} \big( t, x, y, z, \gamma, \chi^{-j} \big) 
    = h^{j} \big( t, x, y, z, \gamma, \chi^{-j}, u^{j} (t, x, y, z, \gamma, \chi^{-j} ) \big),
\end{align*}
for all $(t, x, y, z, \gamma) \in [0,T] \times \Cc( [0,T], \R^{hm}) \times \R \times \R^{hm} \times \M^{hm}$ and given the actions of other managers $\chi^{-j} \in \mathscr X^{-j}$.
\end{assumption}

Contrary to the optimal efforts of agents, who were independent of the efforts of other agents, the maximiser of the Hamiltonian therefore depends on the efforts of other managers, \textit{i.e.}, $\chi^{-j}$. Similarly, these efforts $\chi^{-j}$ will be defined through the maximiser of other managers' Hamiltonian, and will thus depends on $(t,x,y^{-j}, z^{-j}, \gamma^{-j})$, but also on $\chi^j$. This leads us to consider in some way a fixed point of a multidimensional Hamiltonian, where each component represents the Hamiltonian of a manager.

\begin{assumption}\label{ass:unicity_fixedpoint_hamiltonien}
There exists a unique Borel--measurable map $u^{\star} : [0,T] \times \Cc( [0,T], \R^{hm}) \times \R^m \times \M^{hm,m} \times (\M^{hm})^m \longrightarrow \prod_{j=1}^m \Xk^j$, such that for all $j \in \{1, \dots, m \}$, the $j$--th component $u^{j,\star}$ takes values in $\Xk^j$ and satisfies:
\begin{align*}
    \Hc^{j} \big( t, x, y^j, z^j, \gamma^j, u^{-j,\star} (t, x, y, z, \gamma ) \big) 
    = h^{j} \big( t, x, y^j, z^j, \gamma^j, u^{-j,\star} (t, x, y, z, \gamma ), u^{j,\star} (t, x, y, z, \gamma ) \big),
\end{align*}
for all $(t, x) \in [0,T] \times \Cc( [0,T], \R^{hm})$, $y := (y^j)_{j=1}^m \in \R^m$, and $(z,\gamma) \in \M^{hm,m} \times (\M^{hm})^m$ where for all $j \in \{1,\dots, m\}$, $z^j \in \R^{hm}$ is the $j$--th column of $z$ and $\gamma^j \in \M^{hm}$.
\end{assumption}

\begin{remark}
The previous assumption, in particular on the existence of such a function, is classical in multi--agents problems to ensure existence of an equilibrium for the managers \textnormal{(}see for example \citeayn[Assumption 4.1]{elie2019contracting} for a multi--agents problem with drift control only\textnormal{)}. Uniqueness is assumed to simplify the study, as in the similar hypothesis for the agents \textnormal{(}see \textnormal{\Cref{ass:unicity_maximiser_agents})}. Indeed, in the absence of uniqueness, it would be possible to have several Nash equilibria, and it would then be necessary to represent the preferences of the managers and the principal between these different Nash equilibria \textnormal{(}see also \textnormal{\cite[Section 4.1.1]{elie2019contracting}} for an example\textnormal{)}.
\end{remark}

Thanks to \Cref{ass:unicity_fixedpoint_hamiltonien}, we can define the following function, which corresponds to the Hamiltonian of the $j$--th manager under optimal efforts of all managers:
\begin{align}\label{eq:hamiltonian_manager_star}
    \Hc^{j,\star} (t, x, y, z, \gamma) 
    := \Hc^{j} \big( t, x, y^j, z^j, \gamma^j, u^{-j,\star} (t, x, y, z, \gamma ) \big),
\end{align}
for $(t, x, y, z, \gamma) \in [0,T] \times \Cc( [0,T], \R^{hm}) \times \R^m \times \M^{hm,m} \times (\M^{hm})^m$, and taking values in $\R$.

% Apart from the $j$--th manager, we will denote by $\chi^{-j, \star}$ the collection of the effort of other managers such that their Hamiltonians are maximised.

\medskip

We can now define the relevant subset of contracts, similarly as for the agents, except that the contract has to be indexed on $\zeta$, and that the managers' Hamiltonians are coupled. Let $\V := \M^{hm,m} \times (\M^{hm})^m$. For any $\V$--valued $\G$--predictable processes $(Z,\Gamma)$, and any $Y_0^{\rm M} := (Y_0^j)_{j=1}^m \in \R^m$, let us define the multidimensional process $\Yc^{\rm M}$ such that each component $\Yc^{j}$, for $j \in \{1, \dots, m\}$, satisfies:
\begin{align}\label{eq:continuation_utility_manager}
    \Yc_{t}^{j}
    :=  Y_0^j - \int_0^t \Hc^{j,\star} \big( r, \zeta, \Yc_r^{\rm M}, Z_r, \Gamma_r \big) \drm r
    + \int_0^t Z^j_r \cdot \drm \zeta_r
    + \dfrac{1}{2} \int_0^t {\rm Tr} \big[ \Gamma^j_r \drm \langle \zeta \rangle_r \big], \; t \in [0,T],
\end{align}
where $\Hc^{j,\star}$ is defined by \eqref{eq:hamiltonian_manager_star}. Each component $\Yc^{j}$ will thus represent the continuation utility of the $j$--th manager, given the action of others.

\medskip

Similarly as for the agents' problem (see \Cref{rk:solution_ODE}), each component $\Yc^{j}$ of the process $\Yc^{\rm M}$ is defined by \eqref{eq:continuation_utility_manager} as a solution to an \textnormal{ODE} with random coefficients. The following assumption is made to ensure that this \textnormal{ODE} is well defined, so that the solution exists and is unique.
\begin{assumption}\label{ass:hamiltonian_lipschitz}
The multidimensional Hamiltonian $\Hc^{\star}$, whose components are defined by \eqref{eq:hamiltonian_manager_star},
%for all $(t, x, y, z, \gamma) \in [0,T] \times \Cc( [0,T], \R^{hm}) \times \R^m \times (\R^{hm})^m \times (\M^{hm})^m$, 
is uniformly Lipschitz continuous with respect to the variable $y \in \R^m$.
\end{assumption}
One can notice that if, for all $j \in \{1, \dots,m\}$, the discount factor $k^{j,0}$ is not controlled, meaning that $k^{j,0} : [0,T] \times \Cc([0,T], \R^h) \longrightarrow \R$, then the previous assumption is not necessary. Indeed, in this case, the optimal effort of the managers, defined through the function $u^\star$ will be independent of the variable $y \in \R^m$, thus also implying the independence of the Hamiltonian $\Hc^{\star}$. Similarly, in the case of exponential (CARA) utility functions for the managers, a change of variable allows to suppress the dependence of the control on the variable $y$. Therefore, in these two cases, \Cref{ass:hamiltonian_lipschitz} is trivially satisfied, ensuring the well--definition of the process $\Yc^{\rm M}$.

\begin{definition}\label{def:contract_manager}
Let $Y_0^{\rm M} \in \R^m$. We denote by $\Vc$ the set of $\V$--valued $\G$--predictable process $\Zc$, such that
for all $j \in \{1, \dots, m\}$, each component $\Yc^j$ of the $m$--dimensional process $\Yc^{\rm M}$ defined by \eqref{eq:continuation_utility_manager} satisfies the following integrability condition, for some $p > 1$:
\begin{align}\label{eq:integrability_continuation_utility_manager}
    \sup_{\P \in \Pc^{\rm M}} \E^\P \bigg[ \sup_{t \in [0,T]} \big| \Yc^{j}_t \big|^p \bigg] < + \infty,
    \tag{${\rm J}^p_{\rm M}$}
\end{align}
Consider the function $\widebar g^{\rm M} : \Cc([0,T],\R^{hm}) \times \R^m \longrightarrow \R^m$ such that for all $j \in \{1, \dots,m\}$, the $j$--th component is given by the map $\widebar g^{j,0} : \Cc([0,T],\R^{h}) \times \R \longrightarrow \R$ defined in the managers' problem. 
For $\Zc \in \Vc$ and $Y_0^{\rm M} := (Y_0^j)_{j=1}^m \in \R^m$, we consider the $m$--dimensional random variable $\xi^{\rm M} := \widebar g^{\rm M} (\zeta, \Yc_T^{\rm M})$, and denote the corresponding set by $\Xi^{\rm M}$. We will say that $\xi^{\rm M} \in \Xi^{\rm M}$ is a collection of \textbf{revealing contracts} for the managers. In particular, for all $j \in \{1,\dots,m\}$, its $j$--th component satisfies $\xi^{j} = \widebar g^{j,0} (\zeta^{j}, \Yc_T^{j})$, and the corresponding set is denoted by $\Xi^j$.
\end{definition}

\subsubsection{Nash and optimality of revealing contracts}

By considering revealing contracts, we are able to compute the optimal efforts of each manager, which were given informally by the maximiser of their Hamiltonian.
Contrary to the agents' problem, the manager's optimal efforts depend on the efforts of other managers.
Informally, \Cref{ass:unicity_fixedpoint_hamiltonien} is in force to ensure existence and uniqueness of the Nash equilibrium, thus avoiding technical considerations on the preferences between the different Nash equilibria, which, in our opinion, are not relevant for this analysis. These results are rigorously presented in the following proposition.

\begin{proposition}\label{prop:nash_manager}
    Let $Y_0^{\rm M} := (Y_0^j)_{j=1}^m \in \R^m$ and $\Zc := (Z, \Gamma) \in \Vc$. By \textnormal{\Cref{def:contract_manager}}, consider the $m$--dimensional process $\Yc^{\rm M}$ and the associated collection of contracts $\xi^{\rm M} := (\xi^{j})_{j=1}^m \in \Xi^{\rm M}$. Then, $\xi^{\rm M} \in \Cc^{\rm M}$ in the sense of \textnormal{\Cref{def:contract_manager_admissible}} and there exists a unique Nash equilibrium in the sense of \textnormal{\Cref{def:nash_manager_weak}}, \textit{i.e.}, a control $\chi^\star \in \mathscr X$ associated to a probability measure $\P^\star$. This Nash equilibrium is characterised by:
    \begin{enumerate}[label=$(\roman*)$]
    \item for all $j \in \{1, \dots, m\}$, the optimal effort of the $j$--th manager is given by the $j$--th component of the unique fixed point of the multidimensional Hamiltonian, defined through \eqref{ass:unicity_fixedpoint_hamiltonien}, \textit{i.e.}:
    \begin{align}\label{eq:chi_star}
        \chi^{j,\star}_t := u^{j, \star} \big(t, \zeta, \Yc^{\rm M}_t, Z_t, \Gamma_t \big), \; \drm t \otimes \P^\star\text{--a.s. for all } t \in [0,T];
    \end{align}
    \item $Y_0^j = V^{j,0}_0 \big( \xi^{j}, \chi^{-j,\star} \big)$.
% \item $\P^{\star}$ corresponds to the law of $\zeta$ driven by optimal effort of all managers $\chi^{\star}$.
\end{enumerate}
\end{proposition}

The formal proof of the previous result is based on the 2BSDE theory, and is thus reported to \Cref{ss:tech_proof_managers}. It follows the same reasoning as the one developed in the proof of \Cref{prop:nash_agent} (see \Cref{ss:tech_proof_agents}). \Cref{prop:nash_manager} solves the Nash equilibrium for a collection of revealing contracts $\xi^{\rm M} \in \Xi^{\rm M}$ chosen by the principal. At equilibrium, we can write the value function of the $j$--th manager as follows:
\begin{align}\label{eq:V_j0_star}
    V^{j,0,\star} \big( \xi^{\rm M} \big) := V_0^{j,0} \big(\xi^{j}, \chi^{-j,\star} \big), \text{ for all } \; j \in \{1, \dots, m \}.
\end{align}
Finally, it remains to prove that the specialisation of our study to revealing contracts is not restrictive from the principal's point of view. 
This result is given by the following theorem, which echoes \Cref{thm:main_manager} for the manager--agent problem. Its formal proof is postponed to \Cref{ss:tech_proof_managers}.

\begin{theorem}\label{thm:main_principal}
Recalling that the principal's problem is defined by \eqref{eq:pb_principal}, the following equality holds
\begin{align}\label{eq:main_principal}
V^{\rm P}_0
= \sup_{ Y_0 \geq \rho} \widebar V^{\rm P} (Y_0) \; \text{ where } \; 
\widebar V^{\rm P} (Y_0) := \sup_{\Zc \in \Vc} \E^{\P^\star (\Zc)} \Big[ \Kc^{\rm P}_{0,T} g^{\rm P} (\zeta, \xi^{\rm M}) \Big],
\end{align}
where 
\begin{enumerate}[label=$(\roman*)$]
    \item the inequality $Y_0 \geq \rho$ has to be understood componentwise, \textit{i.e.}, for all $j \in \{1, \dots, m\}$ and $i \in \{0, \dots, m\}$, $Y_0^{j,i} \geq \rho^{j,i}$, and ensure that the participation constraint of all workers is satisfied;
    \item $\P^\star (\Zc)$ is the unique Nash equilibrium between the workers, given the control $\Zc \in \Vc$ chosen by the principal;
    \item $\xi^{\rm M}$ is the collection of \textit{revealing contracts} for the managers, thoroughly characterised by the choice of $Y_0^{\rm M} \in \R^m$ and $\Zc \in \Vc$.
\end{enumerate}
\end{theorem}
Note that the choices of the workers' initial continuation utility, namely $Y_0 := (Y_0^{\rm A},Y_0^{\rm M})$, directly impact the initial value of $\zeta$. Moreover, they also have an impact on the Nash equilibrium between the managers. However, again with a view to lighten the notations, this dependency is not explicitly mentioned.

\subsection{Principal's problem}\label{ss:principal_problem}

Following the previous reasoning, in particular under \Cref{ass:zeta_dynamic}, $\zeta$, which is the only variable observable by the principal, is also clearly the only state variable of her problem together the continuation utilities of the managers, namely $\Yc^{\rm M}$, since $\xi^{\rm M} \in \Xi^{\rm M}$ satisfies $\xi^{\rm M} = \widebar g^{\rm M} (\zeta, \Yc_T^{\rm M})$. 

\subsubsection{Canonical space for the principal}

First, we should write the dynamics of $\zeta$ and $\Yc^{\rm M}$ under managers' optimal efforts. With this in mind, and recalling the definition of the map $u^\star$ in \Cref{ass:unicity_fixedpoint_hamiltonien}, we define two functions $\Lambda_{\rm M}^\star, \Sigma_{\rm M}^\star : [0,T] \times \Cc([0,T],\R^{hm}) \times \R^m \times \M^{hm,m} \times (\M^{hm})^m \longrightarrow \R^{dw}, \M^{hm, dw}$, satisfying:
\begin{align}\label{eq:lamb_sig_star}
    \Lambda_{\rm M}^\star (t, x, y, v) := \Lambda_{\rm M} \big(t, x,  u^{\star} (t, x, y, z, \gamma) \big) \; \text{ and } \; 
    \Sigma_{\rm M}^\star (t, x, y, v) := \Sigma_{\rm M} \big(t, \zeta,  u^{\star} (t, x, y, z, \gamma) \big),
\end{align}
for all $(t, x, y) \in [0,T] \times \Cc([0,T],\R^{hm}) \times \R^m$ and $v := (z,\gamma) \in \M^{hm,m} \times (\M^{hm})^m$. 
Similarly, we define the functions
$c^{\rm M, \star}, k^{\rm M, \star} : [0,T] \times \Cc([0,T],\R^{hm}) \times \R^m \times \M^{hm,m} \times (\M^{hm})^m \longrightarrow \R^m$ such that for all $j \in \{1,\dots,m\}$, the $j$--th components $c^{j,\star}$ and $k^{j,\star}$ respectively satisfy:
\begin{align*}
    c^{j,\star} (t, x, y, v) := c^{j,0} (t, x^j, u^{j,\star} (t, x, y, z, \gamma)) \; \text{ and } \; 
    k^{j,\star} (t, x, y, v) := k^{j,0} (t, x^j, u^{j,\star} (t, x, y, z, \gamma)),
\end{align*}
for $(t, x, y, v) \in [0,T] \times \Cc([0,T],\R^{hm}) \times \R^m \times \V$ and more precisely $x := (x^j)_{j=1}^m$, where $x^j \in \Cc([0,T],\R^h)$.

\medskip

With these notations, $\zeta$ is solution to the following SDE:
\begin{align}\label{eq:dyn_zeta_star}
    \zeta_t = \zeta_0 + \int_0^t \Sigma_{\rm M}^\star (s, \zeta, \Yc^{\rm M}_s, \Zc_s) \big( \Lambda_{\rm M}^\star (s, \zeta, \Yc^{\rm M}_s, \Zc_s) \drm s + \drm W_s \big), \; t \in [0,T],
\end{align}
under $\P^\star (\Zc)$, for some $\Zc := (Z,\Gamma) \in \Vc$. Then, we can compute the value at equilibrium of the multidimensional Hamiltonian $\Hc^{\star}$ defined by \eqref{eq:hamiltonian_manager_star}, and use the dynamic \eqref{eq:dyn_zeta_star} of $\zeta$, to write the SDE satisfied by $\Yc^{\rm M}$. More precisely, starting from \eqref{eq:continuation_utility_manager}, we can state that $\Yc^{\rm M}$ is such that each component $\Yc^j$, for all $j \in \{1,\dots,m\}$, satisfies:
\begin{align}\label{eq:continuation_utility_manager_star}
    \Yc_{t}^{j}
    = Y_0^j + \int_0^t \big( c^{j,\star} \big(s, \zeta, \Yc_s^{\rm M}, \Zc_s \big)  + \Yc_s^j k^{j,\star} \big(s, \zeta, \Yc_s^{\rm M}, \Zc_s \big)  \big) \drm s
    + \int_0^t (Z^j_s)^\top \Sigma_{\rm M}^\star (s, \zeta, \Yc^{\rm M}_s, \Zc_s) \drm W_s, \; t \in [0,T].
\end{align}
Note that the column vector process $\Yc^{\rm M}$ taking values in $\R^m$ satisfies the following multidimensional SDE:
\begin{align}\label{eq:continuation_utility_manager_star_dyn}
    \drm \Yc_{t}^{\rm M}
    = \big( c^{\rm M, \star} (t, \zeta, \Yc^{\rm M}_t, \Zc_t) + \Yc_t^{\rm} \cdot k^{\rm M, \star} (t, \zeta, \Yc^{\rm M}_t, \Zc_t) \big)  \drm t
    + Z_t^\top \Sigma_{\rm M}^\star (t, \zeta, \Yc^{\rm M}_t, \Zc_t) \drm W_t, \; t \in [0,T].
\end{align}

\medskip

We are thus led to consider the following canonical space for the principal,
\begin{align*}
    \Omega^{\rm P} := \Cc \big( [0,T], \R^{hm} \big) \times \Cc \big( [0,T], \R^{m} \big) \times \Cc \big( [0,T], \R^{wd} \big) \times \V^{\rm P},
\end{align*}
where $\V^{\rm P}$ is the collection of all finite and positive Borel measures on $[0,T] \times \V$, whose projection on $[0,T]$ is the Lebesgue measure, and we consider the subset $\V^{\rm P}_0$ of all $q \in \V^{\rm P}$ such that $q(\drm s,\drm u) = \delta_{\phi_s}(\drm u) \drm s$, for some Borel function $\phi$. The canonical process is denoted by $(\zeta, \Yc^{\rm M}, W, \Pi^{\rm P})$, and associated canonical filtration is defined by $\F^{\rm P} := (\Fc^{\rm P}_t)_{t \in [0,T]}$ as usual. Then, for any $(t,\psi) \in [0,T] \times \Cc^2_b(\R^{hm} \times \R^m \times \R^{dw}, \R)$, we set
\begin{align*}
    M^{\rm P}_t (\psi) := \psi (\zeta_t, \Yc^{\rm M}_t, W_t) 
    - \int_0^t \int_{\V} \Big( &\ \widetilde \Lambda_{\rm P} (s, \zeta, \Yc^{\rm M}_s, v) \cdot \nabla \psi (\zeta_s, \Yc^{\rm M}_s, W_s) \\
    &+ \frac12  {\rm Tr} \Big[ \nabla^2 \psi (\zeta_s, \Yc^{\rm M}_s, W_s) \big( \widetilde \Sigma_{\rm P} \widetilde \Sigma_{\rm P}^\top \big) (s, \zeta, \Yc^{\rm M}_s, v) \Big] \Big) \Pi^{\rm P} (\mathrm{d}s, \mathrm{d} v),
\end{align*}
where $\widetilde \Lambda_{\rm P}$ and $\widetilde \Sigma_{\rm P}$ are respectively the drift vector and the diffusion matrix of the $(hm + m + dw)$--dimensional vector process $(\zeta, \Yc^{\rm M}, W)$, defined for all $(s, x, y) \in [0,T] \times \Cc([0,T],\R^{hm}) \times \R^m$ and $v := (z,\gamma) \in \V$, by:
\renewcommand*{\arraystretch}{1.2}
\begin{align*}
    \widetilde \Lambda_{\rm P} (s,x,y,v) := 
    \begin{pmatrix}
        (\Sigma_{\rm M}^\star \Lambda_{\rm M}^\star) (s,x,y,v) \\
        c^{\rm M, \star} (s,x,y,v) + y \cdot k^{\rm M, \star}(s,x,y,v) \\
        \mathbf{0}_{w d}
    \end{pmatrix},
    \; 
    \widetilde \Sigma_{\rm P} (s,x,y,u) :=
    \begin{pmatrix}
        \mathbf{0}_{hm,hm} & \mathbf{0}_{hm,m} & \Sigma^\star_{\rm M} (s,x,v) \\
        \mathbf{0}_{m,hm} & \mathbf{0}_{m,m} & z^\top \Sigma^\star_{\rm M} (s,x,y,v) \\
        \mathbf{0}_{wd,hm} & \mathbf{0}_{wd,m} & \mathrm{I}_{wd} &  \\
    \end{pmatrix},
\end{align*}
recalling that $\Lambda_{\rm M}$ and $\Sigma_{\rm M}$ are defined in \Cref{ass:zeta_dynamic}.
%and $z \in \M^{hm,m}$.

\medskip

Similarly as in \Cref{def:weak_formulation} for the initial control problem of the agents, or in \Cref{sss:can_space_init_managers} for the managers' problem, we fix an initial condition for $\Yc^{\rm M}$, namely $Y_0^{\rm M} \in \R^m$, and we define the subset $\Pc^{\rm P}$ of probability measures $\P$ on $(\Omega^{\rm P},\Fc^{\rm P}_T)$ satisfying the following conditions:
\begin{enumerate}[label=$(\roman*)$]
    \item $M^{\rm P}(\psi)$ is a $(\F^{\rm P}, \P)$--local martingale on $[0,T]$ for all $\psi \in \Cc^2_b(\R^{hm} \times \R^m \times \R^{dw},\R)$;
    \item $\P [(\zeta_0,\Yc^{\rm M}_0, W_0) = (\varrho_0,Y^{\rm M}_0,w_0)]=1$;
    \item $\P\big[\Pi^{\rm P} \in \V_0^{\rm P}]=1$.
\end{enumerate}
As usual, we know that for all $\P \in \Pc^{\rm P}$, $\Pi^{\rm P} (\mathrm{d}s,\mathrm{d} u) = \delta_{\Zc^{\P}_s}(\mathrm{d} u)\mathrm{d}s$ $\P$--a.s. for some $\F^{\rm P}$--predictable control process $\Zc^{\P} \in \Vc$, and the representations \eqref{eq:dyn_zeta_star} and \eqref{eq:continuation_utility_manager_star} holds respectively for $\zeta$ and $\Yc^{\rm M}$, driven by the control $\Zc^\P$.

\subsubsection{On solving the principal's problem}

Recall that the principal's problem is initially defined by \eqref{eq:pb_principal}, and then simplified by \Cref{thm:main_principal} into a standard control problem.
Thanks to the previous section, we can finally rigorously write her problem in weak formulation:
\begin{align}\label{eq:main_principal_weak}
\widebar V^{\rm P} (Y_0) = \sup_{\P \in \Pc^{\rm P}} 
\E^{\P} \Big[ \Kc^{\rm P}_{0,T} g^{\rm P} (\zeta, \xi^{\rm M}) \Big], \text{ and thus } \,
V^{\rm P}_0 = \sup_{ Y_0 \geq \rho} \widebar V^{\rm P} (Y_0).
\end{align}

Let $(t, x, y, y^{\rm P}) \in [0,T] \times \Cc([0,T],\R^{hm}) \times \R^m \times \R$, $\nabla := (\nabla^{\zeta}, \nabla^{\Yc} ) \in \R^{hm} \times \R^m$, and $\Delta \in \M^{hm+m}$ the following symmetric block matrix
\renewcommand*{\arraystretch}{1.2}
\begin{align*}
    \Delta :=
    \begin{pmatrix}
        \Delta^{\zeta} & (\Delta^{\zeta, \Yc})^\top \\
        \Delta^{\zeta, \Yc} & \Delta^{\Yc}
    \end{pmatrix},
    \text{ where } \; \Delta^{\zeta} \in \M^{hm}, \; \Delta^{\Yc} \in \M^{m}, \; \Delta^{\zeta, \Yc} \in \M^{m,hm}.
\end{align*}
Given the dynamics of the state variables $\zeta$ and $\Yc^{\rm M}$, we can define the principal's Hamiltonian as follows:
\begin{align}
    \Hc^{\rm P} (t,x,y,y^{\rm P},\nabla,\Delta) := &\ \sup_{v \in \V} h^{\rm P} (t,x,y,y^{\rm P},\nabla,\Delta, v)
\end{align}
where, in addition for $v := (z,\gamma) \in \V$,
\begin{align*}
    h^{\rm P} (t,x,y,y^{\rm P},\nabla,\Delta,v) := &- y^{\rm P} k^{\rm P}(t,x) 
    + \big( \Sigma_{\rm M}^\star \Lambda_{\rm M}^\star \big) (t, x, y, v) \cdot \nabla^{\zeta}
    + \big( c^{\rm M,\star} + y \cdot k^{\rm M,\star} \big)(t,x,y,v) \cdot \nabla^{\Yc} \nonumber \\
    &+ \dfrac{1}{2} {\rm Tr} \big[ \big( \Sigma^\star_{\rm M} (\Sigma^\star_{\rm M})^\top \big) (t, x, y, v) \Delta^{\zeta} \big]
    + \dfrac{1}{2} {\rm Tr} \big[ z^\top \big( \Sigma^\star_{\rm M} ( \Sigma^\star_{\rm M})^\top \big) (t, x, y, v) z \Delta^{\Yc} \big] \nonumber \\
    &+ {\rm Tr} \big[ \big( \Sigma^\star_{\rm M} (\Sigma^\star_{\rm M})^\top \big) (t, x, y, v) z \Delta^{\zeta, \Yc} \big]. \nonumber
\end{align*}
We are then led to consider the following HJB equation for all $(t,x,y) \in [0,T] \times \Cc([0,T],\R^{hm}) \times \R^m$,
\begin{align}\label{eq:HJB_principal}
    - \partial_t V(t,x,y) - \Hc^{\rm P} \big( t,x,y,V(t,x,y),\nabla V(t,x,y),\nabla^2 V(t,x,y) \big) = 0,
\end{align}
with terminal condition $V(T,x,y) = g^{\rm P} (x, \widebar g^{\rm M} (x, y))$. 

\medskip

Given the previous HJB equation, it is clear that the principal's problem $\widebar V^{\rm P}$ boils down to a more standard control problem. Nevertheless, it should be noticed that the previous HJB equation is path--dependent, since the Hamiltonian at time $t \in [0,T]$ depends on the paths of the variable $\zeta$ up to $t$. Therefore, in this general case, solving the principal's problem $\widebar V^{\rm P} (Y_0)$ is equivalent to solving a path--dependent partial differential equation (path--dependent PDE for short) under appropriate conditions for the solution. We refer to the works of \citeay{ekren2016viscosity} \cite{ekren2016viscosity,ekren2016viscositya} for more details on the resolution of this type of problems through the notion of viscosity solutions. Intuitively, the optimal control $\Zc \in \Vc$ will correspond to the maximiser of the Hamiltonian. The final step is then to find the optimal $Y_0 \geq \rho$ in order to maximise the previously obtain value function. 

\medskip

If we consider a Markovian framework, in the sense that the function $g^{\rm P}$ only depends on the terminal value of $\zeta$ (\textit{i.e.} $\zeta_T$) and that the Hamiltonian at time $t \in [0,T]$ only depends on the current value $\zeta_t$, then solving the principal's problem $\widebar V^{\rm P}$ boils down to solving a more standard PDE. In this case, following the line of \citeayn[Theorem 3.9]{cvitanic2018dynamic} we could write a verification result for the problem $\widebar V^{\rm P}$. In particular, assume that there exists a function $V : [0,T] \times \R^{hm} \times \R^m$, \textit{smooth enough}, solution to \textnormal{HJB} equation \eqref{eq:HJB_principal}, and a function $v^\star : [0,T] \times \R^{hm} \times \R^m \longrightarrow \V$ satisfying, for all $(t,x,y) \in [0,T] \times \R^{hm} \times \R^m$,
\begin{align*}
    \Hc^{\rm P} \big(t,x,y,V(t,x,y), \nabla V(t,x,y), \nabla^2 V(t,x,y) \big) = h^{\rm P} \big( t,x,y,V(t,x,y), \nabla V(t,x,y), \nabla^2 V(t,x,y), v^\star (t,x,y) \big).
\end{align*}
Then, intuitively and under \textit{additional appropriate condition} on this two functions, we should obtain that $\widebar V^{\rm P} (Y_0) = V (0, \zeta_0, Y_0^{\rm M})$, and that the process $\Zc^\star$ defined for all $t \in [0,T]$ by $\Zc_t := v^\star (t, \zeta, \Yc_t^{\rm M})$ is an optimal control for the principal. As mentioned above, the final step is to optimise on the initial value of the workers' continuation utility.

\medskip

Finally, the main aspect to notice concerning the principal's problem is that, thanks to the optimal form of contracts for managers, and in particular by \Cref{thm:main_principal}, the dimension of this problem does not explode. More precisely, if the principal supervises $m$ managers, then her problem has $2m$ state variables, potentially multidimensional but of dimension independent of the number of managers. Indeed, on the one hand, each manager $j$ communicates his results through a variable $\zeta^j$, of fixed dimension $h$, which constitutes a state variable for the principal. On the other hand, thanks to the elegant reasoning of \citeayn{sannikov2008continuous}, later developed by \citeayn{cvitanic2018dynamic}, considering in addition the one--dimensional continuation utility of the said manager is sufficient to solve the principal's problem. Therefore, throughout this paper, we have shown that the method used to solve a contracting problem between one principal and one agent can be extended to a hierarchical structure and preserves the same main features, namely that the principal's problem boils down to a more classical control problem with two state variables per agent under her direct supervision.

\section{Conclusion}\label{sec:conclusion}

\textcursive{Summary.}\vspace{-0.3em} In the first part of this paper, we introduce and solve the continuous--time version of \citeauthor{sung2015pay}'s model developed in \cite{sung2015pay}. This opening example highlights the differences between the one--period model and its continuous--time equivalent, in particular concerning the form of the contracts. More precisely, when studying the continuous--time model, we are allowed to consider an extended class of contracts for the managers, indexed in particular on the quadratic variation of the net benefit $\zeta$ observed by the principal. Therefore, in order to rigorously study a continuous--time hierarchy problem, it is not possible to consider the associated discrete--time model with linear contracts, which justifies and even requires the use of the theory of 2BSDEs to deal with problems of moral hazard within a hierarchy. The second part of this paper focuses on a more general model and provides a systematic method to solve any hierarchy problems of this sort, method that can be extended in a straightforward way to a larger scale hierarchy.

\medskip

\textcursive{Extensions.} As mentioned above, the method we developed throughout the theoretical part of this paper can be applied to any hierarchical structure. Moreover, we have assumed that the agents (at the bottom of the hierarchy) do not interact with each other, in the sense that each of them controls his own output and that these outputs are uncorrelated. We could actually assume that they interact, in the same way as the managers finally do. Furthermore, we could also consider, instead of a finite number agents, a continuum of workers with mean--field interaction. There is no reason why this issue could not be addressed by applying the results of \citeayn{elie2019mean} in our framework.
However, several assumptions are necessary to complete our study, notably on the shape of the dynamics of the state variables. Even if they are satisfied in the most common and interesting examples, one might want to weaken those assumptions. Most of these hypotheses prevent the case where a principal does not observe one of the state variables of her agent's problem. Therefore, to hope for an extension of our model, it would be necessary to solve this issue, which is still scarcely addressed in the literature in continuous time (see the work of \citeayn{huang2017optimal} for a particular example).

\section*{Acknowledgements}
The author gratefully acknowledges the support of the ANR project PACMAN ANR--16--CE05--0027, the FACE Foundation -- Thomas Jefferson Fund, as well as the support of the University Gustave Eiffel through a mobility grant. Finally, the author thanks Dylan Possamaï (IEOR Department, Columbia University) for his relevant ideas and useful advice all along the conception of this paper. Nevertheless, all potential errors and opinions expressed in this paper are sole responsibility of the author.

{\small
\bibliography{bibliographyPAH}}

\newpage

\begin{appendices}

\crefalias{section}{appendix}
\crefalias{subsection}{appendix}

\renewcommand{\appendixpagename}{Appendix.}%: for online publication}

\appendixpage

In \Cref{app:sungs_model}, we regroup additional results and proofs that concern the continuous--time version of \citeauthor{sung2015pay}'s model, developed in \Cref{sec:sungmodel}, and its extensions mentioned in \Cref{sec:extensions}. \Cref{app:intuition_extension} provides intuitions for the form of contracts in the general model (\Cref{app:intuition_markovian}), obstacles to the consideration of specific extensions (\Cref{app:extension_dynamic}), as well as simple but interesting examples satisfying the major hypothesis necessary for our study (\Cref{app:assumption_zeta}). Finally, \Cref{sec:2BSDEs} presents the theory of 2BSDE and the associated results relevant to our framework. This section also contains the proofs of the main propositions and theorems established in the paper.

\section{Further comments on Sung's model}\label{app:sungs_model}

Throughout \Cref{sec:sungmodel,sec:extensions}, we compare our results in continuous time with those of the one--period model detailed in \cite{sung2015pay}, but also with the results that would be obtained in a direct contracting framework, \textit{i.e.}, in the case where the principal contracts directly with the agents, without the intermediary of a manager. In the latter case, we are faced with a more traditional principal--agents problem (with a finite number of agents), which can be solved in a straightforward way, by extending the results obtained by \citeayn{holmstrom1987aggregation} to a multitude of agents. The following lemma reports the optimal efforts of the agents as well as the utility of the principal in this case. A similar result is mentioned in \cite{sung2015pay}, but we should refer to the work of \citeayn{koo2008optimal} for a rigorous result, or to the more general model of \citeayn{elie2019contracting}.
\begin{lemma}\label{lem:DC_case}
With direct contracting, the optimal efforts of the workers are given by
\begin{align}\label{eq:effort_DC}
    \alpha^{i, \rm DC} := k^i Z^{i, \rm DC}, \; \text{ where } \; Z^{i, \rm DC} := \dfrac{k^i}{\widetilde R^i} \; \text{ and } \; 
    \widetilde R^i := k^i + R^i |\sigma^i |^2.
\end{align}
Moreover, if $X_0 = 0$, then the value $V^{\rm DC}$ for the principal is equal to:
\begin{align}\label{eq:value_DC}
    V^{\rm DC} := \dfrac{1}{2} \sum_{i=0}^n \dfrac{|k^i|^2}{\widetilde R^i}.
\end{align}
\end{lemma}

\subsection{Proofs related to the initial model}\label{app:ss:proof_sung}

This section contains the proofs of the main results related to the continuous--time version of \citeauthor{sung2015pay}'s model, established in \Cref{sec:sungmodel}.

\begin{proof}[\Cref{prop:manager_sung}]
As in the statement of the proposition, let $(Z,\Gamma) \in \Vc^{\rm b}$. According to \Cref{prop:contract_vol}, the optimal form of contract is given by \eqref{eq:contract_vol}.
% , follows:
% \begin{align*}
%     \xi^{\rm b}_T =  
%     \xi^{\rm b}_0
%     - \int_0^T \Hc^{\rm b} (Z_s, \Gamma_s) \drm s
%     + \int_0^T  Z_s \drm \zeta^{\rm b}_s
%     + \dfrac{1}{2} \int_0^T \big( \Gamma_s + R^{0} |Z_s|^2 \big) \drm \langle \zeta^{\rm b} \rangle_s, \; \xi^{\rm b}_0 \in \R,
% \end{align*}
Recall that the Hamiltonian $\Hc^{\rm b}$ is defined by \eqref{eq:hamiltonian_manager_sung}, and that the dynamic of $\zeta^{\rm b}$ under optimal effort of the agents is given by \eqref{eq:dyn_zeta_b}. Replacing in the value function of the manager, defined by \eqref{eq:manager_pb_sung}, we obtain:
\begin{align*}
    J_0^{0} \big( \xi^{\rm b}, \alpha^{0}, (Z^i)_{i=1}^n \big) = & - \erm^{- R^{0} \xi^{\rm b}_0}
    \exp \bigg(- R^{0} \bigg(
    \int_0^T \Big( h^{\rm b} \big( Z_s, \Gamma_s, (Z_s^i)_{i=1}^n, \alpha^0_s \big) - \Hc^{\rm b} ( Z_s, \Gamma_s) \Big) \drm s \bigg) \\
    & \times  \E^{\P^{0}} \bigg[ \Ec \bigg(- R^{0}
    \int_0^\cdot Z_s \bigg( \sigma^{0} \drm W_s^{0} 
    + \sum_{i=1}^n \sigma^i \big( 1 - Z^i_s \big) \drm W_s^i \bigg) \bigg)_T \bigg],
\end{align*}
where $\Ec$ denotes for the Doléans--Dade exponential and the function $h^{\rm b}$ is defined by:
\begin{align*}
    h^{\rm b} \big( z, \gamma, (z^i)_{i=1}^n, a \big) := 
    \dfrac{1}{2} \gamma | \sigma^0 |^2
    + z a - c^{0} (a)
    + \sum_{i=1}^n \bigg( 
    z \Big( k^i z^{i} - \dfrac{\widetilde R^i}{2} |z^{i}|^2 \Big)  
    + \dfrac{1}{2} \gamma |\sigma^i|^2 |1 - z^i|^2 \bigg).
\end{align*}
Under the appropriate integrability conditions on $Z$, the previously considered Doleans--Dade exponential is a martingale, which implies:
\begin{align*}
    J_0^{0} \big( \xi^{\rm b}, \alpha^{0}, (Z^i)_{i=1}^n \big) = - \erm^{- R^{0} \xi^{\rm b}_0}
    \exp \bigg(- R^{0} \bigg(
    \int_0^T \Big( h^{\rm b} \big( Z_s, \Gamma_s, (Z_s^i)_{i=1}^n, \alpha^0_s \big) - \Hc^{\rm b} ( Z_s, \Gamma_s) \Big) \drm s \bigg)  
    \leq - \erm^{- R^{0} \xi^{\rm b}_0}.
\end{align*}
The inequality in the previous equation stems from the fact that $\Hc^{\rm b} \geq h^{\rm b}$ by definition. In particular, the equality is attained when $(Z^i)_{i=1}^n$ and $\alpha^0$ maximise the function $h^{\rm b}$. More precisely, we have:
\begin{align*}
    k^0 z = \argmax_{a \in \R} \big\{ a z - c^{0} (a) \big\}, \; \textnormal{ and } \; 
    z^{i, \rm b} (z,\gamma) = \argmax_{z^i \in \R} \bigg\{ 
    z \Big(  k^i z^{i} - \frac{\widetilde R^i}{2} |z^{i}|^2 \Big)
    + \dfrac{1}{2} \gamma |\sigma^i|^2 |1 - z^i|^2 \bigg\},
\end{align*}
where the function $z^{i, \rm b}$ is defined in the statement of the proposition by \eqref{eq:zib_star}. 
Therefore, the optimal controls of the manager are given for all $t \in [0,T]$ by $\alpha^{\rm b}_t := k^0 Z_t$ and $Z_t^{i,\rm b} := z^{i,\rm b} (Z_t, \Gamma_t)$.

\medskip

To complete the proof, it remains to compute the dynamics of $\zeta^{\rm b}$ and $\xi^{\rm b}$ under the optimal efforts of the manager. In the one hand, by plugging the optimal effort in the dynamic of $\zeta^{\rm b}$ defined by \eqref{eq:dyn_zeta_b}, it is straightforward to obtain the desired result:
\begin{align*}
    \drm \zeta^{\rm b}_t = \bigg( k^0 Z_t + \sum_{i=1}^n \Big( k^i z^{i, \rm b} (Z_t, \Gamma_t) - \frac{\widetilde R^i}{2} \big| z^{i, \rm b} (Z_t, \Gamma_t) \big|^2 \Big) \bigg) \drm t + \sigma^{0} \drm W_t^{0} + \sum_{i=1}^n \sigma^i \big( 1 - z^{i, \rm b} (Z_t, \Gamma_t) \big) \drm W_t^i.
\end{align*}
Then, using both the previous dynamic for $\zeta^{\rm b}$ and \eqref{eq:contract_vol}, we obtain:
\begin{align*}
    \drm \xi^{\rm b}_t = &\ \Big( h^{\rm b} \big( Z_t, \Gamma_t, (Z^{i, \rm b}_t)_{i=1}^n, \alpha^{\rm b}_t \big) - \Hc^{\rm b} (Z_t, \Gamma_t)  \Big) \drm t
    + c^0 \big(\alpha_t^b \big) \drm t \\
    &+ \dfrac{1}{2} R^{0} Z_t^2 \bigg( |\sigma^{0}|^2 + \sum_{i=1}^n |\sigma^i|^2 \big| 1 - Z^{i, \rm b}_t \big|^2 \bigg) \drm t
    + Z_t \bigg( \sigma^{0} \drm W_t^{0} + \sum_{i=1}^n \sigma^i \big( 1 - Z^{i, \rm b}_t \big) \drm W_t^i \bigg).
\end{align*}
By definition of $h^{\rm b}$ and the optimal controls of the manager, the difference $h^{\rm b}-\Hc^{\rm b}$ is equal to zero. Noticing that the cost $c^0 (\alpha^{\rm b}_t)$ is equal to $k^0 Z_t^2 /2$, we obtain the desired dynamic for $\xi^{\rm b}$, which concludes the proof.
\end{proof}

\begin{proof}[\Cref{prop:principal_sung}]
    Since the principal is risk--neutral, her reward function can be computed as follows:
    \begin{align*}
        J_0^{\rm P} \big(  \xi^{\rm b} \big) = &\ \zeta^{\rm b}_0 - \xi^{\rm b}_0
        + \E^{\P^{\rm b}} \bigg[ \int_0^T \sigma^{0} (1 - Z_t) \drm W_t^{0} \bigg] + \sum_{i=1}^n \E^{\P^{\rm b}} \bigg[ \int_0^T \sigma^i (1 - Z_t) \big( 1 - z^{i, \rm b} (Z_t, \Gamma_t) \big) \drm W_t^i \bigg] \\
        &+ \E^{\P^{\rm b}} \bigg[ \int_0^T \bigg( k^{0} Z_t
        - \dfrac{1}{2} \widetilde R^0 Z_t^2
        + \sum_{i=1}^n h^{i,\rm b} (Z_t,\Gamma_t) \bigg) \drm t \bigg],
    \end{align*}
    for $(Z, \Gamma) \in \Vc^{\rm b}$ and recalling that $\widetilde R^0 := k^0 + R^{0} |\sigma^{0}|^2$.
    First, the expectation of the two stochastic integrals is equal to zero. Then, maximising the reward function is equivalent to maximising inside the last expectation and the integral with respect to time, which leads to the following optimisation problem:
    \begin{align}\label{eq:sup_sung_case}
        \sup_{(z, \gamma) \in \V^{\rm b}} \bigg\{ 
        k^{0} z 
        - \dfrac12 \widetilde R^0 z^2
        + \sum_{i=1}^n h^{i,\rm b} (z,\gamma)
        \bigg\}.
    \end{align}
Since $h^{i,\rm b}$ is strictly concave in $\gamma$, the first--order condition (FOC) is sufficient to obtain the optimal $\gamma$:
\begin{align*}
    0 = 
    \sum_{i=1}^n \Big( k^i \partial_{\gamma} z^{i, \rm b} (z,\gamma)
    - \widetilde R^i z^{i, \rm b} (z,\gamma) \partial_{\gamma} z^{i, \rm b} (z,\gamma)
    + z^2 R^{0} |\sigma^i|^2 \partial_{\gamma} z^{i, \rm b} (z,\gamma) \big( 1 - z^{i, \rm b} (z,\gamma) \big) \Big).
\end{align*}
By computing the derivatives of $z^{i, \rm b} (z,\gamma)$ with respect to $\gamma$ for all $i \in \{1,\dots,n\}$, namely
\begin{align*}
    \partial_{\gamma} z^{i, \rm b} (z,\gamma) =  \dfrac{- R^i |\sigma^i|^4 z }{ \big| \widetilde R^i z - |\sigma^i|^2 \gamma \big|^2},
\end{align*}
one obtain the following FOC
\begin{align*}
    0 = 
    z \big( \gamma + z^3 R^{0} \big)
    \sum_{i=1}^n \dfrac{- |R^i|^2 |\sigma^i|^8}{ \big( \widetilde R^i z - |\sigma^i|^2 \gamma \big)^3},
\end{align*}
and therefore it is optimal to set $\Gamma_t^{\rm b} := -R^0 \big(Z^{\rm b}_t \big)^3$, for all $t \in [0,T]$. Finally, since the optimal control of the principal $(Z^{\rm b}, \Gamma^{\rm b})$ must be in $\Vc^{\rm b}$, the process $Z^{\rm b}$ must be positive. Unfortunately, we cannot obtain an explicit value for $Z^{\rm b}$, but we can estimate it thanks to a simple numerical optimisation, as the maximiser of \eqref{eq:sup_z_sung_case}. Noticing that the optimisation problem \eqref{eq:sup_z_sung_case} does not depend on time nor on the state variable, its maximiser $z^{\rm b}$ is a positive constant. 

\medskip

The point $(ii)$ of the proposition is a simple computation of the contract under the optimal payment rates chosen by the principal:
\begin{align*}%\label{eq:contract_vol_optimal}
    \xi^{\rm b}_T = &\ 
    \xi^{\rm b}_0
    - \int_0^T \Hc^{\rm b} (z^{\rm b}, - R^0 (z^{\rm b})^3) \drm t
    + z^{\rm b} \int_0^T \drm \zeta^{\rm b}_t
    + \dfrac{1}{2} R^0 |z^{\rm b}|^2 \big(1 - z^{\rm b} \big) \int_0^T \drm \langle \zeta^{\rm b} \rangle_t \\
    % = &\ \xi^{0, \star}_0
    % - \Hc^{0} (z^{0,\star}, - R^0 (z^{0,\star})^3) T
    % + z^{0,\star} (\zeta_T - \zeta_0)
    % + \dfrac{1}{2} R^0 \big| z^{0,\star} \big|^2 \big(1 - z^{0,\star} \big) \langle \zeta \rangle_T \\
    = &\ \xi^{\rm b}_0
    - \Hc^{\rm b} (z^{\rm b}, - R^0 (z^{\rm b})^3) T
    + z^{\rm b} (\zeta^{\rm b}_T - \zeta^{\rm b}_0)
    + \dfrac{T}{2} R^0 |z^{\rm b}|^2 (1 - z^{\rm b}) \langle \zeta^{\rm b} \rangle_T.
\end{align*}
The same type of computation leads to the point $(iii)$. Moreover, since the reservation utility of the agents and the managers are equal to $-1$, we obtain that the optimal choice of $\xi^{\rm b}_0$ is zero. To prove the last point of the proposition, it is then sufficient to compute the utility of the principal for the optimal contracts.
\end{proof}

\subsection{Proofs related to the extensions of the initial model}\label{app:ss:proof_extension}

This section contains the proofs of the main results established in \Cref{sec:extensions}, which concern the extensions of the initial model developed in \Cref{sec:sungmodel}.

\begin{proof}[\Cref{prop:cas_pc0_identical}]
    As mentioned in the proposition, we assume that all agents are identical, in the sense that for all $i \in \{1, \dots, n\}$, $k^i = k$, $R^i = R$, $\sigma^i = \sigma$, for some $(k, R, \sigma) \in (\R_+^\star)^3$. In this case, the optimisation problem \eqref{eq:sup_2statevariable} becomes:
    \begin{align*}
        \sup_{(z, \gamma) \in \V^{\rm pc}} \Big\{ 
        k^0 z^{1} - \dfrac12 \widetilde R^0 |z^{1}|^2 + n h^{\cdot, \rm pc} (z, \gamma)
        \Big\},
    \end{align*}
    where $h^{\cdot, \rm pc} := h^{i, \rm pc}$, defined by \eqref{eq:hipc} for all $i \in \{1, \dots, n\}$. By \Cref{lem:DC_case}, to achieve the optimal effort of the direct contracting case for all workers, we should have:
    \begin{align*}
        z^{1,\star} = \dfrac{k^0}{\widetilde R^0} \; \text{ and, for all } \; i \in \{1,\dots,n\}, \; z^{i, \rm pc} (z, \gamma) = z^{\cdot,\rm pc,\star} := \dfrac{k}{\widetilde R}, \text{ where } \; \widetilde R := k + R |\sigma|^2.
    \end{align*}
    Indeed, if $z^1 = z^{1,\star}$ and $z^{\rm \cdot,pc} = z^{\rm \cdot,pc,\star}$, then this imply the optimal effort of the DC case for respectively the manager and the agents.
    Moreover, the value of supremum has to be equal to the value of the principal in the DC case, given by 
    \begin{align*}
        V^{\rm DC} = \dfrac{1}{2} \dfrac{|k^0|^2}{\widetilde R^0} + \dfrac{1}{2} \dfrac{n |k|^2}{\widetilde R}.
    \end{align*}
    Using the expected values of $z^1$ and $z^{\rm pc}$, we obtain the following value for the supremum:
    \begin{align*}
        \dfrac12 \dfrac{|k^0|^2}{\widetilde R^0}
        + \frac{n}{2} \dfrac{|k|^2}{\widetilde R}
        - \dfrac{n}{2} R^{0} |\sigma|^2 
        \inf_{(z, \gamma) \in \V^{\rm pc}} \bigg\{ 
        \bigg| \dfrac{k^0}{\widetilde R^0} + z^{2} \dfrac{k}{\widetilde R} \bigg|^2
        \bigg\}.
    \end{align*}
    The infimum in the previous equation is attained for $z^{2,\star} = - k^0 \widetilde R/ (k \widetilde R^0)$, and we thus obtain the value of the DC case.
    It then remains to solve, for $(\gamma^{12}, \gamma^{22}) \in \R^2$ such that $(z,\gamma) \in \V^{\rm pc}$,
    \begin{align*}
        - \dfrac{k^i z^{1,\star} + |\sigma^i|^2 \gamma^{12} }{ \widetilde R^i z^{2,\star} + |\sigma^i|^2 \gamma^{22}} = \dfrac{k}{\widetilde R},
    \end{align*}
    to ensure that $z^{\cdot, \rm pc}$ has the requested value. This equation is equivalent to:
    \begin{align*}
        \gamma^{12} = \frac{k^0}{\widetilde R^0} R
        - \dfrac{k}{\widetilde R}  \gamma^{22}, \; \text{ for } \; \gamma^{22} < - \frac{k^0 | \widetilde R |^2}{k |\sigma|^2 \widetilde R^0}.
    \end{align*}
    In conclusion, for any $\gamma^{22}$ satisfying the previous inequality, by setting
    \begin{align*}
        z^{1} := \dfrac{k^0}{\widetilde R^0}, \; z^{2} := - \dfrac{k^0 \widetilde R}{k \widetilde R^0} \; \text{ and } \; \gamma^{12} := \dfrac{k^0}{\widetilde R^0} R - \dfrac{k}{\widetilde R} \gamma^{22},
    \end{align*}
    the value function of the principal and the efforts of both the agents and the manager are equal to the those in the DC case.
\end{proof}

\begin{proof}[\Cref{prop:other_reporting}]
$(i)$ To prove the first point of the proposition, we follow the reasoning of the previous subsection by considering contract for the manager of the form \eqref{eq:contract_vol_2D}, indexed on the $2$--dimensional reported variable $\zeta^{\rm b,0}$ by an $\R^2 \times \R^{2 \times 2}$--valued process $(Z,\Gamma)$. We find that the optimal effort of the manager and the optimal payment rate for the $i$--th agents are respectively given for all $t \in [0,T]$ by $\alpha_t^{\rm b, 0} = k^0 Z^{2}_t$ and $Z^{i, \rm b, 0}_t = z^{i, \rm b} (Z_t^{1},\Gamma^{11}_t)$, where $(Z^1, \Gamma^{11}) \in \Vc^{\rm b}$. The principal's problem is then equivalent to:
\begin{align*}
    \sup_{z^2 \in \R} \bigg\{ 
    k^0 z^{2} - \dfrac12 \widetilde R^0 |z^{2}|^2 \bigg\}
    + \sup_{(z^{1}, \gamma^{11}) \in \V^{\rm b}} \sum_{i=1}^n h^{i,\rm b} (z^1, \gamma^{11}),
\end{align*}
recalling that the function $h^{i,\rm b}$ is defined for all $i \in \{1, \dots, n\}$ by \eqref{eq:hib}. 
Therefore, by choosing the constant processes $Z_t^{2} := Z^{0, {\rm DC}}$, defined by \eqref{eq:effort_DC}, and $\Gamma_t^{11} := 0$, for all $t \in [0,T]$, all workers do the optimal effort of the DC case and the supremum becomes:
\begin{align*}
    \dfrac12 \dfrac{|k^0|^2}{\widetilde R^0}
    + \frac{1}{2} \sum_{i=1}^n \dfrac{|k^i|^2}{\widetilde R^i}
    - \frac{1}{2} R^{0} \inf_{z^{1} > 0}
    \sum_{i=1}^n
    |\sigma^i|^2 
    |z^{1}|^2 \bigg| 1 - \dfrac{|k^i|^2}{\widetilde R^i} \bigg|^2.
\end{align*}
The infimum is equal to that in the DC case and can be achieved by considering a sequence $(z^{1,n})_{n \geq 0}$ converging to zero. Therefore, the principal can construct a sequence of contracts, namely $\xi^{{\rm b},0,n}$, with for example $Z^{1,n} := 1/n$, allowing the workers to apply optimal efforts of the DC case and such that her utility converges to $V^{\rm DC}$.

\medskip

$(ii)$ To prove the second point, we consider contracts of the form \eqref{eq:contract_vol_2D}, but indexed on the $3$--dimensional reported variable $\zeta^{\rm pc,0}$, by an $\R^3 \times \R^{3 \times 3}$--valued process $(Z,\Gamma)$. The fact that the principal's problem is also degenerating into the DC case can be shown in the same way as we prove the first point. Moreover, in the case of identical agents, following the same reasoning as in \Cref{sss:2statevariables}, the optimal contract which allow to attain the DC case exists. Indeed, in this case, the principal's problem is equivalent to:
\begin{align*}
    \sup_{(z, \gamma, z^3) \in \V^{\rm pc} \times \R} \bigg\{ 
    k^0 z^{3} 
    - \dfrac12 \widetilde R^0 |z^{3}|^2
    + \sum_{i=1}^n h^{i, \rm pc} (z, \gamma)
    \bigg\},
\end{align*}
recalling that, for all $i \in \{1, \dots, n\}$, $h^{i, \rm pc}$ is defined by \eqref{eq:hipc}.
The optimal payment rate $z^{3}$ is clearly given by the ratio $k^0 / \widetilde R^0$. Moreover, if the agents are identical, in the sense that for all $i \in \{1, \dots, n\}$, $k^i = k$, $R^i = R$, $\sigma^i = \sigma$, for some $(k, R, \sigma) \in (\R_+^\star)^3$,
the previous optimisation problem is equal to
\begin{align*}
    \dfrac12 \dfrac{|k^0|^2}{\widetilde R^0} 
    + n \sup_{(z, \gamma) \in \V^{\rm pc}} \bigg\{ 
    k z^{\cdot,\rm pc}  (z, \gamma)
    - \frac{1}{2} \widetilde R \big| z^{\cdot,\rm pc}  (z, \gamma) \big|^2
    - \dfrac{1}{2} R^{0} |\sigma|^2 
    \big| z^{1} + z^{2} z^{\cdot, pc}(z, \gamma) \big|^2
    \bigg\},
\end{align*}
where $\widetilde R := k + R |\sigma|^2$ and $z^{\cdot,\rm pc}  := z^{i, \rm pc}$ for all $i \in \{1, \dots, n\}$.
By setting
\begin{align*}
    z^{1} = - \dfrac{k}{\widetilde R} z^{2} \; \text{ and } \; \gamma^{12} = - \dfrac{k}{\widetilde R} ( R z^{2} + \gamma^{22}) \; \text{ for all } \; (z^{2}, \gamma^{22}) \in \R^2,
\end{align*}
one obtain $z^{\cdot, \rm pc} (z, \gamma) = k/\widetilde R$ and $z^{1} + z^{2} z^{\cdot,\rm pc} (z, \gamma) = 0$,
implying that the efforts of the workers and the value of the supremum are equal to those in the DC case. 
%This result is not surprising since this was already the case in the example in \Cref{sss:2statevariables}, \Cref{prop:cas_pc0_identical}, where the principal had less controls.
\end{proof}

\section{Intuition and extensions}\label{app:intuition_extension}

\subsection{Intuition in the Markovian framework}\label{app:intuition_markovian}

One of the cornerstones of the approach to continuous--time moral hazard problems, pioneered by \citeayn{sannikov2008continuous}, and studied in full generality by \citeayn{cvitanic2018dynamic}, is to obtain an appropriate probabilistic representation for incentive--compatible contracts. 
Intuitively, we expect that the continuation utility $Y^{j,i}$ of the $(j,i)$--th agent, given a contract $\xi \in \Cc^{j,i}$, in particular $\G^j$--measurable, and actions of other workers subsumed by the collections of outputs $X^j$ and $\widebar X^{-j}$, may be written as follows:
\begin{align*}
    Y_t^{j,i} = v^{j,i} \big( t, X^j_{\cdot \wedge t}, \widebar X_{\cdot \wedge t}^{-j} \big),
\end{align*}
that is, the process $Y^{j,i}$ at time $t$ depends on time $t$ and on the path history of $X^j$ and $\widebar X^{-j}$. Recall that the effort of other agents, $\nu^{-(j,i)}$, are fixed through the probability $\P^{-(j,i)}$, as well as the efforts $\nu^{\rm M}$ of the managers, fixed through $\P^{\rm M}$. We thus consider $\P \in \Pc^{j,i} (\P^{-(j,i)}, \P^{\rm M})$.

\medskip

To intuit the form of contracts used in \Cref{sss:nash_equilibrium_agent}, the focus here is on the Markovian case: we assume that the continuation utility $Y^{j,i}$ can be written at each time $t \in [0,T]$ as a function of $X^j_t$ and $\widebar X^{-j}_t$, and thus not on their paths up to $t$. In particular, if this value function is smooth enough, we can apply It\=o's formula to the process $(\Kc^{j,i,\P}_{0,t} Y_{t}^{j,i})_{t \in [0,T]}$ under $\P$:
\begin{align*}
    \Kc^{j,i,\P}_{0,t} Y_{t}^{j,i} - \Kc^{j,i,\P}_{0,s} Y_{s}^{j,i} = 
    & - \int_s^t k^{j,i} \big( r, X_r^{j,i}, \nu_r^{j,i} \big) \Kc^{j,i,\P}_{0,r} Y_{r}^{j,i} \drm r
    + \int_s^t \Kc^{j,i,\P}_{0,r} \drm Y_{r}^{j,i} \\
    = & \int_s^t \Kc^{j,i,\P}_{0,r} \Big( \partial_t v^{j,i} - k^{j,i} \big( r, X_r^{j,i}, \nu_r^{j,i} \big) v^{j,i} \Big) \drm r
    + \int_s^t \Kc^{j,i,\P}_{0,r} \nabla_{x} v^{j,i}  \cdot \drm X_t^j 
    + \int_s^t \Kc^{j,i,\P}_{0,r} \nabla_{\widebar x} v^{j,i}  \cdot \drm \widebar X_t^{-j} \\
    &+ \dfrac{1}{2} \int_s^t \Kc^{j,i,\P}_{0,r} \Big( 
    {\rm Tr} \Big[ \nabla^2_x v^{j,i}  \drm \langle X^j \rangle_t \Big]
    + {\rm Tr} \Big[ \nabla^2_{\widebar x} v^{j,i} \drm \langle \widebar X^{-j} \rangle_t \Big]
    + 2 {\rm Tr} \Big[ \nabla^2_{x \widebar x} v^{j,i} \drm \langle X^j, \widebar X^{-j} \rangle \Big] \Big).
\end{align*}
Noticing in particular that the outputs are not correlated, we can rewrite the previous form as follows:
\begin{align*}
    \Kc^{j,i,\P}_{0,t} Y_{t}^{j,i} - \Kc^{j,i,\P}_{0,s} Y_{s}^{j,i}
    = \int_s^t \Kc^{j,i,\P}_{0,r} \bigg(
    & \big( \partial_t v^{j,i} - k^{j,i} \big( r, X_r^{j,i}, \nu_r^{j,i} \big) v^{j,i} \big) \drm r 
    + \sum_{\ell = 0}^{n_j}
    \Big( \partial_{x^{j,\ell}} v^{j,i} \drm X_r^{j,\ell}
    + \dfrac{1}{2} \partial^2_{x^{j,\ell}} v^{j,i} \drm \langle X^{j,\ell} \rangle_r \Big) \\
    &+ \sum_{k = 0, \, k \neq j}^{m} \Big( 
    \partial_{\widebar x^k} v^{j,i} \drm \widebar X_r^{k}
    + \dfrac{1}{2} \partial^2_{\widebar x^k} v^{j,i} \drm \langle \widebar X^{k} \rangle_r \Big) \bigg).
\end{align*}
Recall that, for any $t \in [0,T]$, the output $X_t^{j,i}$ of the $(j,i)$--th worker is given by \eqref{eq:XW},
% \begin{align*}
%     \drm X_t^{j,i}= \sigma^{j,i} \big(t, \beta_t^{j,i,\P} \big) \cdot \Big[ \lambda^{j,i} \big(t, \alpha_t^{j,i,\P} \big) \mathrm{d} t +  \mathrm{d} W^{j,i}_t \Big],\; t\in[0,T],\; \P \textnormal{--a.s.}
% \end{align*}
which implies
\begin{align*}
    \drm \widebar X^{k}_t 
    %= \drm \sum_{\ell = 0}^{n_{k}} X_t^{k,\ell} 
    = \sum_{\ell = 0}^{n_{k}} \Lambda^{k,\ell} \big(t, \nu_t^{k, \ell} \big) \mathrm{d} t
    + \sum_{\ell = 0}^{n_{k}} \sigma^{k,\ell} \big(t, \beta_t^{k,\ell} \big) \cdot \mathrm{d} W^{k,\ell}_t, \; \text{for all} \; k \in \{0, \dots, m\}.
\end{align*}
Moreover, the relevant quadratic variations are given by:
\begin{align*}
    \drm \langle X^{j,\ell} \rangle_t &= \big\| \sigma^{j,\ell} \big(t, \beta_t^{j,\ell} \big) \big\|^2 \drm t, \; \text{for all} \; \ell \in \{0, \dots, n_j \}, \; \text{and} \; 
    \drm \langle \widebar X^{k} \rangle_t = \sum_{\ell = 0}^{n_{k}} \big\| \sigma^{k,\ell} \big(t, \beta_t^{k,\ell} \big) \big\|^2 \drm t, \; \text{for all} \; k \in \{0, \dots, m \},
\end{align*}
which implies the following:
\begin{align*}
    \Kc^{j,i,\P}_{0,t} Y_{t}^{j,i} - \Kc^{j,i,\P}_{0,s} Y_{s}^{j,i}
    = & \int_s^t \Kc^{j,i,\P}_{0,r} \Big( \partial_t v^{j,i} 
    + \widetilde h^{j,i} \big(r, X_r^{j,i}, v^{j,i}, \nabla v^{j,i}, \nabla^2 v^{j,i}, \nu_r^{-(j,i)}, \nu_r^M, \nu^{j,i}_r \big)
    + c^{j,i} \big(r, X_r^{j,i}, \nu_r^{j,i} \big)
    \Big) \drm r \\
    &+ \int_s^t  \Kc^{j,i,\P}_{0,r} \bigg( 
    \sum_{\ell = 0}^{n_j}  \partial_{x^{j,\ell}} v^{j,i} \sigma^{j,\ell} \big(r, \beta_r^{j,\ell} \big) \cdot \mathrm{d} W^{j,\ell}_r
    + \sum_{k = 0, k \neq j}^{m} \partial_{\widebar x^k} v^{j,i}   \sum_{\ell = 0}^{n_k} \sigma^{k,\ell} \big(r, \beta_r^{k,\ell} \big) \cdot \mathrm{d} W^{k,\ell}_r \bigg),
\end{align*}
where for $t \in [0,T]$, $(x,y) \in \R^2$, $(z,\widetilde z) \in \R^{n_j+1} \times \R^{m-1}$, $(\gamma, \widetilde \gamma) \in \R^{n_j+1} \times \R^{m-1}$, $(\nu^{-(j,i)}, \nu^{\rm M}) \in \Uc^{-(j,i)} \times \Uc^{\rm M}$, and $u \in U^{j,i}$:
\begin{align*}
    \widetilde h^{j,i} \big(t, x, y, (z,\widetilde z), (\gamma, \widetilde \gamma), \nu_t^{-(j,i)}, \nu_t^M, u \big) := & 
    - c^{j,i} (t, x, u)
    - k^{j,i} (t, x, u) y
    + z^{i}  \Lambda^{j,i} (t, u)
    + \dfrac{1}{2} \gamma^{i} \big\| \sigma^{j,i} (t, b) \big\|^2 \\
    &+ \sum_{\ell = 0, \ell \neq i}^{n_j} z^{\ell} \Lambda^{j,\ell} \big(t, \nu_t^{j, \ell} \big)
    + \dfrac{1}{2} \sum_{\ell = 0, \ell \neq i}^{n_j} \gamma^{\ell} \big\| \sigma^{j,\ell} \big(t, \beta_t^{j,\ell} \big) \big\|^2 \\
    &+ \widetilde z \cdot \bigg( \sum_{\ell = 0}^{n_k} \Lambda^{k,\ell} \big(t, \nu_t^{k,\ell} \big) \bigg)_{k = 1, \, k \neq j}^{m}
    + \dfrac{1}{2} \widetilde \gamma \cdot \bigg( \sum_{\ell = 0}^{n_{k}} \big\| \sigma^{k,\ell} \big(t, \beta_t^{k,\ell} \big) \big\|^2 \bigg)_{k = 0, k \neq j}^{m}.
\end{align*}
recalling that $\nu^{-(j,i)}$ and $\nu^{\rm M}$ are respectively fixed by $\P^{-(j,i)}$ and $\P^{\rm M}$, and $\Lambda^{j,i} (t,u) := \sigma^{j,i}(t,b) \cdot \lambda^{j,i}(t,a)$.

\medskip

Then, under fairly general conditions, the value of the agent's problem is given by $V^{j,i}_0 = v^{j,i}(0, X^j_0, \widebar X^{-j}_0)$, where the function $v^{j,i} : [0,T] \times \R^{n_j+1} \times \R^{m-1} \longrightarrow \R$ can be characterised as the unique viscosity solution (with appropriate
growth at infinity) of the following Hamilton--Jacobi--Bellman
(HJB) equation:
\begin{align*}
    "- \partial_t v^{j,i} \big(t, x^j, \widebar x^{-j} \big) - \sup_{u^{j,i} \in U^{j,i} } \widetilde h^{j,i} \big(t, X_t^{j,i}, v^{j,i}, \nabla v^{j,i}, \nabla^2 v^{j,i}, \nu_t^{-(j,i)}, \nu_t^M, u^{j,i} \big) = 0".
\end{align*}
This implies, for $s = 0$ and $t= T$, and by taking expectation under $\P$,
\begin{align*}
    V_0^{j,i} \big( \P^{j,i}, \P^{\rm M}, \xi^{j,i} \big)
    \geq &\ \E^{\P} \bigg[ \Kc^{j,i,\P}_{0,T} g^{j,i} \big( X_T^{j,i}, \xi^{j,i} \big)
    - \int_0^T \Kc^{j,i,\P}_{0,r}
    c^{j,i} \big(r, X_r^{j,i}, \nu_r^{j,i} \big)
     \drm r  \bigg].
\end{align*}
In particular, the equality in the previous inequality is attained for the maximiser of the Hamiltonian.

\medskip

The same reasoning allows to obtain, still in the Markovian case, the form of the continuation utility $Y^{j,i}$. In particular, by It\=o's formula on $Y^{j,i}$ and using the HJB, we have:
\begin{align*}
    Y_{t}^{j,i}
    = &\ Y_{0}^{j,i} - \int_0^t \sup_{u^{j,i} \in U^{j,i}} \widetilde h^{j,i} \big(r, X_r^{j,i}, v^{j,i}, \nabla v^{j,i}, \nabla^2 v^{j,i}, \nu_r^{-(j,i)}, \nu_r^M, \nu^{j,i}_r \big) \drm r
    + \int_0^t \nabla_{x^j} v^{j,i} \cdot \drm X_r^{j} \\
    &+ \int_0^t \nabla_{\widebar x^{-j}} v^{j,i} \cdot \drm \widebar X_r^{-j}    
    + \dfrac{1}{2} \sum_{\ell = 0}^{n_j} \int_0^t \partial^2_{x^{j,\ell}} v^{j,i} \drm \langle X^{j,\ell} \rangle_r
    + \dfrac{1}{2} \sum_{k = 0, k \neq j}^{m} \int_0^t  \partial^2_{\widebar x^k} v^{j,i} \drm \langle \widebar X^{k} \rangle_r.
\end{align*}
Replacing the quadratic variations and $\widetilde h^{j,i}$ by their values, several simplifications are possible, especially between the terms related to second order derivatives, and we obtain:
\begin{align*}
    Y_{t}^{j,i}
    = &\ Y_{0}^{j,i} - \int_0^t 
    \Big( \sup_{u^{j,i} \in U^{j,i}} h^{j,i} \big( r, X_r^{j,i}, v^{j,i}, \partial_{x^{j,i}} v^{j,i}, \partial^2_{x^{j,i}} v^{j,i}, u^{j,i} \big)
    + H^{j,i} \big(r, (\partial_{x^{j,\ell}} v^{j,i})_{\ell \neq i}, \partial_{\widebar x^{k}} v^{j,i}, \nu^{-(j,i)}, \nu^{\rm M} \big) 
    \Big) \drm r \\
    &+ \int_0^t \nabla_{x^j} v^{j,i} \cdot \drm X_r^{j}
    + \int_0^t \nabla_{\widebar x^{-j}} v^{j,i} \cdot \drm \widebar X_r^{-j}
    + \dfrac{1}{2} \int_0^t \partial^2_{x^{j,i}} v^{j,i} \drm \langle X^{j,i} \rangle_r,
\end{align*}
where $h^{j,i}$ and $H^{j,i}$ are respectively defined by \eqref{eq:agent_hamiltonian_tomax} and \eqref{eq:agent_hamiltonian_1}.

\medskip

On the one hand, the previous reasoning explains the particular form of the $(j,i)$--th agent's Hamiltonian, denoted by $\Hc^{j,i}$ and defined by \eqref{eq:agent_hamiltonian}. In particular, this Hamiltonian is the supremum on the $(j,i)$--th agent's effort of the sum of the two previous terms $h^{j,i}$ and $H^{j,i}$. On the other hand, we also obtain that the $(j,i)$--th agent's continuation utility $Y^{j,i}$ should be parameterised by a triple $\Zc := (Z, \widetilde Z, \Gamma)$, where 
\begin{align*}
    Z := \nabla_x v^{j,i}, \; \widetilde Z := \nabla_{\widebar x} v^{j,i}, \; \text {and} \; \Gamma := \nabla^2_{x^{j,i}} v^{j,i},
\end{align*}
and satisfies for $t \in [0,T]$:
\begin{align*}
    Y_{t}^{j,i}
    =  Y_{0}^{j,i} - \int_0^t \Hc^{j,i} \big( r, X^{j,i}, Y_r^{j,i}, \Zc_r, \widehat \nu^\star \big) \drm r
    + \int_0^t Z_r \cdot \drm X_r^{j} 
    + \int_0^t \widetilde Z_r \cdot \drm \widebar X_r^{-j}
    + \dfrac{1}{2} \int_0^t \Gamma_r \drm \langle X^{j,i} \rangle_r,
\end{align*}
where $\widehat \nu^\star := (\nu^{-(j,i),\star}, \nu^{\rm M})$ and $\Hc^{j,i}$ is defined by \eqref{eq:agent_hamiltonian}. 

\medskip

Therefore, the process $Y^{j,i}$ is the continuation utility of the $(j,i)$--th agent in the Markovian case, and the associated contract $\xi^{j,i}$ should be such that $Y_T^{j,i} = g^{j,i} (X_T^{j,i}, \xi^{j,i})$. However, this form of contract cannot be used directly in the context of a principal--agent problem with moral hazard. Indeed, this form depends explicitly on the efforts of other agents, namely $\nu^{-(j,i)}$, through the Hamiltonian $\Hc^{j,i}$, and these efforts are not supposed to be observable, nor contractible upon, for the manager. Nevertheless, we can overcome this difficulty by replacing $\nu^{-(j,i)}$ by the optimal efforts process of
other agents, which has to be formally computed as the maximiser in the Hamiltonian denoted by $\nu^{-(j,i), \star}$ and defined
by \eqref{eq:hamiltonian_maximiser}. 
Indeed, at equilibrium, each agent should apply their optimal efforts. 
Moreover, one can notice that the Hamiltonian also depends on the effort of the managers. Indeed, for now, the $j$--th manager offers a contract to his agents, given any efforts made by other managers, and thus communicates these efforts to his agents. We will see later, when solving the Nash equilibrium between managers, that the Hamiltonian will naturally also be computed with the optimal efforts of other managers. 

% To summarise, this result gives the intuition for the Hamiltonian defined by \eqref{eq:agent_hamiltonian}, as well as the form of the continuation utility in \Cref{def:contractji}.

\subsection{Extending the dynamics}\label{app:extension_dynamic}

In this subsection, we show the limits when considering each output $X^{j,i}$, for $j \in \{1, \dots, m\}$ and $i \in \{0,\dots, n_j\}$, as a solution to the following SDE
\begin{align}\label{eq:XW_SDE}
    \drm X_t^{j,i} =  \sigma^{j,i} \big(t, X^{j,i}, \beta_t^{j,i} \big) \cdot \Big[ \lambda^{j,i} \big(t, X^{j,i}, \alpha_t^{j,i} \big) \mathrm{d} t +  \mathrm{d}W^{j,i}_t \Big],\; t\in[0,T],\; \P \textnormal{--a.s.}
\end{align}
Following the reasoning developed in \Cref{ss:agent_problem}, the Hamiltonian of the $(j,i)$--th agent is defined by
\begin{align}\label{eq:agent_hamiltonian_SDE}
    \Hc^{j,i} \big(t, x, y, z, \widetilde z, \gamma, \widehat \nu \big) := \sup_{u \in U^{j,i}} h^{j,i} \big(t, x^{j,i}, y, z^i, \gamma, u \big) + H^{j,i} \big( t, x, z^{-i}, \widetilde z, \widehat \nu \big),
\end{align}
for any $\big( t, x, y, z, \widetilde z, \gamma \big) \in [0,T] \times \Cc([0,T], \R^w) \times \R \times \R^{n_j+1} \times \R^{m-1} \times \R$ and $\widehat \nu \in \U_0^{-(j,i)} \times \U_0^{\rm M}$, where
\begin{enumerate}[label=$(\roman*)$]
    \item for $(t, x) \in [0,T] \times \Cc([0,T], \R)$, $(y, z, \gamma) \in \R^3$, $u := (a,b) \in U^{j,i}$,
    \begin{align}\label{eq:agent_h_SDE}
        h^{j,i}(t, x, y, z, \gamma, u) &:= - c^{j,i}(t,x, u) - k^{j,i}(t,x, u)y + \Lambda^{j,i} \big(r, x, u \big) z 
        + \dfrac{1}{2} \big\| \sigma^{j,i}(t,x,b) \big\|^2 \gamma;
    \end{align}
    \item for $(t,x) \in [0,T] \times \Cc([0,T], \R^w)$, $(z, \widetilde z) \in \R^{n_j} \times \R^{m-1}$ and $\widehat \nu := (\nu^{-(j,i)},\nu^{\rm M}) \in \Uc^{-(j,i)} \times \Uc^{\rm M}$,
    \begin{align}\label{eq:agent_hamiltonian_2_SDE}
        H^{j,i} \big( t, x, z, \widetilde z, \widehat \nu \big) := 
        z \cdot \Big( \Lambda^{j,\ell} \big(t, x^{j,\ell}, \widehat \nu_t^{j,\ell} \big) \Big)^{n_j}_{\ell = 0, \, \ell \neq i}
        + \widetilde z \cdot \bigg( \sum_{\ell = 0}^{n_k} \Lambda^{k,\ell} \big(t, x^{k,\ell}, \widehat \nu_t^{k,\ell} \big) \bigg)_{k = 1, \, k \neq j}^{m}.
    \end{align}
\end{enumerate}

First, remark that the $(j,i)$--th agent's Hamiltonian depends on every components of $X$ throughout the drift and volatility functions. Limiting the study to the agent, this is not a problem since we have assumed that the agents observes $X$. However, plugging this Hamiltonian in the contract is not possible, since the $j$--th manager only observes $X^j$ and $\widebar X^{-j}$. Therefore, the manager cannot compute the Hamiltonian part $H^{j,i}$ of his $(j,i)$--th agent. This would lead to an additional assumption, similar to \Cref{ass:zeta_dynamic}, on the shape of the dynamic of $\widebar X^{-j}$ for all $j \in \{1,\dots,m\}$. 
%More precisely, we should assume that for all $j \in \{1,\dots,m\}$, there exist two functions $\sigma^j, \lambda^j$, satisfying appropriate conditions, such that $\widebar X^{-j}$ is solution to the SDE 
% $\drm \widebar X_t^{-j} =  \sigma^{j} (t, \widebar X_t^{-j}, \beta_t^{-j}) \cdot [ \lambda^{j} (t, \widebar X_t^{-j}, \alpha_t^{-j}) \mathrm{d} t +  \mathrm{d}W^{-j}_t]$.

\medskip

Nevertheless, even with this type of assumption, we are faced with a much more serious problem. The part of the Hamiltonian optimised by the agent, \textit{i.e.}, the part given by \eqref{eq:agent_h_SDE}, depends on the output of the agent. Thus, his optimal effort will a priori be a functional of his output. Considering only the relation between the manager and his agent, it is not an issue since both observes the output. However, moving to the problem of another team, since the contract of an agent is written with his Hamiltonian on the optimal efforts of others, it will in fact depend on the output of the agents of another team through their optimal efforts. Since these outputs are not observed by the manager, we are not allowed to write the contract in this way either.

\subsection{On the reporting of the managers}\label{app:assumption_zeta}

The goal of this appendix is to find interesting cases where \Cref{ass:zeta_dynamic} holds. With this in mind, we fix $j \in \{1, \dots, m\}$ and we recall that $\zeta^j_t := f^j \big(t, X^{j}, \xi^{j \setminus 0} \big)$, for some function $f^j$. Assuming that the function $f^j$ is smooth enough, we can apply It\=o's formula to write explicitly the dynamics of $\zeta^j$ with respect to $X^j$ and $\xi^{j \setminus 0}$:
\begin{align*}
    \drm \zeta^j_t = \partial_t f^j \drm t
    + \nabla_x f^j \cdot \drm X^j_t
    + \nabla_y f^j \cdot \drm \xi^{j \setminus 0}_t
    + \dfrac12 \mathrm{Tr} \big[ \nabla_x^2 f^j \drm \langle X^j \rangle_t \big]
    + \dfrac12 \mathrm{Tr} \big[ \nabla_y^2 f^j \drm \langle \xi^{j \setminus 0} \rangle_t \big]
    + \mathrm{Tr} \big[ \nabla_{xy}^2 f^j \drm \langle X^j, \xi^{j \setminus 0} \rangle_t \big].
\end{align*}
Nevertheless, to obtain a dynamics of the form \eqref{eq:dyn_zeta_simple} for $\zeta$, it is necessary to develop the previous equation using the dynamic of $X^j$ and $\xi^{j \setminus 0}$ under the optimal effort of the agents.

\medskip

Given a probability $\P^{\rm M}$ and a collection $\xi^{\rm A}$ of contracts for the agents, both chosen by the managers, the Nash equilibrium between the agents is represented by the probability $\P^\star (\P^{\rm M}, \xi^{\rm A})$, which will be denoted by $\P^\star$ for simplicity. Under this probability and for all $j \in \{1, \dots, m\}$ and $i \in \{0, \dots, n_j\}$, we have the following dynamics for $X^{j,i}$, for all $t \in [0,T]$, $\P^\star$--a.s.
\begin{align*}
    \drm X_t^{j,0}= \sigma^{j,0} \big(t, \beta_t^{j,0} \big) \cdot \Big[ \lambda^{j,0} \big(t, \alpha_t^{j,0} \big) \mathrm{d}t +  \mathrm{d}W^{j,0}_t \Big], \textnormal{ and } 
    \drm X_t^{j,i}= \sigma^{j,i} \big(t, \beta_t^{j,i, \star} \big) \cdot \Big[ \lambda^{j,i} \big(t, \alpha_t^{j,i, \star} \big) \mathrm{d}t +  \mathrm{d}W^{j,i}_t \Big],
\end{align*}
for $i \in \{1, \dots, n_j\}$, where
\begin{align*}
    \big( \alpha_t^{j,i,\star}, \beta_t^{j,i, \star} \big) = u^{j,i, \star} \big(t, Y^{j,i}_t, \big( Z_t^{j,i} \big)^i, \Gamma^{j,i}_t \big), \; \drm t \otimes \P^\star\text{--a.s. for all } t \in [0,T],
\end{align*}
while the effort $\nu^{j,0}$ for all $j$ are fixed through $\P^{\rm M}$. 

\medskip

On the other hand, for all $i \in \{1, \dots, n_j\}$, we have $\xi_t^{j,i} = \widebar g^{j,i} (X_{\cdot \wedge t}^{j,i}, Y_t^{j,i})$, for $t \in [0,T]$,
where $Y^{j,i}$ is defined by \eqref{eq:continuation_utility_agent}.  Under the optimal effort of the agents, some parts of the Hamiltonian simplifies with the drift parts of the stochastic integrals, and we obtain in particular that the dynamic of $Y^{j,i}$ is given by:
\begin{align}\label{eq:dyn_Y_simple}
    \drm Y_{t}^{j,i}
    = &\ \big( c^{j,i} + Y^{j,i}_t  k^{j,i} \big) \big(t, X^{j,i}_t, \nu_t^{j,i,\star} \big) \drm t
    + \big(Z^{j,i}_t\big)^0 \sigma^{j,0} \big(t, \beta_t^{j,0} \big) \cdot \mathrm{d}W^{j,0}_t 
    + \sum_{\ell = 1}^{n_j} \big(Z^{j,i}_t\big)^\ell \sigma^{j,\ell} \big(t, \beta_t^{j,\ell, \star} \big) \cdot \mathrm{d} W^{j,\ell}_t \nonumber \\
    &+ \widetilde Z^{j,i}_t \cdot \bigg( \sigma^{k,0} \big(t, \beta_t^{k,0} \big) \cdot \mathrm{d}W^{k,0}_t
    + \sum_{\ell = 1}^{n_k} \sigma^{k,\ell} \big(t, \beta_t^{k,\ell,\star} \big) \cdot \mathrm{d}W^{k, \ell}_t \bigg)_{k=1, \, k \neq j}^m.
\end{align}

More precisely for the reporting, let us keep in mind that the most consistent forms are those considered in the examples detailed in \Cref{sec:sungmodel,sec:extensions}. We can assume that the $j$--th manager reports in continuous time to the principal the sum of all the outcomes of his working team including him, and the sum of compensations paid to the agents under his supervision, \textit{i.e.}, the two--dimensional variable
\begin{align}\label{eq:zeta_profit_cost}
    \zeta_t^j := \bigg( \sum_{i=0}^{n_j} X_t^{j,i}, \sum_{i=1}^{n_j} \xi_t^{j,i} \bigg), \; \text{ for } \; t \in [0,T],
\end{align}
as considered in the example in the \Cref{sss:2statevariables}. Based on this form of reporting, it will be relatively simple to also consider the case the case where the $j$--th manager only reports in continuous time the net benefits of his working team, as in \citeauthor{sung2015pay}'s model developed in \Cref{sec:sungmodel}, \textit{i.e.}, the one--dimensional variable
\begin{align}\label{eq:zeta_net_profit}
    \zeta_t^j := \sum_{i = 0}^{n_j} X_t^{j,i} - \sum_{i=1}^{n_j} \xi_t^{j,i}, \; \text{ for } \; t \in [0,T].
\end{align}
These two potential choices of reporting are based on the results of \Cref{sss:reporting_denegerate}, which shows that a more accurate reporting of agents' results leads to a degeneracy of the principal's problem into the direct contracting case. 

\medskip

Looking at the two reporting choices mentioned above, it is clear that if $\zeta$ has an independent dynamic in the second case, then the same is true in the first case. We thus focus on a reporting $\zeta$ given by \eqref{eq:zeta_net_profit}. In the following, we detail two interesting cases: 
\begin{enumerate}[label=$(\roman*)$]
    \item the linear case, in the sense that $g^{j,i} (x,y) := g_{\rm x}^{j,i} x + y$, for some $g_{\rm x}^{j,i} \in \R$;
    \item the exponential case, which corresponds to $g^{j,i} (x,y) := - \erm^{- R^{j,i} ( x + y)}$ for some $R^{j,i} > 0$, to cover the examples provided in the first two sections.
\end{enumerate}

\subsubsection{Linear case}

We assume in this section that $g^{j,i} (x,y) := g_{\rm x}^{j,i} x + y$. In this case, remark that $\xi^{j,i}_t = Y^{j,i}_t - g_{\rm x}^{j,i} X^{j,i}_t$. Therefore, the reporting $\zeta^j$ given by \eqref{eq:zeta_net_profit} admits the following dynamics:
\begin{align*}
    \drm \zeta_t^j = \drm X_t^{j,0} + \sum_{i = 1}^{n_j} (1+g_{\rm x}^{j,i}) \drm X_t^{j,i} - \sum_{i=1}^{n_j} \drm Y_t^{j,i}, \; \text{ for } \; t \in [0,T].
\end{align*}
There are two problems in attempting to obtain an independent dynamic: 
\begin{enumerate}[label=$(\roman*)$]
    \item the optimal effort of each agent is a function of his continuation utility, and there is no reason that, finally, the dynamics will only make the sum of $Y^{j,i}$ appear;
    \item the drift parts of the continuation utilities are particular functions of the outputs and the continuation utilities themselves, and, in the same way as for the first point, there is no apparent reason to get the sum at the end.
\end{enumerate}
The easiest way to tackle the first problem is to assume that the discount rate $k^{j,i}$ is not controlled, \textit{i.e.}, $k^{j,i} (t, x, u) = k_{\rm x}^{j,i} (t,x)$, recalling that the function $k^{j,i}_x$ is defined by \Cref{ass:separability}. Under this specification, and given the form \eqref{eq:agent_hamiltonian_tomax} of the Hamiltonian part to be maximised by the agent, we can state that the optimal efforts of the agents do not depend on their continuation utilities anymore. It remains to deal with the second problem. Since the drift part of $Y^{j,i}$ can only depends on $X^{j,i}$ and $Y^{j,i}$, by summing for $i=1$ to $n_j$, there is no way we can obtain something that depends on $X^{j,0}$. It is therefore impossible to obtain a function of $\zeta^j$ in the drift. We are thus led to assume that $c^{j,i} (t,x,u)$ is independent of $x$ and that $k^{j,i}$ is in fact equal to $0$. Under these strong assumptions, it is now clear that the dynamics of $\zeta$ is of the desired form, since it does not depend on the outputs and continuation utilities anymore. One can note that if we had considered the reporting form \eqref{eq:zeta_profit_cost} with $g^{j,i}_x = 0$, we could have let $k^{j,i}$ depends on time $t$. Indeed, in this case, by summing the dynamics of the continuation utilities, we obtain a drift depending only on the sum, which corresponds to the second component of $\zeta^j$.

\begin{lemma}\label{lem:linear_case}
Consider the linear case, \textit{i.e.}, when $g^{j,i}(x,y) := g_{\rm x}^{j,i} x + y$ for some $g_{\rm x}^{j,i} \in \R$. If the reporting $\zeta$ is defined by \eqref{eq:zeta_profit_cost}, then \textnormal{\Cref{ass:zeta_dynamic}} is satisfied if, for all $j \in \{1, \dots, m\}$, $i \in \{1, \dots, n_j\}$, $c^{j,i} (t,x,u) = c_{u}^{j,i} (t,u)$ and $k^{j,i}(t,x,u) = k(t)$ for $(t,x,u) \in [0,T] \times \R \times U^{j,i}$. Moreover, if $\zeta$ is defined by \eqref{eq:zeta_net_profit}, then we have to assume in addition that $k(t) = 0$ for all $t \in [0,T]$.
\end{lemma}

\subsubsection{Exponential case}

We assume in this section that $g^{j,i} (x,y) := - \erm^{- R^{j,i} ( x + y)}$, where $R^{j,i}$ is a positive constant representing the risk--aversion of the $(j,i)$--th agent, and $c^{j,i} \equiv 0$, in order to recover the classical exponential utility case. In this case, we remark that
\begin{align*}
    \xi^{j,i}_t = - \dfrac{1}{R^{j,i}} \ln \big( - Y^{j,i}_t \big) -  X^{j,i}_t.
\end{align*}
Therefore, by applying It\=o's formula and using the dynamics of $Y^{j,i}$ given by \eqref{eq:dyn_Y_simple}, we obtain the following dynamics for $\xi^{j,i}$, for all $t \in [0,T]$:
\begin{align*}
    \drm \xi^{j,i}_t 
    % = &- \dfrac{1}{R^{j,i}} \bigg( \dfrac{1}{Y^{j,i}_t} \drm Y^{j,i}_t - \dfrac{1}{2} \dfrac{1}{\big( Y^{j,i}_t \big& -} \drm \langle Y^{j,i} \rangle_t \bigg) - g_{\rm x}^{j,i} \drm X^{j,i}_t \\
    = & - \dfrac{1}{R^{j,i}} k^{j,i} (t, X^{j,i}_t, \nu_t^{j,i,\star}) \drm t 
    + \dfrac{1}{2} R^{j,i} \big| (\widehat Z_t^{j,i})^0 \big|^2 \big\| \sigma^{j,0} (t, \beta_t^{j,0}) \big\|^2 \drm t
    + \dfrac{1}{2} R^{j,i} \sum_{\ell = 1}^{n_j} \big| ( \widehat Z_t^{j,i})^\ell \big|^2 \big\| \sigma^{j,\ell} (t, \beta_t^{j,\ell, \star}) \big\|^2 \drm t  \\
    &+ \dfrac{1}{2} R^{j,i} \sum_{k=1, \, k \neq j}^m \big| ( \widecheck Z_t^{j,i})^k \big|^2 \bigg( \big\| \sigma^{k,0} (t, \beta_t^{k,0}) \big\|^2
    + \sum_{\ell = 1}^{n_k} \big\| \sigma^{k,\ell} (t, \beta_t^{k,\ell,\star} ) \big\|^2 \bigg) \drm t 
    - g_{\rm x}^{j,i} \drm X^{j,i}_t \\
    &+ ( \widehat Z_t^{j,i})^0 \sigma^{j,0} (t, \beta_t^{j,0}) \cdot \mathrm{d}W^{j,0}_t
    + \sum_{\ell = 1}^{n_j} ( \widehat Z_t^{j,i})^\ell \sigma^{j,\ell} (t, \beta_t^{j,\ell, \star}) \cdot \mathrm{d} W^{j,\ell}_t \\
    &+ \widecheck Z_t^{j,i} \cdot \bigg( \sigma^{k,0} (t, \beta_t^{k,0}) \cdot \mathrm{d}W^{k,0}_t
    + \sum_{\ell = 1}^{n_k} \sigma^{k,\ell} (t, \beta_t^{k,\ell,\star}) \cdot \mathrm{d}W^{k, \ell}_t \bigg)_{k=1, \, k \neq j}^m,
\end{align*}
where for all $t \in [0,T]$,
\begin{align*}
    \widehat Z_t^{j,i} := - \dfrac{Z_t^{j,i}}{R^{j,i} Y_t^{j,i}}, \;
    \widecheck Z_t^{j,i} := - \dfrac{\widetilde Z_t^{j,i}}{R^{j,i} Y_t^{j,i}}, \text{ and } \widehat \Gamma_t^{j,i} := - \dfrac{\Gamma_t^{j,i}}{R^{j,i} Y_t^{j,i}}.
\end{align*}
The previous change of variable is classical when considering exponential utilities, and implies that the optimal control is in fact independent of the continuation utility. Indeed, the Hamiltonian's part to maximise, defined by \eqref{eq:agent_hamiltonian_tomax}, is now given by:
\begin{align*}
    h^{j,i} \big(t, X^{j,i}, Y_t^{j,i}, \big( Z_t^{j,i} \big)^i, \Gamma_t^{j,i}, u \big) 
    &= -  Y_t^{j,i} \bigg( k^{j,i}(t,X^{j,i},u) + R^{j,i} \Lambda^{j,i} (t,u) \big( \widehat  Z_t^{j,i} \big)^i + \dfrac{1}{2} \big\| \sigma^{j,i}(t,b) \big\|^2 R^{j,i}  \widehat \Gamma_t^{j,i} \bigg),
\end{align*}
and its maximiser $\nu_t^{j,i,\star}$ is thus independent of $y$.
Therefore, the first issue in the linear case does not arise in the exponential case. However, in order to obtain an independent dynamic for $\zeta$, we are still led to assume that $k^{j,i}$ is in fact independent of $X^{j,i}$. 

\begin{lemma}\label{lem:exponential_case}
Consider the exponential case, \textit{i.e.}, when $g^{j,i} (x,y) := - \erm^{- R^{j,i} ( x + y)}$ for some $R^{j,i} > 0$. If the reporting $\zeta$ is defined by \eqref{eq:zeta_profit_cost} or \eqref{eq:zeta_net_profit}, then \textnormal{\Cref{ass:zeta_dynamic}} is satisfied if, for all $j \in \{1, \dots, m\}$, $i \in \{1, \dots, n_j\}$, $c^{j,i} (t,x,u) = 0$ and $k^{j,i}(t,x,u) = k_{\rm u}^{j,i}(t,u)$ for $(t,x,u) \in [0,T] \times \R \times U^{j,i}$.
\end{lemma}

\begin{remark}
One may note that the assumption on $k^{j,i}$, mainly that it is bounded \textnormal{(}see {\rm \Cref{ass:separability})}, is made to ensure that the agent's Hamiltonian, defined in \eqref{eq:agent_hamiltonian_tomax}, is Lipschitz in $y$ \textnormal{(}the continuation utility\textnormal{)}, and is not necessary if we only consider \textnormal{CARA} utility functions as in this section.
\end{remark}

\subsubsection{Other kinds of reporting}

Finally, we could imagine alternative types of reporting than those mentioned above. For example, we can assume that the $j$--th manager reports to the principal the sum of the outputs and, separately, the sum of the discounted continuation utilities. In this case, if we arbitrary assume that the agent's optimal efforts are independent of their continuation utility, the dynamic of $\zeta^j$ is also independent of $Y^{j\setminus 0}$. Indeed, using \eqref{eq:dyn_Y_simple}, we have:
\begin{align*}
    \drm \big( \Kc^{j,i,\P^\star}_{0,t} Y_{t}^{j,i} \big)
    = \Kc^{j,i,\P^\star}_{0,t} \Bigg( &\ c^{j,i} (t, X^{j,i}_t, \nu_t^{j,i,\star}) \drm t
    + (Z^{j,i}_t)^0 \sigma^{j,0} (t, \beta_t^{j,0}) \cdot \mathrm{d}W^{j,0}_t 
    + \sum_{\ell = 1}^{n_j} (Z^{j,i}_t)^\ell \sigma^{j,\ell} (t, \beta_t^{j,\ell, \star}) \cdot \mathrm{d} W^{j,\ell}_t \nonumber \\
    &+ \widetilde Z^{j,i}_t \cdot \bigg( \sigma^{k,0} (t, \beta_t^{k,0}) \cdot \mathrm{d}W^{k,0}_t
    + \sum_{\ell = 1}^{n_k} \sigma^{k,\ell} (t, \beta_t^{k,\ell,\star}) \cdot \mathrm{d}W^{k, \ell}_t \bigg)_{k=1, \, k \neq j}^m \Bigg),
\end{align*}
for all $i \in \{1,\dots,n_j\}$. Nevertheless, we still have an issue with the dependency in the output $X$.

\medskip

To prevent this issue, we can also imagine that the $j$--th manager reports the following:
\begin{align*}
    \zeta_t^j := \bigg( \sum_{i=0}^{n_j} X_t^{j,i}, \sum_{i=1}^{n_j} \Kc^{j,i,\P^\star}_{0,t} Y_{t}^{j,i} - \int_0^t c^{j,i} (t, X^{j,i}_t, \nu^{j,i,\star}_t ) \drm t \bigg), \; \text{ for } \; t \in [0,T].
\end{align*}
In this case, under the same assumption as below, namely that the agents' optimal efforts are independent of their continuation utility, then the dynamic of $\zeta^j$ is both independent of $X^j$ and $Y^{j\setminus0}$.

\medskip

In short, there appear to be many cases in which \Cref{ass:zeta_dynamic} is satisfied. Unfortunately, it seems complicated to define a general framework with weak assumptions on the form of the reporting $\zeta$ and on the characteristic functions of the agents ensuring that this hypothesis is satisfied, though the assumption itself can easily be checked on a case--by--case basis.

\section{The underlying theory of 2BSDEs}\label{sec:2BSDEs}

The theoretical framework developed throughout this paper strongly relies on the recent theory of 2BSDEs, which is thus presented in this appendix. More precisely, \Cref{ss:additional_notations_2BSDEs} defines some additional notations. 
%\Cref{app:sss:appendix_anotherrepres} properly defines the canonical space for the 2BSDEs. 
Then, \Cref{ss:link_2BSDEs} (resp. \Cref{ss:link_2BSDEs_manager}) clarifies the link between the Nash for the agents (resp. managers) and the theory of 2BSDEs. Finally, \Cref{ss:tech_proof_agents,ss:tech_proof_managers} regroup the proofs of the propositions and theorems established in the paper, respectively for the managers--agents and the principal--managers problems.

\subsection{Additional notations}\label{ss:additional_notations_2BSDEs}

Throughout this section, let $\X := (\Xc_t)_{t \in [0,T]}$ be an arbitrary filtration on $(\Omega, \Fc_T)$, and $\Pk$ be any set of probability measures on $(\Omega,\Fc_T)$. 

\subsubsection{Filtrations}\label{sss:notation_filtration}

We will denote by $\X_+ :=(\Xc^+_t)_{t\in [0,T]}$ the right limit of $\X$, \textit{i.e.}, $\Xc_t^+ := \bigcap_{s > t} \Xc_s$ for all $t \in [0,T)$ and $\Xc^+_T := \Xc_T$. 
For any $\P\in \Pk$, we denote by $\X^\P:=(\Xc_t^\P)_{t\in [0,T]}$ the completed filtration, where for all $t \in [0,T]$, $\Xc_t^\P$ is the completed $\sigma-$field of $\Xc_t$ under $\P$. Denote also by $\X_+^\P$ the right limit of $\X^\P$, so that $\X_+^\P$ satisfies the usual conditions. 
%We introduce the universally completed filtration $\F^U := (\Fc_t^U)_{t \in [0,T]}$, where $\Fc_t^U := \bigcap_{\P \in \M} \Fc_t^\P$. 
In addition, the filtrations $\X^\Pk := (\Xc_t^\Pk)_{t \in [0,T]}$ and $\X^{\Pk +} := (\Xc_t^{\Pk +})_{t \in [0,T]}$ are defined as follows:
\begin{align*}
    \Xc_t^\Pk := \bigcap_{\P \in \Pk} \Xc_t^\P, \text{ for } t \in [0,T], \; \Xc_t^{\Pk +} := \Xc_{t+}^{\Pk}, \text{ for } t \in [0,T), \text{ and } \Xc_T^{\Pk +} := \Xc_T^{\Pk}.
\end{align*}
Finally, we will use the following notation:
\begin{align}\label{def:Pc_bar}
\Pc (t, \P, \X) :=
 \big\{ \P^\prime \in \Pk \text{ s.t. } \P[E] = \P^\prime[E]~\mbox{for all}~E \in \Xc^+_t \big\}, \; \text{for any} \; (\P,t)\in \Pk \times[0,T].
\end{align}

\subsubsection{Canonical spaces and norms}

Let $\Sigma$ be an $\S^\ell$--valued process, and $p>1$. To properly define the solution of a 2BSDE in our framework, we will have to consider the following spaces with their associated norms:
\begin{enumerate}[label=$(\roman*)$]
    \item $\H_{\ell}^{p}(\X, \Pk, \Sigma)$ the space of $\X^{\Pk}$--progressively measurable $\R^{\ell}$--valued processes $Z$, satisfying:
\begin{align*}
\|Z\|^p_{\H_{\ell}^{p}(\X, \Pk, \Sigma)}
 :=
 \sup_{\P \in \Pk} \E^{\P} \bigg[ \bigg( \int_0^T Z_t^\top \Sigma_t Z_t \drm t\bigg)^{p/2} \bigg] < + \infty;
\end{align*}
    \item $\D^{p}(\X, \Pk)$ the space of $\X^{\Pk +}$--optional $\R$--valued càdlàg processes $Y$, satisfying:
\begin{align*}
 \|Y\|_{\D^{p}(\X, \Pk)}^p :=
 \sup_{\P\in \Pk} \E^\P\bigg[\sup_{0\leq t\leq T}|Y_t|^p\bigg] < + \infty;
\end{align*}
\item $\I^p (\X, \Pk)$ the space of all $\X^{\Pk +}$--optional càdlàg and non--decreasing processes $K$, satisfying $K_0 = 0$, and
\begin{align*}
 \|K \|_{\I^{p}(\X, \Pk)}^p :=
 \sup_{\P\in \Pk} \E^\P \big[ K_T^p \big] < + \infty;
\end{align*}
\end{enumerate}

\subsection{2BSDE representation for an agent}\label{ss:link_2BSDEs}

This section provides a slight adaptation of the 2BSDE theory needed to study and solve the agents' problem.

\subsubsection{Another representation for the set of measures}\label{app:sss:appendix_anotherrepres}

Recall the set of probability measures $\Pc$, specified by \Cref{def:weak_formulation}. The general approach to moral hazard problems by \citeayn{cvitanic2018dynamic} requires to distinguish between the efforts of the agents which give rise to absolutely continuous probability measures in $\Pc$, namely the ones for which only the drift changes, or for which the volatility control changes, while keeping fixed the quadratic variation of $X$. The goal of this subsection is to provide the appropriate formulation in our setting. 

\medskip

For simplicity, we denote by $\Sigma^2(t, b)$ the following diagonal matrix:
\begin{align*}
    \Sigma^2(t, b) := \Sigma(t,b)^\top \Sigma(t,b) = \textnormal{diag} \Big[ \big( \big\| \sigma^{j,i} (t, b^{j,i} ) \big\|^2 \big)_{j,i} \Big], \; \textnormal{ for } t \in [0,T] \textnormal{ and } b \in B,
\end{align*}
recalling that $\Sigma (t,b) \in \M^{dw,w}$ is defined by \eqref{eq:def_vol_X}.
% \begin{align*}
%     \Sigma(t,b) := \bigoplus_{j =1}^m \bigoplus_{i = 0}^{n_j} \sigma^{j,i} (t, b^{j,i})
%     := 
%     \begin{pmatrix}
%         \sigma^{1,0}(t,b^{1,0}) & \textbf{0}_d &  \cdots & \textbf{0}_d \\
%         \textbf{0}_d & \sigma^{1,1} (t,b^{1,1}) & \ddots & \vdots \\
%         \vdots & \ddots & \ddots & \textbf{0}_d \\
%         \textbf{0}_d & \cdots & \textbf{0}_d & \sigma^{m, n_m} (t,b^{m,n_m})
%     \end{pmatrix}.
% \end{align*}

\begin{definition}\label{def:widebar_Pc}
We define by $\widebar \Pc$ the set of probability measures $\widebar \P$ on $(\Omega, \Fc_T)$ such that
\begin{enumerate}[label=$(\roman*)$]
    \item the canonical vector process $(X,W)^\top$ is an $(\F, \widebar \P)$--local martingale for which there exists an $\F$--predictable and $B$--valued process $\beta^{\widebar \P}$ such that the $\widebar \P$--quadratic variation of $(X,W)^\top$ is $\widebar \P$--a.s. equal to
\begin{align*}
    \begin{pmatrix}
        \Sigma^2 \big(t, \beta_t^{\widebar \P} \big)  & \Sigma \big(t, \beta_t^{\widebar \P} \big)^\top  \\
        \Sigma \big(t, \beta_t^{\widebar \P} \big) & \mathrm{I}_{w d} \\
    \end{pmatrix},\; t \in[0,T];
\end{align*}
 %   \item $\widebar\P \circ (X_0)^{-1} = \varrho$, and there exists a measure $\iota$ on $\mathbb R^d\times\mathbb R$ such that $\widebar \P \circ \big((W_0,W_0^{\circ})\big)^{-1} = \iota$ ;
    \item $\widebar\P\big[\Pi \in \U_0]=1$.
\end{enumerate}
\end{definition}

Similarly to \Cref{lem:no_enlarge}, we know that for all $\widebar \P \in \widebar \Pc$, we have the following representation for $X$:
\begin{align*}
    X_t = x_0 + \int_0^t \Sigma \big(s,\beta^{\widebar \P}_s \big)^\top \drm W_s, ~t\in[0,T], \; \widebar\P-\textnormal{a.s.}.
\end{align*}
More precisely, for any $j \in \{1, \dots, m\}$ and $i \in \{0, \dots, n_j\}$,
\begin{align*}
    X^{j,i}_t &= x^{j,i}_0 + \int_0^t \sigma^{j,i} \big(s, \beta^{\widebar \P,j,i}_s) \cdot \drm W^{j,i}_s, ~t\in[0,T], \; \widebar\P-\textnormal{a.s.}.
\end{align*}
Recall that using classical results of \citeayn{bichteler1981stochastic} or \citeayn[Proposition 6.6]{neufeld2014measurability}, we can define a pathwise version of the $\F$--predictable quadratic variation $\langle X \rangle$,
%and $\langle X,W \rangle$, both being $\F$--predictable, 
allowing us to define the $w \times w$ non--negative symmetric matrix $\widehat \sigma_t$ for all $t \in [0,T]$ such that
\begin{align*}%\label{eq:sigmat2}
\widehat \sigma_t^2 := \underset{n \rightarrow +\infty}{\mathrm{limsup}}\; n \big(\langle X \rangle_t-\langle X \rangle_{t-1/n}\big).
\end{align*}
Since $\widehat \sigma_t^2$ takes values in $\S^{w}$, we can naturally define its square root $\widehat \sigma_t$. In particular, we will denote for all $j \in \{1, \dots, m\}$ and $i \in \{0, \dots, n_j \}$, the process $S^{j,i}$, taking values in $\R$, defined as follows:
\begin{align}\label{eq:Stji}
    S_t^{j,i} := \underset{n \rightarrow +\infty}{\mathrm{limsup}}\; n \big(\langle X^{j,i} \rangle_t-\langle X^{j,i} \rangle_{t-1/n}\big), \; \text{for all} \; t\in[0,T].
\end{align}

\begin{definition}\label{def:P_nu_girsanov}
Let $\widebar \P \in \widebar \Pc$ and consider the process $\beta^{\widebar \P}$ associated to $\widebar \P$ in the sense of \textnormal{\Cref{def:widebar_Pc} $(i)$}. For any $A \times B$--valued and $\F$--predictable processes $\nu := (\alpha,\beta)$\footnote{Strictly speaking, the process $\beta$ should be indexed by the measure $\widebar \P$, but we chose to not do so in order to alleviate notations.} such that, for all $t \in [0,T]$, $\Sigma^2 (t, \beta) = \Sigma^2 \big(t, \beta^{\widebar\P} \big)$ $\widebar\P$--a.s., 
we define the equivalent measures $\widebar \P^{\nu}$ by their Radon--Nikodym density on $\Fc_T$,
\begin{align*}
    \frac{\drm \widebar \P^\nu}{\drm \widebar \P} :
    % = &\ \exp \bigg( \sum_{j=1}^{m} \sum_{i=0}^{n_j} \int_0^T \lambda^{j,i} \big(s, \alpha^{j,i}_s \big) \cdot \drm W^{j,i}_s
    % - \frac12 \sum_{j=1}^{m} \sum_{i=0}^{n_j} \int_0^T \big\| \lambda^{j,i} \big(s, \alpha^{j,i}_s \big) \big\|^2 \drm s \bigg) \\
    = &\ \exp \bigg( \int_0^T \lambda (s, \alpha_s) \cdot \drm W_s
    - \frac12 \int_0^T \big\| \lambda (s, \alpha_s) \big\|^2 \drm s \bigg),
\end{align*}
where $\lambda$ is defined as the column vector composed of all the functions $\lambda^{j,i} : [0,T] \times A \longmapsto \R^{d}$, implying that $\lambda$ takes values in $\R^{dw}$.
\end{definition}
% \begin{align*}
% \frac{\mathrm{d}\widebar\P^{\alpha,\beta}}{\mathrm{d}\widebar\P}:=\exp\bigg(-\int_0^T\frac{\alpha_s\cdot\mathbf{1}_d}{\Sigma(\beta_s)}\sigma\big(\beta_s\big)\cdot\mathrm{d}W_s-\frac12\int_0^T(\alpha_s\cdot\mathbf{1}_d)^2\mathrm{d}s\bigg).
% \end{align*}
\noindent Notice that such a measure is well--defined since, for all $j \in \{1,\dots,m\}$ and $i \in \{0,\dots,n_j\}$, $\lambda^{j,i}$ is bounded. It is then immediate to check that the set $\Pc$ coincides exactly with the set of all probability measures of the form $\widebar \P^{\nu}$, which satisfy in addition that there exists $w_0 \in \R^{dw}$ such that $\widebar \P^{\nu} \circ (X_0, W_0)^{-1} = \delta_{(x_0,w_0)}$. For any $\widebar \P\in\widebar\Pc$, we denote by $\widebar \Uc (\widebar\P)$ the set of controls $\nu \in \Uc$ such that $\widebar \P^{\nu}\in\Pc$.

\medskip

Following the reasoning developed in \Cref{ss:hierarchy_theoretical}, it is necessary to characterise the space and the actions of other workers to properly define the admissible response of a considered agent. In particular, this leads to the definitions of $\Pc^{-(j,i)}$ and $\Pc^{\rm M}$ in \Cref{sss:nash_theoretical}, in addition to the definition of $\Pc$ on the whole canonical space in \Cref{ss:theoretical_formulation}. We are therefore led to consider the sets $\widebar \Pc^{-(j,i)}$ and $\widebar \Pc^{\rm M}$ corresponding respectively to $\Pc^{-(j,i)}$ and $\Pc^{\rm M}$, in the same way that we just constructed $\widebar \Pc$ corresponding to $\Pc$ in \Cref{def:P_nu_girsanov}. Similarly to \Cref{def:admissible_agent}, we can then define the set $\widebar \Pc^{j,i} (\widebar \P^{-(j,i)}, \widebar \P^{\rm M})$ of admissible responses $\widebar \P \in \widebar \Pc$ of the $(j,i)$--th agent to some probabilities $\widebar \P^{-(j,i)} \in \widebar \Pc^{-(j,i)}$ and $\widebar \P^{\rm M} \in \widebar \Pc^{\rm M}$ respectively chosen by the other agents and the managers.

\subsubsection{Semilinear Hamiltonian}\label{app:sss:semilinear_hamiltonian}

In the following, in order to focus on the $(j,i)$--th agent, let us consider $j \in \{1, \dots, m\}$ and $i \in \{1, \dots, n_j\}$. We also fix the probabilities $\widebar \P^{-(j,i)} \in \widebar \Pc^{-(j,i)}$ and $\widebar \P^{\rm M} \in \widebar \Pc^{\rm M}$, and consider the associated efforts $\widehat \nu := (\nu^{-(j,i)}, \nu^{\rm M}) \in \Uc^{-(j,i)} \times \Uc^{\rm M}$ of the other workers. In order to lighten the notations, we will consider $\Pk := \widebar \Pc^{j,i} (\widebar \P^{-(j,i)}, \widebar \P^{\rm M})$.

\medskip

For any $t \in [0,T]$, we denote by $\Sc_t^{j,i}$ the image of $B^{j,i}$ by the map $b\in B^{j,i} \longmapsto \| \sigma^{j,i} (t,b) \|^2 \in \R_+$, \textit{i.e.}, $\Sc_t^{j,i} := \{ \| \sigma^{j,i} (t,b) \|^2, \text{ for } b \in B^{j,i} \}$. Conversely, for any $S \in \Sc_t^{j,i}$, we define $\widetilde U_t^{j,i}(S) := \{ (a,b) \in A^{j,i} \times B^{j,i}, \; \text{s.t.} \; \| \sigma^{j,i} (t, b) \|^2 = S \}$. Thanks to these notations, we can isolate the partial maximisation with respect to the squared diffusion in the Hamiltonian of the $(j,i)$--th agent. Indeed, we can define a map $F^{j,i} : [0,T] \times \Cc([0,T], \R) \times \R \times \R^{n_j+1} \times \R^{m-1} \times \Uc^{-(j,i)} \times \Uc^{\rm M} \times \R_+ \longrightarrow \R$ as follows:
\begin{align}\label{eq:generator_2bsde}
    F^{j,i} \big( t, x, y, z, \widetilde z, \widehat \nu, S \big) :=
    \sup_{u \in \widetilde U^{j,i}_t (S) } \widetilde h^{j,i} (t, x, y, z^i, u)
    + H^{j,i} \big( t, z^{-i}, \widetilde z, \widehat \nu \big),
\end{align}
for all $(t, x, y, z, \widetilde z, \widehat \nu, S) \in [0,T] \times \Cc([0,T], \R) \times \R \times \R^{n_j+1} \times \R^{m-1} \times \Uc^{-(j,i)} \times \Uc^{\rm M} \times \R_+$, where in addition for $z^i \in \R$ and $u \in U^{j,i}$, $\widetilde h^{j,i} (t, x, y, z^i, u) := - c^{j,i}(t,x,u) - k^{j,i}(t,x,u)y + \sigma^{j,i}(t,b) \cdot \lambda^{j,i}(t,a) z^i$.
We thus obtain that the Hamiltonian $\Hc^{j,i}$ of the $(j,i)$--th agent, defined by \eqref{eq:agent_hamiltonian}, satisfies:
\begin{align*}
    \Hc^{j,i} \big(t, x, y, z, \widetilde z, \gamma, \widehat \nu \big) = 
    \sup_{S \in \Sc_t^{j,i}} \bigg\{ F^{j,i} \big( t, x, y, z, \widetilde z, \widehat \nu, S \big) + \dfrac{1}{2} S \gamma \bigg\},
\end{align*}
for all $( t, x, y, z, \widetilde z, \gamma, \widehat \nu) \in [0,T] \times \Cc([0,T], \R) \times \R \times \R^{n_j+1} \times \R^{m-1} \times \R \times \Uc^{-(j,i)} \times \Uc^{\rm M}$.

\subsubsection{Best--reaction functions of an agent}\label{sss:agent_best_reac}

In this subsection, for a given admissible contract $\xi^{j,i} \in \Cc^{j,i}$, in the sense of \Cref{def:contract_agent_admissible}, and a pair of probability measures $(\widebar \P^{-(j,i)}, \widebar \P^{\rm M})$ chosen by the other workers, we wish to relate the best--reaction function $V_0^{j,i}(\widebar \P^{-(j,i)}, \widebar \P^{\rm M}, \xi)$ of the $(j,i)$--th agent to an appropriate 2BSDE. 
%With this in mind, we fix throughout this subsection $j \in \{1, \dots, m\}$ and $i \in \{1, \dots, n_j\}$, as well as the probability measures played by other workers, and, in order to lighten the notations, we consider $\Pk := \widebar \Pc^{j,i} (\widebar \P^{-(j,i)}, \widebar \P^{\rm M})$.

\medskip

We define the process $\widetilde \Sigma^{j}$, taking values in the set of diagonal positive $(n_j + m-1)$--dimensional matrices, by:
\begin{align*}
\widetilde \Sigma^{j}_t := \mathrm{diag} \Big[ \big( S_t^{j,\ell} \big)_{\ell \in \{0, \dots, n_j\}} \big] \oplus \mathrm{diag} \bigg[ \bigg( \sum_{\ell = 1}^{n_k} S^{k,\ell}_t \bigg)_{k \in \{1, \dots, m\} \setminus \{j\} } \bigg],\; t\in[0,T],
\end{align*}
recalling that $S^{k,\ell}$ is defined for all $k \in \{1, \dots, m\}$ and all $\ell \in \{0, \dots, n_k\}$ by \eqref{eq:Stji}. Note that the process $\widetilde \Sigma^{j}$ represents the (pathwise) quadratic variation of the vector $(X^j, \widebar X^{-j})$, which corresponds to the state variables of any agents with manager $j$. 

\medskip

Given an admissible contract $\xi^{j,i} \in \Cc^{j,i}$, we are led to consider the following 2BSDE, indexed by $(j,i)$:
\begin{align}\label{eq:2bsde}
    Y_t = g^{j,i} \big( X^{j,i}, \xi^{j,i} \big) + \int_t^T F^{j,i} \big(s, X^{j,i}, Y_s, Z_s, \widetilde Z_s, \widehat \nu_s, S^{j,i}_s \big)  \drm s 
    - \int_t^T Z_s \cdot \drm X^{j}_s
    - \int_t^T \widetilde Z_s \cdot \drm \widebar X^{-j}_s 
    + \int_t^T \drm K_s,
    \tag{2BSDE $(j,i)$}
\end{align}
where $F^{j,i}$ is defined by \eqref{eq:generator_2bsde}. The following definition adapts the classic notion of 2BSDE to our framework, using the notations defined in \Cref{sss:notation_filtration}. In particular, recall that we consider here $\Pk := \widebar \Pc^{j,i} (\widebar \P^{-(j,i)}, \widebar \P^{\rm M})$, and that $\G^j$ is the natural filtration generated by $X^j$ and $\widebar X^{-j}$.

\begin{definition} \label{def:2BSDE}
We say that $(Y,(Z,\widetilde Z), K)$ is a solution to \textnormal{\ref{eq:2bsde}} if \textnormal{\ref{eq:2bsde}} holds $\Pk$--\textnormal{q.s.}, and if for some $k>1$, 
$Y \in \D^{k}(\G^j, \Pk)$, $(Z, \widetilde Z) \in \H_{n_j+m-1}^{k}( \G^j, \Pk, \widetilde \Sigma^j)$, $K \in \I^k (\G^j, \Pk)$, where $K$ satisfies in addition the following minimality condition
\begin{align*}
    0 = \essinf_{ \widebar{\mathbb{P}}^\prime \in \widebar{\mathcal{P}} (t, \widebar{\mathbb{P}},\G^{j} ) }^{\widebar{\mathbb P}}
    \E^{\P'} \Big[ K_T - K_t \Big|(\Gc^{j}_t)^{\widebar\P+}\Big],
 ~0\leq t\leq T,
 \ \widebar{\P}\textnormal{--a.s. for all } 
 \widebar\P \in \Pk,
 \end{align*}
recalling that $\widebar \Pc (t, \widebar{\mathbb{P}},\G^{j})$ is defined by \eqref{def:Pc_bar}, and $(\G^{j})^{\widebar\P+}$ is the right limit of the completion of $\G^j$ under $\widebar \P$.
\end{definition}

The following result relates the solution to the above 2BSDE to the best--reaction function of the $(j,i)$--th agent. 
\begin{proposition}\label{prop:genbestreac}
Fix $(\widebar \P^{-(j,i)}, \widebar \P^{\rm M}) \in \widebar \Pc^{-(j,i)} \times \widebar \Pc^{\rm M}$, as well as $\xi^{j,i} \in \Cc^{j,i}$. Let $(Y, (Z, \widetilde Z), K)$ be a solution to \textnormal{\ref{eq:2bsde}}.
We have
\begin{align*}
    V_0^{j,i} \big( \widebar \P^{-(j,i)}, \widebar \P^{\rm M}, \xi^{j,i} \big) = \sup_{\widebar \P \in \widebar \Pc^{j,i} (\widebar \P^{-(j,i)}, \widebar \P^{\rm M})} \E^{\widebar \P}[Y_0].
\end{align*}
Conversely, the \textnormal{(}dynamic\textnormal{)} value function $V_t^{j,i} (\widebar \P^{-(j,i)}, \widebar \P^{\rm M}, \xi^{j,i})$ always provides the first component $"Y"$ of a solution to \textnormal{\ref{eq:2bsde}}. Moreover, any optimal effort $\nu^{j,i}$, and the optimal measure $\widebar \P$ must be such that 
\begin{align}\label{eq:characterisation_star_ji}
 K=0,\; \widebar{\P}^{\nu} \textnormal{--a.s.}, \; \text{ and } \; \nu^{j,i} \in 
 \argmax_{u \in \widetilde U^{j,i}_t (S^{j,i}_t) } \widetilde h^{j,i} (t, X^{j,i}, Y_t, Z_t^i, u),\; \widebar{\P}^{\nu} \textnormal{--a.s.},
\end{align}
where $\widebar \P^{\nu}$ is defined from $\widebar \P$ and $\nu$ by {\rm \Cref{def:P_nu_girsanov}}, where $\nu$ results from the collection of all workers' efforts, and is thus composed by the optimal effort $\nu^{j,i}$ of the $(j,i)$--th agent, and by the efforts of other workers $(\nu^{-(j,i)},\nu^{\rm M})$ fixed through the pair of probability measures $(\widebar \P^{-(j,i)}, \widebar \P^{\rm M})$.
\end{proposition}

%\todo[inline]{$\Sigma$ par forcément inversible, faire la même preuve que dans CPT, avec la condition $F(z)-F(z') < Sigma^{1/2}(z-z')$ de Nizar Mete Jinfeng.}
\begin{proof}
The proof is classical and follows the lines of \citeayn[Proof of Propositions 4.5 and 4.6]{cvitanic2018dynamic}. We thus only mention here why the assumptions required to apply the results of \citeayn{possamai2018stochastic} are satisfied in our framework.

\medskip

First of all, recall that $k^{j,i}$, $\sigma^{j,i}$, and $\lambda^{j,i}$ are bounded for all $j \in \{1, \dots, m\}$ and $i \in \{1,\dots, n_j\}$.
As in \cite[Proof of Proposition 4.5]{cvitanic2018dynamic}, it follows from the definition of admissible controls that $F^{j,i}$ satisfies the Lipschitz continuity assumptions required in \cite[Assumption 2.1 $(i)$]{possamai2018stochastic}. Indeed, in the one hand we have, for all $(t, x, y, z) \in [0,T] \times \Cc([0,T], \R) \times \R \times \R$, $S^{j,i} \in \Sc_t^{j,i}$, and $(y', z') \in \R^2$,
\begin{align*}
    \bigg| \sup_{u \in \widetilde U^{j,i}_t (S^{j,i}) } \widetilde h^{j,i} (t, x, y, z, u) - \sup_{u \in \widetilde U^{j,i}_t (S^{j,i}) } \widetilde h^{j,i} (t, x, y', z', u) \bigg| 
    % \\
    % \leq &\ \big| k^{j,i} \big|_{\infty} |y-y'| 
    % + \big| \lambda^{j,i} \big|_{\infty} \bigg| \sup_{u \in \widetilde U^{j,i}_t (S^{j,i}) } \big\{ \mathbf{1}_d \cdot \sigma^{j,i}(t,b)  z \big\} 
    % - \sup_{u \in \widetilde U^{j,i}_t (S^{j,i}) } \big\{  \mathbf{1}_d \cdot \sigma^{j,i}(t,b) z' \big\} \bigg| \\
    \leq \big| k^{j,i} \big|_{\infty} |y-y'| + \big| \lambda^{j,i} \big|_{\infty} \big| S^{j,i} \big|^{1/2} |z-z'|.
\end{align*}
On the other hand, for all $(t, z, \widetilde z, \widehat \nu) \in [0,T] \times \R^{n_j+1} \times \R^{m} \times \Uc^{-(j,i)} \times \Uc^{\rm M}$, and $(z', \widetilde z') \in \R^{n_j+1} \times \R^{m}$, we have
\begin{align*}
    \big| H^{j,i} \big( t, z, \widetilde z, \widehat \nu \big) - H^{j,i} \big( t, z', \widetilde z', \widehat \nu \big) \big|
    \leq \sum_{\ell = 0, \; \ell \neq i}^{n_j} 
    \big| \lambda^{j,\ell} \big|_{\infty}
    \big| S^{j,\ell} \big|^{1/2}
    \big| z^\ell - (z')^\ell \big|
    + \sum_{k = 1, k \neq j}^{m} 
    \bigg| \sum_{\ell = 0}^{n_k} \lambda^{k,\ell} \bigg|_{\infty}
    \bigg| \sum_{\ell = 0}^{n_k} S^{k,\ell} \bigg|^{1/2}
    \big| \widetilde z^k - (\widetilde z')^k \big|.
\end{align*}
Combining the two inequalities, we obtain that $F^{j,i}$ is Lipschitz in $y$ and in $(\widetilde \Sigma^j)^{1/2} (z, \widetilde z)$, as requested in \cite[Assumption 2.1 $(i)$]{possamai2018stochastic}.

\medskip

Moreover, by \Cref{def:contract_agent_admissible} of the set of admissible contracts $\Cc^{j,i}$, the terminal condition $g^{j,i}(X^{j,i}, \xi)$ satisfies \eqref{eq:integrability_contract_agent}. Using in addition the integrability condition \eqref{eq:integ_cond_cost} for $c^{j,i}$, it then follows that the terminal condition $g^{j,i}(X^{j,i}, \xi)$ and $F^{j,i}$ satisfy the integrability properties in \cite[Assumption 1.1 $(ii)$]{possamai2018stochastic}.
%, using the fact that the densities from probabilities in $\Pc$ to probabilities in $\widebar\Pc$ have moments of any order, uniformly on the measures.
In addition, \cite[Assumption 3.1]{possamai2018stochastic} is also satisfied thanks to the integrability condition \eqref{eq:integ_cond_cost} for $c^{j,i}$, as explained in \cite[Proof of Proposition 4.5]{cvitanic2018dynamic}. Next, \cite[Assumption 1.1 $(iii)$--$(v)$]{possamai2018stochastic} are also satisfied by the set of measures $\Pk$, see for instance \cite{nutz2013constructing}. Finally, the set $\Pk$ is saturated in the sense of \cite[Definition 5.1]{possamai2018stochastic}, see 
\cite[Remark 5.1]{possamai2018stochastic}.
\end{proof}

\subsubsection{Characterisation of the Nash equilibrium between agents}\label{sss:nash_agents}

With \Cref{prop:genbestreac} in hand, we can now characterise a Nash equilibria between the agents, thanks to a collection of decoupled 2BSDEs, reminiscent of the multidimensional BSDE obtained in the setting of \citeayn{elie2019contracting} where only the drift of the canonical process was controlled.

\begin{theorem}\label{th:genbestreac}
Let $\widebar \P^{\rm M} \in \widebar \Pc^{\rm M}$, as well as a collection $\xi^{\rm A} \in \Cc^{\rm A}$ of contracts for the agents. A probability measure $\P^\star$ belongs to $\Pc^{\rm A,\star}(\widebar \P^{\rm M},\xi^{\rm A})$, in the sense of \textnormal{\Cref{def:nash_agent}}, if and only if $\P^\star=\widebar \P^{\nu^\star}$ where $\nu^{\star}:=(\alpha^{\P^\star},\beta^{\P^\star})$ is such that for any $j \in \{1, \dots, m\}$ and any $i \in \{1,\dots, n_j\}$, 
 \[
 K^{j,i}=0,\; \widebar{\P}^{\nu^\star} \textnormal{--a.s.}, \; \text{ and } \; \nu^{j,i,\star} =
 \argmax_{u \in \widetilde U^{j,i}_t (S^{j,i}_t) } \widetilde h^{j,i} \big(t, X^{j,i}, Y^{j,i}_t, (Z^{j,i}_t)^i, u \big),\; \widebar{\P}^{\nu^{\star}}\textnormal{--a.s.},
 \]
where $(Y^{j,i},(Z^{j,i},\widetilde Z^{j,i}), K^{j,i})$ is a solution to \textnormal{\ref{eq:2bsde}}, in the sense of \textnormal{\Cref{def:2BSDE}}.
\end{theorem} 

\begin{proof}
As in the statement of the theorem, we fix $\widebar \P^{\rm M} \in \widebar \Pc^{\rm M}$, as well as a collection $\xi^{\rm A} \in \Cc^{\rm A}$ of contracts for the agents. Recall that by definition of the set $\Cc^{\rm A}$, the collection of contracts $\xi^{\rm A}$ leads to a unique Nash equilibrium between the agents (see \Cref{def:contract_manager_admissible}).
By \Cref{prop:genbestreac}, we have a characterisation of the best--reaction function of the $(j,i)$--th agent to an arbitrary pair of probability measures $(\widebar \P^{-(j,i)}, \widebar \P^{\rm M})$ chosen by the managers and the other agents. A Nash equilibrium $\P^\star$ then necessitates only that for each $j \in \{1, \dots,m\}$ and $i \in \{1, \dots, n_j \}$, $\P^\star$ is the best--reaction function of the $(j,i)$--th agent to $(\widebar \P^{-(j,i)}, \widebar \P^{\rm M})$, that are respectively defined as the restrictions of $\P^\star$ to $\Omega^{-(j,i)}$ and $\Omega^{\rm M}$. In other words, the probability $\P^\star$ and the associated effort $\nu^\star = (\nu^{j,i,\star})_{j,i}$ have to satisfy \eqref{eq:characterisation_star_ji} for all $j \in \{1, \dots, m\}$ and $i \in \{1,\dots, n_j\}$, which is exactly what is written in the statement of the theorem. 
\end{proof}

\subsection{Technical proofs for the agents' problem}\label{ss:tech_proof_agents}

Thanks to the results established in the previous subsection, we now have everything in hand to prove \Cref{prop:nash_agent} and \Cref{thm:main_manager}.

\begin{proof}[\Cref{prop:nash_agent}]
As in the statement of the proposition, we fix $\P^{\rm M}$ a probability chosen by the managers. For all $j \in \{1, \dots, m\}$ and $i \in \{0, \dots, n_j\}$, let $Y^{j,i}_0 \in \R$ and $\Zc^{j,i} \in \Vc^{j,i}$, and consider the continuation utility $Y^{j,i}$ as well as the associated contract $\xi^{j,i}$ defined through \Cref{def:contractji}. Note that each $\xi^{j,i}$ naturally satisfies the properties in order to be admissible in the sense of \Cref{def:contract_agent_admissible}, and that it suffices to prove uniqueness of the Nash equilibrium to ensure that $\xi^{\rm A} \in \Cc^{\rm A}$. We will first show that the equilibrium suggested in the proposition is indeed a Nash equilibrium, and then that it is unique.

\medskip

\textcursive{Existence of a Nash equilibrium.}\vspace{-0.4em} We first fix $j \in \{1, \dots, m\}$ and $i \in \{0, \dots, n_j\}$ to focus on the $(j,i)$--th agent. We assume for now that other agents are playing according to $\nu^{-(j,i),\star}$, \textit{i.e.}, $\P^{-(j,i)} = \P^{-(j,i),\star}$, and resume by $\widehat \nu^\star := (\nu^{-(j,i),\star}, \nu^{\rm M}) \in  \Uc^{-(j,i)} \times \Uc^{M}$ the effort of other workers. We look for the best admissible response of the $(j,i)$--th agent to this probability $\P^{-(j,i),\star}$, and to the probability $\P^{\rm M}$ chosen by the managers, \textit{i.e.}, a probability $\P \in \Pc^{j,i} (\P^{-(j,i), \star}, \P^{\rm M})$ that maximises his utility. First, by assumption, the contract $\xi^{j,i}$ belongs to the set $\Xi^{j,i}$, thus the continuation utility $Y^{j,i}$ satisfies the formula in \Cref{def:contractji}, for $Y_0^{j,i} \in \R$ and $\Zc^{j,i} = (Z^{j,i}, \widetilde Z^{j,i}, \Gamma^{j,i}) \in \Vc^{j,i}$, \textit{i.e.}, for all $t \in [0,T]$,
\begin{align*}
    Y_{t}^{j,i}
    = &\ Y_{0}^{j,i} - \int_0^t \Hc^{j,i} \big( r, X^{j,i}, Y_r^{j,i}, \Zc^{j,i}_r, \widehat \nu_r \big) \drm r
    + \int_0^t Z^{j,i}_r \cdot \drm X_r^{j} 
    + \int_0^t \widetilde Z^{j,i}_r \cdot \drm \widebar X_r^{-j}
    + \dfrac{1}{2} \int_0^t \Gamma^{j,i}_r \drm \langle X^{j,i} \rangle_r.
\end{align*}
Define then, for $t \in [0,T]$,
\begin{align*}
    K^{j,i}_t := \int_0^t \bigg( \Hc^{j,i} \big( r, X^{j,i}, Y_r^{j,i}, \Zc^{j,i}_r,  \widehat \nu_r \big)
    - \frac12 \Gamma^{j,i}_r S^{j,i}_r 
    - F^{j,i} \big(r, X^{j,i}, Y^{j,i}_r, Z^{j,i}_r, \widetilde Z^{j,i}_r,  \widehat \nu_r, S^{j,i}_r \big) \bigg) \mathrm{d}r,
\end{align*}
where $F^{j,i}$ is given by \eqref{eq:generator_2bsde}. Replacing in the previous form of the continuation utility, we thus obtain:
\begin{align*}
    Y_{t}^{j,i}
    = &\ Y_{0}^{j,i} 
    - \int_0^t F^{j,i} \big(r, X^{j,i}, Y^{j,i}_r, Z^{j,i}_r, \widetilde Z^{j,i}_r,  \widehat \nu_r, S^{j,i}_r \big) \mathrm{d}r
    + \int_0^t Z^{j,i}_r \cdot \drm X_r^{j} 
    + \int_0^t \widetilde Z^{j,i}_r \cdot \drm \widebar X_r^{-j}
    - \int_0^t \drm K^{j,i}_r.
\end{align*}
Finally, recalling the contract satisfies $\xi^{j,i} = \widebar g^{j,i} (X^{j,i}, Y^{j,i}_T)$, we have $Y_T^{j,i} = g^{j,i} (X^{j,i}, \xi^{j,i})$ and we can rewrite the previous equation in a backward form as follows:
\begin{align*}
    Y_{t}^{j,i}  
    = &\ Y_{T}^{j,i}
    + \int_t^T F^{j,i} \big(r, X^{j,i}, Y^{j,i}_r, Z^{j,i}_r, \widetilde Z^{j,i}_r,  \widehat \nu_r, S^{j,i}_r \big) \mathrm{d}r
    - \int_t^T Z^{j,i}_r \cdot \drm X_r^{j} 
    - \int_t^T \widetilde Z^{j,i}_r \cdot \drm \widebar X_r^{-j}
    + \int_t^T \drm K^{j,i}_r,
\end{align*}
which exactly corresponds to \ref{eq:2bsde} under $\Pk := \widebar \Pc^{j,i} (\widebar \P^{-(j,i), \star}, \widebar \P^{\rm M})$. By definition of $F^{j,i}$, we can directly check that $K^{j,i}$ is always a non--decreasing process, which vanishes on the support of any probability measure corresponding to the efforts $\nu^{j,i,\star}$ defined in the statement of the proposition. 
% Indeed,
% \begin{align*}
%     \drm K^{j,i}_s
%     = & \bigg( \sup_{S \in \Sc^{j,i}_s} \bigg\{ F^{j,i} \big(s, X^{j,i}, Y^{j,i}_s,  \big( Z^{j,i}_s \big)^i, S^{j,i}_s \big) + \dfrac{1}{2} \Gamma^{j,i}_s S \bigg\}
%     - F^{j,i} \big(s, X^{j,i}, Y^{j,i}_s,  \big( Z^{j,i}_s \big)^i, S^{j,i}_s \big) 
%     - \frac12 \Gamma^{j,i}_s S^{j,i}_s \bigg) \mathrm{d}s, \; s \in [0,T].
% \end{align*}
To ensure that $(Y^{j,i}, (Z^{j,i}, \widetilde Z^{j,i}), K^{j,i})$ solves \ref{eq:2bsde}, it therefore remains to check that all the integrability requirements in \Cref{def:2BSDE} are satisfied. The one for $Y^{j,i}$ is immediate by definition of the set $\Vc^{j,i}$. The required integrability on $(Z^{j,i}, \widetilde Z^{j,i}, K^{j,i})$ then follows from \citeayn[Theorem 2.1 and Proposition 2.1]{bouchard2018unified}.

\medskip

\textcursive{Uniqueness.}\vspace{-0.1em} We have therefore obtained that the candidate provided in the statement of the proposition was indeed an equilibrium. Let us now prove uniqueness. Let $\nu^{-(j,i)}$ be the arbitrary efforts of other agents, and the associated probability measure $\P^{-(j,i)}$. In this case, the continuation utility of the $(j,i)$--th agent, given a contract $\xi^{j,i} \in \Xi^{j,i}$, does not satisfies \ref{eq:2bsde}, since other agents' efforts are not necessarily optimal anymore. Nevertheless, $\Xi^{j,i} \subset \Cc^{j,i}$ and by \Cref{prop:genbestreac}, the optimal effort $\nu^{j,i,\star}$ is the maximiser of the map $F^{j,i}$, which coincides with the maximiser \eqref{eq:hamiltonian_maximiser} of his Hamiltonian $\Hc^{j,i}$ given by \eqref{eq:agent_hamiltonian}. By \Cref{ass:unicity_maximiser_agents}, this optimal effort is unique, and in particular does not depend on $\nu^{-(j,i)}$ (nor on $\nu^{M}$). To sum up, given a contract in $\Xi^{j,i}$ and for arbitrary efforts $\nu^{-(j,i)}$ of others, the agent $(j,i)$ has a unique optimal effort $\nu^{j,i, \star}$, independent of others actions. We can therefore conclude that the optimal effort of each agent is given by the maximiser of his Hamiltonian. This induces a unique equilibrium in terms of efforts, given by $\nu^\star = (\alpha^\star, \beta^\star)$, and inducing the law $\P^\star$. It is therefore the unique equilibrium.
\end{proof}

% Moreover, since for all $k \in \{1, \dots, m\} \setminus \{j\}$, the efforts $\nu^{k,0} = (\alpha^{k,0},\beta^{k,0})$ are fixed through $\P^{\rm M}$, the dynamic of $X^{k,0}$ is given by:
% \begin{align*}
%     \drm X_t^{k,0}= \sigma^{k,0} \big(t, \beta_t^{k,0} \big) \cdot \Big[ \lambda^{k,0} \big(t, \alpha_t^{k,0} \big) \mathrm{d}t + \mathrm{d}W^{k,0}_t \Big],\; t\in[0,T],\; \P \textnormal{--a.s.}
% \end{align*}
% Recall then that, under $\P^{-(j,i),\star}$, for all $k \in \{1, \dots, m\} \setminus \{j\}$ and all $\ell \in \{1, \dots, n_k\}$,
% \begin{align*}
%     \drm X_t^{k,\ell}= \sigma^{k,\ell} \big(t, \beta_t^{k,\ell,\star} \big) \cdot \Big[ \lambda^{k,\ell} \big(t, \alpha_t^{k,\ell,\star} \big) \mathrm{d}t +  \mathrm{d}W^{k, \ell}_t \Big],\; t\in[0,T],\; \P \textnormal{--a.s.},
% \end{align*}
% which implies, in particular, that the following dynamic for $\widebar X^k$ holds $\P$--a.s. for all $t\in[0,T]$:
% \begin{align*}
%     \drm \widebar X_t^{k} = \sigma^{k,0} \big(t, \beta_t^{k,0} \big) \cdot \Big[ \lambda^{k,0} \big(t, \alpha_t^{k,0} \big) \mathrm{d}t +  \mathrm{d}W^{k,0}_t \Big] 
%     + \sum_{\ell = 1}^{n_k} \sigma^{k,\ell} \big(t, \beta_t^{k,\ell,\star} \big) \cdot \Big[ \lambda^{k,\ell} \big(t, \alpha_t^{k,\ell,\star} \big) \mathrm{d}t +  \mathrm{d}W^{k, \ell}_t \Big],
% \end{align*}

\begin{proof}[\Cref{thm:main_manager}]
Before explaining how to prove the aforementioned result, notice that if we can prove that the restriction to revealing contracts in $\Xi^{j,i}$ is without loss of generality, then the equality \eqref{eq:main_manager} is trivial. Indeed, as mentioned before the theorem and by \Cref{def:contractji}, given a constant $Y_0^{j,i} \in \R$, choosing a contract $\xi^{j,i} \in \Xi^{j,i}$ is strictly equivalent of choosing a triple of payment rates $\Zc^{j,i} := (Z^{j,i}, \widetilde Z^{j,i}, \Gamma^{j,i}) \in \Vc^{j,i}$.

\medskip

The fact that the restriction to revealing contracts in the sense of \Cref{def:contractji} is without loss of generality relies on arguments similar to the ones developed in the aforementioned works \cite{cvitanic2018dynamic,elie2018tale, aid2019optimal, elie2019contracting}. We thus consider an arbitrary collection $\xi^{\rm A} \in \Cc^{\rm A}$ of contracts, such that each agent $(j,i)$ receives the contract $\xi^{j,i} \in \Cc^{j,i}$. Starting from the admissible collection of contracts $\xi^{\rm A}$, the goal is to show that for all agents, we can define an approximation $\xi^\eps$ of his contract $\xi^{j,i}$, leading to the same Nash equilibrium, such that the associated continuation utility $Y^\eps$ has the required dynamics \eqref{eq:continuation_utility_agent}, and moreover, $\xi^\eps = \xi^{j,i}$ at the Nash equilibrium. This will thus ensure that the manager receives the same value when considering revealing contracts in the sense of \Cref{def:contractji} instead of arbitrary admissible contracts.

\medskip

First, using \Cref{prop:genbestreac} and \Cref{th:genbestreac}, we know that for a collection of contracts $\xi^{\rm A} \in \Cc^{\rm A}$ and a probability $\P^{\rm M}$ chosen by the managers, there exists a unique Nash equilibrium $\P^\star \in \Pc^{\rm A,\star} (\P^{\rm M}, \xi^{\rm A})$, associated to an optimal effort $\nu^{\star} \in \Uc$, satisfying for any $j \in \{1, \dots, m\}$ and $i \in \{1,\dots, n_j\}$, 
 \[
 K^{j,i}=0,\; \P^{\star} \textnormal{--a.s.}, \; \text{ and } \; \nu^{j,i,\star} \in 
 \argmax_{u \in \widetilde U^{j,i}_t (S^{j,i}_t) } \widetilde h^{j,i} (t, X^{j,i}, Y^{j,i}_t, \big(Z^{j,i}_t \big)^i, u \big),\; \P^{\star} \textnormal{--a.s.},
 \]
where $(Y^{j,i}, (Z^{j,i}, \widetilde Z^{j,i}), K^{j,i})$ is a solution to \ref{eq:2bsde}, in the sense of \Cref{def:2BSDE}. 

\medskip

Given an arbitrary but admissible collection $\xi^{\rm A} \in \Cc^{\rm A}$ of contracts, the idea is to use the aforementioned solution $(Y^{j,i}, (Z^{j,i}, \widetilde Z^{j,i}), K^{j,i})$ to \ref{eq:2bsde} to construct an approximation of the contract $\xi^{j,i}$. Since the reasoning is similar for all agents, we fix $j \in \{1, \dots, m\}$, $i \in \{1, \dots, n_j\}$, and focus on the approximation of the $(j,i)$--th agent's contract $\xi^{j,i}$. 
The main difference between contracts in $\Cc^{j,i}$ and $\Xi^{j,i}$ comes from whether the process $K^{j,i}$ above is absolutely continuous with respect to Lebesgue measure or not. Since it is not in general, we will approximate it by a sequence of absolutely continuous ones. With this in mind, fix some $\eps>0$, and define the absolutely continuous approximation of $K^{j,i}$:
\[
    K^{\eps}_t :=
    \frac1\eps\int_{(t-\eps)^+}^t K^{j,i}_s \mathrm{d}s,\; t\in[0,T].
\]
Recalling the notation $\Pk := \widebar \Pc^{j,i} (\widebar \P^{-(j,i)}, \widebar \P^{\rm M})$, we have that $ K^{\eps}$ is $(\G^j)^{\Pk}$--predictable, non--decreasing $\Pk$--q.s. and
\begin{align}\label{eq:kk}
    K^{\eps} = 0,
    \; \P^\star \mbox{--a.s. for all} \;
    \P^\star \in \Pc^{\rm A,\star} (\widebar \P^{\rm M}, \xi^{\rm A}).
\end{align}
We next define for any $t\in[0,T]$ the process
\begin{align}\label{Yeps}
    Y^\eps_t :=
    Y^{j,i}_0
    - \int_0^t F^{j,i} \big(s, X^{j,i}, Y^\eps_s, Z^{j,i}_s, \widetilde Z^{j,i}_s, \widehat{\nu}_s^\star, S^{j,i}_s \big)  \drm s 
    + \int_0^t Z^{j,i}_s \cdot \drm X^{j}_s
    + \int_0^t \widetilde Z^{j,i}_s \cdot \drm \widebar X^{-j}_s 
    - \int_0^t \drm K^\eps_s,
\end{align}
where $\widehat \nu^\star := (\nu^{-(j,i),\star}, \nu^{\rm M}) \in  \Uc^{-(j,i)} \times \Uc^{M}$ denotes for the effort of others under $\P^\star \in \Pc^{\rm A,\star} (\widebar \P^{\rm M}, \xi^{\rm A})$, as defined in \Cref{prop:nash_agent}.
We first verify that $(Y^\eps, (Z^{j,i}, \widetilde Z^{j,i}), K^{\eps})$ solves \ref{eq:2bsde}, with terminal condition $Y^\eps_T$ and generator $F^{j,i}$. First, by \eqref{eq:kk}, $K^{\eps}$ clearly satisfies the required minimality condition. Then, noticing that $0 \leq K^\eps \leq K^{j,i}$, $K^\eps$ inherits the integrability of $K^{j,i}$. By the integrability of $Y^{j,i}$, $(Z^{j,i}, \widetilde Z^{j,i})$ and $K^{j,i}$, the stability of solutions to SDEs with Lipschitz generator implies that $\sup_{\widebar\P\in\widebar\Pc}\E^{\widebar\P}\big[| Y_T^\eps |^{p}\big]<\infty$. 
Therefore, $(Y^\eps, (Z^{j,i}, \widetilde Z^{j,i}), K^{\eps})$ is solution to \ref{eq:2bsde} in the sense of \Cref{def:2BSDE}, which implies by \cite[Theorem 4.4]{possamai2018stochastic} the following estimates:
\begin{align}\label{estimates-eps}
    \|Y^\eps\|_{\S^{\bar p} (\G^j, \Pk)} + \big\| \big(Z^{j,i}, \widetilde Z^{j,i} \big) \big\|_{\H^{\bar p}_{n_j + m -1} (\G^j, \Pk, \widetilde \Sigma^j) } < \infty,\;
    \mbox{for} \; \bar p \in(1,p).
\end{align}
We finally observe that a probability measure $\P$ satisfies $K^{j,i}=0$, $\P$--a.s. if and only if it satisfies $K^\eps=0$, $\P$--a.s. An approximation $\xi^\eps$ of the admissible contract $\xi^{j,i}$ can thus be defined as a particular function of the terminal value of $Y^\eps$, more precisely by $\xi^{\varepsilon} := \widebar g^{j,i} (X^{j,i}, Y^\eps_T)$, recalling that $\widebar g^{j,i}$ corresponds to the inverse of $g^{j,i}$ with respect to the second variable. In other words, the approximation $\xi^\eps$ satisfies $Y^\eps_T = g^{j,i} (X^{j,i}, \xi^{\varepsilon})$.

\medskip

To prove that the previously defined contract $\xi^\eps$ is a revealing contract, meaning that it belongs to the set $\Xi^{j,i}$, we should in particular make the parameter $\Gamma$ appear. With this in mind, notice that for any $(t,\omega, x,y,z, \widetilde z) \in [0,T] \times \Omega \times \Cc([0,T],\R) \times \R \times \R^{n_j+1} \times \R^{m-1}$, the map
\begin{align}\label{surjective}
    \gamma \longmapsto  \Hc^{j,i} (t, x, y, z, \widetilde z, \gamma, \widehat \nu^\star)
    - \frac12 \gamma S(\omega) - F^{j,i} (t, x, y, z, \widetilde z, \widehat \nu^\star, S(\omega)),
\end{align}
is surjective on $(0, + \infty)$. Indeed, it is non--negative, by definition of $\Hc^{j,i}$ and $F^{j,i}$, convex, continuous on the interior of its domain, and is coercive by the boundedness of the functions $\lambda^{j,i}$, $\sigma^{j,i}$, $k^{j,i}$ and $c^{j,i}$.
Let $\dot{K}^\eps$ denote the density of the absolutely continuous process $K^\eps$ with respect to the Lebesgue measure. Applying a classical measurable selection argument (the maps appearing here are continuous, and we can use the results from \cite{benes1970existence, benes1971existence}), we may deduce the existence of a $\G^j$--predictable process $\Gamma^\eps$ such that
\[
    \dot{K}_s^{\eps}
    = \Hc^{j,i} \big( r, X^{j,i}, Y_s^{\eps}, Z^{j,i}_s, \widetilde Z^{j,i}_s, \Gamma^\eps_s, \widehat \nu_s^\star \big)
    - \frac12 \Gamma^\eps_s S^{j,i}_s 
    - F^{j,i} \big(r, X^{j,i}, Y^{\eps}_s, Z^{j,i}_s, \widetilde Z^{j,i}_s, \widehat \nu_s^\star, S^{j,i}_s \big), \; s \in [0,T].
\]
Indeed, if $\dot{K}_s^{\eps}>0$, the existence of $\Gamma^\eps_s$ is clear from \eqref{surjective}, and if $\dot{K}_s^{\eps}=0$, $\Gamma^\eps_s$ can be chosen arbitrarily. Substituting in \eqref{Yeps}, it follows that the following representation for $Y^\eps$ holds
\begin{align*}
    Y^\eps_t :=
    &\ Y^{j,i}_0
    - \int_0^t \Hc^{j,i} \big( r, X^{j,i}, Y_s^{\eps}, Z^{j,i}_s, \widetilde Z^{j,i}_s, \Gamma^\eps_s, \widehat \nu_s^\star \big) \drm s
    + \int_0^t Z^{j,i}_s \cdot \drm X^{j}_s
    + \int_0^t \widetilde Z^{j,i}_s \cdot \drm \widebar X^{-j}_s
    + \frac12 \int_0^t  \Gamma^\eps_s \drm \langle X^{j,i} \rangle_s.
\end{align*}

This shows that the continuation utility $Y^\eps$ has the required dynamics \eqref{eq:continuation_utility_agent}. The fact that the contract $\xi^\eps$ induced by $Y^\eps$ belongs to $\Xi^{j,i}$ then stems from \eqref{estimates-eps}. Moreover, notice that the admissible contract $\xi^{j,i}$ and its approximation $\xi^\eps$ coïncides at the equilibrium, in the sense that $\xi^\varepsilon=\xi^{j,i},\; \mathbb P^\star$--a.s. This reasoning is true for all $j \in \{1, \dots, m\}$ and $i \in \{1, \dots , n_j \}$, and we have therefore constructed a well--suited approximation of the collection $\xi^{\rm A}$ of contract belonging to $\Xi^{\rm A}$. Using \Cref{prop:nash_agent,prop:genbestreac}, we can then conclude as in the proof of \cite[Theorem 3.6]{cvitanic2018dynamic}, since both collection of contracts lead to the same unique Nash equilibrium.
\end{proof}

\subsection{2BSDE representation for a manager}\label{ss:link_2BSDEs_manager}

This section provides the slight adaptation of the 2BSDE theory needed to study and solve the managers problem.

\subsubsection{Another representation for the set of measures}\label{app:sss:another_rep_manager}
Following the reasoning developed in \Cref{app:sss:appendix_anotherrepres} for the agents' problem, we need to distinguish between the efforts of the managers which give rise to absolutely continuous probability measures, namely the ones for which only the drift changes, or for which the volatility control changes, while keeping fixed the quadratic variation of $\zeta$. 

\begin{definition}\label{def:widebar_Pc_manager}
We define by $\widebar \Pc^{\rm M}$ the set of probability measures $\widebar \P$ on $(\Omega^{\rm M}, \Fc^{\rm M}_T)$ such that
\begin{enumerate}[label=$(\roman*)$]
    \item the canonical vector process $(\zeta,W)^\top$ is an $(\F^{\rm M}, \widebar \P)$--local martingale for which there exists an $\F^{\rm M}$--predictable and $\Xk$--valued process $\chi^{\widebar \P}$ such that the $\widebar \P$--quadratic variation of $(\zeta,W)^\top$ is $\widebar \P$--a.s. equal to
\begin{align*}
    \begin{pmatrix}
        \Sigma_{\rm M} \Sigma_{\rm M}^\top \big(t,\zeta, \chi_t^{\widebar \P} \big)  & \Sigma_{\rm M} \big(t,\zeta, \chi_t^{\widebar \P} \big)  \\
        \Sigma_{\rm M}^\top \big(t,\zeta, \chi_t^{\widebar \P} \big) & \mathrm{I}_{w d} \\
    \end{pmatrix},\; t \in[0,T],
\end{align*}
recalling that $\Sigma_{\rm M}$ is defined in \textnormal{\Cref{ass:zeta_dynamic}};
 %   \item $\widebar\P \circ (X_0)^{-1} = \varrho$, and there exists a measure $\iota$ on $\mathbb R^d\times\mathbb R$ such that $\widebar \P \circ \big((W_0,W_0^{\circ})\big)^{-1} = \iota$ ;
    \item $\widebar\P\big[\Pi \in \X_0]=1$.
\end{enumerate}
\end{definition}

We thus know that for all $\widebar \P \in \widebar \Pc^{\rm M}$, we have the following representation for $\zeta$:
\begin{align*}
    \zeta_t = \zeta_0 + \int_0^t \Sigma_{\rm M} \big(s,\zeta, \chi^{\widebar \P}_s \big) \drm W_s, \; t\in[0,T], \; \widebar\P-\textnormal{a.s.}
\end{align*}
We can also define a pathwise version of the $\F^{\rm M}$--predictable quadratic variation $\langle \zeta \rangle$, allowing us to define the $hm \times hm$ non--negative symmetric matrix $\widehat \Sigma_t$ for all $t \in [0,T]$ such that
\begin{align}\label{eq:pathwise_qv_zeta}
\widehat \Sigma_t := \underset{n \rightarrow +\infty}{\mathrm{limsup}}\; n \big(\langle \zeta \rangle_t-\langle \zeta \rangle_{t-1/n}\big).
\end{align}
Since $\widehat \Sigma$ takes values in $\S^{hm}$, we can naturally define its square root $\widehat \Sigma_t^{1/2}$.

\begin{definition}\label{def:P_nu_girsanov_manager}
Let $\widebar \P \in \widebar \Pc^{\rm M}$ and consider the process $\chi^{\widebar \P}$ associated to $\widebar \P$ in the sense of \textnormal{\Cref{def:widebar_Pc_manager} $(i)$}. For any $\Xk$--valued and $\F^{\rm M}$--predictable processes $\chi$ such that
\begin{align*}
    \Sigma_{\rm M} \Sigma_{\rm M}^\top \big(t,\zeta, \chi_t \big) = \Sigma_{\rm M} \Sigma_{\rm M}^\top \big(t,\zeta, \chi_t^{\widebar \P} \big), \; t\in[0,T], \; \widebar\P-\textnormal{a.s.},
\end{align*}
we define the equivalent measures $\widebar \P^{\chi}$ by their Radon--Nikodym density on $\Fc^{\rm M}_T$,
\begin{align*}
    \frac{\drm \widebar \P^\chi}{\drm \widebar \P} := 
    \exp \bigg( \int_0^T \Lambda_{\rm M} (s, \zeta, \chi_s) \cdot \drm W_s
    - \frac12 \int_0^T \big\| \Lambda_{\rm M} (s, \zeta, \chi_s) \big\|^2 \drm s \bigg).
\end{align*}
\end{definition}

\noindent Notice that such a measure is well--defined since $\Lambda_{\rm M}$ is bounded. It is then immediate to check that the set $\Pc^{\rm M}$ coincides exactly with the set of all probability measures of the form $\widebar \P^{\chi}$, which satisfy in addition that there exists $\iota \in \R^{wd}$ such that $\widebar \P^{\chi} \circ (\zeta_0, W_0)^{-1} = \delta_{(\varrho,\iota)}$. For any $\widebar \P\in\widebar\Pc^{\rm M}$, we denote by $\widebar{\mathscr{X}} (\widebar\P)$ the set of controls $\chi \in \mathscr X$ such that $\widebar \P^{\chi} \in \Pc^{\rm M}$.

\medskip

Following the reasoning developed in \Cref{ss:manager_problem}, it is necessary to characterise the space and the actions of a considered manager in response of other managers' choices. In particular, this leads to the definition of $\Pc^{j} (\chi^{-j})$ in \Cref{sss:weak_manager}, in addition to the definition of $\Pc$ on the whole canonical space in \Cref{sss:can_space_init_managers}, when actions of other managers are fixed through $\chi^{-j} \in \mathscr X^{-j}$. We are therefore led to consider the set $\widebar \Pc^{j} (\chi^{-j})$ corresponding to $\Pc^{j} (\chi^{-j})$, in the same way that we just constructed $\widebar \Pc^{\rm M}$ corresponding to $\Pc^{\rm M}$ in \Cref{def:P_nu_girsanov_manager}.

\subsubsection{Best--reaction function of a manager}\label{app:sss:manager_best_reac}

In this subsection, for a given admissible contract $\xi^{j} \in \Cc^{j,0}$, in the sense of \Cref{def:contract_manager_admissible}, and the choices of other managers, namely $\chi^{-j}$, we wish to relate the best--reaction function $V_0^{j,0}(\xi^{j,0}, \chi^{-j})$ of the $j$--th manager, defined by \eqref{eq:manager_weak_formulation}, to an appropriate 2BSDE. 
With this in mind, we fix throughout the following $j \in \{1, \dots, m\}$ in order to focus on the $j$--th manager, as well as the effort of other managers summarised by $\chi^{-j} \in \mathscr X^{-j}$. For simplicity, we will denote $\Pk := \widebar \Pc^j (\chi^{-j})$.

\medskip

First, we should adapt the notations defined in \Cref{app:sss:semilinear_hamiltonian} for the agents' problem to the managers' problem, by defining for any $(t,x) \in [0,T] \times \Cc([0,T], \R^{hm})$, 
\begin{align*}
    \Sc_t^{j} (x, \chi^{-j}) &:= \Big\{ \Sigma_{\rm M} \Sigma_{\rm M}^\top (t,x, u \otimes_j \chi_t^{-j}) \in \S^{hm}_+, \text{ for } u \in \Xk^{j} \big\}, \\
    \text{and } \; \widetilde \Xk_t^{j}(x,\chi^{-j}, S) &:= \big\{ u \in \Xk^{j}, \text{ s.t. } \Sigma_{\rm M} \Sigma_{\rm M}^\top (t,x, u \otimes_j \chi_t^{-j}) = S \big\}, \; \text{ for } \; S \in \Sc_t^{j}(x, \chi^{-j}).
\end{align*}

Thanks to these notations, we can isolate the partial maximisation with respect to the squared diffusion in the Hamiltonian of the $j$--th manager. Indeed, we can define a map $F^{j} : [0,T] \times \Cc([0,T], \R^{hm}) \times \R \times \R^{hm} \times \mathscr X^{-j} \times \S^{hm}_+ \longrightarrow \R$ as follows:
\begin{align}\label{eq:generator_2bsde_manager}
    F^{j} \big( t, x, y, z, \chi^{-j}, S \big) &:=
    \sup_{u \in \widetilde \Xk^j_t (x, \chi^{-j}, S)} 
    \widetilde h^{j} ( t, x, y, z, \chi^{-j}, u),
\end{align}
for all $(t, x, y, z, \chi^{-j}, S) \in [0,T] \times \Cc([0,T], \R^{hm}) \times \R \times \R^{hm} \times \mathscr X^{-j} \times \S^{hm}_+$ and where, in addition,
\begin{align*}
    \widetilde h^{j} ( t, x, y, z, \chi^{-j}, u) := &- \big( c^{j,0} + y k^{j,0} \big) (t,x^j,u) + \big( \Sigma_{\rm M} \Lambda_{\rm M} \big) \big( t,x, u \otimes_j \chi^{-j} \big) \cdot z, \;\text{ for } \; u \in \widetilde \Xk_t^{j}(x,\chi^{-j}, S).
\end{align*}
We thus obtain that the Hamiltonian $\Hc^{j}$ of the $j$--th manager, defined by \eqref{eq:hamiltonian_manager}, satisfies:
\begin{align*}
    \Hc^{j} \big(t, x, y, z, \gamma, \chi^{-j} \big) = 
    \sup_{S \in \Sc_t^{j}(x, \chi^{-j}) } \bigg\{ F^{j} \big( t, x, y, z, \chi^{-j}, S \big) + \dfrac{1}{2} {\rm Tr} \big[ S \gamma \big] \bigg\},
\end{align*}
for all $(t, x, y, z, \gamma, \chi^{-j}) \in [0,T] \times \Cc([0,T], \R^{hm}) \times \R \times \R^{hm} \times \M^{hm} \times \mathscr X^{-j}$.

\medskip

Given an admissible contract $\xi^{j} \in \Cc^{j,0}$, we are led to consider the following 2BSDE, indexed by $j$:
\begin{align}\label{eq:2bsde_manager}
    \Yc_t = g^{j,0} \big( \zeta^j, \xi^{j} \big) 
    + \int_t^T F^{j} \big(s, \zeta, \Yc_s, Z_s, \chi^{-j}, \widehat \Sigma_s \big)  \drm s 
    - \int_t^T Z_s \cdot \drm \zeta_s
    + \int_t^T \drm K_s,
    \tag{2BSDE $j$}
\end{align}
where $F^{j}$ is defined by \eqref{eq:generator_2bsde_manager}. The following definition, which echoes \Cref{def:2BSDE}, adapts the classic notion of 2BSDE to our framework, using the notations defined in \Cref{sss:notation_filtration}. In particular, recall that we consider here $\Pk := \widebar \Pc^{j} (\chi^{-j})$, and that $\G$ is the natural filtration generated by $\zeta$.

\begin{definition}\label{def:2BSDE_manager}
We say that $(\Yc,Z,K)$ is a solution to \textnormal{\ref{eq:2bsde_manager}} if \textnormal{\ref{eq:2bsde_manager}} holds $\Pk$--\textnormal{q.s.}, and if for some $k>1$, 
$\Yc \in \D^{k}(\G, \Pk)$, $Z \in \H_{hm}^{k}( \G, \Pk, \widehat \Sigma)$, $K \in \I^k (\G, \Pk)$, where $K$ satisfies in addition the following minimality condition
\begin{align*}
    0 = \essinf_{ \widebar{\mathbb{P}}^\prime \in \widebar{\mathcal{P}} (t, \widebar{\mathbb{P}},\G) }^{\widebar{\mathbb P}}
    \E^{\P'} \Big[ K_T - K_t \Big| \Gc_t^{\widebar\P+}\Big],
 ~0\leq t\leq T,
 \ \widebar{\P}-a.s.\mbox{ for all }
 \widebar\P \in \Pk,
 \end{align*}
recalling that $\widebar \Pc (t, \widebar{\mathbb{P}},\G)$ is defined by \eqref{def:Pc_bar}, and $\G^{\widebar\P+}$ is the right limit of the completion of $\G$ under $\widebar \P$.
\end{definition}

The following result relates the solution to the above 2BSDE to the best--reaction function of the $j$--th manager. 
\begin{proposition}\label{prop:genbestreac_manager}
Fix $\chi^{-j} \in \mathscr X^{-j}$, as well as $\xi^{j} \in \Cc^{j,0}$. Let $(\Yc, Z, K)$ be a solution to \textnormal{\ref{eq:2bsde_manager}}.
We have
\begin{align*}
    V_0^{j,0} \big( \xi^{j}, \chi^{-j} \big) = \sup_{\widebar \P \in \widebar \Pc^{j} (\chi^{-j})} \E^{\widebar \P}[\Yc_0].
\end{align*}
Conversely, the \textnormal{(}dynamic\textnormal{)} value function $V_t^{j,0} (\xi^{j}, \chi^{-j})$ always provides the first component $"\Yc"$ of a solution to \textnormal{\ref{eq:2bsde_manager}}. Moreover, any optimal effort $\widetilde \chi^{j,\star}$, and the optimal measure $\widebar \P$ must be such that 
\begin{align}\label{eq:characterisation_star_ji_manager}
    K=0,\; \widebar{\P}^{\widetilde \chi} \textnormal{--a.s.}, \; \text{ and } \;
    \widetilde \chi_t^{j, \star} \in
    \argmax_{u \in \widetilde \Xk^j_t (\zeta, \chi^{-j}, \widehat \Sigma_t) }
    \widetilde h^{j} ( t, \zeta, \Yc_t, Z_t, \chi_t^{-j}, u), \; t \in [0,T], \; \widebar{\P}^{\widetilde \chi} \textnormal{--a.s.},
\end{align}
where $\widebar \P^{\widetilde \chi}$ is defined from $\widebar \P$ and $\widetilde \chi$ by {\rm \Cref{def:P_nu_girsanov_manager}}, where $\widetilde \chi$ results from the collection of all managers' efforts, and is thus composed by the optimal effort $\widetilde \chi^{j,\star}$ of the $j$--th manager, and by the arbitrary efforts of other workers $\chi^{-j} \in \mathscr X^{-j}$. More precisely, $\chi := \widetilde \chi^{j,\star} \otimes_j \chi^{-j}$.
\end{proposition}

\begin{proof}
As in the proof of \Cref{prop:genbestreac}, it suffices to mention why the assumptions required to apply the results of \cite{possamai2018stochastic} are satisfied within this framework. First of all, recall that by definition of $k^{j,0}$ in \Cref{sss:def_manager_pb}, and $\Sigma_{\rm M}$, $\Lambda_{\rm M}$ in \Cref{ass:zeta_dynamic}, these functions are assumed to be bounded.
As in \cite[Proof of Proposition 4.5]{cvitanic2018dynamic}, it follows that $F^{j}$ satisfies the Lipschitz continuity assumptions required in \cite[Assumption 2.1 $(i)$]{possamai2018stochastic}. 
Indeed, for all $(t, x, y, z) \in [0,T] \times \Cc([0,T], \R^{hm}) \times \R \times \R^{hm}$, $S \in \Sc_t^j (x, \chi^{-j})$, and $(y', z') \in \R \times \R^{hm}$, we have
\begin{align*}
    \big| F^{j} \big( t, x, y, z, \chi^{-j}, S \big) - F^{j} \big( t, x, y', z', \chi^{-j}, S \big) \big|
    \leq |k^{j,0}|_{\infty} |y-y'| + | \Lambda_{\rm M} |_{\infty} \big| S^{1/2} (z-z') \big|,
\end{align*}
ensuring that $F^{j}$ is Lipschitz continuous in $y$ and in $S^{1/2} z$, as requested in \cite[Assumption 2.1 $(i)$]{possamai2018stochastic}.
Moreover, by \Cref{def:contract_manager_admissible} of the set of admissible contracts $\Cc^{j,0}$, the terminal condition $g^{j,0}(\zeta^j, \xi^j)$ satisfies \eqref{eq:integrability_contract_manager}. Using in addition the integrability condition \eqref{eq:integ_cond_cost_manager} for $c^{j,0}$, it then follows that the terminal condition $g^{j,0}(\zeta^j, \xi^j)$ and $F^{j}$ satisfy the integrability properties in \cite[Assumption 1.1 $(ii)$]{possamai2018stochastic}. Indeed, we have, for some $p>1$,
\begin{align*}
    &\ \sup_{\P \in \Pk} \E^{\P} \bigg[ \big| g^{j,0} (\zeta^j, \xi^j) \big|^p\bigg] < + \infty, \; \text{ by \Cref{eq:integrability_contract_manager}} \\
    \text{and } \;
    &\ \sup_{\P \in \Pk} \E^{\P} \bigg[ \int_t^T \big| F^j (s, \zeta, 0, 0, \chi^{-j}, \widehat \Sigma_s) \big|^p \drm s \bigg]
    \leq \sup_{\P \in \Pc} \E^{\P} \bigg[ \int_0^T \big| c^{j,0} (t, \zeta^j, \chi^{j,\P}) \big|^p \drm s \bigg] 
    < + \infty, \; \text{ by \eqref{eq:integ_cond_cost_manager}}, \; \forall t \in [0,T].
\end{align*}
In addition, \cite[Assumption 3.1]{possamai2018stochastic} is also satisfied thanks to the integrability condition \eqref{eq:integ_cond_cost_manager} for $c^{j,0}$, as explained in \cite[Proof of Proposition 4.5]{cvitanic2018dynamic}. Next, \cite[Assumption 1.1 $(iii)$--$(v)$]{possamai2018stochastic} are also satisfied by the set of measures $\Pk := \widebar \Pc^j (\chi^{-j})$, see for instance \cite{nutz2013constructing}. Finally, the set $\Pk$ is saturated in the sense of \cite[Definition 5.1]{possamai2018stochastic}, see 
\cite[Remark 5.1]{possamai2018stochastic}.
\end{proof}

\subsubsection{Characterisation of the Nash equilibrium between managers}\label{sss:nash_manager}

With \Cref{prop:genbestreac_manager} in hand, we can now characterise a Nash equilibria between the managers, thanks to a collection of coupled 2BSDEs, reminiscent of the multidimensional BSDE obtained in the setting of \citeayn{elie2019contracting} where only the drift of the canonical process was controlled.

\begin{theorem}\label{th:genbestreac_manager}
Let $\xi^{\rm M} \in \Cc^{\rm M}$ be the collection of contracts for the managers, meaning that the $j$--th manager receives a contract $\xi^{j} \in \Cc^{j,0}$. The unique Nash equilibrium $\P \in \Pc^{\rm M,\star}(\xi^{\rm M})$, in the sense of \textnormal{\Cref{def:nash_manager_weak}}, is characterised by $\P =\widebar \P^{\widetilde \chi^\star}$ where $\widetilde \chi^\star$ is such that for any $j \in \{1, \dots, m\}$, 
\[
K_t^{j}=0\; 
%\widebar{\P}^{\chi^\star} \textnormal{--a.s.}, \;
\text{ and } \; \widetilde \chi^{j,\star} =
\argmax_{u \in \widetilde \Xk^j_t (\zeta, \widetilde \chi^{-j, \star}, \widehat \Sigma_t) } \widetilde h^{j} ( t, \zeta, \Yc^j_t, Z^j_t, \widetilde \chi_t^{-j, \star}, u), \; \text{ for all } \; t \in [0,T], \; \widebar{\P}^{\widetilde \chi^\star} \textnormal{--a.s.},
\]
where $(\Yc^{j},Z^{j}, K^{j})$ is a solution to \textnormal{\ref{eq:2bsde_manager}}, in the sense of \textnormal{\Cref{def:2BSDE_manager}}, on $\widebar \Pc^j(\widetilde \chi^{-j,\star})$. 
\end{theorem} 

\begin{proof}
As in the statement of the theorem, we fix a collection $\xi^{\rm M} \in \Cc^{\rm M}$ of contracts for the managers. Recall that by definition of the set $\Cc^{\rm M}$, the collection of contracts $\xi^{\rm M}$ leads to a unique Nash equilibrium between the managers (see \Cref{def:contract_manager_admissible}).
By \Cref{prop:genbestreac_manager}, we have a characterisation of the best--reaction function of the $j$--th manager to an arbitrary tuple of controls $\chi^{-j}$ chosen by the other managers. A Nash equilibrium $\P^\star$ associated to an optimal effort $\widetilde \chi^{\star} := (\widetilde \chi^{j,\star})_{j=1}^m$ then necessitates only that for all $j \in \{1, \dots,m\}$, $\P^\star$ is the best--reaction function of the $j$--th managers to $\widetilde \chi^{-j,\star}$. In other words, the probability $\P^\star$ and the associated effort $\widetilde \chi^\star$ have to satisfy \eqref{eq:characterisation_star_ji_manager} for all $j \in \{1, \dots, m\}$, which is exactly what is written in the statement of the theorem.
\end{proof}

The result of the previous theorem leads us to consider, if it exists, a map $\widetilde u$ taking values in $\R^m$, satisfying for all $j \in \{1, \dots,m\}$
\begin{align}\label{eq:u_tilde_j}
\widetilde u^{j} (t, x, y, z, S) =
\argmax_{u \in \widetilde \Xk^j_t (x, \widetilde u^{-j} (t, x, y, z, S), S)} \widetilde h^{j} ( t, x, y^j, z^j, \widetilde u^{j} (t, x, y, z, S), u),
\end{align}
for $(t,x) \in [0,T] \times \Cc([0,T],\R^{hm})$, $y = (y^j)_{j=1}^m \in \R^{m}$, $z = (z^j)_{j=1}^m \in (\R^{hm})^m$ and $S \in \S^{hm}_+$.
Let then define $F^\star : [0,T] \times \Cc([0,T],\R^{hm}) \times \R^{m} \times (\R^{hm})^m \times \S^{hm}_+ \longrightarrow \R^m$ such that each component satisfies
\begin{align*}
    F^{j,\star} (t, x, y, z, S) := F^j (t, x, y^j, z^j, \widetilde u^{j} (t, x, y, z, S)), \; j \in \{1, \dots, m\}.
\end{align*}
We can then consider a triple $(\Yc, Z, K)$, solution of a multidimensional 2BSDE, in the sense that each component $(\Yc^j, Z^j,K^j)$ is solution to the following \textnormal{2BSDE}
\begin{align}\label{eq:2bsde_manager_star}
    \Yc^j_t = g^{j,0} \big( \zeta^j, \xi^{j} \big) 
    + \int_t^T F^{j,\star} \big(s, \zeta, \Yc_s, Z_s, \widehat \Sigma_s \big)  \drm s 
    - \int_t^T Z^j_s \cdot \drm \zeta_s
    + \int_t^T \drm K^j_s, \; \Pc^j(\widetilde \chi^{-j,\star})\textnormal{--q.s.},
\end{align}
where $\widetilde \chi^{\star}$ is defined component by component by $\widetilde \chi_t^{j,\star} := \widetilde u^{j} (t, \zeta, \Yc_t, Z_t, \widehat \Sigma_t)$, for all $t \in [0,T]$ and $j \in \{1, \dots, m\}$.

\medskip

This multidimensional 2BSDE is an extension of the pair of 2BSDEs considered by \citeayn[Equations (3.22--3.23)]{possamai2018zero} in their framework of a zero--sum game with two interacting players. One can also relate this multidimensional 2BSDE to the Mean--Field and Mc--Kean Vlasov 2BSDEs obtained respectively by \citeayn{elie2019mean} and \citeayn{barrasso2020controlled} in a framework with a continuum of agents with Mean--Field interactions.

\subsection{Technical proofs for the managers' problem} \label{ss:tech_proof_managers}

Thanks to the results established in the previous subsection, we now have everything in hand to prove \Cref{prop:nash_manager} and \Cref{thm:main_principal}.

\begin{proof}[\Cref{prop:nash_manager}]

As in the statement of the proposition, let $Y_0^{\rm M} := (Y_0^{j})_{j=1}^m$ and $\Zc := (Z,\Gamma) \in \Vc$. By \Cref{def:contract_manager}, consider the $m$--dimensional process $\Yc^{\rm M} := (\Yc^j)_{j=1}^m$ as well as the associated collection of contracts $\xi^{\rm M} := (\xi^j)_{j=1}^m \in \Xi^{\rm M}$. Note that each $\xi^j$ naturally satisfies the properties in order to be admissible in the sense of \Cref{def:contract_manager_admissible}, and that it suffices to prove uniqueness of the Nash equilibrium to ensure that $\xi^{\rm M} \in \Cc^{\rm M}$.

\medskip

\textcursive{Existence of a Nash equilibrium.}\vspace{-0.2em} We first fix $j \in \{1, \dots, m\}$, and we assume that other managers apart form the $j$--th are playing according to $\chi^{-j,\star}$, defined by the first point of the proposition, \textit{i.e.},
\begin{align}\label{eq:chi_ell_star}
    \chi^{-j,\star} = (\chi^{\ell,\star})_{\ell=1, \, \ell \neq j}^m, \; \text{ where } \; \chi_t^{\ell,\star} = u^{\ell,\star} (t, \zeta, \Yc, Z, \Gamma), \; t \in [0,T], \; \ell \in \{1,\dots, m\} \setminus \{j\}.
\end{align} 
We look for the best admissible response of the $j$--th manager with respect to the effort of others, \textit{i.e.}, a probability $\P \in \Pc^{j} (\chi^{-j, \star})$ that maximises his utility. More precisely, we want to prove that $\chi^{j,\star}$ is also given by \eqref{eq:chi_ell_star}. By assumption on the contract, namely that $\xi^j \in \Xi^j$, we have in particular that the continuation utility $\Yc^j$ of the $j$--th manager satisfies the formula \eqref{eq:continuation_utility_manager}. Recalling the definition of $\Hc^{j,\star}$ in \eqref{eq:hamiltonian_manager_star}, \textit{i.e.}, for $(t, x, y, z, \gamma) \in [0,T] \times \Cc( [0,T], \R^{hm}) \times \R^m \times (\R^{hm})^m \times (\M^{hm})^m$,
\begin{align*}
    \Hc^{j,\star} (t, x, y, z, \gamma) 
    := \Hc^{j} \big( t, x, y^j, z^j, \gamma^j, u^{-j,\star} (t, x, y, z, \gamma ) \big),
\end{align*}
it is easy to see that $\Yc^j$ satisfies the following:
\begin{align*}
    \Yc_{t}^{j}
    =  Y_0^j - \int_0^t \Hc^{j} \big( r, \zeta, \Yc_r^j, Z_r^j, \Gamma_r^j, \chi_r^{-j,\star} \big) \drm r
    + \int_0^t Z^j_r \cdot \drm \zeta_r
    + \dfrac{1}{2} \int_0^t {\rm Tr} \big[ \Gamma^j_r \drm \langle \zeta \rangle_r \big], \; t \in [0,T].
\end{align*}
Recalling that we have denoted by $\widehat \Sigma$ the pathwise version of the quadratic variation $\langle \zeta \rangle$ (see \eqref{eq:pathwise_qv_zeta}), we define
\begin{align*}
    K^{j}_t := \int_0^t \Big( \Hc^{j} \big( r, \zeta, \Yc_r^j, Z_r^j, \Gamma_r^j, \chi_r^{-j,\star} \big)
    - \frac12 {\rm Tr} \big[\Gamma^{j}_r \widehat \Sigma_r \big]
    - F^j \big(r, \zeta, \Yc^j_r, Z^j_r, \chi_r^{-j,\star}, \widehat \Sigma_r \big)
    \Big) \mathrm{d}r,
\end{align*}
for $t \in [0,T]$, where $F^{j}$ is defined by \eqref{eq:generator_2bsde_manager}. Replacing in the previous form of the continuation utility, we obtain:
\begin{align*}
    \Yc_{t}^{j}
    = &\ Y_0^j 
    - \int_0^t F^{j} \big(r, \zeta, \Yc^j_r, Z^j_r, \chi_r^{-j,\star}, \widehat \Sigma_r \big) \mathrm{d}r
    + \int_0^t Z^{j}_r \cdot \drm \zeta_r
    - \int_0^t \drm K^{j}_r.
\end{align*}
Finally, recalling the contract satisfies $\xi^{j} = \widebar g^{j} (\zeta^j, \Yc^{j}_T)$, we have $\Yc_T^{j} = g^{j} (\zeta^j, \xi^{j})$ and we can rewrite the previous equation in a backward form as follows:
\begin{align*}
    \Yc_{t}^{j}  
    = &\ g^{j} (\zeta^j, \xi^{j})
    + \int_t^T F^{j} \big(r, \zeta, \Yc^j_r, Z^j_r, \chi_r^{-j,\star}, \widehat \Sigma_r \big) \mathrm{d}r
    - \int_t^T Z^{j}_r \cdot \drm \zeta_r
    + \int_t^T \drm K^{j}_r,
\end{align*}
which exactly corresponds to \ref{eq:2bsde_manager} under $\Pk := \widebar \Pc^{j} (\chi^{-j, \star})$. By definition of $F^{j}$, we can directly check that $K^{j}$ is always a non--decreasing process, which vanishes on the support of any probability measure corresponding to the efforts $\chi^{j,\star}$ defined in the statement of the proposition. 
To ensure that $(\Yc^{j}, Z^{j}, K^{j})$ solves \ref{eq:2bsde_manager}, it therefore remains to check that all the integrability requirements in \Cref{def:2BSDE_manager} are satisfied. The one for $\Yc^{j}$ is immediate by definition of the set $\Vc$. The required integrability on $(Z^{j}, K^{j})$ then follows from \citeayn[Theorem 2.1 and Proposition 2.1]{bouchard2018unified}.

\medskip

\textcursive{Uniqueness.}\vspace{-0.1em} We have therefore obtained that the candidate provided in the statement of the proposition was indeed an equilibrium. Let us now prove uniqueness. Let $\chi^{-j}$ be the arbitrary efforts of other managers. In this case, the continuation utility of the $j$--th manager, given a contract $\xi^{j} \in \Xi^{j}$, does not satisfies \ref{eq:2bsde_manager}, since other agents' efforts are not necessarily optimal anymore. Nevertheless, $\Xi^{j} \subset \Cc^{j,0}$ and by \Cref{prop:genbestreac_manager}, an optimal effort $\chi^{j,\star}$ is a maximiser of the map $F^{j}$, which coincides with a maximiser of his Hamiltonian $\Hc^{j}$ given by \eqref{eq:hamiltonian_manager}. Recall that the existence of such a maximiser is ensured by \Cref{ass:existence_maximiser_manager}, but uniqueness is not assumed. Nevertheless, this reasoning is valid for all managers, implying that $\chi^{\star}$ should satisfies for all $j \in \{1, \dots, m\}$, $\chi^{j,\star}_t = u^{j} \big(t, \zeta, \Yc^j, Z^j, \Gamma^j, \chi^{-j,\star} \big)$.
This condition is equivalent to the definition of the map $u^\star$ in \Cref{ass:unicity_fixedpoint_hamiltonien}. Moreover, since by \Cref{ass:unicity_fixedpoint_hamiltonien} there is a unique map $u^\star$ guaranteeing the maximization of each manager's Hamiltonian simultaneously, this induces a unique equilibrium in terms of efforts, given by $\chi^\star$.
\end{proof}

\begin{proof}[\Cref{thm:main_principal}]

The main point is to prove that the restriction to revealing contracts in the sense of \Cref{def:contract_manager} is without loss of generality. This proof relies on arguments similar to the ones developed in the proof of \Cref{thm:main_manager} and initially in \cite[Proof of Theorem 3.6]{cvitanic2018dynamic}. We thus consider an arbitrary collection $\xi^{\rm M} \in \Cc^{\rm M}$ of contracts, in the sense that the $j$--th manager receives the contract $\xi^{j} \in \Cc^{j,0}$. Starting from this admissible collection of contracts, the goal is to prove that, for all managers, we can define an approximation $\xi^{j,\eps}$ of his contract $\xi^{j}$, such that the new collection of contracts $\xi^\eps$ is admissible and gives the same Nash equilibrium. Moreover, we should verify that for all $j \in \{1, \dots, m\}$, $\xi^{j,\eps} = \xi^{j}$ at the Nash equilibrium, and that the associated continuation utility $\Yc^{{\rm M}, \eps}$ satisfies the representation \eqref{eq:continuation_utility_manager}, in addition to required integrability conditions, to ensure that $\xi^{\eps}$ is a collection of revealing contracts, in the sense of \Cref{def:contract_manager}.

\medskip

First, using \Cref{prop:genbestreac_manager} and \Cref{th:genbestreac_manager}, we know that for a collection of contracts $\xi^{\rm M} \in \Cc^{\rm M}$, there exists a unique equilibrium $\P^\star \in \Pc^{\rm M,\star} (\xi^{\rm M})$, associated to an optimal effort $\widetilde \chi^{\star} \in \mathscr X$, satisfying for any $j \in \{1, \dots, m\}$, 
\begin{align}\label{eq:charac_chistar}
K^{j}=0,\; \P^{\star} \textnormal{--a.s.}, \; \text{ and } \; \widetilde \chi_t^{j,\star} = 
\argmax_{u \in \widetilde \Xk_t^j (\zeta, \widetilde \chi^{-j,\star}, \widehat \Sigma_t) } \widetilde h^{j} (t, \zeta, \Yc^{j}_t, Z^{j}_t, \widetilde \chi^{-j,\star}_t, u \big),\; \P^{\star} \textnormal{--a.s.},
\end{align}
where $(\Yc^j, Z^j, K^j)$ is a solution to \ref{eq:2bsde_manager}, in the sense of \Cref{def:2BSDE_manager}. 
% We can then consider, thanks to the definition of $\widetilde u$ in \eqref{eq:u_tilde_j}, that the optimal efforts $\widetilde \chi^{\star}$ of the managers satisfy:
% \begin{align*}
%     \widetilde \chi^{\star} = (\widetilde \chi^{\ell,\star})_{\ell=1}^m, \; \text{ where } \; \widetilde \chi_t^{\ell,\star} =  \widetilde u^{\ell} (t, \zeta, \Yc_t, Z_t, \widehat \Sigma_t), \; t \in [0,T], \; \ell \in \{1,\dots, m\}.
% \end{align*}

\medskip

Given this arbitrary but admissible collection $\xi^{\rm M}$ of contracts, the idea is to use the aforementioned solution $(\Yc^j, Z^j, K^j)$ to \ref{eq:2bsde_manager} to construct an approximation $\xi^{j,\eps}$ of the contract $\xi^{j}$. Similar to the agents' problem, let us fix some $\eps>0$, and define, for all $j \in \{1, \dots, m\}$, the absolutely continuous approximation of $K^{j}$:
\[
    K^{j,\eps}_t :=
    \frac1\eps\int_{(t-\eps)^+}^t K^{j}_s \mathrm{d}s,\; t\in[0,T].
\]
The process $K^{j,\eps}$ naturally inherits some properties of $K^{j}$.
More precisely, given the effort of other managers $\chi^{-j}$, and recalling the notation $\Pk := \widebar \Pc^{j} (\chi^{-j})$, we have that $K^{j,\eps}$ is $\G^{\Pk}$--predictable, non--decreasing $\Pk$--q.s. and
\begin{align}\label{eq:kk_manager}
    K^{j,\eps} = 0,
    \; \P^\star \mbox{--a.s. for all} \;
    \P^\star \in \Pc^{\rm M,\star} (\xi^{\rm M}).
\end{align}
We next define the $m$--dimensional process $\Yc^\eps$ such that each component $\Yc^{j,\eps}$ satisfies, for any $t\in[0,T]$, 
\begin{align}\label{Yeps_manager}
    \Yc^{j,\eps}_t :=
    Y_0^j
    - \int_0^t F^{j} \big(s, \zeta, \Yc^{j,\eps}_s, Z^j_s, \widetilde \chi_s^{-j,\star}, \widehat \Sigma_s \big) \drm s 
    + \int_0^t Z^{j}_s \cdot \drm \zeta_s
    - \int_0^t \drm K^{j,\eps}_s, \; \Pc^{j} (\widetilde \chi^{-j,\star})\text{--q.s.},
\end{align}
recalling that the optimal efforts $\widetilde \chi^{-j, \star}$ of other managers are defined defined omega per omega through \eqref{eq:charac_chistar}.
% \begin{align*}
%     \widetilde \chi^{-j,\star} = (\widetilde \chi^{\ell,\star})_{\ell=1, \, \ell \neq j}^m, \; \text{ where } \; \widetilde \chi_t^{\ell,\star} =  \widetilde u^{\ell} (t, \zeta, \Yc^\eps_t, Z_t, \widehat \Sigma_t), \; t \in [0,T], \; \ell \in \{1,\dots, m\} \setminus \{j\}.
% \end{align*} 
% This implies that the previously defined process $\Yc^\eps$ is in fact solution of a coupled multidimensional 2BSDE.

\medskip

We first verify that for all $j \in \{1, \dots, m\}$, $(\Yc^{j,\eps}, Z^{j}, K^{j,\eps})$ solves \ref{eq:2bsde_manager} under $\Pk := \widebar \Pc^{j} (\chi^{-j, \star})$, with terminal condition $\Yc^{j,\eps}_T$ and generator $F^{j}$. With this in mind, let us fix $j \in \{1,\dots,m\}$ as well the other components, namely $(\Yc^{-j, \eps}, Z^{-j}, K^{-j,\eps})$. First, by \eqref{eq:kk_manager}, $K^{j,\eps}$ clearly satisfies the required minimality condition. Then, noticing that $K^{j,\eps} \leq K^{j}$, $K^{j,\eps}$ inherits the integrability of $K^j$, and moreover we can verify that $\sup_{\widebar\P\in\widebar\Pc^{\rm M}} \E^{\widebar\P} [| \Yc^{j,\eps}_T |^{p}] <\infty$, similarly to the equivalent proof for the manager--agents problem.
Therefore, by \cite[Theorem 4.4]{possamai2018stochastic}, we have the following estimates
\begin{align}\label{estimates-eps_manager}
    \|\Yc^{j,\eps}\|_{\S^{\bar p} (\G, \Pk)} + \big\| Z^j \big\|_{\H^{\bar p}_{hm} (\G, \Pk, \widehat \Sigma) } < \infty,\;
    \mbox{for} \; \bar p \in(1,p).
\end{align}
We finally observe that a probability measure $\P$ satisfies $K^{j}=0$, $\P$--a.s. if and only if it satisfies $K^\eps=0$, $\P$--a.s. An approximation $\xi^{j,\eps}$ of the admissible contract $\xi^j$ can thus be defined omega per omega from the terminal value of $\Yc^{j,\eps}$ by $\xi^{j,\eps} := \widebar g^{j,0} (\zeta^j, \Yc^{j, \eps}_T)$, recalling that $\widebar g^{j,0}$ corresponds to the inverse of $g^{j,0}$ with respect to the second variable. In other words, the approximation $\xi^{j,\eps}$ satisfies $\Yc^{j, \eps}_T := g^{j,0} (\zeta^j, \xi^{j,\eps})$.

\medskip

To prove that the previously defined contract $\xi^{j,\eps}$ is a revealing contract, meaning that it belongs to the set $\Xi^{j,0}$, we should in particular make the parameter $\Gamma$ appears in the representation \eqref{Yeps_manager}. With this in mind, notice that for any $(t,\omega, x, y, z) \in [0,T] \times \Omega^{\rm M} \times \Cc([0,T],\R^{hm}) \times \R^{hm} \times \R^{n_j+1}$, the map
\begin{align}\label{surjective_manager}
    \gamma \longmapsto  \Hc^{j} (t, x, y, z, \gamma, \widetilde \chi^{-j,\star} )
    - \frac12 {\rm Tr} \big[ \gamma S(\omega) \big] - F^{j} (t, x, y, z, \widetilde \chi^{-j,\star}, S(\omega)),
\end{align}
is surjective on $(0, + \infty)$. Indeed, it is non--negative, by definition of $\Hc^{j}$ and $F^{j}$, convex, continuous on the interior of its domain, and is coercive by the boundedness of the functions $\Lambda_{\rm M}$, $\Sigma_{\rm M}$, $k^{j,0}$ and $c^{j,0}$.
Let $\dot{K}^{j,\eps}$ denote the density of the absolutely continuous process $K^{j,\eps}$ with respect to the Lebesgue measure. Applying a classical measurable selection argument (the maps appearing here are continuous, and we can use the results from \cite{benes1970existence, benes1971existence}), we can deduce the existence of a $\G$--predictable process $\Gamma^{j,\eps}$ such that
\[
    \dot{K}_s^{j,\eps}
    = \Hc^{j} \big( r, \zeta, \Yc_s^{j,\eps}, Z^{j}_s, \Gamma^{j,\eps}_s, \widetilde \chi^{-j,\star}_s \big)
    - \frac12 {\rm Tr} \big[\Gamma^{j,\eps}_s \widehat \Sigma_s\big] 
    - F^{j} \big(r, \zeta, \Yc^{j,\eps}_s, Z^{j}_s, \widetilde \chi^{-j,\star}_s, \widehat \Sigma_s \big), \; s \in [0,T].
\]
Indeed, if $\dot{K}_s^{j,\eps}>0$, the existence of $\Gamma^{j,\eps}_s$ is clear from \eqref{surjective_manager}, and if $\dot{K}_s^{j,\eps}=0$, $\Gamma^{j,\eps}_s$ can be chosen arbitrarily. Substituting in \eqref{Yeps_manager}, it follows that the following representation for $\Yc^{j,\eps}$ holds
\begin{align*}
    \Yc^{j,\eps}_t :=
    &\ Y_0^j
    - \int_0^t \Hc^{j} \big( r, \zeta, \Yc_s^{j,\eps}, Z^{j}_s, \Gamma^{j,\eps}_s, \widetilde \chi^{-j,\star}_s \big) \drm s
    + \int_0^t Z^{j}_s \cdot \drm \zeta_s
    + \frac12 \int_0^t {\rm Tr} \big[ \Gamma^{j,\eps}_s \drm \langle \zeta \rangle_s \big].
\end{align*}
This shows that the continuation utility $\Yc^\eps$ has the required dynamics \eqref{eq:continuation_utility_manager}, since, at equilibrium, the effort $\widetilde \chi^{\star} = \chi^{\star}$ and is unique. The fact that the contract $\xi^{j,\eps}$ induced by $Y^{j,\eps}$ belongs to $\Xi^{j}$ then stems from \eqref{estimates-eps_manager}. Moreover, notice that the admissible contract $\xi^{j}$ and its approximation $\xi^{j,\eps}$ coincides at the equilibrium, in the sense that $\xi^{j,\eps}=\xi^{j},\; \mathbb P^\star$--a.s. This reasoning is true for all $j \in \{1, \dots, m\}$, and we have therefore constructed a well--suited approximation of the collection $\xi^{\rm M}$ of contracts, belonging to $\Xi^{\rm M}$. Using \Cref{prop:genbestreac_manager,prop:nash_manager}, we can then conclude as in the proof of \cite[Theorem 3.6]{cvitanic2018dynamic} since both collection of contracts lead to the same unique Nash equilibrium.

\medskip

Finally, the equality \eqref{eq:main_principal} is now trivial. Indeed, by \Cref{def:contract_manager}, choosing a collection $\xi^{\rm M} \in \Xi^{\rm M}$ of contracts is strictly equivalent of choosing both a pair of payment rates $\Zc := (Z, \Gamma) \in \Vc$ to index the contract of each manager respectively on $\drm \zeta$ and $\drm \langle \zeta \rangle$, and a constant $Y_0^{\rm M} := (Y^j_0)_{j=1}^m \in \R^m$. Nevertheless, for all $j \in \{1, \dots, m\}$, the constant $Y_0^{j} \in \R$ has to be chosen so that the participation constraint for the $j$--th manager is satisfied. Moreover, the principal also chooses the initial value $Y^{\rm A}_0 \in \R$ of the agents' continuation utility, such that their participation constraints, \textit{i.e.}, \Cref{eq:participation_agent}, are satisfied. Using \Cref{prop:nash_agent,prop:nash_manager} respectively for the agents and the managers, these conditions are satisfied if and only if:
\begin{align*}
    Y_0^{j,i} = V_0^{j,i,\star} (\chi^\star)  \geq \rho^{j,i} \; \text{ and } \;
    Y_0^{j} = V_0^{j,0, \star} (\xi^{\rm M}) \geq \rho^{j,0}, \; \text{ for all } j \in \{1,\dots,m\}, \; i \in \{0, \dots, n_j\},
\end{align*}
recalling that $V^{j,i,\star}$ and $V_0^{j,0, \star}$ are respectively defined by \eqref{eq:V_ji_star} and \eqref{eq:V_ji_star}. This justifies the equality \eqref{eq:main_principal} and ends the proof.
\end{proof}

\end{appendices}

\end{document}